\documentclass[11pt, a4paper]{amsart}
\usepackage{amsmath,amssymb,amsthm,mathtools,wasysym,calc,verbatim,tikz,url,hyperref,mathrsfs,cite,fullpage,bbm}
\usepackage[noabbrev,capitalize,nameinlink]{cleveref}
\usepackage[shortlabels]{enumitem}

\usepackage{comment}

\numberwithin{equation}{section}
\newtheorem{theorem}{Theorem}[section]
\newtheorem{lemma}[theorem]{Lemma}
\newtheorem{corollary}[theorem]{Corollary}

\newtheorem{claim}[theorem]{Claim}
\newtheorem{fact}[theorem]{Fact}
\newtheorem{definition}[theorem]{Definition}

\theoremstyle{definition}

\theoremstyle{remark}
\newtheorem{remark}[theorem]{Remark}
\newtheorem*{remark*}{Remark}

\newcommand{\snorm}[1]{\lVert#1\rVert}
\newcommand{\norm}[1]{\bigg\lVert#1\bigg\rVert}
\newcommand{\sang}[1]{\langle #1 \rangle}

\newcommand{\imod}[1]{~\mathrm{mod}~#1}
\newcommand{\eps}{\varepsilon}

\newcommand{\mb}{\mathbb}

\newcommand{\mbm}{\mathbbm}
\newcommand{\mc}{\mathcal}
\newcommand{\mf}{\mathfrak}
\newcommand{\mr}{\mathrm}

\newcommand{\ol}{\overline}
\newcommand{\on}{\operatorname}

\newcommand{\wh}{\widehat}
\newcommand{\wt}{\widetilde}

\begin{document}

\title{Quasipolynomial bounds on the inverse theorem for the Gowers $U^{s+1}[N]$-norm}

\author[A1]{James Leng}
\address{Department of Mathematics, UCLA, Los Angeles, CA 90095, USA}
\email{jamesleng@math.ucla.edu}

\author[A2]{Ashwin Sah}
\author[A3]{Mehtaab Sawhney}
\address{Department of Mathematics, Massachusetts Institute of Technology, Cambridge, MA 02139, USA}
\email{\{asah,msawhney\}@mit.edu}
\thanks{Leng was supported by NSF Graduate Research Fellowship Grant No.~DGE-2034835. Sah and Sawhney were supported by NSF Graduate Research Fellowship Program DGE-2141064.}

\begin{abstract}
We prove quasipolynomial bounds on the inverse theorem for the Gowers $U^{s+1}[N]$-norm. The proof is modeled after work of Green, Tao, and Ziegler and uses as a crucial input recent work of the first author regarding the equidistribution of nilsequences. In a companion paper, this result will be used to improve the bounds on Szemer\'{e}di's theorem. 
\end{abstract}

\maketitle

\section{Introduction}\label{sec:introduction}
We recall the definition of the Gowers $U^s$-norm on $\mb{Z}/N\mb{Z}$ and $[N]$. Throughout we let $[N] = \{1,\ldots,N\}$.
\begin{definition}\label{def:gowers-norm}
Given $f\colon\mb{Z}/N\mb{Z}\to\mb{C}$ and $s\ge 1$, we define
\[\snorm{f}_{U^s(\mb{Z}/N\mb{Z})}^{2^s}=\mb{E}_{x,h_1,\ldots,h_s\in\mb{Z}/N\mb{Z}}\Delta_{h_1,\ldots,h_s}f(x)\]
where $\Delta_hf(x)=f(x)\ol{f(x+h)}$ is the multiplicative discrete derivative (extended to lists by composition). Given a natural number $N$ and a function $f\colon[N]\to\mb{C}$, we choose a number $\wt{N}\ge 2^{s}N$ and define $\wt{f}\colon\mb{Z}/\wt{N}\mb{Z}\to\mb{C}$ via $\wt{f}(x) = f(x)$ for $x\in[N]$ and $0$ otherwise. Then 
\[\snorm{f}_{U^s[N]} := \snorm{\wt{f}}_{U^s(\mb{Z}/\wt{N}\mb{Z})}/\snorm{\mbm{1}_{[N]}}_{U^s(\mb{Z}/\wt{N}\mb{Z})}.\]
\end{definition}
\begin{remark*}
This is known to be well-defined and independent of $\wt{N}$, and a norm if $s\ge 2$; see \cite[Lemma~B.5]{GT10}.
\end{remark*}

Our main result is quasi-polynomial bounds on the inverse theorem for the Gowers $U^{s+1}$-norm over the integers. This builds on earlier work \cite[Section~8]{Len23} of the first author which handled the case of the $U^{4}$-norm. 
\begin{theorem}\label{thm:main}
Fix $\delta\in (0,1/2)$. Suppose that $f\colon[N]\to\mb{C}$ is $1$-bounded and
\[\snorm{f}_{U^{s+1}[N]}\ge\delta.\]
Then there exists a nilmanifold $G/\Gamma$ of degree $s$, complexity at most $M$, and dimension at most $d$ as well as a function $F$ on $G/\Gamma$ which is at most $K$-Lipschitz such that 
\[|\mb{E}_{n\in[N]}[f(n)\ol{F(g(n)\Gamma)}]|\ge\eps,\]
where we may take
\[d\le\log(1/\delta)^{O_s(1)}\emph{ and }\eps^{-1},K,M\le\exp(\log(1/\delta)^{O_s(1)}).\]    
\end{theorem}
\begin{remark*}
Throughout this paper, we will abusively write $\log$ for $\max(\log(\cdot), e^e)$; this is to avoid issues with small numbers. 
\end{remark*}

We have not formally defined a nilmanifold or notions of complexity; our definition is identical to that in work of Green and Tao \cite{GT12} and will be recalled precisely in Sections~\ref{sec:nilmani} and \ref{sec:complex}. 

In the companion paper to this work \cite{LSS24c}, we will use Theorem~\ref{thm:main} in order to improve the long standing bounds of Gowers \cite{Gow98,Gow01a} on Szemer\'{e}di's theorem. 
\begin{theorem}[Theorem~1.1 in~\cite{LSS24c}]\label{thm:szem-strong}
Let $r_k(N)$ denote the size of the largest $S\subseteq[N]$ such that $S$ has no $k$-term arithmetic progressions. For $k\ge 5$, there is $c_k\in(0,1)$ such that
\[r_k(N)\ll N\exp(-(\log\log N)^{c_k}).\]
\end{theorem}

\subsection{History and previous results}\label{sub:history}
A long standing conjecture of Erd\H{o}s and Tur\'{a}n \cite{ET36} stated that $r_k(N) = o(N)$. In full generality, this conjecture remained open until a combinatorial tour de force of Szemer\'{e}di \cite{Sze70,Sze75} which established the Erd\H{o}s and Tur\'{a}n conjecture. 
\begin{theorem}\label{thm:szem}
For $k\ge 3$, we have that 
\[r_k(N) = o_k(N).\]
\end{theorem}
Due to uses of the van der Waerden theorem and the regularity lemma (which was introduced in this work), Szemer\'{e}di's density saving over the trivial bound was exceedingly small. In particular, Szemer\'{e}di's result provided no improvement on known bounds for van der Waerden's theorem which was part of Erd\H{o}s and Tur\'{a}n's original motivation. 

The first result in the effort to prove reasonable bounds for $r_k(N)$, e.g.~giving a density saving of at least a finite iterated logarithmic type, came from work of Roth \cite{Rot54} which proved
\[r_3(N)\ll N(\log\log N)^{-1}.\]
Being based on Fourier analysis, the methods used in this paper did not obviously generalize to $k\ge 4$. An estimate for $r_k(N)$ which was ``reasonable'' would have to wait until pioneering work of Gowers \cite{Gow98,Gow01a}. 

The starting point of work of Gowers \cite{Gow98,Gow01a} is noting via an iterative application of the Cauchy--Schwarz inequality that if a set $A$ of density $\delta$ in $[N]$ has no $(s+2)$-term arithmetic progressions then $\snorm{f}_{U^{s+1}[N]}\ge\delta^{O_s(1)}$ where $f$ is a shifted indicator function of the set. In doing so, Gowers provided the correct notion of ``psuedorandomness'' generalizing Fourier coefficients which was suitable for understanding arithmetic patterns in subsets of the integers and therefore created ``higher order Fourier analysis''. The key technical ingredient in work of Gowers was a certain ``local inverse theorem'' for the $U^{s+1}[N]$-norm. Gowers proved that given a $1$-bounded function $f$ such that $\snorm{f}_{U^{s+1}[N]}\ge\delta$, there exists a decomposition of $[N]$ into arithmetic progressions of length roughly $N^{c_s}$ and a $1$-bounded function $g$ which is constant along these arithmetic progressions such that
\[\mb{E}_{x\in[N]} f(x) \ol{g(x)} \ge\delta^{O_s(1)};\]
i.e., $f$ correlates with $g$. This result, coupled with the density increment strategy as introduced by Roth \cite{Rot54}, provided the bound 
\[r_k(N)\ll N(\log\log N)^{-c_k}\]
for Szemer\'{e}di's theorem. These bounds have remained the best known for general $k$ until this work. For the sake of comparison, a long sequence of works have attacked the special case of $k=3$, culminating in a recent breakthrough work of Kelley and Meka \cite{KM23} which proved 
\[r_3(N)\ll N\exp(-c(\log N)^{1/12});\]
the constant $1/12$ was refined to $1/9$ in work of Bloom and Sisask \cite{BS23}. The only other improvements to the bound of Gowers were due to works of Green and Tao \cite{GT09,GT17} which ultimately established that 
\[r_4(N)\ll N(\log N)^{-c},\]
and very recent work of the authors \cite{LSS24} which handled the case $k=5$ of Theorem~\ref{thm:szem-strong}.

Notice however that the ``local inverse theorem'' of Gowers only gives correlations on arithmetic progressions of length $N^{c_s}$ and that the converse of this result is not true. In particular, a function may have small $U^{s+1}[N]$-norm and still correlate with a function which is constant on progressions of length $N^{c_s}$. To construct such an example, break $[N]$ into consecutive segments of length $\sqrt{N}$ and include each segment with probability $1/2$; while this set with high probability has large ``local correlations'' it has polynomially small Gowers norm. To obtain a full inverse result (analogous to the quality of Freiman's theorem, say), one must carefully pin down the global structure as well. Such a task is not straightforward, since the natural generalization of Fourier characters to exponentials of polynomials does not suffice.

A crucial development in the theory towards the inverse conjecture for the Gowers norm was the discovery of the role of nilpotent Lie groups. In groundbreaking work, Furstenberg \cite{Fur77} gave an alternate proof of Szemer\'{e}di based on ergodic theory; this work naturally led to seeking to understand certain nonconventional ergodic averages. In works of Conze and Lesigne \cite{CL84} and Furstenberg and Weiss \cite{FW96} regarding nonconventional ergodic averages, \emph{nilmanifolds} $G/\Gamma$ where $G$ is nilpotent and $\Gamma$ is a discrete cocompact subgroup were brought to the forefront. Host and Kra \cite{HK05} and independently Ziegler \cite{Zie07}, proved convergence of such nonconventional ergodic averages. Crucial to these works was establishing that such averages are controlled by projections on certain characteristic factors which naturally give rise to nilmanifolds. The role of nilsequences (derived from \emph{polynomial sequences} on nilmanifolds) was further highlighted in work of Bergelson, Host, and Kra \cite{BHK05}.

The statement of the inverse conjecture (without the given quantification) we will prove was first formulated in work of Green and Tao \cite{GT10}. Conditional on this inverse conjecture and that the M{\"o}bius function does not correlate with nilsequences, Green and Tao were able to prove asymptotic counts for all linear patterns in the primes of ``finite complexity'', vastly generalizing the celebrated Green--Tao theorem \cite{GT08c}. Both of these conjectures were resolved; the second being resolved in work of Green and Tao \cite{GT12c} while the first was resolved in work of Green, Tao, and Ziegler \cite{GTZ12}. We remark the cases $s=2$ and $s=3$ of the inverse conjecture were proven earlier by Green and Tao \cite{GT08b} and Green, Tao, and Ziegler \cite{GTZ11} respectively. A crucial ingredient in the cases $s\ge 3$ was work of Green and Tao \cite{GT12} on the equidistribution behavior of polynomial orbits on nilmanifolds. An alternative approach to the inverse conjecture was initiated by Szegedy \cite{Sze12}, involving the development of the theory of nilspaces by Camarena and Szegedy \cite{CS10}; a detailed treatment of these papers was given by Candela \cite{Can17, Can17a}. This nilspace approach has been further developed in works of Gutman, Manners, and Varj\'{u}\cite{GMV20,GMV19,GMV20b}. Both of these approaches to the inverse theorem, however, at least formally, gave no bounds on the complexity or dimension of the nilsequences with which the function correlates in the cases $s\ge 4$. A third approach due to Manners \cite{Man18} will be discussed later in this section.

We remark that the study of the inverse conjecture for the Gowers norm makes sense beyond the setting of functions on the interval or on the cyclic group $\mb{Z}/N\mb{Z}$. Work of Bergelson, Tao, and Ziegler \cite{BTZ10} and Tao and Ziegler \cite{TZ10b,TZ12b} resolved the analogue of the inverse conjecture for the Gowers norm over $\mb{F}_p^n$. Candela and Szegedy \cite{CS22} gave a version of the inverse conjecture for the Gowers norm over all compact abelian groups. This final work falls within the context of giving proofs which, broadly speaking, attempt to handle various abelian groups in a uniform manner. There has been substantial further work in this rough direction including works of Jamneshan and Tao \cite{JT23}, Jamneshan, Shalom, and Tao \cite{JST24}, and Candela, Gonz\'{a}lez-S\'{a}nchez, and Szegedy \cite{CGS23}.

The inverse theorem has had numerous further applications within additive combinatorics; we highlight just two. First, Tao and Ziegler \cite{TZ18} gave an asymptotic for the number of polynomial patterns $x+P_1(y),\ldots,x+P_j(y)$ in the primes where $P_1(0) = \cdots = P_j(0) = 0$ with top degree terms $P_1,\ldots,P_j$ being distinct. Second, works of Green and Tao \cite{GT10b} and Altman \cite{Alt22,Alt22b} used the inverse conjecture in combination with an arithmetic regularity lemma to establish the true complexity conjectures of Gowers and Wolf \cite{GW10}.

Due to its importance in the theory of additive patterns, establishing quantitative bounds on the inverse theorem for the Gowers norm has been seen as a central problem in additive combinatorics, with Green suggesting it as ``perhaps the biggest open question in the subject'' \cite[Problem~56]{GreOp}. For the case of $s = 2$, work of Green and Tao \cite{GT08b} gave quantitative bounds for the inverse theorem over the integers and work of Sanders \cite{San12} combined with the strategy in \cite{GT08b} proves Theorem~\ref{thm:main} for the case of $s=2$. For general $s$, until roughly five years ago no quantitative bounds were known for the inverse theorem and this was considered a major open problem. This state of affairs was substantially improved in remarkable work of Manners \cite{Man18} which proves a version of the inverse theorem where, in the notation of Theorem~\ref{thm:main},
\[d \le \delta^{-O_s(1)} \text{ and } \eps^{-1},K,M\le \exp(\exp(\delta^{-O_s(1)})).\]
This result was subsequently used as a crucial input in work of Tao and Ter{\"a}v{\"a}inen \cite{TT21} to give an effective result for the counts of linear equations in the primes. We remark that a quantitative version of the inverse conjecture over finite fields of high characteristic was proven in work of Gowers and Mili\'cevi\'c \cite{GM17,GM20}. 

At the highest level, the quantitative proofs of Manners \cite{Man18} and Gowers and Mili\'cevi\'c \cite{GM17,GM20} examine when the iterated derivatives of a function are $0$ with positive probability. Deriving useful information from this hypothesis over finite fields and the integers are very different problems but fundamentally one glues information from higher derivatives together into information regarding lower derivatives iteratively.

Our proof instead operates via induction on $s$ and attempts to glue degree $(s-1)$ nilmanifolds into a degree $s$ one exactly as in work of Green, Tao, and Ziegler \cite{GTZ12}. Our proof in fact is very closely modeled on their work and borrows large sections of their work essentially verbatim. In fact, we believe that the proof in \cite{GTZ12}, if appropriately quantified, itself yields a bound involving $O(s^2)$ many iterated exponentials. The primary improvement of our proof over theirs stems from the use of improved quantitative equidistribution results on nilmanifolds \cite{Len23, Len23b} rather than the results of \cite{GT12}. The reason we obtain quasi-polynomial bounds is that our proof, even though it inducts on $s$, gives quasi-polynomial bounds for each step of the induction. Since an iterated composition of finitely many quasi-polynomial functions is still quasi-polynomial, it follows that our bounds should remain quasi-polynomial. In contrast, we believe that the proof in \cite{GTZ12}, appropriately quantified, results in adding $O(t)$ iterated exponentials in each step $t$ of the induction, which when iterated totals $O(s^2)$ iterated exponentials. Here the results of \cite{Len23,Len23b} play a crucial role in eliminating the logarithms accumulated in the induction step. We further remark that the case $s=3$ of the main theorem (e.g. the $U^4$-inverse theorem) of the strength in Theorem~\ref{thm:main} was proven earlier by the first author in \cite[Section~8]{Len23} and may useful stepping stone for reading this paper (although this paper is logically independent).

\subsection{Organization of the paper I}\label{sub:organization-1}
We briefly discuss the next three sections of the paper. In Section~\ref{sec:nilmani}, we define a number of basic notions regarding nilmanifold and set various conventions which will be used throughout the paper. Our conventions differ in various extremely minor ways from those in the work of Green, Tao, and Ziegler \cite{GTZ12} but we record them explicitly to recall a number of definitions which will be used throughout the paper. In Section~\ref{sec:complex}, we set various complexity notions that will be given throughout the paper. In the case of nilmanifolds which are given a degree filtration (as is the case in Theorem~\ref{thm:main}), our conventions match those of Green and Tao \cite{GT12}. Given these notions in hand, we will be in position to outline the main proof in greater detail in Section~\ref{sec:outline}.

\subsection*{Acknowledgements}
The first author thanks Terence Tao for advisement. The authors thank Ben Green and Terence Tao for useful discussions regarding \cite{GTZ12,GTZ24}. The authors are grateful to Dan Altman, Ben Green, and Zach Hunter for comments. Finally the authors are especially grateful to Sarah Peluse for exceptionally detailed and useful comments on the manuscript.

\section{Conventions on nilmanifolds}\label{sec:nilmani}
We will recall a large portion of setup regarding nilsequences. In order to discuss this in a quantitative manner, various complexity notions are required which are formally defined in Section~\ref{sec:complex}. This section contains little more than bare definitions; a number of these concepts are developed and motivated in a beautiful manner in \cite[Section~6]{GTZ12}.

\subsection{Basic group theory}\label{sec:groups}
We briefly record various basic group theory notations which will be used throughout the paper; our notation is identical to that of \cite[Section~3]{GTZ12}.

Given a group $G$ and a subset $A$, we define $\sang{A}$ to be the subgroup generated by the subset $A$. Given a collection of subgroups $(H_i)_{i\in I}$ in $G$, we define $\bigvee_{i\in I} H_i$ to be the smallest subgroup containing all the $H_i$. Given $h,k\in G$, we denote the commutator of $h$ and $k$ to be
\[[h,k] = h^{-1}k^{-1}hk.\]

Given a sequence of elements $g_1,\ldots,g_r\in G$, we define the set of $(r-1)$-fold commutators inductively. The $0$-fold commutators of the set $g_i$ is simply $g_i$. For $r>1$, an $(r-1)$-fold commutator is $[w,w']$ where $w$ and $w'$ are $(s-1)$-fold and $(s'-1)$-fold commutators of $g_{i_1},\ldots,g_{i_s}$ and $g_{i_1'},\ldots,g_{i_{s'}'}$ respectively with $\{i_1,\ldots,i_s\} \cup\{i_{1}',\ldots, i_{s'}'\} = \{1,\ldots, r\}$ and $s + s' = r$. For instance, $[[g_3,g_4],[g_1,g_2]]$ and $[g_1,[g_3,[g_2,g_4]]]$ are $3$-fold commutators of $g_1$, $g_2$, $g_3$, and $g_4$.  

We let $H\leqslant G$ denote that $H$ is a subgroup of $G$. Given $H,K\leqslant G$, we denote the commutator subgroup 
\[[H,K] = \sang{[h,k]\colon h\in H, k\in K}.\]

The following pair of elementary lemmas will be used throughout the paper to verify various commutator identities; the first is \cite[Lemma~3.1]{GTZ12}. 

\begin{lemma}\label{lem:commute}
Let $H = \sang{A}$ and $K = \sang{B}$ be normal subgroups of a nilpotent group $G$. Then $[H,K]$ is also normal and is generated by the $(i+j-1)$-fold iterated commutators of $a_1,\ldots,a_i,b_1,\ldots,b_j$ over all choices of $a_1,\ldots,a_i\in A$, $b_1,\ldots,b_j\in B$ and $i,j\ge 1$.
\end{lemma}

This implies (see \cite[p.~1242]{GTZ12}) that for families $(H_i)_{i\in I}$, $(K_j)_{j\in J}$ which are normal in a nilpotent group $G$, 
\[\Big[\bigvee_{i\in I}H_i,\bigvee_{j\in J}K_j\Big] = \bigvee_{i\in I, j\in J}[H_i,K_j].\]

We next require that normality and various filtration conditions can be checked at the level of generators. 
\begin{lemma}\label{lem:com-check}
Suppose $K\leqslant H$ with $H=\sang{A}, K=\sang{B}$ where $A=A^{-1}$ and $B=B^{-1}$. Then:
\begin{itemize}
    \item If $[a,b]\in K$ for all $a\in A$ and $b\in B$ then $K$ is normal in $H$.
    \item Suppose $L\leqslant K\cap H$ is a normal subgroup with respect to both $K$ and $H$, and suppose for $a\in A$, $b\in B$, we have $[a,b]\in L$. Then $[H,K]\leqslant L$.
\end{itemize}
\end{lemma}
\begin{remark*}
Suppose we wish to prove that $(G_i)_{i\in I}$ forms an $I$-filtration (see Definition~\ref{def:arb-filtration}). This lemma implies that it suffices to check the commutator filtration conditions simply at the level of generators: if for each $i,j\in I$ we know $[g_i,g_j]\in G_{i+j}$ for all generators $g_i$ for $G_i$ and $g_j$ for $G_j$, then we can deduce that $G_{i+j}$ is normal in $G_i$ using the first bullet point above, and then deduce that $[G_i,G_j]\leqslant G_{i+j}$ using the second bullet point above.
\end{remark*}
\begin{proof}
For $a\in A, b\in B$ we have $[a,b]\in K$ hence $a^{-1}b^{-1}a\in K$. Since $B=B^{-1}$ generates $K$, we find $a^{-1}Ka\leqslant K$. Since $A$ generates $H$, we deduce that $K$ is normal in $H$.

For the second item, note that
\[[xy,z] = y^{-1}[x,z]y\cdot [y,z]\text{ and }[x,zy] = [x,y] \cdot y^{-1}[x,z]y.\]
Repeatedly expanding $[h,k]$ for $h\in H, k\in K$ into generators proves the result.
\end{proof}

Finally, and most importantly, we will require the following versions of the Baker--Campbell--Hausdorff formula (see \cite[(3.2)]{GTZ12}). Given $g_1,g_2$ in a nilpotent group $G$ and $n_1,n_2\in\mb{N}$, we have 
\begin{equation}\label{eq:bch-dis}
g_1^{n_1}g_2^{n_2} = g_2^{n_2}g_1^{n_1} \prod_{a}g_a^{P_{a}(n_1,n_2)}
\end{equation}
where $g_a$ ranges over all iterated commutators of $g_1$ and $g_2$ with at least $1$ copy of each and $P_a(n_1,n_2)\colon\mb{Z}\times\mb{Z}\to\mb{Z}$ is a polynomial in $n_1$ and $n_2$. Furthermore if $g_a$ involves $d_1$ copies of $g_1$ and $d_2$ copies of $g_2$ we have that $P_a$ has degree at most $d_1$ in $n_1$ and degree at most $d_2$ in $n_2$. Here the $a$ have been ordered in some arbitrary manner.

If $G$ is a connected, simply connected nilpotent Lie group, then we denote the Lie algebra of $G$ as $\log G$ and let $\exp\colon\log G\to G$ denote the exponential map while $\log\colon G\to\log G$ is the inverse (the exponential map being a homeomorphism in this situation). When we refer to nilpotent Lie groups, they will henceforth be connected and simply connected. For $g\in G$ and $t\in\mb{R}$, we define 
\[g^t = \exp(t\log g).\]
The Baker--Campbell--Hausdorff formula also implies that 
\[\exp(t_1\log g_1 + t_2\log g_2) = g_1^{t_1}g_2^{t_2}\prod_{a}g_a^{R_a(t_1,t_2)}\]
where $g_a$ ranges over all iterated commutators of $g_1$ and $g_2$ with at least $1$ copy of each and $R_a$ is a polynomial with rational coefficients satisfying identical degree constraints to $P_a$. Finally we require the following, most standard version, of the Baker--Campbell--Hausdorff formula which states that if $X,Y\in\log G$, then
\[\exp(X)\exp(Y) = \exp\Big(X + Y + \frac{1}{2}[X,Y] + \cdots\Big)\]
where the remaining terms in the expansion are iterated commutators in $X$ and $Y$ with all higher terms having at least one ``copy'' of $X$ and $Y$ within them. In particular, this implies that 
\begin{equation}\label{eq:bch-com}
\exp(-X)\exp(-Y) \exp(X)\exp(Y) = \exp\big([X,Y] + \cdots\big)
\end{equation}
where are all higher order terms have at least one copy of $X$ and $Y$ in them and are $r$-fold commutators with $r\ge 3$. In all versions of Baker--Campbell--Hausdorff, it is important for us that nilpotency means these expressions are finite.

\subsection{Filtrations}\label{sub:def-filtration}
We next require the notion of an ordering and an associated filtration (see \cite[Definition~6.7]{GTZ12}). 
\begin{definition}\label{def:arb-filtration}
An \emph{ordering} $I = (I,\preceq,+,0)$ is a set $I$ with a distinguished element $0$, binary operation $+\colon I\times I\to I$, and a partial order $\preceq$ on $I$ such that   
\begin{itemize}
    \item $+$ is associative and commutative with $0$ acting as an identity element;
    \item $\preceq$ has $0$ as the minimal element;
    \item For all $i,j,k\in I$, if $i\preceq j$ then $i+k\preceq j+k$;
    \item The initial segments $\{i\in I\colon i\preceq d\}$ are finite for all $d$.
\end{itemize}
We define the following three orderings, with addition being the standard addition:
\begin{itemize}
        \item The \emph{degree} ordering is given by the standard ordering on $\mb{N}$, denoted $I=\mb{N}$ for short;
        \item The \emph{degree-rank} ordering is given by $\{(d,r)\in\mb{N}^2\colon 0\le r\le d\}$ with the ordering that $(d',r')\preceq(d,r)$ if $d'<d$ or $d' = d$ and $r'\le r$, denoted $I=\mr{DR}$ for short;
        \item The \emph{multidegree} ordering is given by $\mb{N}^k$ with $(i_1',\ldots,i_k')\preceq(i_1,\ldots,i_k)$ when $i_j'\le i_j$ for all $1\le j\le k$, denoted $I=\mb{N}^k$ for short.
\end{itemize}
An \emph{$I$-filtration of $G$} is a collection of subgroups $G_I=(G_i)_{i\in I}$ such that $G_0=G$ and:
\begin{itemize}
    \item (Nesting) If $i,j\in I$ are such that $i\preceq j$ then $G_i\geqslant G_j$;
    \item (Commutator) For $i,j\in I$, we have $[G_i,G_j]\leqslant G_{i+j}$.
\end{itemize}
We say that a filtered group $G$ has \emph{degree $\le d$} (for $d\in I$) if $G_i$ is trivial for $i\not\preceq d$. $G$ has \emph{degree $\subseteq J$} for a downset $J$ if $G_i$ is trivial whenever $i\notin J$.
\end{definition}
Note that the commutator condition implies nested subgroups are normal within each other. We next define degree, degree-rank, and multidegree filtrations.
\begin{definition}\label{def:filt-prec}
Given $d\in\mb{N}$, we say a group $G$ is given a \emph{degree filtration of degree $d$} if:
\begin{itemize}
    \item $G$ is given a $\mb{N}$-filtration $(G_i)_{i\in\mb{N}}$ with degree $\le d$;
    \item $G_0 = G_1$.
\end{itemize}
Given $(d,r)\in\mb{N}^2$ with $0\le r\le d$, $G$ is given a \emph{degree-rank filtration of degree-rank $(d,r)$} if:
\begin{itemize}
    \item $G$ is given a $\mr{DR}$-filtration $(G_i)_{i\in\mr{DR}}$ with degree $\le(d,r)$;
    \item $G_{(0,0)}=G_{(1,0)}$ and $G_{(i,0)}=G_{(i,1)}$ for $i\ge 1$. (We also let $G_{(i,j)}=G_{(i+1,0)}$ for $j>i$.)
\end{itemize}
The associated degree filtration with respect to this degree-rank filtration is $(G_{(i,0)})_{i\ge 0}$.

Given $(d_1,\ldots,d_k)\in\mb{N}^k$, $G$ is given a \emph{multidegree filtration of multidegree $J$} (where $J\subseteq\mb{N}^k$ is a downset) if:
\begin{itemize}
    \item $G$ is given a $\mb{N}^k$-filtration $(G_i)_{i\in\mb{N}^k}$ with degree $\subseteq J$;
    \item $G_{\vec{0}} = \bigvee_{i=1}^k G_{\vec{e_i}}$.
\end{itemize}
The associated degree filtration with respect to the multidegree filtration is $(\bigvee_{|\vec{i}|=i}G_{\vec{i}})_{i\ge 0}$. Here $|\vec{i}| = i_1 + \ldots + i_k$.
\end{definition}
\begin{remark*}
This definition imposes some additional equalities of subgroups in order to say a group is given a degree-rank filtration versus a $\mr{DR}$-filtration (for example). In particular, the concept of ``degree-rank'' filtration and $\mr{DR}$-filtration are distinct. The difference is minor, but causes a number of technical checks to be required, most notably in Appendix~\ref{app:nilcharacters}. We will almost exclusively operate with these additional conditions; this is so that we can invoke equidistribution theory safely.
\end{remark*}

We now define polynomial sequences of an $I$-filtered group. The notion of a polynomial sequence for a group $G$ given a degree-rank filtration will be the same as treating this ordering as a $\mr{DR}$-filtration; the same applies for degree and multidegree filtrations. 
\begin{definition}\label{def:I-polynomial}
Given $g\colon H\to G$ a map between groups (not necessarily a homomorphism) and $h\in H$, we define the derivative $\partial_hg\colon H\to G$ via $\partial_hg(n)=g(hn)g(n)^{-1}$ for all $n\in H$. If $H,G$ are $I$-filtered, we say that this map $g$ is \emph{polynomial} if for all $m\ge 0$ and $i_1,\ldots,i_m\in I$, we have
\[\partial_{h_1}\cdots\partial_{h_m}g(n)\in G_{i_1+\cdots+i_m}\]
for all choices of $h_j\in H_{i_j}$ and $n\in H_0$. The space of all polynomial maps with respect to this data is denoted $\on{poly}(H_I\to G_I)$.
\end{definition}

We will require various general properties of polynomial sequences established in \cite[Appendix~B]{GTZ12}. We will only consider $H=\mb{Z}^k$ for $k\ge 1$ and the following $I$-filtrations on $H$.

\begin{definition}\label{def:Z-filtration}
We define the following filtrations on $H=\mb{Z}^k$:
\begin{itemize}
    \item The \emph{(domain) degree filtration} is with $I=\mb{N}$ the degree ordering and $H_0=H_1=\mb{Z}^k$, and $H_i=\{0\}$ for $i\ge 2$;
    \item The \emph{(domain) multidegree filtration} is with $I=\mb{N}^k$ the multidegree ordering, $H_{\vec{0}}=\mb{Z}^k$, $H_{\vec{e}_i}=\mb{Z}\vec{e}_i$ for $i\in[k]$, and $H_{\vec{v}}=\{0\}$ otherwise, where $\vec{e}_i$ forms the standard basis of $\mb{Z}^k$;
    \item The \emph{(domain) degree-rank filtration} is with $I=\mr{DR}$ the degree-rank ordering and $H_{(0,0)}=H_{(1,0)}=\mb{Z}^k$ and $H_{(d,r)}=\{0\}$ otherwise.
\end{itemize}
\end{definition}

We now define the notion of a nilmanifold, which is essentially a compact quotient of a filtered nilpotent Lie group.
\begin{definition}\label{def:nilmanifold}
We define an \emph{$I$-filtered nilmanifold $G/\Gamma$} to be the data of a connected, simply connected nilpotent Lie group $G$ with $I$-filtration (of Lie subgroups) and discrete cocompact subgroup $\Gamma\leqslant G$ which is rational with respect to $G_I$ (i.e., $\Gamma_i:=\Gamma\cap G_i$ is cocompact in $G_i$ for all $i\in I$). We say it has degree $\le d$ or $\subseteq J$ if $G$ has degree $\le d$ or $\subseteq J$.

If $I = \mb{N}$ and the $I$-filtration is furthermore a degree filtration with degree $\le d$, then $G/\Gamma$ is a degree $d$ nilmanifold. If $I = \mr{DR}$ and the $I$-filtration is furthermore a degree-rank filtration with degree $\le (d,r)$, then $G/\Gamma$ is a degree-rank $(d,r)$ nilmanifold. Finally if $I = \mb{N}^k$ and the $I$-filtration is furthermore a multidegree filtration with degree $\subseteq J$, then $G/\Gamma$ is a multidegree $J$ nilmanifold. 
\end{definition}
\begin{remark*}
Note that $\Gamma$ can naturally be given the structure of an $I$-filtered group $\Gamma_I$.
\end{remark*}

We finally (very occasionally) will require the lower central series of a group $G$. 
\begin{definition}\label{def:lower-central}
Given a nilpotent group $G$, define the lower central series inductively via $G_{(0)} = G_{(1)} = G$ and $G_{(i+1)} = [G,G_{(i)}]$. The \emph{step} of $G$ is the minimal $j$ such that $G_{(j+1)}=\mr{Id}_G$. 
\end{definition}

\subsection{Horizontal tori and Taylor coefficients}\label{sub:def-horizontal-taylor}
The next notion, that of a horizontal character, plays a vital role when discussing the equidistribution of nilsequences. 
\begin{definition}\label{def:horiz}
Given a connected, simply connected nilpotent group $G$ and a discrete, cocompact subgroup $\Gamma$, a \emph{horizontal character} $\eta$ is a continuous homomorphism $\eta\colon G\to\mb{R}$ such that $\eta(\Gamma)\subseteq\mb{Z}$. We say a horizontal character is nontrivial when $\eta$ is not identically zero.
\end{definition}
\begin{remark*}
Throughout the literature on nilmanifolds, horizontal characters are continuous homomorphisms $\eta\colon G\to\mb{R}/\mb{Z}$ such that $\eta$ annihilates $\Gamma$. It is straightforward to prove (via using Mal'cev bases) that these two notions are identical up to taking $\imod 1$. The reason we operate with the above definition is that the kernel of $\eta$ as defined is then a subspace of $G/[G,G]\simeq\mb{R}^{\dim(G)-\dim([G,G])}$.
\end{remark*}

We next require the notion of horizontal tori with respect to a degree-rank filtration. These tori will play a starring role in Sections~\ref{sec:sunflower}, \ref{sec:linear}, and \ref{sec:nil-first}; our definition is exactly that of \cite[Definition~9.6]{GTZ12}.
\begin{definition}\label{def:horizi}
Let $G$ be a degree-rank filtered nilpotent Lie group with filtration $G_{\mr{DR}}=(G_{(d,r)})_{(d,r)\in\mr{DR}}$. Given a subgroup $\Gamma$ of $G$, we define various \emph{horizontal tori} for $i\ge 1$ as
\begin{align*}
\on{Horiz}_i(G) &:= G_{(i,1)}/G_{(i,2)},\\
\on{Horiz}_i(\Gamma) &:= (\Gamma\cap G_{(i,1)})/(\Gamma\cap G_{(i,2)}),\\
\on{Horiz}_i(G/\Gamma) &:= \on{Horiz}_i(G)/\on{Horiz}_i(\Gamma).
\end{align*}
Given a polynomial sequence $g\in\on{poly}(\mb{Z}_{\mr{DR}}\to G_{\mr{DR}})$ we define the \emph{$i$-th horizontal Taylor coefficient} to be
\begin{align*}
\on{Taylor}_i(g) &:= \partial_1\cdots\partial_1 g(n) \imod G_{(i,2)}\in\on{Horiz}_i(G),\\
\on{Taylor}_i(g\Gamma) &:= \on{Taylor}_i(g) \imod\on{Horiz}_i(\Gamma)\in\on{Horiz}_i(G/\Gamma),
\end{align*}
where we take $i$ iterated derivatives.
\end{definition}

We also require the notion of $i$-th horizontal characters.
\begin{definition}\label{def:i-th-horiz}
Consider a nilmanifold $G/\Gamma$ with a degree-rank filtration. A continuous homomorphism $\eta\colon G_{(i,1)}\to\mb{R}$ is an \emph{$i$-th horizontal character} if $\eta(G_{(i,2)}) = 0$ and $\eta(G_{(i,1)}\cap\Gamma) \subseteq \mb{Z}$.
\end{definition}

The name Taylor coefficient is also used in the context of Taylor coefficients of polynomial factorizations. The following elementary lemma relates these two notions; we remark that a very closely related proof appears in \cite[Lemma~A.8]{GT10b}. 
\begin{lemma}\label{lem:taylor-relation}
Let $G$ be given a degree-rank filtration of degree-rank $(d,r)$ and consider a sequence $g\in\on{poly}(\mb{Z}_{\mr{DR}}\to G_{\mr{DR}})$. Then we may write $g(n)=\prod_{i=0}^d g_i^{\binom{n}{i}}$ for elements $g_i\in G_{(i,0)}$ and for $1\le i\le d$ we have
\[\on{Taylor}_i(g)=g_i\imod G_{(i,2)}.\]
\end{lemma}
\begin{proof}
The representation of $g(n)$ in the specified product form follows immediately from the existence of Taylor expansion, see \cite[Lemma~B.9]{GTZ12}. 

We next prove $\on{Taylor}_j(g)=g_j\imod G_{(j,2)}$ for each $1\le j\le d$ individually. Notice that it suffices to consider $\wt{g}(n)$ which is $g(n) \imod G_{(j,2)}$, i.e., we consider the group $G/G_{(j,2)}$ with quotiented filtration. This group is easily seen to be at most $j$-step nilpotent and furthermore $[G_{(i,1)}/G_{(j,2)}, G_{(j-i,1)}/G_{(j,2)}] = \mr{Id}_{G/G_{(j,2)}}$ for $0\le i\le j$ (one should check the cases $j=1$ and $i\in\{0,j\}$ manually). Let $\wt{G}_i = G_{(i,1)}/G_{(j,2)}$ for $0\le i\le j$ and note $\wt{G}_0=\wt{G}_1$.

We see that $\wt{G}_0\geqslant\cdots\geqslant\wt{G}_j$ is an $\mb{N}$-filtration for $\wt{G}_0$ with $[\wt{G}_i,\wt{G}_{j-i}]=\mr{Id}_{\wt{G}_0}$ for all $0\le i\le j$. Note that $\wt{g}(n) = \prod_{i=0}^{j}\wt{g}_i^{\binom{n}{i}}$ where $\wt{g}_i$ is $g_i\imod G_{(j,2)}$.

It suffices to prove the claim that $\wt{g}(n+1)\wt{g}(n)^{-1} = \prod_{i=0}^{j-1}(\wt{g}_i')^{\binom{n}{i}}$ with $\wt{g}_i'\in\wt{G}_{i+1}$ and $\wt{g}_{j-1}' = \wt{g}_j$. If this is the case, then we may modify the filtration $\wt{G}_0\geqslant\wt{G}_{1}\geqslant\cdots\geqslant\wt{G}_{j}$ by stripping off the top group, which maintains the necessary inductive properties. Iterating this procedure $j$ times we obtain the desired Taylor equality. 

This claim is a consequence of the Taylor expansion for general polynomial sequences and the Baker--Campbell--Hausdorff formula and counting the depths of nested commutators. The crucial reason that $\wt{g}_{j-1}' = \wt{g}_j$ is that any ``higher order'' terms which arise in the Baker--Campbell--Hausdorff formula and could contribute are in fact annhilated due to $[\wt{G}_i,\wt{G}_{j-i}]=\mr{id}_{\wt{G}_0}$ for $0\le i\le j$. 
\end{proof}

We also have the following linearity of the $i$-th Taylor coefficients.
\begin{lemma}\label{lem:coeff-mult}
Assume the setup of Lemma~\ref{lem:taylor-relation}. We have
\[\on{Taylor}_i(gh) = \on{Taylor}_i(g) + \on{Taylor}_i(h)\]
and if 
\[g(n) = \exp\bigg(\sum_{i=0}^dg_i \binom{n}{i}\bigg)\]
for $g_i\in\log(G_i)$ we have 
\[\on{Taylor}_i(g) = \exp(g_i) \imod G_{(i,2)}.\]
\end{lemma}
\begin{remark}
Note that $G_{(i,1)}/G_{(i,2)}$ is abelian and hence additive notation may be used when considering Taylor coefficients. 
\end{remark}
\begin{proof}
The first claim follows from Lemma~\ref{lem:taylor-relation}, the Baker--Campbell--Hausdorff formula, and the commutator relationship that $[G_{(i,0)},G_{(j-i,0)}]=[G_{(i,1)},G_{(j-i,1)}]\subseteq G_{(j,2)}$. (Note that this is using $G_{(0,0)}=G_{(0,1)}=G_{(1,0)}$ in the case $i=0$.)

For the second claim, suppose that
\[g(n) = \exp\bigg(\sum_{i=0}^dg_i \binom{n}{i}\bigg) = \prod_{i=0}^{s}(g_i')^{\binom{n}{i}}.\]
Via iterated applications of the Baker--Campbell--Hausdorff formula and the commutator relationship that $[G_{(i,0)},G_{(j-i,0)}]=[G_{(i,1)},G_{(j-i,1)}]\subseteq G_{(j,2)}$, we see that $g_j' = \exp(g_j) \imod G_{(j,2)}$ and the result follows. 
\end{proof}

\subsection{Vertical tori and nilcharacters}\label{sub:def-vertical-nilcharacter}
Given a polynomial sequence $g$ on an $I$-filtered Lie group with $I=\mb{N}$, we can define a sequence of vectors by considering a smooth vector-valued function $F$ on $G/\Gamma$ and looking at $F(g(n)\Gamma)$. However, we will be particularly interested in those which ``have a Fourier coefficient'' with respect to various subgroups of the center.

\begin{definition}\label{def:vert-char}
Consider a nilmanifold $G/\Gamma$ and a function $F\colon G/\Gamma\to\mb{C}$. Given a connected, simply connected subgroup $T$ of the center $Z(G)$ which is rational (i.e., $\Gamma\cap T$ is cocompact in $T$) and a continuous homomorphism $\eta\colon T\to\mb{R}$ such that $\eta(T\cap\Gamma)\subseteq \mb{Z}$, if 
\[F(gx)=e(\eta(g))F(x)\emph{ for all }g\in T\]
we say that $F$ has a \emph{$T$-vertical character} (or \emph{$T$-vertical frequency}) $\eta$.
\end{definition}
\begin{remark*}
Note that $T/(\Gamma\cap T)$ is isomorphic to a torus and thus one can modify functions under consideration to have vertical characters via appropriate Fourier decomposition.
\end{remark*}

A particular case which will arise frequently in our applications comes from the fact that given a filtration satisfying the conditions of Definition~\ref{def:filt-prec}, we have that the ``bottom group'' is contained in the center. For example, a group $G$ given a degree filtration of degree $d$ satisfies $[G,G_d] = [G_1,G_d] = \mr{Id}_G$ hence $G_d\leqslant Z(G)$. One special class of functions with a vertical frequency which will be of particular importance is that of nilcharacters.  

\begin{definition}\label{def:nilcharacter}
A \emph{nilcharacter} of degree $d$ and output dimension $D$ is the following data. Consider an $I$-filtered nilmanifold $G/\Gamma$ of degree $d$ such that $[G,G_d] = \mr{Id}_G$ and an $I$-filtered abelian group $H$. Let $g\in\on{poly}(H_I\to G_I)$ and consider function $F\colon G/\Gamma\to\mb{C}^D$ such that:
\begin{itemize}
    \item $\snorm{F(x)}_2 = 1$ for all $x\in G/\Gamma$ pointwise;
    \item $F(g_dx) = e(\eta(g_d))F(x)$ for all $g_d\in G_d$ where $\eta$ is some continuous homomorphism $G_d\to\mb{R}$ such that $\eta(\Gamma\cap G_d) \subseteq\mb{Z}$.
\end{itemize}
The values of the nilcharacter are given by $\chi\colon H\to\mb{C}^D$ where $\chi(n)=F(g(n)\Gamma)$ for $n\in H$.
\end{definition}
\begin{remark*}
We work with vector-valued nilcharacters for precisely the same topological reason given in \cite[p.~1254]{GTZ12}.
\end{remark*}

\subsection{Additional miscellaneous conventions}\label{sub:conventions}
We end this section with a brief discussion of various miscellaneous conventions. Throughout the paper we use $\{\cdot\}$ to denote the map $\mb{R}\to(-1/2,1/2]$ (or $\mb{R}/\mb{Z}\to(-1/2,1/2]$, abusively) which takes the representative $\imod{1}$ closest to $0$. Furthermore given $x\in\mb{R}/\mb{Z}$ and $y\in\mb{R}$ we will treat $x-y\in\mb{R}/\mb{Z}$ in the obvious manner. As used above, we let $e\colon\mb{R}/\mb{Z}\to\mb{C}$ denote the exponential function $e(x)=\exp(2\pi ix)$, which is lifted to $\mb{R}$ in the obvious manner.

We use standard asymptotic notation. Given functions $f=f(n)$ and $g=g(n)$, we write $f=O(g)$, $f\ll g$, $g=\Omega(f)$, or $g\gg f$ to mean that there is a constant $C$ such that $|f(n)|\le Cg(n)$ for sufficiently large $n$. We write $f\asymp g$ or $f=\Theta(g)$ to mean that $f\ll g$ and $g\ll f$, and write $f=o(g)$ or $g=\omega(f)$ to mean $f(n)/g(n)\to0$ as $n\to\infty$. Subscripts indicate dependence on parameters.

Finally in various arguments throughout the paper it will be convenient to denote appropriately bounded functions as $b(n)$ or $b(n_1,\ldots,n_k)$, and $B(n), B(n_1,\ldots,n_k)$ when vector-valued. When using such notation, the functions $b, B$ may change from line to line and within a line may refer to different functions.

\section{Various complexity notions}\label{sec:complex}
\subsection{Rationality of bases and Lipschitz norms}\label{sub:bases}
We will now discuss the definitions chosen for complexity of nilmanifolds. We start by defining first- and second-kind coordinates given a basis $\mc{X}$ for $\log G$.

\begin{definition}\label{def:coordinate}
Consider a connected, simply connected nilpotent Lie group $G$ of dimension $d$. Given a basis $\mc{X} = \{X_1,\ldots,X_d\}$ of $\log G$ and $g\in G$, there exists $(t_1,\ldots,t_d)\in\mb{R}^d$ such that 
\[g = \exp(t_1X_1 + t_2X_2+\cdots+ t_dX_d).\]
We define \emph{Mal'cev coordinates of first-kind} $\psi_{\exp}=\psi_{\exp,\mathcal{X}}\colon G\to\mb{R}^d$ for $g$ relative to $\mc{X}$ by 
\[\psi_{\exp}(g) := (t_1,\ldots,t_d).\]
Given $g\in G$ there also exists $(u_1,\ldots,u_d)\in\mb{R}^d$ such that 
\[g = \exp(u_1X_1)\cdots\exp(u_dX_d),\]
and we define the \emph{Mal'cev coordinates of second-kind} $\psi=\psi_{\mathcal{X}}\colon G\to\mb{R}^d$ for $g$ relative to $\mc{X}$ by
\[\psi(g) := (u_1,\ldots,u_d).\]
\end{definition}

Note that the above definition does not account for the cocompact subgroup $\Gamma$. The next set of definitions account for how ``rational'' $\mc{X}$ is with respect to itself and $\Gamma$.

\begin{definition}\label{def:heigh}
The \emph{height} of a number $x$ is $\max(|a|,|b|)$ if $x = a/b$ with $\gcd(a,b) = 1$ and $\infty$ if $x$ is irrational.
\end{definition}
\begin{definition}\label{def:Mal'cev}
Given a nilmanifold $G/\Gamma$ of dimension $d$, consider a basis $\mc{X} = \{X_1,\ldots,X_d\}$ of $\mf{g}=\log G$. $\mc{X}$ is said to be a \emph{weak basis of rationality $Q$} with respect to $\Gamma$ if:
\begin{itemize}
    \item There exist rationals $c_{ijk}$ of height at most $Q$ such that
    \[[X_i,X_j] = \sum_{k}c_{ijk}X_k;\]
    \item There exists integer $1\le q\le Q$ such that 
    \[q\cdot\mb{Z}^d\subseteq\psi_{\mr{exp},\mc{X}}(\Gamma)\subseteq q^{-1}\cdot\mb{Z}^d.\]
\end{itemize}
$\mc{X}$ is a \emph{Mal'cev basis} of $\log G$ with respect to $\Gamma$ of rationality $Q$ if:
\begin{itemize}
    \item There exist rationals $c_{ijk}$ of height at most $Q$ such that
    \[[X_i,X_j] = \sum_{k}c_{ijk}X_k;\]
    \item $\psi_{\mc{X}}(\Gamma) = \mb{Z}^d$.
\end{itemize}
We say that $\mc{X}$ has the \emph{degree $k$ nesting property} if there exist $\ell_1\le\cdots\le\ell_k$ such that if $\mf{g}_t=\on{span}_{\mb{R}}(X_{\ell_t+1},\ldots,X_m)$ then $[\mf{g},\mf{g}]\subseteq\mf{g}_1$, $[\mf{g},\mf{g}_{\ell}]\subseteq\mf{g}_{\ell+  1}$ and $[\mf{g},\mf{g}_k]=0$.

Finally we say that a Mal'cev basis is \emph{adapted} to a sequence of nesting subgroups $G = G_0\geqslant G_1\geqslant G_2\geqslant\cdots\geqslant G_\ell \geqslant\mr{Id}_G$ if 
\[\on{span}_{\mb{R}}(\{X_j\colon d-\dim(G_i)<j\le d\}) = \log G_i\]
for $1\le i\le\ell$.
\end{definition}

We next state the definition of the Lipschitz property for a function on $G/\Gamma$.
\begin{definition}\label{def:Lip}
We define a metric $d = d_{G,\mc{X}}$ on $G$ by
\[d(x,y) := \inf\bigg\{\sum_{i=1}^{n}\min(\snorm{\psi(x_ix_{i+1}^{-1})},\snorm{\psi(x_{i+1}x_{i}^{-1})})\colon n\in\mb{N}, x_1,\ldots,x_{n+1}\in G, x_1 = x, x_{n+1} = y\bigg\},\]
where $\snorm{\cdot}$ denotes the $\ell^\infty$-norm on $\mb{R}^m$, and define a metric on $G/\Gamma$ by
\[d(x\Gamma,y\Gamma) = \inf_{\gamma,\gamma'\in\Gamma}d(x\gamma,y\gamma').\]
Furthermore, for any function $F\colon G/\Gamma\to\mb{C}$ we define 
\[\snorm{F}_{\mr{Lip}} := \snorm{F}_{\infty} + \sup_{x\neq y\in G/\Gamma}\frac{|F(x)-F(y)|}{d(x,y)}.\]
Given a function $F\colon G/\Gamma\to\mb{C}^D$ such that $F = (F_1,\ldots,F_D)$ we define 
\[\snorm{F}_{\mr{Lip}} := \max_{1\le i\le D}\snorm{F_i}_{\mr{Lip}}.\]
\end{definition}
\begin{remark*}
Note that the metric on $G$ is right-invariant. We may omit the subscript $\mc{X}$ for the distance function when clear from context.
\end{remark*}

\subsection{Complexity of nilmanifolds}\label{sub:nil-complexity}
We now define the complexity of a nilmanifold with respect to either a degree or a degree-rank filtration.
\begin{definition}\label{def:complex-1}
Let $s\ge 1$ be an integer and let $M\ge 1$. A nilmanifold $G/\Gamma$ of degree $s$, dimension $d$, and \emph{complexity} at most $M$ consists of a degree $s$ filtration of $G$ along with a Mal'cev basis $\mc{X} = \{X_1,\ldots,X_d\}$ of $\log G$ which satisfies the following:
\begin{itemize}
    \item $\{X_1,\ldots,X_d\}$ is a Mal'cev basis for $\log G$ with respect to $\Gamma$ of rationality at most $M$;
    \item $\mc{X}$ is adapted to the sequence of subgroups $(G_i)_{i\in\mb{N}}$.
\end{itemize}
Analogously a nilmanifold $G/\Gamma$ of degree-rank $(s,r)$, dimension $d$, and complexity at most $M$ consists of a degree-rank $(s,r)$ filtration of $G$ along with a Mal'cev basis $\mc{X} = \{X_1,\ldots,X_d\}$ of $\log G$ which satisfies the following:
\begin{itemize}
    \item $\{X_1,\ldots,X_d\}$ is a Mal'cev basis for $\log G$ with respect to $\Gamma$ of rationality at most $M$;
    \item $\mc{X}$ is adapted to the sequence of subgroups $(G_i)_{i\in\mr{DR}}$.
\end{itemize}
\end{definition}
\begin{remark*}
The only difference in complexity for a degree versus degree-rank filtration is that we require the Mal'cev basis to be adapted with respect to the appropriate filtration. This definition unfortunately \emph{does not} extend to the case of multidegree filtrations since the set of subgroups do not nest in a total order. Furthermore note that a degree-rank nilmanifold of complexity $M$ is also a degree nilmanifold of the same complexity by taking the associated degree filtration.

Finally, whenever discussing the complexity of nilmanifolds, this is always with respect to a given Mal'cev basis $\mc{X}$. We will abusively write phrases such as ``nilmanifold $G/\Gamma$ of complexity $M$'' throughout the paper; such a statement should always be understood with a corresponding implicitly provided adapted Mal'cev basis of the Lie algebra. 
\end{remark*}
\begin{remark*}
We will also in passing require the notion of a degree $0$ nilmanifold. A degree $0$ nilmanifold is simply the trivial group $\mr{Id}_G$. All scalar-valued functions on degree $0$ nilmanifolds are constants and the Lipschitz norm is defined to be the absolute value of this constant.
\end{remark*}

We will next need the notion of a rational subgroup with respect to a Mal'cev basis; this will be crucial when giving the definition of complexity with respect to a multidegree filtration.
\begin{definition}\label{def:rat-subgroup}
A closed, connected subgroup $G'\leqslant G$ is $Q$-rational with respect to a basis $\mc{X} = \{X_1,\ldots,X_m\}$ of $\log G$ if $\log G'$ has a basis $\mc{X}' = \{X_1',\ldots,X_{m'}'\}$ where $X_i' = \sum_{j=1}^m c_{ij}X_j$ for $1\le i\le m'$ with $c_{ij}\in\mb{Q}$ having heights bounded by $Q$.
\end{definition}

We will repeatedly use the following fact about rational subgroups without further comment. 
\begin{fact}\label{fact:rat-subgroup}
Suppose $G$ is a connected, simply connected nilpotent Lie group of step $s$ and dimension $d$ with a discrete cocompact subgroup $\Gamma$. Suppose that $G/\Gamma$ has a weak basis $\mc{X}$ of rationality at most $Q$. Let $H_1,\ldots,H_j$ be subgroups which are each $Q$-rational and normal in $G$. Then 
\[H=\bigvee_{i=1}^j H_i\]
is an $O_s(Q^{O_s(d^{O_s(1)})})$-rational subgroup.
\end{fact}
\begin{proof}
Let $\mc{X}^{i}$ denote the underlying basis of $H_i$ witnessing low height. By applying Baker--Campbell--Hausdorff, we have that $\log H$ is spanned by taking all $(\le s)$-fold commutators of elements in $\mc{X}^{i}$ (possibly for different $i$). Each such element of the Lie algebra is easily seen to be a $O_s(Q^{O_s(d^{O_s(1)})})$-rational combination of $\mc{X}$ (using the weak basis property of $\mc{X}$). Taking a subset of these commutators which forms a basis of $\log H$ gives the desired result.
\end{proof}

We are now in position to define the complexity of a multidegree nilsequence. This definition is admittedly rather artificial but is designed to be the most flexible given various lemmas scattered throughout the literature.
\begin{definition}\label{def:complex-2}
Consider a downset $J$ with respect to the multidegree ordering on $\mb{N}^k$. Consider a group $G$ with a multidegree filtration of degree $\subseteq J$. Recall the associated degree filtration
\[G_i = \bigvee_{\vec{v}:|\vec{v}|=i}G_{\vec{v}}\]
and define the associated degree to be $\sup_{\vec{v}\in J}|\vec{v}|$. We say a multidegree $J$ nilmanifold $G/\Gamma$ of dimension $d$ with Mal'cev basis $\mc{X}$ has complexity at most $M$ if:
\begin{itemize}
    \item $\{X_1,\ldots,X_d\}$ is a Mal'cev basis for $\log G$ with respect to $\Gamma$ of rationality at most $M$;
    \item $\mc{X}$ is adapted to the sequence of subgroups $(G_i)_{i\in\mb{N}}$;
    \item $G_{\vec{v}}$ is an $M$-rational subgroup for all $\vec{v}\in\mb{N}^k$.
\end{itemize}
\end{definition}

We next note the trivial fact that complexity is bounded appropriately with respect to taking direct products; we implicitly invoke this when handling the complexity of direct products.
\begin{fact}\label{fct:product-rat}
Consider nilmanifolds $G/\Gamma$, $H/\Gamma'$ given degree $s$ filtrations $(G_i)_{i\ge 0}$, $(H_i)_{i\ge 0}$ and adapted Mal'cev bases $\mc{X},\mc{X}'$ each of complexity at most $M$. Then $(G\times H)/(\Gamma\times\Gamma')$ has complexity at most $M$ with respect to the Mal'cev basis 
\[\mc{X}^\ast=\{(X,0)\colon X\in\mc{X}\}\cup\{(0,X')\colon X'\in\mc{X}'\}.\]
$\mc{X}^\ast$ may be adapted to the degree $s$ filtration $G_i\times H_i$ by creating an ordering with suffixes 
\begin{align*}
&\big\{(X_j,0)\colon X_j\in\mc{X},0\le\dim(G)-j<\dim(G_i)\big\}\cup\big\{(0,X_j')\colon X_j'\in\mc{X}',0\le\dim(H)-j<\dim(H_i)\big\}.
\end{align*}
Furthermore given $F\colon G/\Gamma\to\mb{C}$ and $F'\colon H/\Gamma'\to\mb{C}$ which are $M$-Lipschitz, 
\[\wt{F}((g,h)(\Gamma\times\Gamma')):=F(g\Gamma)F'(h\Gamma')\]
is $3M^2$-Lipschitz on $(G\times H)/(\Gamma\times\Gamma')$. Analogous statements hold for degree-rank filtrations and multidegree filtrations. 
\end{fact}

We finally end by noting that quotients by normal subgroups of bounded rationality have appropriate complexity.
\begin{lemma}\label{lem:quotient-rat}
Consider a nilmanifold $G/\Gamma$ with $G$ given a degree $s$ filtration $(G_i)$ and of complexity at most $M$ with respect to an adapted Mal'cev basis $\mc{X}$. 

Suppose that $H$ is a normal subgroup of $G$ which is $M$-rational with respect to $\mc{X}$. Then the quotient nilmanifold $(G/H)/(\Gamma/(\Gamma\cap H))$ may be given an adapted Mal'cev basis $\mc{X}^\ast$, where the degree $s$ filtration is $(G_i/(G_i\cap H))$, which is an $M^{O_s(d^{O_s(1)})}$-rational combination of
\[\mc{X}'=\{X\imod\log H\colon X\in\mc{X}\}.\]
Analogous statements hold for degree-rank filtrations and multidegree filtrations. Finally if $H\leqslant Z(G)$ and $F$ is an $M$-Lipschitz function on $G/\Gamma$ which is $H$-invariant then $F$ descends to $(G/H)/(\Gamma/(\Gamma\cap H))$ and is $M^{O_s(d^{O_s(1)})}$-Lipschitz with respect to $\mc{X}^\ast$.
%The Lipschitz bound does not require $H$ to be in $Z(G)$.
\end{lemma}
\begin{proof}
We may find a subset $S$ such that 
\[\mc{X}'=\{X_i\imod\log H\colon X_i\in\mc{X},i\in S\}\]
is a basis for $\log(G/H)$. Since $H$ is $M$-rational with respect to $\mathcal{X}$, it follows from Cramer's rule that for $j \not\in S$, $X_j \imod \log H$ is a $M^{O_s(d^{O_s(1)})}$-combination of $X_i \imod \log H$ with $i\in S$ . Hence, $\mc{X}'$ is a weak Mal'cev basis for $(G/H)/(\Gamma/(\Gamma\cap H))$ of rationality $M^{O_s(d^{O_s(1)})}$. By \cite[Lemma~B.11]{Len23b}, we may find a Mal'cev basis adapted to $(G/H)/(\Gamma/(\Gamma\cap H))$ with complexity $M^{O_s(d^{O_s(1)})}$. Now, if $H\leqslant Z(G)$ and $F$ is $M$-Lipschitz on $G/\Gamma$ which is $H$-invariant, it follows trivially that $F$ descends to a function $\overline{F}$ on $(G/H)/(\Gamma/(\Gamma\cap H))$. The Lipschitz bounds for $\overline{F}$ follow from \cite[Lemma~B.3]{Len23b}.
%Using 
%\[\exp(g) \imod H = \exp(g \imod\log H)\]
%and \cite[Lemma~B.2]{Len23b} it follows that $\mc{X}'$ is a weak Mal'cev basis for $(G/H)/(\Gamma/(\Gamma\cap H))$ of rationality $M^{O_s(d^{O_s(1)})}$. The desired result then follows from \cite[Lemma~B.11]{Len23b}.
\end{proof}

\subsection{Size of vertical and horizontal characters}\label{sub:char-complexity}
We now define the size of vertical and horizontal characters. We first define the size of a horizontal character.
\begin{definition}\label{def:horiz-complex}
Given a nilmanifold $G/\Gamma$ and a Mal'cev basis $\mc{X}$, note that any horizontal character $\eta\colon G\to\mb{R}$ can be expressed in the form
\[\eta(g) = k\cdot\psi(g)\]
for some $k\in\mb{Z}^{\dim(G)}$. We define the \emph{size} of the horizontal character as $\snorm{k}_\infty$.
\end{definition}

We next define the size of an $i$-th horizontal character.
\begin{definition}\label{def:horiz-i-complex}
Consider a nilmanifold $G/\Gamma$ with $G$ given a degree-rank filtration of degree-rank $(s,r)$ and a Mal'cev basis $\mc{X} = \{X_1,\ldots,X_{\dim(G)}\}$ adapted to the degree-rank filtration. Note that any $i$-th horizontal character $\eta_i\colon G_{(i,1)}\to\mb{R}$ can be expressed in the form
\[\eta_i(g) = k\cdot\psi(g)\]
with $k\in\mb{Z}^{\dim(G)}$ for some $k$ which is nonzero only on coordinates between $\dim(G)-\dim(G_{(i,1)})<j\le\dim(G)-\dim(G_{(i,2)})$. We define the \emph{size} of the $i$-th horizontal character as $\snorm{k}_\infty$.
\end{definition}

We finally define the size of a vertical character. 
\begin{definition}\label{def:vert-complex}
Consider a nilmanifold $G/\Gamma$ with $G$ given a degree filtration of degree $k$ and a Mal'cev basis $\mc{X} = \{X_1,\ldots,X_{\dim(G)}\}$ adapted to the degree filtration. Consider a continuous vertical character $\xi\colon T\to\mb{R}$ from a rational subgroup $T\leqslant Z(G)$. We define the \emph{height} of $\xi$ as
\[\sup_{x\neq y\in T/(\Gamma\cap T)}\frac{|\xi(x) - \xi(y)|}{d_G(x\Gamma,y\Gamma)};\]
this will be denoted as $|\xi|$.
\end{definition}
\begin{remark*}
We now justify the terminology ``height'' given for the complexity of a vertical character. Suppose that $G/\Gamma$ has complexity $M$ (given $\mc{X}$) with respect to a degree filtration of degree $d$ and $T$ is $Q$-rational. We have that $T$ has a Mal'cev basis which is a $(QM)^{O_k(d^{O(1)})}$-rational combination of $\mc{X}$ by \cite[Lemma~B.12]{Len23b}; denote this $\mc{X}'$. By \cite[Lemma~B.9]{Len23b}, we have that for $x,y\in T$, 
\[d_{G,\mc{X}}(x\Gamma,y\Gamma) \le (QM)^{O_k(d^{O(1)})}d_{T,\mc{X}'}(x(\Gamma\cap T),y(\Gamma\cap T)) \le (QM)^{O_k(d^{O(1)})}d_{G,\mc{X}}(x\Gamma,y\Gamma).\]
With respect to $\mc{X}'=\{X_1',\ldots,X_{\dim(T)}'\}$, we have that $\xi$ is an integer vector and the definition of height is equivalent up to a multiplicative factor of $(QM)^{O_k(d^{O(1)})}$ to the height of this vector.
\end{remark*}

\subsection{Correlation}\label{sub:correlation-def}
We will also require the notion of a sequence being biased of some order. 
\begin{definition}\label{def:corr}
A function $f\colon[N]\to\mb{C}^D$ is \emph{$s$-biased} of correlation $\eta$, complexity $M$, and dimension $d$ if there exists a nilmanifold $G/\Gamma$ with a degree $s$ filtration such that $G$ has dimension at most $d$, $G/\Gamma$ has complexity at most $M$, and there exists an $M$-Lipschitz function $F$ and a polynomial sequence $g\in\on{poly}(\mb{Z}_\mb{N}\to G_\mb{N})$ such that
\[\snorm{\mb{E}_{n\in[N]}[f(n)\ol{F(g(n)\Gamma)}]}_{\infty}\ge\eta.\]
We will denote this as $f\in\on{Corr}(s,\eta,M,d)$.
\end{definition}

\subsection{Miscellaneous complexity notions}\label{sub:misc-complexity}
We will also require the following definition regarding smoothness norms of polynomial sequences. 
\begin{definition}\label{def:smooth}
Given $\vec{v}\in\mb{N}^k$ and $\vec{n}\in\mb{N}^k$, we define 
\[\binom{\vec{n}}{\vec{v}}=\prod_{i=1}^{k}\binom{n_i}{v_i}.\] Any polynomial sequence $g\colon\mb{Z}^{k}\to\mb{R}$ can be expressed uniquely as \[g(\vec{n}) = \sum_{\vec{\ell}\in\mb{N}^k}\alpha_{\vec{\ell}}\binom{\vec{n}}{\vec{\ell}}\]
with $\alpha_{\vec{\ell}}\in\mb{R}$.
We define 
\[\snorm{g}_{C^{\infty}[N]} := \max_{\vec{\ell}\neq \vec{0}}N^{|\vec{\ell}|}\cdot\snorm{\alpha_{\vec{\ell}}}_{\mb{R}/\mb{Z}}\]
where $|\vec{\ell}| = \sum_{j=1}^{k}\ell_j$.
\end{definition}
\begin{remark*}
Note that the above definition is only sensitive to the values of $g \imod 1$. 
\end{remark*}

We now define when a polynomial sequence is rational and smooth.
\begin{definition}
Consider a nilmanifold $G/\Gamma$ given either a degree, degree-rank, or multidegree filtration with Mal'cev basis $\mc{X}$ and $g$ a domain $\mb{Z}^k$ polynomial sequence on $G$ with respect to the given filtration. We say that $g$ is \emph{$(M,N)$-smooth} if:
\begin{itemize}
    \item $d_{G,\mc{X}}(g(\vec{0}),\mr{id}_G)\le M$;
    \item $d_{G,\mc{X}}(g(\vec{v}),g(\vec{v}+\vec{e}_i))\le M \cdot N^{-1}$ for $\vec{v}\in[N]^{k}$ and $1\le i\le k$.
\end{itemize}
We say that $g$ is \emph{$M$-rational} if there is $1\le m\le M$ such that for all $\vec{n}\in\mb{N}^k$ we have that 
\[\psi_{\mc{X}}(g(\vec{n}))\in\frac{1}{m}\cdot\mb{Z}^{\dim(G)}.\]
\end{definition}

\section{Proof outline}\label{sec:outline}
We are now in position to discuss the proof of Theorem~\ref{thm:main}; as our proof is closely modeled on that of Green, Tao, and Ziegler \cite{GTZ12}, the announcement of \cite{GTZ11b} may prove a useful starting point for certain readers. For various parts of this outline we will restrict to the case of the $U^5$-inverse theorem and discuss the proof as if the analysis were performed with bracket polynomials.

\subsection{Induction on degree and additive quadruples}
Suppose that $f\colon[N]\to\mb{C}$ is $1$-bounded such that 
\[\snorm{f}_{U^5[N]}\ge\delta.\]
Via the inductive definition of the Gowers norm, we have for $\delta^{O(1)}N$ values of $h\in[N]$ that 
\[\snorm{\Delta_hf}_{U^4[N]}\ge\delta^{O(1)}.\]
Call this set of indices $H$. Applying Theorem~\ref{thm:main} inductively (when converted to bracket polynomials; see e.g.~\cite[Proposition~1.4]{LSS24}) we may choose $d_1,d_2,d_3\le\log(1/\delta)^{O(1)}$ and coefficients $a_{i,h}$ etc.~such that
\begin{align*}
\bigg|\mb{E}_{n\in[N]} \Delta_h f(n) \cdot e\bigg(&\sum_{i=1}^{d_1}a_{i,h} n[b_{i,h} n][c_{i,h}n] + \sum_{i=1}^{d_2}d_{i,h} n^2[e_{i,h} n] + \sum_{i=1}^{d_3} f_{i,h}n[g_{i,h}n] \\
&\qquad\qquad\qquad+j_hn^3 + \ell_hn^2 + m_hn\bigg)\bigg|\ge\exp(-\log(1/\delta)^{O(1)});
\end{align*}
we have padded with extra coefficients to make the dimensions $d_i$ not $h$-dependent. Set
\[\ol{G_h(n)} = e\bigg(\sum_{i=1}^{d_1}a_{i,h} n[b_{i,h} n][c_{i,h}n] + \sum_{i=1}^{d_2}d_{i,h} n^2[e_{i,h} n] + \sum_{i=1}^{d_3} f_{i,h}n[g_{i,h}n] + j_hn^3 + \ell_h n^2 + m_hn\bigg).\]

For the sake of clarity, we will let $L_h(n)$ denote terms of degree $\le 2$ which are possibly $h$-dependent. We have
\[\bigg|\mb{E}_{n\in[N]} \Delta_h f(n) \cdot e\bigg(\sum_{i=1}^{d_1}a_{i,h} n[b_{i,h} n][c_{i,h}n] + \sum_{i=1}^{d_2}d_{i,h} n^2[e_{i,h} n] + j_hn^3 + L_h(n) \bigg)\bigg|\ge\exp(-\log(1/\delta)^{O(1)}).\]
The first crucial step, via a Cauchy--Schwarz argument due to Gowers \cite{Gow98} (see \cite[Proposition~6.1]{GTZ11} or Lemma~\ref{lem:CS-basic}) is that for many additive quadruples $(h_1,h_2,h_3,h_4)$, i.e.~$h_1+h_2=h_3+h_4$, we have 
\[|\mb{E}_{n\in[N]}G_{h_1}(n) G_{h_2}(n+h_1-h_4) \ol{G_{h_3}(n)} \ol{G_{h_4}(n+h_1-h_4)}|\ge\exp(-\log(1/\delta)^{O(1)}).\]

\subsection{Sunflower and linearization for the top degree-rank}
Via bracket polynomial manipulations, we see that the ``top degree-rank'' term of the above expression is 
\[\sum_{i=1}^{d_1}(a_{i,h_1} n[b_{i,h_1} n][c_{i,h_1}n] +a_{i,h_2} n[b_{i,h_2} n][c_{i,h_2}n]  -a_{i,h_3} n[b_{i,h_3} n][c_{i,h_3}n] -a_{i,h_4} n[b_{i,h_4} n][c_{i,h_4}n]).\]
The heart of the proof is demonstrating that these ``top degree-rank terms line up'' in an appropriate sense across a dense additive tuples in $H$. Such a conclusion is at least plausible since for generic coefficients the associated bracket polynomial equidistributes$\imod{1}$, which would violate the given condition on $G_{h_1}(n) G_{h_2}(n+h_1-h_4) \ol{G_{h_3}(n)} \ol{G_{h_4}(n+h_1-h_4)}$. One possibility where the top degree-rank term is exactly zero is when we can write $a_{i,h_1} = a_i h_1$, $b_{i,h_1} = b_i^\ast$, $c_{i,h_1} = c_i^\ast$. The heart of the matter is that, up to controlled modifications, this is the only way for that to occur in a robust sense. 

The first modification is that we can replace in the above example the expression $a_{i,h_1} = a_i h_1$ with $a_{i,h_1} = \Theta_i\{\Theta_i' h_1\}$ or more generally a bracket linear form. The second modification is that we may not get a description that respects the presented structure of the sum. Instead the coordinates of the bracket linear form may only appear in these ``fixed'', ``fixed'', ``bracket linear'' triples after a linear change of variables. We prove the existence of this structure in two steps, as in \cite{GTZ12}. The first step proves that the bracket form is ``fixed'', ``fixed'', ``$h$-dependent'' and the second step then proves that the ``$h$-dependent'' part in fact has a bracket linear structure. These steps will fall under the names \emph{sunflower} and \emph{linearization} respectively.

\subsection{Degree-rank iteration}
Once we have learned this refined form for $\sum_{i=1}^{d_1}a_{i,h}n[b_{i,h} n][c_{i,h}n]$, we iterate and then learn the refined form for the next highest degree-rank term $\sum_{i=1}^{d_2}d_{i,h}n^2[e_{i,h}n]$, and then finally we learn the refined form for $j_hn^3$. Given these refined forms, Green, Tao, and Ziegler prove that the top degree terms in fact have the form of a multidegree $(1,3)$ nilsequence (in variables $h$ and $n$). Finally given such a correlation, a symmetrization argument as in \cite{GTZ12} concludes the proof. We remark here that while terms such as $a_{i,h}$ and $e_{i,h}$ correspond to Taylor coefficients on the top degree horizontal torus, terms such as $d_{i,h}$ belong to the second horizontal torus, and $j_h$ to the third horizontal torus. Furthermore to handle terms of the form $\sum_{i=1}^{d_2}d_{i,h}n^2[e_{i,h}n]$ correctly we must realize such terms via a degree-rank $(3,2)$ nilmanifold, hence the need for the finer degree-rank notion.

\subsection{Nilcharacters and horizontal tori}
We now make this description more precise in terms of nilcharacters and horizontal tori. Let $F(g_h(n)\Gamma) = G_h(n)$ be a nilcharacter of degree-rank $(s,r)$; here $e(an[bn][cn])$ should be thought of as an ``almost'' degree-rank $(3,3)$ nilcharacter and $e(an[bn^2])$ as an ``almost'' degree-rank $(3,2)$ nilcharacter. The sunflower step proves that the nilsequence $F(g_h(n)\Gamma)$ can be realized as a bracket polynomial whose top degree-rank part is a sum of terms with $(r-1)$ iterated brackets where each term consists of $(r-1)$ $h$-independent phases of $g_h$, and possibly one $h$-dependent phase of $g_h$. Here, ``phase'' will correspond to components of the Taylor coefficients of $g_h$, $\on{Taylor}_i(g_h)$. This corresponds to showing that the $i$-th horizontal torus $G_{(i,1)}/G_{(i,2)}$ contains vector spaces $V_{i,\mr{Dep}}\leqslant V_i$ such that:
\begin{itemize}
    \item $\on{Taylor}_i(g_h)-\on{Taylor}_i(g_{h'})\in V_{i,\mr{Dep}}$ and $\on{Taylor}_i(g_h) \in V_{i}$;
    \item If $i_1+\cdots+i_r=s$, then $[v_{i_1},v_{i_2},\ldots,v_{i_r}] = 0$ whenever $v_{i_j}\in V_{i_j}$ and there are at least two indices $j$ such that $v_{i_j} \in V_{i_j,\mr{Dep}}$.
\end{itemize}
Here we have implicitly descended an iterated commutator to the vector spaces $G_{(i,1)}/G_{(i,2)}$ which corresponds to a multilinear form in this case. Such a result is proven via combining quantitative equidistribution theory of nilsequences \cite{Len23, Len23b} with a ``Furstenberg--Weiss argument'' as in \cite{GTZ11, GTZ12, Len23}; see \cite{Tao21} for further examples of the Furstenberg--Weiss argument.

The linearization step then proves that the remaining $h$-dependent phases are ``bracket linear'' in $h$. In practice, we require an additional case that the $h$-dependent phase may be a \emph{petal} phase: a top degree-rank term with the petal phase can be realized as a ``lower order term'', or more precisely a bracket phase with at most $(r-2)$ iterated brackets or of total degree at most $s-1$. Thus, the statement we ultimately prove is that we may decompose a subspace of the $i$-th horizontal torus into the sum of three linearly disjoint vector spaces $W_{i,\ast}$, $W_{i,\mr{Lin}}$, and $W_{i,\mr{Pet}}$ such that:
\begin{align*}
\on{Taylor}_i(g_h) &\in W_{i,\ast} + W_{i,\mr{Lin}} + W_{i,\mr{Pet}},\\
\on{Taylor}_i(g_h) - \on{Taylor}_i(g_{h'})&\in W_{i,\mr{Lin}} + W_{i,\mr{Pet}},
\end{align*}
and the projection of $\on{Taylor}_i(g_h)$ onto $W_{i,\mr{Lin}}$ is bracket linear. In addition, we require that if $i_1+\cdots+i_r=s$, then $[v_{i_1},v_{i_2},\ldots,v_{i_r}]=0$ whenever $v_i\in W_{i,\ast}+W_{i,\mr{Lin}}+W_{i,\mr{Pet}}$ and either $v_{i_j}\in W_{i_j,\mr{Pet}}$ for at least one index $j$ or $v_{i_j}\in W_{i_j,\mr{Lin}}$
for at least two distinct indices $j$. Thus even though we have not improved our understanding of the Taylor coefficients on $W_{i,\mr{Pet}}$ we have the improved the vanishing of the top degree-rank commutator bracket on this vector space. The linearization step is proved by a combination of quantitative equidistribution theory of nilmanifolds \cite{Len23, Len23b} and inverse sumset theory. We refer the reader to \cite{Len23} for a simpler case of the argument given here.

\subsection{Quantitative bounds}
The heart of this paper is performing the sunflower and linearization steps efficiently. Green, Tao, and Ziegler \cite{GTZ12} accomplish this (when unwinding the correspondence between nilmanifolds and bracket polynomials) via iteratively learning relations between the coefficients $a_{i,h},b_{i,h},c_{i,h}$ and performing a dimension reduction argument.\footnote{This is performed in \cite[Section~10]{GTZ12} via a ``rank minimality'' argument; this requires passing to an ultralimit. When performed in finitary language this becomes a dimension reduction argument and is also present in the proof of \cite[Theorem~D.5]{GTZ12}.} Furthermore, the underlying equidistribution theorem used in the work of Green, Tao, and Ziegler \cite{GTZ12}, proven in work of Green and Tao \cite{GT12}, relies on an induction on dimension argument. The use of any induction on dimension argument essentially immediately results in $O(s)$ iterated logarithms and thus must be avoided.

The use of induction on dimension in the equidistribution theorem was avoided in work of the first author \cite{Len23,Len23b}. The key point in Sections~\ref{sec:sunflower} and \ref{sec:linear} therefore is to perform the sunflower and linearization steps without any use of induction on dimension. The precise details, while mainly utilizing elementary linear algebra, require a bit of precision. This argument, extending the case of the $U^4$-inverse theorem from \cite{Len23}, demonstrates that a dimension-independent number of applications of equidistribution theory is sufficient to derive the necessary decrease in degree-rank. (Note that the argument in \cite{GTZ12} morally uses that one can in fact assume that there are no ``short linear relations'' between various coefficients, but such a result necessitates exponential in dimension dependencies in the exponent.) Another crucial point in our work is that the length of the associated bracket linear form that is obtained not ``very long''. This is, by now, a standard consequence of the quasi-polynomial bounds of Sanders \cite{San12b} towards the polynomial Bogolyubov conjecture.

We finally remark that the quantitative equidistribution theorem we use is slightly different than the one derived in work of the first author \cite{Len23,Len23b}. The work of the first author is most naturally phrased as factoring ill-distributed polynomial sequences into a smooth part, a rational part, and a polynomial sequence which (up to taking a certain quotient) lives in a lower \emph{step} nilmanifold. For our purposes, it is critical to instead lower the \emph{degree} of the nilmanifold. This is most easily seen from the above bracket polynomial example where we are attempting to linearize a function of the form
\[e\bigg(\sum_{i=1}^{d_3} f_{i,h}n[g_{i,h}n] + j_hn^3 + \ell_h n^2 + m_hn\bigg).\]
At this step we wish to linearize $j_hn^3$ instead of handling the terms $f_{i,h}n[g_{i,h}n]$; the $j_hn^3$ term, while having the highest degree, does not correspond to the highest step part of the nilmanifold. This phenomenon only occurs when proving the $U^{s+1}$-inverse theorem for $s\ge 4$. Thus a crucial ingredient in our work is bootstrapping, as a black box, the efficient version of equidistribution with respect to step in order to obtain an efficient version of equidistribution with respect to degree; this is Theorem~\ref{thm:equi-deg}.

\subsection{Organization of the paper II}\label{sub:organization-2}
In Section~\ref{sec:equi} we prove the necessary quantitative equidistribution theorem with respect to degree. In Section~\ref{sec:setup}, we perform the setup and give various definitions which will be used to perform the sunflower and linearization steps. In Section~\ref{sec:gowers}, we derive that many additive quadruples exhibit a bias. In Section~\ref{sec:sunflower} we perform the sunflower step while in Section~\ref{sec:linear} we perform the linearization step. In Sections~\ref{sec:nil-first} and \ref{sec:nil-first-2} we then convert information regarding the Taylor coefficients into correlation with a multidegree $(1,s-1)$ nilsequence and a nilsequence of lower degree-rank. Iterating this argument we eventually obtain correlation with a mutltidegree $(1,s-1)$ nilsequence. In Section~\ref{sec:sym}, we symmetrize this nilsequence to obtain Theorem~\ref{thm:main}.

Appendix~\ref{app:approx-hom} collects certain standard results regarding approximate homomorphisms (this is ultimately where work of Sanders \cite{San12b} is invoked). In Appendix~\ref{app:defer}, we collect a number of miscellaneous propositions which are deferred throughout the paper. Finally in Appendix~\ref{app:nilcharacters} we collect a number of propositions regarding nilcharacters.

\section{Efficient equidistribution theory of nilsequences}\label{sec:equi}
In order to state the primary equidistribution input of this paper we will need the notion of when an element in $G/[G,G]$ and a horizontal character are orthogonal.
\begin{definition}\label{def:orthogonal}
Consider a nilmanifold $G/\Gamma$, a horizontal character $\eta\colon G\to\mb{R}$, and $w\in G/[G,G]$. We say that $\eta$ and $w$ are \emph{orthogonal} if $\eta(w) = 0$.
\end{definition}

The primary equidistribution input into our results will be the following result of the first author \cite[Theorem~3]{Len23b}. This result is ultimately the driving force of this paper. 
\begin{theorem}\label{thm:step-equi}
Fix an integer $\ell\ge 1$, $\delta\in(0,1/10)$, $M,d\ge 1$, and $F\colon G/\Gamma\to\mb{C}$. Suppose that $G$ is a dimension $d$, at most $s$-step connected, simply connected nilpotent Lie group with a given degree $k$ filtration, and the nilmanifold $G/\Gamma$ is complexity at most $M$ with respect to this filtration. Let $g$ be a polynomial sequence on $G$ with respect to this filtration.

Furthermore suppose that $\snorm{F}_{\mr{Lip}}\le 1$ and $F$ has $G_{(s)}$-vertical frequency $\xi$ such that the height of $\xi$ is bounded by $M/\delta$. Suppose that $N\ge(M/\delta)^{\Omega_{k,\ell}(d^{\Omega_{k,\ell}(1)})}$ and
\[\big|\mb{E}_{\vec{n}\in[N]^\ell}F(g(\vec{n})\Gamma)\big|\ge\delta.\]
There exists an integer $0\le r\le\dim(G/[G,G])$ such that:
\begin{itemize}
    \item We have horizontal characters $\eta_1,\ldots,\eta_r\colon G\to\mb{R}$ with heights bounded by $(M/\delta)^{O_{k,\ell}(d^{O_{k,\ell}(1)})}$;
    \item For all $1\le i\le r$, we have $\snorm{\eta_i\circ g}_{C^{\infty}[N]} \le(M/\delta)^{O_{k,\ell}(d^{O_{k,\ell}(1)})}$
    \item For any $w_1,\ldots,w_s\in G/[G,G]$ such that $w_i$ are orthogonal to all of $\eta_1,\ldots,\eta_r$, we have
    \[\xi([[[w_1, w_2], w_3],\ldots,w_s]) = 0.\]
\end{itemize}
\end{theorem}
\begin{remark}\label{rem:height}
Note that $G_{(s)}$ (and in fact any group in the lower central series) is seen to be $O_{s,k}(M^{O_{s,k}(1)})$-rational due to Lemma~\ref{lem:commute}. This guarantees that the height definition used in \cite{Len23b} and here are compatible. 
\end{remark}
\begin{remark*}
Let $W = \bigcap_{i=1}^r\on{ker}(\eta_i)$. The crucial property of the lemma output is that
\[\wt{G} := W/\ker(\xi)\]
is trivially seen to be at most $(s-1)$-step nilpotent. (Note that if $G$ is abelian then we have that $\wt{G}$ is trivial.) This is due to the fact that defining $W = W_0 = W_1$ and $W_j = [W_1,W_{j-1}]$ for $j\ge 2$ yields $W_{(s)}\leqslant G_{(s)}$ and $\xi(W_{(s)}) = 0$. Additionally, the statement in \cite[Theorem~3]{Len23b} assumes $G$ is exactly $s$-step nilpotent and $\xi$ is nonzero. In the case when $G$ is strictly less than $s$-step nilpotent, taking no horizontal characters (i.e., $W=G$) gives the desired statement. Furthermore when $\xi$ is zero we may similarly take no horizontal characters and note that the final statement is vacuous. 
\end{remark*}

The following variant of Theorem~\ref{thm:step-equi} will essentially be the primary equidistribution tool in our paper. For the sake of argumentation, we first prove the result in the case when the vertical frequency considered lives on a $1$-dimensional torus and then bootstrap to the general case. 

This theorem and its proof are motivated by \cite[Lemma~E.11]{GTZ12}. The key point is that Theorem~\ref{thm:step-equi} allows us to give a procedure that relies on an induction on step rather than an induction on dimension. The main technical issue is at each stage we pass to a quotient group given by quotienting the kernel of a certain vertical character and thus we must iteratively ``lift'' these factorizations. 

\begin{theorem}\label{thm:equi-deg}
Let $\ell\ge 1$ be an integer, $\delta\in (0,1/10)$, $M\ge 1$, and $F\colon G/\Gamma\to\mb{C}$. Suppose that $G$ is dimension $d$, is $s$-step nilpotent with a given degree $k$ filtration, and the nilmanifold $G/\Gamma$ is complexity at most $M$ with respect to this filtration.

Suppose that $T\leqslant Z(G)$ is a $1$-dimensional subgroup of the center which is $M$-rational with respect to $G$. Further suppose that $F$ has a nonzero $T$-vertical character $\xi$ with $|\xi|\le M/\delta$, $\snorm{F}_{\mr{Lip}}\le M$, $N\ge(M/\delta)^{\Omega_{k,\ell}(d^{\Omega_{k,\ell}(1)})}$, and $g$ is a polynomial sequence with respect to the degree $k$ filtration such that $g(0) = \mr{id}_G$. Then if
\[\big|\mb{E}_{\vec{n}\in[N]^{\ell}}F(g(\vec{n})\Gamma)\big|\ge\delta\]
there exists a factorization
\[g = \eps g'\gamma\]
such that:
\begin{itemize}
    \item $\eps(0)=g'(0)=\gamma(0)=\mr{id}_G$;
    \item $g'$ lives in an $(M/\delta)^{O_{k,\ell}(d^{O_{k,\ell}(1)})}$-rational subgroup $H$ such that $H\cap T = \mr{Id}_G$;
    \item $\gamma$ is an $(M/\delta)^{O_{k,\ell}(d^{O_{k,\ell}(1)})}$-rational polynomial sequence;
    \item $\eps$ is an $((M/\delta)^{O_{k,\ell}(d^{O_{k,\ell}(1)})},N)$-smooth polynomial sequence.
\end{itemize}
\end{theorem}
\begin{proof}
The proof proceeds by iteratively ``simplifying'' $g$ to live on successively lower-step nilmanifolds. We treat $\ell$ as constant and allow implicit constants to depend on $\ell$.

\noindent\textbf{Step 1: Iteration setup.}
We will define a sequence of parameters $M_i,\delta_i$ and $Q_i,N_i,v_i$ (where the domain of $\vec{n}$ at stage $i$ will be $v_i+Q_i\cdot[N_i]^\ell$) satisfying:
\[
M_{i+1} \le (M_i/\delta_i)^{O_{k}(d^{O_{k}(1)})},\quad\delta_{i+1}\ge (\delta_i/M_i)^{O_{k}(d^{O_{k}(1)})};\]
\[Q_{i+1} \le Q_i \cdot (M_i/\delta_i)^{O_{k}(d^{O_{k}(1)})},\quad N_{i+1} \ge N_i \cdot (\delta_i/M_i)^{O_{k}(d^{O_{k}(1)})},\quad Q_{i+1} \cdot N_{i+1} + \snorm{v_{i+1}}_{\infty} \le N.\]
During the iteration, we have a sequence of nilpotent Lie groups 
\[G^{0},G^{1},\ldots,G^{t},\ldots\]
such that $G^{t}$ is at most $(s-t)$-step nilpotent with associated lattice $\Gamma^{t}$ and is complexity at most $M_t$. This in particular will imply that there are at most $s$ stages in the iteration. We also maintain a sequence of subgroups
\[K^{0},\ldots,K^{t},\ldots\]
which are $M_t$-rational subgroups of $G$. 

We will define homomorphisms $\pi_{t+1}\colon G^{t}\to G^{t}/\on{ker}(\xi_t) =: \wt{G}^{t+1}$, where $\xi_t$ is a $G^{t}_{(s-t)}$-frequency (recall $H_{(i)}$ denotes the lower central series filtration of a group $H$). $G^{t+1}$ will be an appropriately rational subgroup of $\wt{G}^{t+1}$. We will always maintain the invariant that $\on{ker}(\xi_t)\cap(\pi_t\circ \cdots \circ \pi_1(T)) = \mr{Id}_{G^t}$. We will furthermore maintain that the function $F_t$ has a $\pi_t\circ\cdots\circ\pi_1(T)$-character given by descending $\xi$ on $G$ via $\pi_t\circ\cdots\circ\pi_1$.

We inductively maintain the following pair of relations:
\begin{itemize}
    \item $\pi_{t}\circ \cdots \circ \pi_1(K^{t}) = G^t$;
    \item $\pi_t \circ \cdots \circ \pi_1(g_t) = \wt{g}_t$;
\end{itemize}
where $g_t$ and $\wt{g}_t$ are polynomial sequences living in $K^t$ and $G^t$ respectively.

The iteration terminates when $G^{t}\cap (\pi_t\circ \cdots \circ \pi_1(T)) = \mr{Id}_{G^{t}}$. Before termination note that $G^{t} \cap (\pi_t\circ \cdots \circ \pi_1(T)) = \pi_t\circ \cdots \circ \pi_1(T)$ since $\pi_t\circ \cdots \circ \pi_1(T)$ is $1$-dimensional. Note that this in particular ensures that before the termination of the iteration, $\pi_t\circ \cdots \circ \pi_1(T)$ is well-defined even though $\pi_j$ is not fully defined on the image of $\pi_{j-1}$! Using the invariant that $\on{ker}(\xi_t)\cap(\pi_t\circ \cdots \circ \pi_1(T)) = \mr{Id}_{G^t}$ we also have that $\xi$ (defined on $T$) naturally descends to $G^t$. We define $J^{t} = \pi_t\circ \cdots \circ \pi_1(T)$.

Furthermore at each stage of the iteration we have that 
\[g_t = \eps_{t+1} \cdot g_{t+1} \cdot\gamma_{t+1}\]
where:
\begin{itemize}
    \item $\eps_{t+1}$ and $\gamma_{t+1}$ are polynomial sequences lying in $K_t$;
    \item $g_{t+1}$ is a polynomial sequence lying in $K_{t+1}$;
    \item $\gamma_{t+1}$ is $M_{t+1}$-rational;
    \item $\eps_{t+1}$ is $(M_{t+1},N)$-smooth.
\end{itemize}
Finally, in each stage of the iteration we will maintain a function $F_t\colon G^{t}/\Gamma^{t}\to\mb{C}$ such that
\[|\mb{E}_{\vec{n}\in v_t + Q_t \cdot [N_t]^{\ell}}[F_t(\wt{g_t}(\vec{n})\Gamma^{t})]| \ge\delta_t.\]
Throughout the iterations, nilmanifolds at stage $i$ will have complexity bounded by $M_i$, $F_i$ is $M_i$-Lipschitz, and various horizontal and vertical characters constructed will have size and height bounded by $M_i$. The starting conditions are $G^0=G$, $\Gamma^0=\Gamma$, $M_0=M$, $F_0=F$, $N_0=N$, $v_0=0$, $Q_0=1$, $\delta_0=\delta$, $K^0=G$ (and $J^0=T$), and $g_0=\wt{g}_0=g$.

\noindent\textbf{Step 2: Applying equidistribution.} We now run a single step of the iteration. We have 
\[|\mb{E}_{\vec{n}\in v_t + Q_t \cdot [N_t]^{\ell}}[F_t(\wt{g}_t(n)\Gamma^{t})]|\ge\delta_t.\]
By definition, we have that $F_t$ has a $J^t$-frequency (a descent of $\xi$); this is not sufficient to apply Theorem~\ref{thm:step-equi}. We perform an additional Fourier-analytic step to obtain a $G^t_{(s-t)}$-vertical frequency.  Since $F_t$ is $M_t$-Lipschitz, via \cite[Lemma~A.6]{Len23b} we may write
\[F_t(z\Gamma^t) = \sum_{|\xi'|\le(M_t/\delta_t)^{O_{k}(d^{O_k(1)})}}F_{\xi',t}(z\Gamma^t) + \tau(z\Gamma^t)\]
such that 
\begin{itemize}
    \item $F_{\xi',t}$ has $G^{t}_{(s-t)}$-vertical frequency $\xi'$;
    \item $\snorm{\tau}_{\infty}\le\delta_t/2$;
    \item $F_{\xi',t}$ is $(M_t/\delta_t)^{O_{k}(d^{O_k(1)})}$-Lipschitz on $G^t/\Gamma^t$.
\end{itemize}
Given this representation, recall that $F_t$ has $\xi$ (appropriately descended) as a $J^t$-vertical frequency. We abusively write this as $\xi$. Therefore
\begin{align*}
F_t(z\Gamma^t) &= \int_{g\in J^t/\Gamma^t}e(-\xi(g))F_t(zg\Gamma^t)dJ^t(g)\\
&= \sum_{|\xi'|\le(M_t/\delta_t)^{O_{k}(d^{O_k(1)})}}\int_{g\in J^t/\Gamma^t} e(-\xi(g))F_{\xi',t}(zg\Gamma^t)dJ^t(g) + \int_{g\in J^t/\Gamma^t}e(-\xi(g))\tau(zg\Gamma^t)dJ^t(g)\\
&= \sum_{|\xi'|\le(M_t/\delta_t)^{O_{k}(d^{O_k(1)})}}\wt{F}_{\xi',t}(z\Gamma^t) + \int_{g\in J^t/\Gamma^t}e(-\xi(g))\tau(zg\Gamma^t)dJ^t(g),
\end{align*}
where $dJ^t$ represents the Haar measure on $J^t/\Gamma^t$. Thus $F_t$ may be decomposed into a sum of functions with $G^{t}_{(s-t)}$-vertical characters up to an $L^\infty$ error of $\delta_t/2$. Furthermore, each vertical character $\xi'$ in question must agree with $\xi$ on $J^t\cap G^{t}_{(s-t)}$. If not, then the corresponding integral in the second line will average to $0$ and we may remove it.

Applying Pigeonhole, there exists $|\xi'|\le(M_t/\delta_t)^{O_k(d^{O_k(1)})}$ such that
\begin{equation}\label{eq:equi-deg-1}
|\mb{E}_{\vec{n}\in v_t + Q_t \cdot [N_t]^{\ell}}[\wt{F}_{\xi',t}(\wt{g}_t(\vec{n})\Gamma^{t})]|\ge(\delta_t/M_t)^{O_k(d^{O_k(1)})}.
\end{equation}
We have the following trichotomy:
\begin{itemize}
    \item $\xi'$ is nonzero and $J^t\cap G^t_{(s-t)}=J^t$;
    \item $\xi'$ is nonzero and $J^t\cap G^t_{(s-t)}=\mr{Id}_{G^t}$;
    \item $\xi'=0$ in $\wh{G^t_{(s-t)}}$.
\end{itemize}

We define $\pi_{t+1}\colon G^t\to G^t/\ker(\xi')=:\wt{G}^{t+1}$ (in particular we let $\xi_{t+1} = \xi'$). Let $\wt{\Gamma}^{t+1} := \Gamma^{t}/(\Gamma^{t}\cap\ker(\xi')) = \pi_{t+1}(\Gamma_t)$. We now apply Theorem~\ref{thm:step-equi} to \eqref{eq:equi-deg-1}, obtaining horizontal characters $\eta_1,\ldots,\eta_r\colon G^t\to\mb{R}$. Let their common kernel be $H^\ast$ and let $G^{t+1}=\pi_{t+1}(H^\ast)\leqslant\wt{G}^{t+1}$. Note that in the case when $G^{t}$ is abelian, we do not necessarily have that $\eta_i$ are trivial on $\on{ker}(\xi')$. However replacing $H^{\ast}$ by $H^{\ast}\on{ker}(\xi')$ (and abusively referring to this as $H^{\ast}$), we may then replace $\eta_i$ by $\eta_i'$ which instead cutout $H^{\ast}\on{ker}(\xi')$ and note that $\eta_i'$ may be taken to be $(M_t/\delta_t)^{O_k(d^{O_k(1)})}$-height integer combination of $\eta_i$. We abusively rename these characters as $\eta_i$ and then proceed with the proof in this edge abelian case.

By applying \cite[Lemma~A.1]{Len23b}, we obtain a factorization of $\wt{g}_t(Q_tn+v_t)$ into three nilsequences which are ``smooth'', supported on a rational subgroup, and ``rational''. We may change variables and then apply $\pi_{t+1}$ to obtain
\[\pi_{t+1}(\wt{g}_t)=:\eps_{t+1}^\ast g_{t+1}^\ast\gamma_{t+1}^\ast\]
where:
\begin{itemize}
    \item $g_{t+1}^\ast\in G^{t+1}$, and $G^{t+1}$ is at most $(s-t-1)$-step nilpotent. Furthermore $G^{t+1}$ is trivially seen to be $(M_t/\delta_t)^{O_k(d^{O_k(1)})}$-rational with respect to $\wt{G}_{t+1}$;
    \item $\gamma_{t+1}^\ast$ is an $(M_t/\delta_t)^{O_k(d^{O_k(1)})}$-rational polynomial sequence within $\wt{G}^{t+1}$;
    \item $\eps_{t+1}^\ast$ is $((M_t/\delta_t)^{O_k(d^{O_k(1)})},N)$-smooth.
\end{itemize}
We remark that changing variables is easily seen to not affect the smoothness and rationality in a substantial manner due to the bounds on $Q_t$. We can see that the step of $G^{t+1}$ decreases appropriately.

\noindent\textbf{Step 3: Lifting the factorization data.}
Note that $G^{t+1}$ can be defined via a set of horizontal characters $\eta_1',\ldots,\eta_{r'}'$ of $\wt{G}^{t+1}$ such that
\[G^{t+1} = \{x\in\wt{G}^{t+1}\colon\eta_i'(x) = 0\text{ for all } 1\le i\le r'\}\]
and we let $\Gamma^{t+1}=G^{t+1}\cap\wt{\Gamma}^{t+1}$. Note that the $\eta_i'$ are the natural descentions of $\eta_i$ as we have $\eta_i$ are trivial on $\on{ker}(\xi')$; this is precisely why we earlier modified the characters in the abelian case.  

We define 
\[K^{t+1} = \{x\in K^{t}\colon\eta_i'(\pi_{t+1}\circ\pi_{t}\circ\cdots\circ \pi_1(x)) = 0\text{ for all } 1\le i\le r'\}.\]
The trivial (but key) point is that $\pi_{t+1}\circ\pi_{t}\circ\cdots\circ \pi_1(K^{t+1})\leqslant G^{t+1}$. The key issue is noting that the map is well-defined; this is because $\pi_{t}\circ\cdots\circ\pi_1(K^{t})\leqslant G^{t}$ by induction so that we are allowed to apply $\pi_{t+1}$ to any such values. We further see that $\pi_{t+1}\circ\pi_{t}\circ\cdots\circ \pi_1(K^{t+1}) = G^{t+1}$ because $\pi_{t+1}\circ\pi_t\circ\cdots\circ\pi_1(K^t)=\pi_{t+1}(G^t)=\wt{G}^{t+1}$ and $K^{t+1}$ is the subgroup of $K^t$ such that the image under $\pi_{t+1}\circ\cdots\circ\pi_1$ is precisely in the intersection of kernels defining $G^{t+1}$ within $\wt{G}^{t+1}$.

Recall by induction that
\[\pi_t \circ \cdots \circ \pi_1(g_t) = \wt{g}_t\]
and thus 
\[\pi_{t+1} \circ \cdots \circ \pi_1(g_t) = \eps_{t+1}^\ast g_{t+1}^\ast \gamma_{t+1}^\ast.\]
Applying $\eta_i'$, we find that that there exists a nonzero integer $T_i\le (M_t/\delta_t)^{O_{k}(d^{O_{k}(1)})}$ such that 
\begin{equation}\label{eq:equi-deg-2}
\snorm{T_i \cdot\eta_i'(\pi_{t+1} \circ \cdots \circ \pi_1(g_t))}_{C^{\infty}[N]}\le(M_t/\delta_t)^{O_{k}(d^{O_{k}(1)})}.
\end{equation}

We now claim that $\eta_i'(\pi_{t+1}\circ \cdots \circ \pi_1(\cdot))$ is a horizontal character on $K^{t}$. It is a homomorphism since the $\pi_i$ are homomorphisms and it is well-defined by the above. In addition, we may inductively show that $\pi_{t + 1} \circ \cdots \circ \pi_1(\Gamma\cap K^t) = \wt{\Gamma}^{t + 1}$ and $\pi_{t+1}\circ\cdots\circ\pi_1(\Gamma\cap K^{t+1})=\Gamma^{t+1}$. Hence $\eta_i'(\pi_{t + 1} \circ \cdots \circ \pi_1(\Gamma\cap K^t))\leqslant\mb{Z}$, which verifies the property of being a horizontal character. That the horizontal character has appropriately bounded height is an immediate consequence of induction and the fact that $|\xi'|\le (M_t/\delta_t)^{O_k(d^{O_k(1)})}$.

Now we use this data to construct the required factorization. By applying \cite[Lemma~A.1]{Len23b} with the horizontal characters $T_i \cdot\eta_i'(\pi_{t+1} \circ \cdots \circ \pi_1)$ defined on $K^t$ with the hypotheses \eqref{eq:equi-deg-2}, we may write
\[g_{t} = \eps_{t+1}'g_{t+1}'\gamma_{t+1}'\]
where:
\begin{itemize}
    \item $g_{t+1}'$ takes values in $K_{t+1}$;
    \item $\eps_{t+1}'$ and $\gamma_{t+1}'$ take values in $K_{t}$;
    \item $\gamma_{t+1}'$ is an $(M_t/\delta_t)^{O_k(d^{O_k(1)})}$-rational polynomial sequence;
    \item $\eps_{t+1}'$ is $((M_t/\delta_t)^{O_k(d^{O_k(1)})},N)$-smooth.
\end{itemize}

Then $Q'$ denote the least common multiple of the periods of the $\ell$ different directions for $\gamma_{t+1}'\Gamma$; note that such periods exist and we have $Q'\le (M_t/\delta_t)^{O_k(d^{O_k(1)})}$ by \cite[Lemma~B.14]{Len23b}. Divide $v_t + Q_t \cdot [N_t]^{\ell}$ into boxes of common difference $Q_tQ'$. By Pigeonhole there exists $v'$ such that 
\[|\mb{E}_{\vec{n}\in v'+ Q'Q_t \cdot [N_t/Q']^{\ell}}[\wt{F}_{\xi',t}(\wt{g}_t(\vec{n})\Gamma^{t})]|\ge(\delta_t/M_t)^{O_k(d^{O_k(1)})}.\]
Note that 
\[\wt{g}_t = \pi_t \circ \cdots \circ \pi_1(g_t) = \pi_t \circ \cdots \circ \pi_1(\eps_{t+1}') \cdot\pi_t \circ \cdots \circ \pi_1(g_{t+1}') \cdot\pi_t \circ \cdots \circ \pi_1(\gamma_{t+1}').\]

Since the differences we are considering are divisible by $Q'$, there is $\gamma_{\mr{Rep}}$ such that
\[\gamma_{\mr{Rep}}^{-1}\gamma_{t+1}'(v'+ Q'Q_t\cdot\vec{n})\in\Gamma\]
for all $\vec{n}\in\mb{Z}^{\ell}$, where $\gamma_{\mr{Rep}}\in K^{t}$ and $d_G(\gamma_{\mr{Rep}},\mr{id}_G)\le(M_t/\delta_t)^{O_{k}(d^{O_{k}(1)})}$. Since $\pi_t \circ \cdots\circ\pi_1(\Gamma\cap K^{t})\leqslant\Gamma^{t}$ we have that 
\[|\mb{E}_{\vec{n}\in v'+ Q'Q_t \cdot [N_t/Q']^{\ell}}[\wt{F}_{\xi',t}(\pi_t \circ \cdots \circ \pi_1(\eps_{t+1}'\gamma_{\mr{Rep}}) \cdot\pi_t \circ \cdots \circ \pi_1(\gamma_{\mr{Rep}}^{-1}g_{t+1}'\gamma_{\mr{Rep}}) \Gamma^{t})]|\ge(\delta_t/M_t)^{O_k(d^{O_k(1)})}.\]

\noindent\textbf{Step 4: Completing the induction.}
The first key polynomial sequence we shall define is 
\[g_{t+1} = \gamma_{\mr{Rep}}^{-1}\cdot g_{t+1}'\cdot\gamma_{\mr{Rep}}.\]
Note that $K^{t+1}$ is normal within $K^{t}$ and since $\gamma_{\mr{Rep}}\in K^{t}$ we have that $g_{t+1}$ takes on values in $K^{t+1}$ as desired. Further let $\eps_{t+1} = \eps_{t+1}' \cdot\gamma_{\mr{Rep}}$ and $\gamma_{t+1} = \gamma_{\mr{Rep}}^{-1} \cdot\gamma_{t+1}'$; these are trivially seen to lie in $K_t$ and have the necessary rationality and smoothness properties due to the above analysis.

We now break $[N_t/Q']^{\ell}$ into a collection of boxes of length $N_{t+1} \ge N_t/Q' \cdot (M_t/\delta_t)^{-O_k(d^{O_k(1)})}$. There exists a box such that 
\[|\mb{E}_{\vec{n}\in v''+ Q'Q_t \cdot [N_{t+1}]^{\ell}}[\wt{F}_{\xi',t}(\pi_t \circ \cdots \circ \pi_1(\eps_{t+1}' \cdot\gamma_{\mr{Rep}}) \cdot\pi_t \circ \cdots \circ \pi_1(g_{t+1}) \Gamma^{t})]|\ge(\delta_t/M_t)^{O_k(d^{O_k(1)})}.\]
Taking $N_{t+1}$ sufficiently small, we may replace the initial ``smooth'' polynomial sequence $\eps_{t+1}^\ast$ by $\eps^\ast\in K^{t}$ where $d_G(\eps^{\ast},\mr{id}_G)\le (M_t/\delta_t)^{O_{k}(d^{O_{k}(1)})}$ such that 
\[|\mb{E}_{\vec{n}\in v''+ Q'Q_t \cdot [N_{t+1}]^{\ell}}[\wt{F}_{\xi',t}(\pi_t \circ \cdots \circ \pi_1(\eps^\ast) \cdot\pi_t \circ \cdots \circ \pi_1(g_{t+1}) \Gamma^{t})]|\ge(\delta_t/M_t)^{O_k(d^{O_k(1)})}.\]

The new function $F_{t+1}$ is given by descending $g\mapsto\wt{F}_{\xi',t}(\pi_t \circ \cdots \circ \pi_1(\eps^\ast) \cdot g\Gamma^t)$ from $G^t$ to $\wt{G}^{t+1}$ (and later we may implicitly restrict to $G^{t+1}$). Explicitly, for $g\in G^t$ we have
\[\wt{F}_{\xi',t}(\pi_t \circ \cdots \circ \pi_1(\eps^\ast) g \Gamma^{t+1}) = F_{t+1}(\pi_{t+1}(g)\wt{\Gamma}^{t+1})\]
which is possible because $\wt{F}_{\xi',t}$ has vertical frequency $\xi'$. Therefore we have
\begin{equation}\label{eq:equi-deg-3}
|\mb{E}_{\vec{n}\in v''+ Q'Q_t \cdot [N_{t+1}]^{\ell}}[F_{t+1}(\pi_{t+1}\circ\pi_t \circ \cdots \circ \pi_1(g_{t+1}(\vec{n})) \wt{\Gamma}^{t+1})]|\ge(\delta_t/M_t)^{O_k(d^{O_k(1)})}.
\end{equation}
We let
\[\wt{g}_{t+1} := \pi_{t+1}\circ\pi_t \circ \cdots \circ \pi_1(g_{t+1})\]
and we may replace $\wt{\Gamma}^{t+1}$ with $\Gamma^{t+1} = \wt{\Gamma}^{t+1} \cap G^{t+1}$ in \eqref{eq:equi-deg-3}.

We now check that $\on{ker}(\xi')\cap J^t = \mr{id}_{G^{t}}$, which is one of the invariants we are maintaining (we take $\xi_t=\xi'$). We will have to distinguish between cases:
\begin{itemize}
    \item If $\xi'$ is nonzero and $J^{t}\cap G^{t}_{(s-t)} = J^{t}$ note that $\on{ker}(\xi') \cap J^{t} = \mr{Id}_{G^{t}}$. This is due to the fact that $\xi'$ restricted to $J^{t}$ is (the descended version of) $\xi$ which is nonzero as given.
    \item If $\xi'$ is nonzero and $J^{t}\cap G^{t}_{(s-t)} = \mr{Id}_G^{t}$ then note that $\on{ker}(\xi')\cap J^t\leqslant J^t\cap G^{t}_{(s-t)} = \mr{Id}_{G^{t}}$.
    \item If $\xi' = 0$ then note that as $\xi$ (appropriately descended) was nonzero we have that $J^{t}\cap G^{t}_{(s-t)} = \mr{Id}_G^{t}$ is forced in this case. The result then follows as in the previous step.
\end{itemize}

Now, if $G^{t+1}\cap\pi_{t+1}\circ \cdots \circ\pi_1(T) = \pi_{t+1}\circ \cdots \circ\pi_1(T)$ then we continue with the iteration and do not terminate. If we have reached termination, we therefore have that $G^{t+1}\cap\pi_{t+1}\circ \cdots \pi_1(T) = \mr{Id}_{G^{t+1}}$. We claim that this implies that $K^{t+1}\cap T = \mr{Id}_G$ (and therefore we may take the output group to be $H = K^{t+1}$). For the sake of contradiction, instead suppose $T\leqslant K^{t+1}$ (since $T$ is $1$-dimensional). Applying $\pi_{t+1}\circ \cdots \circ \pi_1$ we have that 
\[\pi_{t+1}\circ \cdots \circ \pi_1(T)\leqslant \pi_{t+1}\circ \cdots \circ \pi_1(K^{t+1}) = G^{t+1}\]
which contradicts the termination condition.

Finally, note that if $G^{t+1}\cap\pi_{t+1}\circ \cdots \pi_1(T) = \pi_{t+1}\circ \cdots \pi_1(T)$ then $F_{t+1}$ when viewed as a function on $G^{t+1}/\Gamma^{t+1}$ is seen to have a nonzero $\pi_{t+1}\circ \cdots \pi_1(T)$ vertical character (which is given by descending $\xi$ on $G$ in through $\pi_{t+1}\circ \cdots \pi_1$ in the obvious manner), so one can continue in the iteration in this case.

\noindent\textbf{Step 5: Fixing the value at $0$.}
To see that this completes the proof, if the iteration terminates at some stage $t$ then note that 
\[g = \eps_1 \cdots \eps_{t} \cdot g_t \cdot\gamma_t \cdots\gamma_1.\]
Using that the product of smooth sequences are appropriately smooth and analogously for rational sequences allows us to deduce the necessary outputs. However, we have not guaranteed that the values of the factorization are the $\mr{id}_G$ at $0$. For this, let $g_t(0) = \{g_t(0)\}[g_t(0)]$ with $[g_t(0)]\in K^{t}\cap\Gamma$ and $d_{G}(\{g_t(0)\},\mr{id}_G)\le (M/\eps)^{O_{k}(d^{O_{k}(1)})}$. We then have that 
\[g = \eps_1 \cdots \eps_{t} \cdot\{g_t(0)\} \cdot (\{g_t(0)\}^{-1} g_t [g_t(0)]^{-1}) \cdot [g_t(0)] \cdot\gamma_t \cdots\gamma_1.\]
As $g(0) = 0$, we have that $\tau = [g_t(0)] \cdot\gamma_t(0) \cdot\cdots \gamma_1(0)$ satisfies $d_{G}(\tau,\mr{id}_G)\le (M/\eps)^{O_{k}(d^{O_{k}(1)})}$ and $\tau$ is $(M/\eps)^{O_{k}(d^{O_{k}(1)})}$-rational. Thus
\[g = \eps_1 \cdots \eps_{t} \cdot\{g_t(0)\} \tau \cdot (\tau^{-1}\{g_t(0)\}^{-1} g_t [g_t(0)]^{-1}\tau) \cdot\tau^{-1} [g_t(0)] \cdot\gamma_t \cdot\gamma_1\]
and note that $(\tau^{-1}\{g_t(0)\}^{-1} g_t [g_t(0)]^{-1}\tau)$ takes value in the conjugated subgroup $\tau^{-1} K^{t}\tau$ which is $(M/\eps)^{O_{k}(d^{O_{k}(1)})}$-rational by \cite[Lemma~B.15]{Len23b}. Note however that despite modifying the output group $H$ via conjugation, we have $\tau^{-1} K^{t}\tau \cap T = \tau^{-1} K^{t}\tau \cap\tau^{-1}T\tau = \mr{Id}_G$ as desired.
\end{proof}

We now remove the assumption of a $1$-dimensional vertical torus via a reduction to this case.
\begin{corollary}\label{cor:equi-deg}
Let $\ell\ge 1$ be an integer, $\delta\in (0,1/10)$, $M\ge 1$, and $F\colon G/\Gamma\to\mb{C}$. Suppose that $G$ is dimension $d$, is $s$-step nilpotent with a given degree $k$ filtration, and the nilmanifold $G/\Gamma$ is complexity at most $M$ with respect to this filtration.

Suppose that $T\leqslant Z(G)$ is a subgroup of the center which is $M$-rational. Further suppose that $F$ has a nonzero $T$-vertical character $\xi$ with $|\xi|\le M/\delta$, $\snorm{F}_{\mr{Lip}}\le M$, $N\ge(M/\delta)^{\Omega_{k,\ell}(d^{\Omega_{k,\ell}(1)})}$, and $g$ is a polynomial sequence with respect to the degree $k$ filtration. Then if
\[\big|\mb{E}_{\vec{n}\in[N]^{\ell}}F(g(\vec{n})\Gamma)\big|\ge\delta\]
there exists a factorization
\[g = \eps g'\gamma\]
such that:
\begin{itemize}
    \item $g'$ lives in an $(M/\delta)^{O_{k,\ell}(d^{O_{k,\ell}(1)})}$-rational subgroup $H$ such that $\xi(H\cap T) = 0$;
    \item $\gamma$ is an $(M/\delta)^{O_{k,\ell}(d^{O_{k,\ell}(1)})}$-rational polynomial sequence;
    \item $\eps$ is an $((M/\delta)^{O_{k,\ell}(d^{O_{k,\ell}(1)})},N)$-smooth polynomial sequence.
\end{itemize}
Furthermore if $g(0)=\mr{id}_G$ then we may take $\eps(0)=g'(0)=\gamma(0)=\mr{id}_G$.
\end{corollary}
\begin{proof}
We first reduce to the case where $g(0) = \mr{id}_G$ as is standard. We factor $g(0) = \{g(0)\}[g(0)]$ such that $[g(0)]\in\Gamma$ and $\psi_{G}(\{g(0)\})\in[0,1)^{\dim(G)}$. Replacing $F$ by $F(\{g(0)\} \cdot)$ and $g$ by $\{g(0)\}^{-1}g[g(0)]^{-1}$ we may clearly reduce to the case where $g(0) = \mr{id}_G$ at the cost of replacing $M$ by $M^{O_{k}(d^{O_k(1)})}$ which leaves the conclusion unchanged. 

Using Lemma~\ref{lem:quotient-rat} to bound the complexity of $G/\on{ker}(\xi)$ and noting that $F$ descends to an $(M/\delta)^{O_{k}(d^{O_k(1)})}$-Lipschitz function on $G/\on{ker}(\xi)$, by Theorem~\ref{thm:equi-deg} we have that 
\[(g \imod \on{ker}(\xi)) = \eps g' \gamma\]
where $\eps,g',\gamma$ satisfy:
\begin{itemize}
    \item $\eps(0) = g'(0) = \gamma(0) = \mr{id}_{G/\on{ker}(\xi)}$;
    \item $g'$ lives in an $(M/\delta)^{O_{k,\ell}(d^{O_{k,\ell}(1)})}$-rational subgroup $H$ such that $H\cap (T/\on{ker}(\xi)) = \mr{id}_{G/\on{ker}(\xi)}$;
    \item $\gamma$ is an $(M/\delta)^{O_{k,\ell}(d^{O_{k,\ell}(1)})}$-rational polynomial sequence;
    \item $\eps$ is an $((M/\delta)^{O_{k,\ell}(d^{O_{k,\ell}(1)})},N)$-smooth polynomial sequence.
\end{itemize}

We now ``lift'' this factorization. Consider the Mal'cev basis $\mc{X}'$ for $G/\on{ker}(\xi)$. For each element $X_i'\in\mc{X}'$ we may lift to $Z_i\in\log G$ such that:
\begin{itemize}
    \item $\exp(X_i') = \exp(Z_i)\imod\on{ker}(\xi)$;
    \item $d_{G}(\exp(Z_i),\mr{id}_G)\le (M/\delta)^{O_k(d^{O_k(1)})}$;
    \item $Z_i$ is an $(M/\delta)^{O_k(d^{O_k(1)})}$-rational combination of the elements of $\mc{X}$.
\end{itemize}
Writing $\eps$ as 
\[\eps(\vec{n}) = \exp\bigg(\sum_{|\vec{i}| \le k} \mf{\eps}_{\vec{i}} \binom{\vec{n}}{\vec{i}}\bigg)\]
where $\mf{\eps}_{\vec{i}}\in\log(G_{|\vec{i}|}/(\on{ker}(\xi)\cap G_{|\vec{i}|}))$, we lift via the above mapping on $\mc{X}'$ to 
\[\wt{\eps}(n) = \exp\bigg(\sum_{|\vec{i}| \le k} \wt{\mf{\eps}}_{\vec{i}} \binom{\vec{n}}{\vec{i}}\bigg)\]
where $\wt{\eps}_{\vec{i}}\in\log(G_{|\vec{i}|})$ and analogously for $g',\gamma$.

We easily see that $\wt{\eps}$ is an $((M/\delta)^{O_{k,\ell}(d^{O_{k,\ell}(1)})},N)$-smooth polynomial sequence, that $\wt{\gamma}$ is an $(M/\delta)^{O_{k,\ell}(d^{O_{k,\ell}(1)})}$-rational polynomial sequence, and that $\wt{g}'$ takes values in the subgroup $H' = \exp(\log(H) + \log(\on{ker}(\xi)))$. Furthermore $H'$ is seen to be $(M/\delta)^{O_{k,\ell}(d^{O_{k,\ell}(1)})}$-rational and $\xi(H'\cap T) = 0$. Finally note that $\wt{\eps} \imod \on{ker}(\xi) = \eps$ and analogously for $\wt{g}',\wt{\gamma}$. Therefore
\[g \cdot (\wt{\eps} \wt{g}' \wt{\gamma})^{-1} \equiv \mr{id}_G \imod \on{ker}(\xi)\]
as polynomial sequences. Thus 
\[g = g \cdot (\wt{\eps} \wt{g}' \wt{\gamma})^{-1}\cdot (\wt{\eps} \wt{g}' \cdot\wt{\gamma})=\wt{\eps} \cdot ((g \cdot (\wt{\eps} \wt{g}' \wt{\gamma})^{-1}) \cdot\wt{g}') \cdot\wt{\gamma}\]
gives the desired factorization noting that $\on{ker}(\xi)\leqslant H'$ and $\on{ker}(\xi)$ is central and therefore $g \cdot (\wt{\eps}\wt{g}'\wt{\gamma})^{-1}$ may be commuted to the right.
\end{proof}

\section{Setup for Sunflower and Linearization Iteration}\label{sec:setup}
We now set up the iteration which will take up the bulk of the following four sections. The idea is to inductively assume the statement of Theorem~\ref{thm:main} for $s-1$ (i.e., the quantitative inverse theorem for the $U^s[N]$-norm) and the remaining goal is to prove it for $s$. The key step is to show that for many $h\in[N]$, $\Delta_h f$ correlates with a multidegree $(1,s-1)$ nilcharacter; this is a quantitative version of \cite[Theorem~7.1]{GTZ12}. For the remainder of the analysis until Section~\ref{sec:sym} we will be concerned with the notion of a correlation structure, which can be thought of as refining the notion in Definition~\ref{def:corr} with intermediate bracket information.
\begin{definition}\label{def:correlation-structure}
A \emph{correlation structure} associated to the function $f\colon[N]\to\mb{C}$ with parameters $\rho$, $M$, $d$, and $D$ and degree-rank $(s-1,r^\ast)$ is the following data:
\begin{itemize}
    \item A subset $H\subseteq[N]$ such that $|H|\ge\rho N$;
    \item A multidegree $(1,s-1)$ nilcharacter $\chi(h,n)$ that lives on a nilmanifold $G^\ast/\Gamma^\ast$ where $\chi$ has a $G^\ast_{(1,s-1)}$-vertical frequency $\eta^\ast$. Furthermore $G^\ast/\Gamma^\ast$ has dimension bounded by $d$ and complexity bounded by $M$, the function $F^\ast$ underlying $\chi$ is $M$-Lipschitz, $\eta^\ast$ has height bounded by $M$, and the output dimension of $\chi$ is bounded by $D$. We let $g(h,n)$ denote the underlying polynomial sequence of $\chi$;
    \item A collection of degree-rank $(s-1,r^\ast)$ nilcharacters $\chi_h(n)$ which live on $G/\Gamma$ where every $\chi_h$ has the same $G_{(s-1,r^\ast)}$-vertical frequency $\eta$. Furthermore $G/\Gamma$ has dimension bounded by $d$ and complexity bounded by $M$ (with Mal'cev basis $\mc{X}$), the function underlying $\chi_h$ is $M$-Lipschitz, $\eta$ has height bounded by $M$, and $\chi_h$ has output dimension bounded by $D$. We let $g_h(n)$ denote the polynomial sequence underlying $\chi_h$. Finally, the function underlying $\chi_h$, which we will denote $F$, is independent of $h$;
    \item The polynomial sequences satisfy $g_h(0) = \mr{id}_G$;
    \item For all $h\in H$ we have
    \[\Delta_h f(n) \otimes \ol{\chi(h,n)} \otimes \ol{\chi_h(n)} \in\on{Corr}(s-2,\rho, M, d).\]
\end{itemize}
\end{definition}

If the input function $f$ we are considering for the proof of Theorem~\ref{thm:main} satisfies 
\[\snorm{f}_{U^{s+1}[N]}\ge\delta,\]
then our proof will always maintain bounds of the form
\[\rho^{-1}, M, D \le\exp(\log(1/\delta)^{O_s(1)})\text{ and } d\le\log(1/\delta)^{O_s(1)}\]
on intermediate correlation structures, although the precise dependence may decay over roughly $s$ stages (wherein we reduce $r^\ast$ from $s-1$ to $0$).

To get started, we first note that given a function $f$ with large $U^{s+1}$-norm we may associate to it a correlation structure of degree-rank $(s-1,s-1)$; this is little more than chasing definitions and applying induction.

\begin{lemma}\label{lem:base-case}
Fix $\delta\in (0,1/2)$ and $s\ge 2$. Assume Theorem~\ref{thm:main} for $s-1$. Let $f\colon[N]\to\mb{C}$ be a $1$-bounded function such that 
\[\snorm{f}_{U^{s+1}[N]}\ge\delta.\]
Then there exists a degree-rank $(s-1,s-1)$ correlation structure associated to $f$ with parameters $\rho$, $M$, $d$, and $D$ such that 
\[\rho^{-1}, M, D \le\exp(\log(1/\delta)^{O_s(1)})\emph{ and } d\le\log(1/\delta)^{O_s(1)}.\]
\end{lemma}
\begin{proof}
Note that $\snorm{f}_{U^{s+1}[N]}\ge\delta$ implies that 
\[\mb{E}_{h\in[N]}\snorm{\Delta_h f}_{U^s[N]}^{2^s}\ge\delta^{O_s(1)};\]
this implicitly uses that $\snorm{\Delta_h f}_{U^s[N]} = \snorm{\Delta_{-h} f}_{U^s[N]}$ and that $\Delta_h f$ is identically zero for $|h|>N$. 

Therefore there exists $H\subseteq[N]$ with $|H|\ge\delta^{O_s(1)}N$ such that 
\[\snorm{\Delta_hf}_{U^s[N]}^{2^s}\ge\delta^{O_s(1)}\]
for $h\in H$.

By induction on Theorem~\ref{thm:main}, we may assume that for all such $h\in H$ there exists $G_h/\Gamma_h$ with degree $s-1$ filtration and an associated polynomial sequence $g_h(\cdot)$ such that 
\[\mb{E}_{h\in[N]}[\Delta_h f(n) \ol{F_h(g_h(n)\Gamma)}]\ge\rho\]
where $G_h/\Gamma_h$ has complexity bounded by $M$ and dimension bounded by $d$. We may take 
\[M,\rho^{-1}\le\exp(\log(1/\delta)^{O_s(1)})\text{ and } d\le\log(1/\delta)^{O_s(1)}.\]
Note that via writing $g_h(0) = \{g_h(0)\}[g_h(0)]$ where $\psi_{G_h}(\{g_h(0)\})\in[0,1)^{\dim(G_h)}$ and $[g_h(0)]\in\Gamma_h$, we have that 
\begin{align*}
F_h(g_h(n)\Gamma) &= F_h(\{g_h(0)\}\{g_h(0)\}^{-1}g_h(n)[g_h(0)]^{-1}\cdot [g_h(0)]\Gamma) \\
&= F_h(\{g_h(0)\}\{g_h(0)\}^{-1}g_h(n)[g_h(0)]^{-1}\Gamma)
\end{align*}
Note that $g_h'(n) = \{g_h(0)\}^{-1}g_h(n)[g_h(0)]^{-1}$ has $g_h'(0) = \mr{id}_{G_h}$ and $F_h' = F_h(\{g_h(0)\}\cdot)$ is appropriately Lipschitz (as $\{g_h(0)\}$ has appropriately bounded coordinates by \cite[Lemma~B.2]{Len23b}). Therefore without loss we may assume that $g_h(0) = \mr{id}_G$ for all $h\in H$. 

Next note that there are only 
$O_s(M)^{O_s(d^{O(1)})}$ nilmanifolds of dimension at most $d$ with degree $(s-1)$ filtration of complexity bounded by $M$ (up to isomorphism). This follows from Lie's third theorem on the correspondence between Lie algebras and connected, simply connected Lie groups and counting the total possible number of different structure constants and filtration choices for the Lie algebra. Therefore by Pigeonhole we may assume, at the cost of decreasing the size of set $H$ by a multiplicative factor of $O_s(M)^{O_s(-d^{O(1)})}$, that $G_h/\Gamma_h = G/\Gamma$ (and the corresponding filtration) is independent of $h\in H$.

We next remove the dependence on $h$ for the function $F_h$. Let $\gamma$ be a parameter to be chosen later; by applying Lemma~\ref{lem:nilmanifold-partition-of-unity} we may write 
\[F_h(g\Gamma) = \sum_{j\in I}\tau_j(g\Gamma)^2 \cdot F_h(g\Gamma)\]
where $|I|\le (1/\gamma)^{O_s(d^{O_s(1)})}$, every $g\Gamma$ is supported on at most $2^{O_s(d)}$ many terms, and $\tau_j$ are $(M/\gamma)^{O_s(d^{O_s(1)})}$-Lipschitz. Furthermore each $\tau_j$ is supported on a width $2\gamma$ cube near the origin (in Mal'cev coordinates); see the third item of Lemma~\ref{lem:nilmanifold-partition-of-unity} for a precise description. Since $F$ is an $M$-Lipschitz function, and choosing $\gamma$ to be sufficiently small with respect to $(\rho/M)^{O_s(d^{O_s(1)})}$, we find that
\[\sup_{g\in G}|F_h(g\Gamma) - \sum_{j\in I}a_j\tau_j(g\Gamma)^2| \le\rho/2\]
by taking $a_j$ to be the mean of $F_h$ on the support of $\tau_j$. Note that $|a_j|\le M$. Pigeonholing over $j\in I$ and decreasing $\rho$ and the size of $H$ by appropriate factors of $O_s(M)^{O_s(-d^{O_s(1)})}$, we may assume that $F_h = F$ for all $h\in H$.

We finally want to replace $F$ by a nilcharacter with a vertical frequency and the claimed output dimension bound. We first give $G$ a degree-rank $(s-1,s-1)$ filtration induced by its degree $s-1$ filtration. This is done via \cite[Example~6.11]{GTZ12} (i.e., $G_{(d,r)}$ is generated by iterated commutators which either have filtration depths adding to greater than $d$ or adding to exactly $d$ with at least $r$ participating elements). Lemma~\ref{lem:commute} guarantees each subgroup is $M^{O_s(d^{O_s(1)})}$-rational. Via \cite[Lemma~B.11]{Len23b}, we may give $G$ a Mal'cev basis adapted to this degree-rank $(s-1,s-1)$ filtration with complexity $M^{O_s(d^{O_s(1)})}$.

Via Fourier expansion (see \cite[Lemma~A.6]{Len23b}) and the triangle inequality we may additionally assume that $F$ has a vertical $G_{(s-1,s-1)}$-frequency $\eta$ \footnote{We apply \cite[Lemma~A.6]{Len23b} to the degree filtration $G_{(0,0)} = G_{(1,0)}\geqslant G_{(2,0)}\geqslant \cdots \geqslant G_{(s-1,0)} \geqslant G_{(s-1,s-1)}\geqslant\mr{Id}_G$.} with height at most $O_s(M/\rho)^{O_s(d^{O_s(1)})} = \exp(\log(1/\delta)^{O_s(1)})$. Given $F$, there exists a nilcharacter $F_{\eta}$ by Lemma~\ref{lem:nil-exist} with vertical frequency $\eta$, output dimension bounded by $2^{O_s(d)}$, and such that each coordinate is $O_s(M)^{O_s(d^{O_s(1)})}$-Lipschitz. The function $(F/(2\snorm{F}_{\infty}), F_{\eta} \cdot\sqrt{1-|F/(2\snorm{F}_{\infty})|^2})$ demonstrates that without loss of generality, we may assume $F$ is a coordinate of a nilcharacter.

To complete the deduction, we take $G^\ast/\Gamma^\ast$ to be the trivial nilmanifold and $g(h,n)$ to be a constant sequence. 
\end{proof}

The heart of this paper is the following quantification of \cite[Theorem~7.2]{GTZ12}, the proof of which is the goal of the next few sections culminating in Section~\ref{sub:proof-of-induc}.
\begin{lemma}\label{lem:induc}
Fix $s\ge 2$ and $1\le r^\ast\le s-1$. Suppose $f\colon[N]\to\mb{C}$ is a $1$-bounded function and $N\ge\exp(\Omega_s((d\log(MD/\rho))^{\Omega_s(1)}))$.

Furthermore suppose that there exists a degree-rank $(s-1,r^\ast)$ correlation structure associated to $f$ with parameters $\rho$, $M$, $d$, and $D$. Then there exists a degree-rank $(s-1,r^\ast-1)$ correlation structure associated to $f$ with parameters $\rho'$, $M'$, $d'$, and $D'$ such that 
\[\rho'^{-1},M',D'\le\exp(O_s((d\log(MD/\rho))^{O_s(1)})) \emph{ and }d'\le O_s((d\log(MD/\rho))^{O_s(1)}).\]
\end{lemma}

Combining Lemma~\ref{lem:induc} along with the observation that degree-rank $(s-1,0)$ nilmanifolds induce a degree $(s-2)$ filtration (coming from the groups $G_{(i,0)}$), we immediately obtain the following. In particular, these can now be ``hidden'' inside the nilmanifolds implicit in $\on{Corr}(\cdot,\cdot,\cdot,\cdot)$.

\begin{theorem}\label{thm:iter-end}
Fix $\delta\in (0,1/2)$ and $s\ge 2$. Assume Theorem~\ref{thm:main} for $s-1$. Let $f\colon[N]\to\mb{C}$ be a $1$-bounded function such that 
\[\snorm{f}_{U^{s+1}[N]}\ge\delta.\]

Then the following data exists:
\begin{itemize}
    \item A subset $H\subseteq[N]$ of size at least $\rho N$;
    \item A multidegree $(1,s-1)$ nilcharacter $\chi(h,n)$ which lives on a nilmanifold $G^\ast/\Gamma^\ast$ where $\chi$ has a $G^\ast_{(1,s-1)}$-vertical frequency $\eta^\ast$. Furthermore $G^\ast/\Gamma^\ast$ has dimension bounded by $d$ and complexity bounded by $M$, the function underlying $\chi$ is $M$-Lipschitz, $\eta^\ast$ has height bounded by $M$, and the output dimension of $\chi$ is bounded by $D$;
    \item For all $h\in H$ we have that 
    \[\Delta_h f(n) \otimes \ol{\chi(h,n)}\in\on{Corr}(s-2,\rho, M, d).\]
\end{itemize}
Furthermore, we can find such data satisfying
\[\rho^{-1},M,D\le\exp(\log(1/\delta)^{O_s(1)})\emph{ and }d\le\log(1/\delta)^{O_s(1)}.\]
\end{theorem}
\begin{remark*}
The case when $N$ is small (i.e., $N\le\exp(\log(1/\delta)^{O_s(1)})$) is handled via noting that $\snorm{\Delta_h f}_{L^2[N]} \ge\exp(\log(1/\delta)^{O_s(1)})\cdot N^{-O(1)}$ for many $h$ and then applying Fourier analysis. Such an analysis always loses factors of $N$ and thus is only useful in this crude edge case. We will not comment further on such issues.
\end{remark*}

\section{On a Cauchy--Schwarz Argument of Gowers}\label{sec:gowers}
The proof of Lemma~\ref{lem:induc} is performed in a sequence of stages. We first deduce that the functions correlating with $\Delta_h f$ are not arbitrary. Indeed for many additive quadruples $(h_1,h_2,h_3,h_4)$ we have that the associated tensor product of $\chi_h(n)$ exhibits correlation with a degree $(s-2)$ nilsequence.

We first need the following elementary Fourier-analytic lemma which converts correlation on long progressions to correlation with a major-arc Fourier phase; this is essentially \cite[Lemma~3.5(ii)]{GTZ11}.
\begin{lemma}\label{lem:major-arc}
Let $\delta\in (0,1/2)$. Suppose that $g\colon[N]\to\mb{C}$ is $1$-bounded and there exists an arithmetic progression $P$ of length $\delta N$ with common difference $q$ within $[N]$ such that 
\[\big|\mb{E}_{n\in P} g(n)\big|\ge\delta.\]
Then there exists $\Theta\in\mb{R}$ such that $\snorm{q\Theta}_{\mb{R}/\mb{Z}}\le\delta^{-O(1)}N^{-1}$ and  
\[\big|\mb{E}_{n\in[N]}e(\Theta n) g(n)\big|\ge\delta^{O(1)}N.\] 
\end{lemma}
\begin{proof}
Extend $g$ to be zero beyond the interval $[N]$. Let $P'$ be the arithmetic progression of length $\delta^{2}N$ with common difference $q$ centered at $0$. We have 
\[\bigg|\sum_{n\in\mb{Z}}(\mbm{1}_{P}\ast (|P'|^{-1} \mbm{1}_{P'}))(n) g(n)\bigg|\ge\delta^{O(1)}N.\]
Via Fourier inversion, we have
\[\bigg|\int_{\Theta\in\mb{T}}\wh{g}(\Theta)\ol{\wh{\mbm{1}_{P}}(\Theta) \wh{\mbm{1}_{P'}}(\Theta)} d\Theta\bigg| \ge\delta^{O(1)} N^2.\]
Now via standard bounds on linear exponential sums, we have 
\begin{align*}
|\wh{\mbm{1}_{P}}(\Theta)|, |\wh{\mbm{1}_{P'}}(\Theta)|&\lesssim\min(\snorm{q\Theta}_{\mb{R}/\mb{Z}}^{-1},N).
\end{align*}
Since $|\wh{g}(\Theta)|\le N$, we have that
\[\bigg|\int_{\snorm{q\Theta}_{\mb{R}/\mb{Z}}\ge T/N}\wh{g}(\Theta)\wh{\mbm{1}_{P}}(\Theta)\wh{\mbm{1}_{P'}}(\Theta) d\Theta \bigg|\lesssim N^2/T.\]
Therefore, taking $T = \delta^{-O(1)}$ sufficiently large we have that 
\[N^2\int_{\snorm{q\Theta}_{\mb{R}/\mb{Z}}\le T/N}|\wh{g}(\Theta)|d\Theta\ge\bigg|\int_{\snorm{q\Theta}_{\mb{R}/\mb{Z}}\le T/N}\wh{g}(\Theta)\wh{\mbm{1}_{P}}(\Theta)\wh{\mbm{1}_{P'}}(\Theta) d\Theta\bigg|\ge\delta^{O(1)}N^2.\]
Thus 
\[\sup_{\snorm{q\Theta}_{\mb{R}/\mb{Z}}\le T/N}|\wh{g}(\Theta)|\ge\delta^{O(1)}T^{-1}N,\]
which is exactly the desired conclusion (recalling that $T = \delta^{-O(1)}$). 
\end{proof}

The following lemma is due ultimately to Gowers but essentially appears as \cite[Proposition~6.1]{GTZ11}. We include the proof for the sake of completeness.
\begin{lemma}\label{lem:CS-basic}
Suppose $\delta\in(0,1/2)$, $f_1,f_2\colon[N]\to\mb{C}$ are $1$-bounded, and $\chi_h\colon\mb{Z}\to\mb{C}$ are all $1$-bounded. Suppose that 
\[\mb{E}_{h\in[N]}|\mb{E}_{n\in[N]}f_2(n)\Delta_hf_1(n)\ol{\chi_h(n)}|\ge\delta.\]
Then there exists $\Theta$ such that 
$\snorm{\Theta}_{\mb{R}/\mb{Z}}\le\delta^{-O(1)}/N$ and 
\[\mb{E}_{\substack{h_1 + h_2 = h_3 + h_4\\ h_i\in[N]}}\bigg|\mb{E}_{n\in[N]}\chi_{h_1}(n)\chi_{h_2}(n + h_1 - h_4)\ol{\chi_{h_3}(n)} \ol{\chi_{h_4}(n+ h_1 - h_4)} \cdot e\big(\Theta  n \big)\bigg|\ge\delta^{O(1)}.\]
\end{lemma}
\begin{proof}
Note that we assume that $\chi_h(n) = 0$ for $h\notin[N]$ and that $\chi_h(n) = 0$ for $n\notin[N]$ via replacing $\chi_h(n)$ with $\chi_h(n) \cdot\mbm{1}_{n\in[N]}$; we will remove this truncation at the end of the argument. We extend these functions by $0$ to $\mb{Z}/\wt{N}\mb{Z}$ where $\wt{N}$ is a prime between $4N$ and $8N$.

By Cauchy--Schwarz, we have 
\[\mb{E}_{h\in\mb{Z}/\wt{N}\mb{Z}}|\mb{E}_{n\in\mb{Z}/\wt{N}\mb{Z}}f_2(n)\Delta_h f_1(n) \ol{\chi_h(n)}|^2\gg\delta^2.\]
Expanding, this is equivalent to 
\[\mb{E}_{h\in\mb{Z}/\wt{N}\mb{Z}}\mb{E}_{n_1,n_2\in\mb{Z}/\wt{N}\mb{Z}}f_2(n_1)f_1(n_1)\ol{f_1(n_1+h)}\ol{f_2(n_2)f_1(n_2)}f_1(n_2+h)\ol{\chi_h(n_1)}\chi_h(n_2)\gg\delta^2.\]
We set $n = n_1$, $k = n_2 - n_1$, and $m = n_1 + h$ and find that 
\[\mb{E}_{m,n\in\mb{Z}/\wt{N}\mb{Z},k\in\mb{Z}/\wt{N}\mb{Z}}\Delta_k(f_2f_1)(n) \Delta_k \ol{f_1(m)} \Delta_k \ol{\chi_{m-n}(n)}\gtrsim \delta^2.\]
This implies that 
\[\mb{E}_{k\in\mb{Z}/\wt{N}\mb{Z}}|\mb{E}_{m,n\in\mb{Z}/\wt{N}\mb{Z}}\Delta_k(f_2f_1)(n) \Delta_k \ol{f_1(m)} \Delta_k \ol{\chi_{m-n}(n)}|^{4}\gtrsim \delta^8.\]
Recall the box-norm inequality that for $a,b,\Phi$ which are $1$-bounded, we have 
\begin{align}
|\mb{E}_{n,m\in\mb{Z}/\wt{N}\mb{Z}}a(n)b(m)\Phi(n,m)|^{4}&\le\big(\mb{E}_{n\in\mb{Z}/\wt{N}\mb{Z}}|\mb{E}_{m\in\mb{Z}/\wt{N}\mb{Z}}b(m)\Phi(n,m)|\big)^{4}\notag\\
&\le\big(\mb{E}_{n\in\mb{Z}/\wt{N}\mb{Z}}|\mb{E}_{m\in\mb{Z}/\wt{N}\mb{Z}}b(m)\Phi(n,m)|^2\big)^{2}\notag\\
& = \big(\mb{E}_{n\in\mb{Z}/\wt{N}\mb{Z}}\mb{E}_{m,m'\in\mb{Z}/\wt{N}\mb{Z}}b(m)\ol{b(m')}\Phi(n,m)\ol{\Phi(n,m')}|\big)^{2}\notag\\
& = \big(\mb{E}_{m,m'\in\mb{Z}/\wt{N}\mb{Z}}|\mb{E}_{n\in\mb{Z}/\wt{N}\mb{Z}}\Phi(n,m)\ol{\Phi(n,m')}|\big)^{2}\notag\\
& \le\mb{E}_{m,m'\in\mb{Z}/\wt{N}\mb{Z}}|\mb{E}_{n\in\mb{Z}/\wt{N}\mb{Z}}\Phi(n,m)\ol{\Phi(n,m')}|^2\notag\\
& = \mb{E}_{n,n',m,m'\in\mb{Z}/\wt{N}\mb{Z}}\Phi(n,m)\ol{\Phi(n,m')}\ol{\Phi(n',m)}\Phi(n',m')\big).\label{eq:box-norm-inequality}
\end{align}
Applying this for each fixed $k$, we have that 
\[\mb{E}_{k\in \mb{Z}/\wt{N}\mb{Z}}\mb{E}_{n,n',m,m'\in\mb{Z}/\wt{N}\mb{Z}}\Delta_k \ol{\chi_{m-n}(n)} \Delta_k \ol{\chi_{m'-n'}(n')} \Delta_k \chi_{m'-n}(n) \Delta_k \chi_{m-n'}(n')\gtrsim \delta^8.\]
Take $m'-n = h_1$, $m-n' = h_2$, $m-n = h_3$, $m'-n' = h_4$. Note that $n'-n = h_1 - h_4$ and $h_1 + h_2 = h_3 + h_4$ and noting that $n$ and $n+k$ range over the whole cyclic group, this is exactly
\[\mb{E}_{\substack{h_1 + h_2 = h_3 + h_4\\ h_i\in \mb{Z}/\wt{N}\mb{Z}}}\big|\mb{E}_{n\in\mb{Z}/\wt{N}\mb{Z}}\chi_{h_1}(n)\chi_{h_2}(n + h_1 - h_4)\ol{\chi_{h_3}(n)} \ol{\chi_{h_4}(n+ h_1 - h_4)}\big|^2\gtrsim \delta^8.\]

Since $\chi_h(n) = 0$ identically for $h\notin [N]$, we in fact have 
\[\mb{E}_{\substack{h_1 + h_2 = h_3 + h_4\\ h_i\in[N]}}\big|\mb{E}_{n\in\mb{Z}/\wt{N}\mb{Z}}\chi_{h_1}(n)\chi_{h_2}(n + h_1 - h_4)\ol{\chi_{h_3}(n)} \ol{\chi_{h_4}(n+ h_1 - h_4)}\big|^2\gtrsim \delta^8.\]

For the inner sum, recall that we ``truncated'' $\chi_h(n)$ with $\mbm{1}_{n\in[N]}$. In particular, extracting the truncation term we have that 
\[\mb{E}_{\substack{h_1 + h_2 = h_3 + h_4\\ h_i\in[N]}}\bigg|\mb{E}_{n\in[N]}\mbm{1}_{1\le n + h_1-h_4\le N}\chi_{h_1}(n)\chi_{h_2}(n + h_1 - h_4)\ol{\chi_{h_3}(n)} \ol{\chi_{h_4}(n+ h_1 - h_4)}\bigg|^2\gtrsim \delta^8.\]
Via an application of Lemma~\ref{lem:major-arc}, there exist choices of $\Theta_{\vec{h}}$ with $\snorm{\Theta_{\vec{h}}}_{\mb{R}/\mb{Z}}\le\delta^{-O(1)}/N$ such that 
\[\mb{E}_{\substack{h_1 + h_2 = h_3 + h_4\\ h_i\in[N]}}\big|\mb{E}_{n\in[N]}\chi_{h_1}(n)\chi_{h_2}(n + h_1 - h_4)\ol{\chi_{h_3}(n)} \ol{\chi_{h_4}(n+ h_1 - h_4)} e(\Theta_{\vec{h}}n)\big|^2\gtrsim \delta^{O(1)}.\]
Rounding $\Theta_{\vec{h}}$ to a lattice of spacing $\delta^{O(1)}/N$ and Pigeonholing then gives the desired result. 
\end{proof}

The next proof will require defining the notion when two nilcharacters are ``equivalent'' (i.e., have the same symbol in a quantified sense of \cite[Appendix~E]{GTZ12}).
\begin{definition}\label{def:equiv}
We say nilcharacters $\chi,\chi'$ are \emph{$(M,D,d)$-equivalent for multidegree $J$} if $\chi,\chi'$ have output dimensions bounded by $D$ and all coordinates of 
\[\chi\otimes \ol{\chi'}\]
can be represented as sums of at most $M$ nilsequences of multidegree $J$ such that the underlying functions of each nilsequence are $M$-Lipschitz and the underlying nilmanifolds have complexity bounded by $M$ and dimension bounded by $d$.
\end{definition}

The key reason for the definition of equivalence is the following proposition, which states that given equivalent nilcharacters $\chi$ and $\chi'$, correlations with them are equivalent modulo introducing a term of multidegree $J$. This is a finitary quantification of \cite[Lemma~E.7]{GTZ12}.
\begin{lemma}\label{lem:equiv}
Given a function $f\colon\Omega\to\mb{C}^{L}$ and nilcharacters $\chi,\chi'$ which are $(M,D,d)$-equivalent for multidegree $J$, if
\[\snorm{\mb{E}_{\vec{n}\in\Omega} f(\vec{n}) \otimes \chi(\vec{n})}_{\infty}\ge\rho\]
then 
\[\snorm{\mb{E}_{\vec{n}\in\Omega} f(\vec{n}) \otimes \chi'(\vec{n}) \cdot\psi(\vec{n})}_{\infty}\ge (\rho/(MD))^{O(1)},\]
where $\psi$ can be taken to be one of the nilsequences used as part of a represention of one of the coordinates in $\chi\otimes\ol{\chi'}$. In particular, $\psi$ is a nilsequence of multidegree $J$ such that underlying nilmanifold has complexity bounded by $M$ and dimension bounded by $d$ and the underlying function has Lipschitz constant bounded by $M$.
\end{lemma}
\begin{remark*}
The additional condition that $\psi$ can be taken to be an explicit nilsequence occurring in a witness for the equivalence of $\chi,\chi'$ is used primarily to allow us to Pigeonhole the choice of $\psi$ in cases where we may need to apply this statement ``on average''.
\end{remark*}
\begin{proof}
Notice that since $\chi'$ is a nilcharacter, we have that the trace of 
\[\chi' \otimes \ol{\chi'} \]
is the constant function $1$. Furthermore note that the trace is the sum of at most $D$ coordinates of $\chi' \otimes \ol{\chi'}$ and therefore 
\[\snorm{\mb{E}_{\vec{n}\in\Omega} f(\vec{n}) \otimes \chi(\vec{n}) \otimes \ol{\chi'(\vec{n})} \otimes \chi'(\vec{n})}_{\infty}\ge\rho/D.\]
Consider the coordinate of $f(\vec{n}) \otimes \chi(\vec{n}) \otimes \ol{\chi'(\vec{n})} \otimes \chi'(\vec{n})$ which achieves the $L^\infty$ above, and in particular the associated coordinate of $\chi(\vec{n}) \otimes \ol{\chi'(\vec{n})}$ that contributes. Applying the definition of equivalence and the triangle inequality, there exists $\psi(\vec{n})$ of the desired form such that
\[\snorm{\mb{E}_{\vec{n}\in\Omega} f(\vec{n}) \otimes \chi'(\vec{n})\cdot\psi(\vec{n})}_{\infty}\ge(\rho/(MD))^{O(1)}.\qedhere\]
\end{proof}

We are now in position to prove the quantification of \cite[Proposition~7.3]{GTZ12}. We remark that there was an error in the published version of \cite[Proposition~8.3]{GTZ12} which affected the proof of \cite[Proposition~7.3]{GTZ12}. We quantify a closely related approach to that given in the erratum \cite{GTZ24}. For our proof we require various quantifications of \cite[Appendix~E]{GTZ12}; all of these are completely mechanical.

\begin{lemma}\label{lem:four-fold-biased}
Fix $s \ge 3$ and $1\le r^\ast\le s-1$. Let $f\colon[N]\to\mb{C}$ is a $1$-bounded function. Suppose that $f$ has a correlation structure with parameters $\rho$, $M$, $d$, and $D$ and associated nilcharacters $\chi(h,n)$ and $\chi_h(n)$. Then for at least $(MD/\rho)^{-O_s(d^{O_s(1)})} N^3$ quadruples $h_1,h_2,h_3,h_4\in H$ with $h_1 + h_2 = h_3 + h_4$ we have
\[\chi_{h_1}(n)\otimes \chi_{h_2}(n + h_1-h_4)\otimes \ol{\chi_{h_3}(n)}\otimes \ol{\chi_{h_4}(n + h_1-h_4)} \in\on{Corr}(s-2,\rho',M',d')\]
with 
\[\rho'^{-1}, M'\le(MD/\rho)^{O_s(d^{O_s(1)})}\emph{ and }d' \le O_s(d^{O_s(1)}).\]
\end{lemma}
\begin{remark}\label{rem:quadruple-2}
For $s=2$, the same statement holds modulo a correction term of $e(\Theta n)$ where $\Theta$ is such that $\snorm{\Theta}_{\mb{R}/\mb{Z}}\le (MD/\rho)^{O_s(d^{O_s(1)})}/N$.
\end{remark}
\begin{proof}
By definition of correlation structures we have for $h\in H$ that
\[\snorm{\mb{E}_{n\in[N]}(\Delta_h f)(n) \otimes \ol{\chi(h,n)} \otimes \ol{\chi_h(n)} \cdot\ol{\psi_h(n)}}_{\infty} \ge\rho\]
where $\psi_h$ is a nilsequence of degree $(s-2)$ whose underlying function is at most $M$-Lipschitz on a nilmanifold of complexity at most $M$ and dimension at most $d$. Setting $\chi_h(n)$ to be zero for $h\notin H$ we have
\[\mb{E}_{h\in[N]}\snorm{\mb{E}_{n\in[N]}f(n) \ol{f(n+h)} \otimes \ol{\chi(h,n)} \otimes \ol{\chi_h(n)} \cdot\ol{\psi_h(n)}}_{\infty} \ge\rho^2.\]
Twisting $\psi_h$ by an appropriate $h$-dependent constant complex phase so as to make the $L^\infty$ values be realized as positive real numbers, we may assume that 
\[\snorm{\mb{E}_{h\in[N]}\mb{E}_{n\in[N]}f(n) \ol{f(n+h)} \otimes \ol{\chi(h,n)} \otimes \ol{\chi_h(n)} \cdot\ol{\psi_h(n)}}_{\infty} \ge\rho^2/D^2.\]

By Lemma~\ref{lem:multilinear}, we have that $\chi(h,n)$ is $((MD)^{O_s(d^{O_s(1)})}, (MD)^{O_s(d^{O_s(1)})}, d^{O_s(1)})$-equivalent for degree $(s-1)$ to some $\wt{\chi}(h,n,\ldots,n)$ which is a multidegree $(1,\ldots,1)$ nilcharacter with output dimension, complexity of underlying nilmanifold, Lipschitz constant of underlying function for each coordinate, and vertical frequency height all bounded by $(MD)^{O_s(d^{O_s(1)})}$. ($\wt{\chi}$ has $s$ total arguments.) Thus, applying Lemma~\ref{lem:equiv}, we have that
\begin{align*}
\snorm{\mb{E}_{n,h\in[N]}f(n) \ol{f(n+h)} \otimes \ol{\wt{\chi}(h,n,\ldots,n)} \otimes \ol{\chi_h(n)} \cdot\ol{\psi_h(n)} \cdot \wt{\psi}(h,n)}_{\infty} \ge (MD/\rho)^{-O_s(d^{O_s(1)})},
\end{align*}
where $\wt{\psi}(h,n)$ is a degree $(s-1)$ nilsequence where the underlying function has Lipschitz norm and complexity of underlying nilmanifold bounded $(MD)^{O_s(d^{O_s(1)})}$ while the dimension of the underlying nilmanifold is bounded by $O_s(d^{O_s(1)})$. The nilsequence $\wt{\psi}(h,n)$ can also be viewed as a multidegree $(0,s-1) \cup (s-1,s-2)$ nilsequence. (I.e., we take the union of the down-sets generated by these elements.) Furthermore, the underlying function has Lipschitz norm and complexity of underlying nilmanifold bounded $(MD)^{O_s(d^{O_s(1)})}$ while the dimension of the underlying nilmanifold is bounded by $O_s(d^{O_s(1)})$.

Thus, applying Lemma~\ref{lem:split} (splitting) we have
\begin{align*}
\snorm{\mb{E}_{n,h\in[N]}f(n) \ol{f(n+h)} \otimes \ol{\wt{\chi}(h,n,\ldots,n)} \otimes \ol{\chi_h(n)} \cdot\ol{\wt{\psi_h}(n)} \cdot b(n)}_{\infty} \ge (MD/\rho)^{-O_s(d^{O_s(1)})}
\end{align*}
where $\wt{\psi_h}$ are degree $(s-2)$ nilsequences in $n$ where complexity and Lipschitz constant are bounded by $(MD)^{O_s(d^{O_s(1)})}$ and the dimension of the underlying nilmanifold is bounded by $O_s(d^{O_s(1)})$ while $b(n)$ is $(MD)^{O_s(d^{O_s(1)})}$-bounded. Therefore, applying Lemma~\ref{lem:CS-basic}, we have
\begin{align*}
&\mb{E}_{\substack{h_1 + h_2 = h_3 + h_4\\ h_i\in[N]}}\snorm{\mb{E}_{n\in[N]}\wt{\chi}(h_1,n,\ldots,n) \otimes \wt{\chi}(h_2,n + h_1 -h_4,\ldots,n + h_1 -h_4) \otimes \ol{\wt{\chi}(h_3,n,\ldots,n)} \\
&\otimes \ol{\wt{\chi}(h_4,n + h_1 -h_4,\ldots,n + h_1 -h_4)} \otimes \chi_{h_1}(n) \otimes \chi_{h_2}(n+h_1-h_4) \otimes \ol{\chi_{h_3}(n)} \otimes \ol{\chi_{h_4}(n+h_1-h_4)}\\
&\cdot \wt{\psi_{h_1}}(n)\wt{\psi_{h_2}}(n)\ol{\wt{\psi_{h_3}}(n)}\ol{\wt{\psi_{h_4}}(n)}e(\Theta n)}_{\infty} \ge (MD/\rho)^{-O_s(d^{O_s(1)})}.
\end{align*}
We may combine $\wt{\psi_{h_1}}(n)\wt{\psi_{h_2}}(n)\ol{\wt{\psi_{h_3}}(n)}\ol{\wt{\psi_{h_4}}(n)}e(\Theta n)$ to form $\psi_{h_1,h_2,h_3,h_4}^\ast(n)$ which is degree $(s-2)$ in $n$ and with identical complexity bounds to $\wt{\psi_{h_1}}$ modulo changing implicit constant. Additionally, we may twist $\psi_{h_1,h_2,h_3,h_4}^\ast$ by an $(h_1,h_2,h_3,h_4)$-dependent complex phase to bring the outer expectation inside the norm. Thus we have
\begin{align*}
&\snorm{\mb{E}_{\substack{h_1 + h_2 = h_3 + h_4\\ h_i\in[N]}}\mb{E}_{n\in[N]}\wt{\chi}(h_1,n,\ldots,n) \otimes \wt{\chi}(h_2,n + h_1 -h_4,\ldots,n + h_1 -h_4) \otimes \ol{\wt{\chi}(h_3,n,\ldots,n)} \\
&\qquad\otimes \ol{\wt{\chi}(h_4,n + h_1 -h_4,\ldots,n + h_1 -h_4)}\otimes \chi_{h_1}(n) \otimes \chi_{h_2}(n+h_1-h_4) \otimes \ol{\chi_{h_3}(n)}\\
&\qquad\otimes \ol{\chi_{h_4}(n+h_1-h_4)} \cdot \psi^{\ast}_{h_1,h_2,h_3,h_4}(n)}_{\infty} \ge (MD/\rho)^{-O_s(d^{O_s(1)})}.
\end{align*}

By Lemma~\ref{lem:multilinear}, $\chi(h_2,n+h_1-h_4,\ldots,n+h_1-h_4)$ is $((MD)^{O_s(d^{O_s(1)})}, (MD)^{O_s(d^{O_s(1)})}, d^{O_s(1)})$-equivalent for degree $(s-1)$ to 
\[\bigotimes_{k=0}^{s-1}\chi(h_2,n,\ldots,n,h_1-h_4,\ldots,h_1-h_4)\]
where there are $s-k-1$ copies of $n$ and $k$ copies of $h_1-h_4$ and we have a similar expansion for $\chi(h_4,n+h_1-h_4,\ldots,n+h_1-h_4)$. Note that all terms in this expansion except for $k = 0$ may be absorbed into $\psi^{\ast}$. Therefore applying Lemma~\ref{lem:equiv}, we have that 
\begin{align*}
&\snorm{\mb{E}_{\substack{h_1 + h_2 = h_3 + h_4\\ h_i\in[N]}}\mb{E}_{n\in[N]}\wt{\chi}(h_1,n,\ldots,n) \otimes \wt{\chi}(h_2,n,\ldots,n) \otimes \ol{\wt{\chi}(h_3,n,\ldots,n)} \\
&\qquad\otimes \ol{\wt{\chi}(h_4,n,\ldots,n)} \otimes \chi_{h_1}(n) \otimes \chi_{h_2}(n+h_1-h_4) \otimes \ol{\chi_{h_3}(n)}\\
&\qquad\otimes \ol{\chi_{h_4}(n+h_1-h_4)} \cdot \psi^{\ast}_{h_1,h_2,h_3,h_4}(n) \cdot \tau(n,h_1,h_2,h_3,h_4)}_{\infty} \ge (MD/\rho)^{-O_s(d^{O_s(1)})};
\end{align*}
here $\tau(n,h_1,h_2,h_3,h_4)$ is a degree $(s-1)$ nilsequence where the underlying function has Lipschitz norm and complexity of underlying nilmanifold bounded by $(MD)^{O_s(d^{O_s(1)})}$ while the dimension of the underlying nilmanifold is bounded by $O_s(d^{O_s(1)})$ and we have folded certain terms into $\psi^{\ast}$ while guaranteeing it is a degree $(s-2)$ nilsequence (and the complexity bounds have not changed modulo implicit constants). Finally via Lemma~\ref{lem:multilinear}, we have that
\[\wt{\chi}(h_1,n,\ldots,n) \otimes \wt{\chi}(h_2,n,\ldots,n) \otimes \ol{\wt{\chi}(h_3,n,\ldots,n)}\otimes \ol{\wt{\chi}(h_4,n,\ldots,n)}\]
and $\wt{\chi}(h_1+h_2-h_3-h_4,n,\ldots,n)$ are $((MD)^{O_s(d^{O_s(1)})}, (MD)^{O_s(d^{O_s(1)})}, d^{O_s(1)})$-equivalent for degree $(s-1)$. Thus applying Lemma~\ref{lem:equiv}, we have
\begin{align*}
&\mb{E}_{\substack{h_1 + h_2 = h_3 + h_4\\ h_i\in[N]}}\snorm{\mb{E}_{n\in[N]}\chi_{h_1}(n) \otimes \chi_{h_2}(n+h_1-h_4) \otimes \ol{\chi_{h_3}(n)} \otimes \ol{\chi_{h_4}(n+h_1-h_4)}\\
&\qquad\qquad\cdot\wt{\chi}(0,n,\ldots,n) \psi^{\ast}_{h_1,h_2,h_3,h_4}(n) \tau(n,h_1,h_2,h_3,h_4)}_{\infty} \ge (MD/\rho)^{-O_s(d^{O_s(1)})};
\end{align*}
here we have folded in various terms into $\tau(n,h_1,\ldots,h_4)$ and the complexity bounds have not changed modulo implicit constants. Note that by Lemma~\ref{lem:specialization}, $\wt{\chi}(0,n,\ldots,n)$ is a degree $(s-1)$ nilsequence in $n$ and thus may abusively also be absorbed into $\tau$. Finally noting that a degree $(s-1)$ nilsequence may also be viewed as a multidegree $(s-1,0,\ldots,0) \cup (s-2,s-1,\ldots,s-1)$ nilsequence and thus applying Lemma~\ref{lem:split} we have 
\begin{align*}
&\snorm{\mb{E}_{\substack{h_1 + h_2 = h_3 + h_4\\ h_i\in[N]}}\mb{E}_{n\in[N]} \chi_{h_1}(n) \otimes \chi_{h_2}(n+h_1-h_4) \otimes \ol{\chi_{h_3}(n)} \otimes \ol{\chi_{h_4}(n+h_1-h_4)}\cdot \psi^{\ast}_{h_1,h_2,h_3,h_4}(n)b(n)}_{\infty}\\
&\qquad\ge (MD/\rho)^{-O_s(d^{O_s(1)})},
\end{align*}
where $b(n)$ is an $(MD)^{O_s(d^{O_s(1)})}$-bounded function and $\psi^{\ast}$ has been modified but the underlying complexity bounds have not changed modulo implicit constants. Note $\psi^\ast$ is degree $(s-2)$.

We now reparameterize with
\[h_1 = m-n, h_2 = m'-n', h_3 = m' - n, h_4 = m-n'.\]
By approximating with regions where we take $m,m',n,n'$ to live in short intervals, there exist intervals $I_1,\ldots,I_4$ each of density $(MD/\rho)^{-O_s(d^{O_s(1)})}$ in $[\pm 2N]$ such that 
\begin{align*}
\norm{\mb{E}_{m\in I_1,m'\in I_2, n\in I_3,n'\in I_4}&\chi_{m-n}(n)\otimes\chi_{m'-n'}(n')\otimes\ol{\chi_{m'-n}(n)}\otimes\ol{\chi_{m-n'}(n')}\\
&\qquad\cdot\psi^{\ast}_{m-n,m'-n',m'-n,m-n'}(n)b(n)}_\infty\ge (MD/\rho)^{-O_s(d^{O_s(1)})}
\end{align*}
where $\psi_{m-n,m'-n',m'-n,m-n'}$ is a degree $(s-2)$ nilsequence. Now by Cauchy--Schwarz, duplicating the variable $m$ and denoting the copies by $m,m''$, we obtain
\begin{align*}
\norm{\mb{E}_{m,m''\in I_1,m'\in I_2,n\in I_3,n'\in I_4}&\chi_{m-n}(n)\otimes\ol{\chi_{m''-n}(n)}\otimes\ol{\chi_{m-n'}(n')}\otimes\chi_{m''-n'}(n')\\
&\cdot\psi^{\ast}_{m-n,m'-n',m'-n,m-n',m''-n,m''-n'}(n)}_\infty\ge(MD/\rho)^{-O_s(d^{O_s(1)})}.
\end{align*}
Note that every term not involving $m$ was removed using appropriate boundedness. Now we may Pigeonhole on $m'-n = t$ and deduce
\begin{align*}
\norm{\mb{E}_{m,m''\in I_1,n\in I_3,n'\in I_4}&\chi_{m-n}(n)\otimes\ol{\chi_{m''-n}(n)}\otimes\ol{\chi_{m-n'}(n')}\otimes\chi_{m''-n'}(n')\\
&\qquad\cdot\psi^{\ast}_{m-n,m-n',m''-n,m''-n'}(n) \cdot \mbm{1}[n+t\in I_2]}_\infty\ge(MD/\rho)^{-O_s(d^{O_s(1)})}.
\end{align*}
Let $m''-n = h_1$, $m-n' = h_2$, $m-n = h_3$, and $m''-n' = h_4$ (abusively). We have
\begin{align*}
\norm{&\mb{E}_{n\in[N]}\mb{E}_{\substack{h_1 + h_2 = h_3 + h_4\\h_i\in[\pm N]}} \chi_{h_1}(n) \otimes \chi_{h_2}(n + h_1 - h_4) \otimes \ol{\chi_{h_3}(n)} \otimes \ol{\chi_{h_4}(n + h_1 - h_4)}\\
&\quad\cdot\chi_{h_1,h_2,h_3,h_4}(n) \cdot\mbm{1}[n + h_3,n+h_1\in I_1, n\in I_3,n+h_1-h_4\in I_4, n+t\in I_2]}
_{\infty} \ge (MD/\rho)^{-O_s(d^{O_s(1)})}
\end{align*}
where $\chi_{h_1,h_2,h_3,h_4}(n)$ is a degree $(s-2)$ nilsequence (for each fixed $h_1,h_2,h_3,h_4$) where the underlying nilmanifold and Lipschitz constant of underlying function are bounded by $(MD/\rho)^{O_s(d^{O_s(1)})}$ and the dimension is bounded by $O_s(d^{O_s(1)})$.

Therefore, by the triangle inequality we have that 
\begin{align*}
\mb{E}&_{\substack{h_1 + h_2 = h_3 + h_4\\h_i\in[\pm N]}}\norm{\mb{E}_{n\in[N]} \chi_{h_1}(n) \otimes \chi_{h_2}(n + h_1 - h_4) \otimes \ol{\chi_{h_3}(n)} \otimes \ol{\chi_{h_4}(n + h_1 - h_4)}\\
&\qquad\chi_{h_1,h_2,h_3,h_4}(n) \cdot\mbm{1}[n + h_3,n+h_1\in I_1, n\in I_3,n+h_1-h_4\in I_4, n+t\in I_2]}
_{\infty} \gtrsim (MD/\rho)^{-O_s(d^{O_s(1)})}.
\end{align*}
Finally, the last term is the indicator of an $\vec{h}$-dependent interval. Applying Lemma~\ref{lem:major-arc} (and noting that $s\ge 3$ allows us to fold in the major arc Fourier term) completes the proof.
\end{proof}

\section{Sunflower Step}\label{sec:sunflower}
For the next stage of our proof, as outlined in Section~\ref{sec:outline}, we wish to provide more structure on $h$-dependent nilcharacters $\chi_h$ given information about additive quadruples as established in Section~\ref{sec:gowers}. As setup we will require the notion of a rational subspace with respect to a specified basis, and establish some basic control over Taylor coefficients of bounded polynomial sequences.
\begin{definition}\label{def:rat-subspace}
A vector subspace $V'\leqslant V$ is \emph{$Q$-rational with respect to $V$} given the basis $\mc{B}=\{B_1,\ldots,B_{\dim(V)}\}$ (of $V$) if there exists a basis $\mc{B}' = \{B_1',\ldots,B_{\dim(V')}'\}$ of $V'$ such that each $B_j'$ is a linear combination of elements of $\mc{B}$ with coefficients of height at most $Q$.
\end{definition}
\begin{lemma}\label{lem:small-Taylor}
Consider a nilmanifold $G/\Gamma$ given a degree-rank filtration of degree rank $(s,r)$, dimension $d$, and complexity at most $M$. Let $\mc{X}$ denote the underlying adapted Mal'cev basis and assign the basis 
\[\mc{X}_i = (\mc{X}\cap\log(G_{(i,1)}))/\log(G_{(i,2)})\]
for $G_{(i,1)}/G_{(i,2)}$. Suppose $\eps$ is a polynomial sequence such that
\[d_{G,\mc{X}}(\mr{id}_G,\eps(n))\le M\]
for $n\in[N]$. Then for $1\le i\le s$, we have 
\[d_{G_{(i,1)}/G_{(i,2)}, \mc{X}_i}(\on{Taylor}_i(\eps),\mr{id}_{G_{(i,1)}/G_{(i,2)}}) \le M^{O_s(d^{O_s(1)})} N^{-i}.\] 
\end{lemma}
\begin{proof}
We may write 
\[\eps(n) = \exp\bigg(\sum_{j=0}^{s}\eps_j\binom{n}{j}\bigg)\]
where $\eps_j\in\log(G_{(j,0)})$. By Lemma~\ref{lem:taylor-relation}, we have 
\[\on{Taylor}_i(\eps) = \exp(\eps_i) \imod G_{(i,2)}.\]
We have that 
\[\snorm{\psi_{\mr{exp}}(\eps(n))}_{\infty}\le M^{O_s(d^{O_s(1)})}\]
for all $n\in[N]$ by \cite[Lemmas~B.1,~B.3]{Len23b}. This implies that
\[\norm{\sum_{t=0}^{j}(-1)^{t}\binom{j}{t}\psi_{\mr{exp}}(\eps(t\cdot\lfloor N/(2j)\rfloor +1 ))}_{\infty}\le M^{O_s(d^{O_s(1)})}.\]
This is exactly the $j$-th discrete derivative and thus terms coming from $\eps_i$ with $i<j$ vanish. This implies that 
\[\snorm{\eps_j N^{j} \imod \log(G_{(j,2)})}_{\infty}\le M^{O_s(d^{O_s(1)})},\]
where the basis we assign to $\log(G_{(j,1)}/G_{(j,2)})$ is $\mc{X}_j$. The result follows by dividing by $N^{-j}$ and noting, by say \cite[Lemma~2.6]{LSS24}, that the distance in first- and second-kind coordinates is comparable. 
\end{proof}

We now come to the first of two crucial arguments in this paper where we ``improve'' the correlation structure. At the cost of restricting the set $H$, we force the Taylor coefficients of $g_h$, the polynomial sequences underlying the $\chi_h$, to live in certain restricted subspaces and their differences to lie in an even finer restriction.

This step is closely related to the ``sunflower'' arguments of \cite[Step~1]{GTZ11} and \cite[Lemma~11.3]{GTZ12}; a quantitative version for the $U^4$-inverse theorem due to the first author can be found in \cite{Len23}. The precise statement of the lemma should also be compared with \cite[Theorem~11.1(i)]{GTZ12}. We note however that unlike \cite{GTZ11,GTZ12}, our proof is completely free of any iteration (or equivalently passing to a subgroup where polynomial sequences are ``totally equidistributed'', which necessitates too much loss in the relevant parameters).

Thus, the crucial point of the following technical statement is the final condition, which essentially captures that two $h$-dependent frequencies in the improved correlation structure cannot ``simultaneouly'' affect the bottom degree-rank portion.

\begin{lemma}\label{lem:sunflower}
Fix $s\ge 2$ and $1\le r^\ast\le s-1$. Let $f\colon[N]\to\mb{C}$ be a $1$-bounded function. Suppose that $f$ has a degree-rank $(s-1,r^\ast)$ correlation structure with parameters $\rho$, $M$, $d$, and $D$ and that $N\ge (MD/\rho)^{O_s(d^{O_s(1)})}$ and data labeled as in Definition~\ref{def:correlation-structure}. Furthermore let $\mc{X}_i = (\mc{X}\cap\log(G_{(i,1)}))/\log(G_{(i,2)})$.

We output a new degree-rank $(s-1,r^\ast)$ correlation structure for $f$ with parameters
\begin{align*}
\rho'^{-1}&\le(MD/\rho)^{O_s(d^{O_s(1)})},\quad M'\le O(M),\quad D'= D,\quad d'\le O(d),
\end{align*}
with set $H'\subseteq H$, with multidegree $(1,s-1)$ nilcharacter $\chi'(h,n)={F^\ast}'(g'(h,n){\Gamma^\ast}')$ on $(G^\ast)'=G^\ast\times\mb{R}$, with $h$-dependent nilcharacters $\chi_h'$ having underlying polynomial sequences $g_h'(n)=F'(g_h'(n)\Gamma)$ on $G'=G$. This correlation structure satisfies:
\begin{itemize}
    \item $(G^\ast)'$ is given the multidegree filtration
    \[(G^\ast)'_{(i,j)} = (G^\ast)_{(i,j)} \times\{0\}\]
    if $(i,j)\neq (0,0)$ or $(0,1)$. For $(i,j)\in\{(0,0),(0,1)\}$ we set
    \[(G^\ast)'_{(i,j)} = (G^\ast)_{(i,j)} \times\mb{R}.\]
    We have ${F^\ast}'((x,z) (\Gamma^\ast\times\mb{Z})) = F^\ast(x \Gamma^\ast) \cdot e(z)$. We have $g'(h,n) = (g(h,n), \Theta n)$ for some appropriate value of $\Theta$;
    \item There exists a collection of $\mb{R}$-vector spaces $V_{i,\mr{Dep}}\leqslant V_i\leqslant G_{(i,1)}/G_{(i,2)}$ which are all $(MD/\rho)^{O_s(d^{O_s(1)})}$-rational with respect to $\exp(\mc{X}_i)$ for each $i$;
    \item For $1\le i\le s-1$ and $h,h_1,h_2\in H'$ we have 
    \[\on{Taylor}_i(g_h')\in V_i,\qquad\on{Taylor}_i(g_{h_1}') - \on{Taylor}_i(g_{h_2}')\in V_{i,\mr{Dep}};\]
    \item $F'$ is $M'$-Lipschitz and has the same vertical frequency $\eta$ as $F$;
    \item For integers $i_1 + \cdots + i_{r^\ast} = s-1$, suppose that $v_{i_\ell}\in V_{i_\ell}$ and for at least two distinct indices $\ell_1,\ell_2$ we have $v_{i_{\ell_1}}\in V_{i_{\ell_1},\mr{Dep}}$ and $v_{i_{\ell_2}}\in V_{i_{\ell_2},\mr{Dep}}$. Then for $w$ which is any $(r^\ast-1)$-fold commutator of $v_{i_1},\ldots,v_{i_{r^\ast}}$, we have
    \[\eta(w) = 0.\]
\end{itemize}
\end{lemma}
\begin{remark*}
Consider $g_{i_j} = \exp(X_{i_j})$ with $g_{i_j}\in G_{i_j,0}$ for $1\le j\le r^\ast$ and $i_1 + \cdots + i_{r^\ast} = s-1$. Fixing any $(r^\ast-1)$-fold commutator $w$ of $g_{i_1},\ldots,g_{i_{r^\ast}}$, repeated application of the commutator version of Baker--Campbell--Hausdorff (e.g.~\eqref{eq:bch-com}) implies that
\[w = \exp([X_{i_1},\ldots,X_{i_{r^\ast}}])\]
where the associated commutator has the same ``form'' as that defining $w$. (All higher terms are annihilated since $G$ has degree-rank $(s-1,r^\ast)$.) Note that this implies that one can define the associated commutator given inputs in $G_{(i_1,1)}/G_{(i_1,2)},\ldots,G_{(i_{r^\ast},1)}/G_{(i_{r^\ast},2)}$ and furthermore we see that the associated commutator form on the Lie algebra is a multilinear form of the vector arguments (since $G_{(i,1)}/G_{(i,2)}$ and $G_{(s-1,r^\ast)}$ are real vector spaces and the commutator bracket on the Lie algebra is multilinear).
\end{remark*}
\begin{proof}
We first note that the statement of the lemma is trivial for $r^\ast = 1$ since we may take $V_i = V_{i,\mr{Dep}} = G_{(i,1)}'/G_{(i,2)}'$; it is impossible to have two distinct indices in the final bullet point. Taking $g_h'(n) = g_h(n)$ and $g'(h,n) = (g(h,n), 0)$ completes the proof in this case. For $s = 2$, the only possible case is $r^\ast = 1$ and therefore for the remainder of the proof we will consider $s\ge 3$. Similarly, if $\eta$ is trivial, the result is once again immediate. Thus throughout the remainder of the proof we will assume that $s-1\ge r^\ast\ge 2$ and $\eta$ is nontrivial.

\noindent\textbf{Step 1: Setup for invoking equidistribution theory.}
By Lemma~\ref{lem:four-fold-biased}, we have
\begin{align*}
\norm{\mb{E}\bigg[\chi_{h_1}(n)\otimes \chi_{h_2}(n + h_1-h_4)\otimes \ol{\chi_{h_3}(n)}\otimes \ol{\chi_{h_4}(n + h_1-h_4)}\cdot\psi_{\vec{h}}(g_{\vec{h}}(n)\Gamma')\bigg]}_{\infty} \ge (MD/\rho)^{-O_s(d^{O_s(1)})}
\end{align*}
for at least $(MD/\rho)^{-O_s(d^{O_s(1)})}$ fraction of additive quadruples $h_1+h_2=h_3+h_4$. Furthermore $g_{\vec{h}}(n)$ is a polynomial sequence on a group $G_{\mr{Error}}$ which has a degree $(s-2)$ filtration, dimension bounded by $O_s(d^{O_s(1)})$, and the complexity of $G_{\mr{Error}}/\Gamma_{\mr{Error}}$ and the Lipschitz constant of the function for $\psi_{\vec{h}}$ are bounded by $(MD/\rho)^{O_s(d^{O_s(1)})}$. Note that a priori $G_{\mr{Error}}/\Gamma_{\mr{Error}}$ and the associated Mal'cev basis depend on $\vec{h}$. However, applying Pigeonhole on the choice of the associated structure constants allows us to assume, at the cost of passing to a density $(MD/\rho)^{-O_s(d^{O_s(1)})}$ subset of the additive quadruples, that $G_{\mr{Error}}/\Gamma_{\mr{Error}}$ is independent of $\vec{h}$. Finally, we may assume as usual that $g_{\vec{h}}(0) = \mr{id}_{G_{\mr{Error}}}$ via by-now standard manipulations. 

We now consider the group $\wt{G} = G\times G \times G \times G \times G_{\mr{Error}}$. $\wt{G}$ may naturally be given a degree-rank $(s-1,r^\ast)$ product filtration (where we use \cite[Example~6.11]{GTZ12} to assign $G_{\mr{Error}}$ a degree-rank $(s-2,s-2)$ structure) and Mal'cev basis. Furthermore if $\chi_{h_i}(n) = F(g_{h_i}(n))$ we have that the five-fold function $F(x_1\Gamma)\otimes F(x_2\Gamma)\otimes\ol{F(x_3\Gamma)}\otimes\ol{F(x_4\Gamma)}\cdot\psi_{\vec{h}}(x_5\Gamma_{\mr{Error}})$ has a vertical frequency $\eta_{\mr{Prod}}=(\eta,\eta,-\eta,-\eta,0)$. (Note that $(G_{\mr{Error}})_{(s-1,i)} = \mr{Id}_{G_{\mr{Error}}}$ for all $i\ge 0$.)

For the sake of convenience, we set \[g_{\vec{h}}^\ast(n) = (g_{h_1}(n),g_{h_2}(n + h_1-h_4),g_{h_3}(n),g_{h_4}(n + h_1-h_4),g_{\vec{h}}(n))\] and note that the function $F\otimes F\otimes \ol{F}\otimes \ol{F} \cdot\psi_{\vec{h}}$ is seen to be $M^{O_s(d^{O_s(1)})}$-Lipschitz on $\wt{G}$. Note that by the second item of Lemma~\ref{lem:coeff-mult}, we immediately have that 
\[\on{Taylor}_i(g_{h_2}(n + h_1-h_4)) = \on{Taylor}_i(g_{h_2}(n ))\]
for $1\le i\le s-1$ and analogously for $g_{h_4}(n+h_1-h_4)$. 

\noindent\textbf{Step 2: Invoking equidistribution theory.}
By applying Corollary~\ref{cor:equi-deg} (since $\eta$ is nonzero), there exists a $(MD/\rho)^{O_s(d^{O_s(1)})}$-rational subgroup $J=J_{\vec{h}}$ of $\wt{G}$ such that $\eta_{\mr{Prod}}(J\cap\wt{G}_{(s-1,r^\ast)}) = 0$ and such that 
\[g_{\vec{h}}^\ast = \eps_{\vec{h}}\cdot\wt{g_{\vec{h}}} \cdot\gamma_{\vec{h}}\]
where:
\begin{itemize}
    \item $\eps_{\vec{h}}(0) = \wt{g_{\vec{h}}}(0) = \gamma_{\vec{h}}(0) = \mr{id}_{\wt{G}}$;
    \item $\wt{g_{\vec{h}}}$ takes values in $J$;
    \item $\gamma_{\vec{h}}$ is $(MD/\rho)^{O_s(d^{O_s(1)})}$-rational (with respect to the lattice $\Gamma\times\Gamma\times\Gamma\times\Gamma\times\Gamma_{\mr{Error}}$);
    \item $d(\eps(n),\eps(n-1))\le(MD/\rho)^{O_s(d^{O_s(1)})}N^{-1}$ for $n\in[N]$.
\end{itemize}
By passing to a subset of additive quadruples of density $(MD/\rho)^{-O_s(d^{O_s(1)})}$ we may in fact assume that the group $J$ is independent of $\vec{h}$ under consideration.

We define
\[J_i' := (J\cap\wt{G}_{(i,1)})/(J\cap\wt{G}_{(i,2)}),\qquad J_i := \tau_i(J_i')\]
where $\tau_i\colon\on{Horiz}_i(G)^{\otimes 4}\times\on{Horiz}_i(G_{\mr{Error}})\to\on{Horiz}_i(G)^{\otimes 4}$ is the natural projection map to the four-fold product. Since $\eta_{\mr{Prod}}(J\cap\wt{G}_{(s-1,r)}) = 0$ (due to the output of Corollary~\ref{cor:equi-deg}), we have 
\[\eta_{\mr{Prod}}([J_{i_1}',\ldots,J_{i_{r^\ast}}']) = 0\]
for $i_1' + \ldots + i_{r^{\ast}}' = s-1$ where the commutator bracket is taken with respect to $\wt{G}$ and $[\cdot,\ldots,\cdot]$ denotes any possible $(r^\ast-1)$-fold commutator bracket. 

Since $G_{\mr{Error}}$ has been given a degree-rank $<(s-1,r^\ast)$ filtration, we have that in fact 
\[\eta_{\mr{Prod}}([J_{i_1},\ldots,J_{i_{r^\ast}}]) = 0\]
where we abusively descend $\eta_{\mr{Prod}}$ to $G^{\otimes 4}$. Less formally, we are noting that the final coordinate of elements in $\wt{G}$ play no role in commutators of the depth being considered.

\noindent\textbf{Step 3: Furstenberg--Weiss commutator argument.}
We now perform the crucial Furstenberg--Weiss commutator argument. Given $T\subseteq[4]$, we define $\pi_T((v_1,\ldots,v_4)) = (v_i)_{i\in T}$ with the coordinates represented in increasing order of index.

We define
\begin{align*}
\pi_{123}(J_i)^\ast &= \pi_{123}(J_i) \cap\{(v,0,0)\colon v\in\on{Horiz}_i(G)\},\\
\pi_{124}(J_i)^\ast &= \pi_{124}(J_i) \cap\{(v,0,0)\colon v\in\on{Horiz}_i(G)\}.
\end{align*}
Note that $\pi_{123}(J_i)^\ast$ and $\pi_{124}(J_i)^\ast$ may (abusively) be viewed as subspaces of $\on{Horiz}_i(G)$. The crucial claim is that 
\[\eta([v_{i_1},\ldots, v_{i_{r^\ast}}]) = 0\]
if $i_1+\cdots+i_{r^\ast} = s-1$, each $v_{i_\ell}\in\pi_1(J_{i_\ell})$, and for two distinct indices $\ell_1,\ell_2$ we have that $v_{i_{\ell_1}}\in\pi_{123}(J_{i_{\ell_1}})^\ast$ and $v_{i_{\ell_2}}\in\pi_{124}(J_{i_{\ell_2}})^\ast$. Note that $\eta$ lives on $G$ and the commutator brackets are taken with respect to $G$, not $G^{\otimes 4}$. The Furstenberg--Weiss commutator argument is required to capture precisely this difference. 

Note that an element $v_i\in\pi_1(J_{i_\ell})$ lifts to an element $\wt{v_{i_\ell}}$ of the form $(v_{i_\ell},\cdot,\cdot,\cdot)\in\on{Horiz}_{i_\ell}(G)^{\otimes 4}$. Furthermore note that $v_{i_\ell}\in\pi_{123}(J_{i_\ell})$ ``lifts'' to an element $\wt{v_{i_\ell}}$ of the form $(v_{i_\ell},0,0,\cdot)\in\on{Horiz}_{i_\ell}(G)^{\otimes 4}$ while $v_{i_\ell}\in\pi_{124}(J_{i_\ell})$ lifts to an element $\wt{v_{i_\ell}}$ of the form $(v_{i_\ell},0,\cdot,0)\in\on{Horiz}_{i_\ell}(G)^{\otimes 4}$.

Given the above setup, we have 
\[[\wt{v_{i_{1}}},\ldots,\wt{v_{i_{{r^\ast}}}}] = ([v_{i_{1}},\ldots,v_{i_{{r^\ast}}}],\mr{id}_G,\mr{id}_G,\mr{id}_G).\]
To see this note that the iterated commutator of elements in $G\times\mr{Id}_G\times\mr{Id}_G \times G$ (with any elements in $G^{\otimes 4}$) remains in the subgroup $G\times\mr{Id}_G\times\mr{Id}_G \times G$; an analogous fact holds true for $G\times\mr{Id}_G\times G \times\mr{Id}_G$. Since we assumed that our commutator contains elements in both $G\times\mr{Id}_G\times\mr{Id}_G \times G$ and $G\times\mr{Id}_G\times G \times\mr{Id}_G$, the commutator must in fact live in $G\times\mr{Id}_G\times\mr{Id}_G \times\mr{Id}_G$, and the first coordinates of the desired commutators is trivially seen to match. 

Recalling that we have
\[\eta_{\mr{Prod}}([J_{i_1},\ldots,J_{i_{r^\ast}}]) = 0,\]
and noting that $\eta_{\mr{Prod}}$ descends to $\eta$ on the subgroup $G_{(s-1,r^\ast)}\times\mr{Id}_{G}^{\otimes 3}$, we have 
\[\eta([v_{i_{1}},\ldots,v_{i_{{r^\ast}}}]) = 0\]
as claimed.

\noindent\textbf{Step 4: Finding $(h_2,h_3)$ and $(h_2',h_4')$ which extend to many ``good'' $h_1$.}
Recall that we are looking at the at least $(MD/\rho)^{-O_s(d^{O_s(1)})}$ fraction of additive quadruples $(h_1,h_2,h_3,h_4)\in H^4\subseteq[N]^4$ which are such that $\wt{g_{\vec{h}}}$ lives on a specified subgroup $J$. Call this set of quadruples $\mc{S}$.

So by Markov, there are at least $(MD/\rho)^{-O_s(d^{O_s(1)})}N$ many $h_1\in[N]$ which extend to at least $(MD/\rho)^{-O_s(d^{O_s(1)})}N^2$ quadruples in $\mc{S}$. Thus there are at least $(MD/\rho)^{-O_s(d^{O_s(1)})}N^5$ pairs of additive tuples of the form
\[(h_1,h_2,h_3,h_1+h_2-h_3),~(h_1,h_2',h_1 + h_2'-h_4',h_4')\in\mc{S}.\] 

By averaging, there exists a pair of pairs $(h_2,h_3)$
and $(h_2',h_4')$ such that there are at least $(MD/\rho)^{-O_s(d^{O_s(1)})}N$ many $h_1\in[N]$ which live in such additive tuples. We fix such a pair of pairs and define $\mc{T}$ to denote the set of $h_1\in[N]$ such that $(h_1,h_2,h_3,h_1+h_2-h_3)\in\mc{S}$ and $(h_1,h_2',h_1+h_2'-h_4',h_4')\in\mc{S}$.

\noindent\textbf{Step 5: Extracting coefficient data.}
Consider $h_1\in\mc{T}$ and define
\[h^{123} = (h_1,h_2,h_3,h_1+h_2-h_3),\quad h^{124} = (h_1,h_2',h_1+h_2'-h_4',h_4').\]
Recall $\mc{X}_i = (\mc{X}\cap\log(G_{(i,1)}))/\log(G_{(i,2)})$ and assign the basis $\exp(\mc{X}_i)$ to $G_{(i,1)}/G_{(i,2)}$ (viewed as a vector space). 
Finally we assign the basis $\mc{Z}_i=\bigcup_{Y_i\in\exp(\mc{X}_i)}\{(Y_i,0,0), (0,Y_i,0),(0,0,Y_i)\}$ to $(G_{(i,1)}/G_{(i,2)})^{\otimes 3}$.

By Lemma~\ref{lem:coeff-mult}, we have
\begin{align*}
\on{Taylor}_i(g_{h^{123}}^\ast) &= \on{Taylor}_i(\eps_{h^{123}}) + \on{Taylor}_i(\wt{g_{h^{123}}})  + \on{Taylor}_i(\gamma_{h^{123}}) \\
\on{Taylor}_i(g_{h^{124}}^\ast) &= \on{Taylor}_i(\eps_{h^{124}}) + \on{Taylor}_i(\wt{g_{h^{124}}})  + \on{Taylor}_i(\gamma_{h^{124}}) 
\end{align*}
Therefore, by Lemma~\ref{lem:small-Taylor}, for all $h_1\in\mc{T}$ we have 
\begin{align*}
\on{dist}(\on{Taylor}_i((g_{h_1},g_{h_2},g_{h_3})), \pi_{123}(J_i) + T_{h_1}^{-1}\on{Horiz}_i(\Gamma^{\otimes 3})) \le (MD/\rho)^{O_s(d^{O_s(1)})} N^{-i},\\
\on{dist}(\on{Taylor}_i((g_{h_1},g_{h_2'},g_{h_4'})), \pi_{124}(J_i) + T_{h_1}'^{-1}\on{Horiz}_i(\Gamma^{\otimes 3})) \le (MD/\rho)^{O_s(d^{O_s(1)})} N^{-i},
\end{align*}
where $T_{h_1}$ and $T_{h_1}'$ are positive integers bounded by $(MD/\rho)^{O_s(d^{O_s(1)})}$. Here we have identified the basis $\mc{Z}_i$ (for $(G_{(i,1)}/G_{(i,2)})^{\otimes 3}$) with the standard basis vectors in $\mb{R}^{3\dim(\on{Horiz}_i(G))}$ and taken the $L^{\infty}$ metric on the latter (for the notion of $\on{dist}$). At the cost of shrinking the set $\mc{T}$ by a multiplicative factor of $(MD/\rho)^{-O_s(d^{O_s(1)})}$ we may assume that $T_{h_1} = T$ and $T_{h_1}' = T'$ for all $h_1\in\mc{T}$.

We now consider a basis $\mc{B}_i$ for $\pi_{123}(J_i)$ which is in row-echelon form where one orders the coordinates corresponding to second copy of $G$ (in the four-fold $G^{\otimes 4}$) at the front, then the third copy, and then the first copy. In particular, the ``final block'' of basis vectors span $\pi_{123}(J_i)^\ast$. Note that one can take such $\mc{B}_i$ such that the coordinates are integers bounded by $(MD/\rho)^{O_s(d^{O_s(1)})}$ due to the rationality of $\pi_{123}(J_i)$.

For $h_1,h_1'\in\mc{T}$, we have
\begin{align}\label{eq:sunflower-1}
\on{Taylor}_i((g_{h_1},g_{h_2},g_{h_3})) &= \sum_{R_j\in\mc{B}_i}a_jR_j + T^{-1}\mb{Z}^{3\dim(\on{Horiz}_i(G))} + v_{h_1}\\
\on{Taylor}_i((g_{h_1'},g_{h_2},g_{h_3})) &= \sum_{R_j\in\mc{B}_i}a_j'R_j + T^{-1}\mb{Z}^{3\dim(\on{Horiz}_i(G))} + v_{h_1'}\label{eq:sunflower-2}
\end{align}
where $\snorm{v_{h_1}}_{\infty}, \snorm{v_{h_1'}}_{\infty}\le (MD/\rho)^{O_s(d^{O_s(1)})}N^{-i}$. For each basis vector $R_j\in\mc{B}_i$ where the first nonzero element is either in coordinates corresponding to second or third copy of $G$, there exists a dual vector which is zero on the coordinates corresponding to the first copy of $G$ and whose inner product with all of $\mc{B}_i$ but $R_j$ is zero.

Call this vector $v_j$ and note one may take $v_j$ to have integral coordinates bounded by $(MD/\rho)^{O_s(d^{O_s(1)})}$ and divisible by $T$. Then from \eqref{eq:sunflower-1} and \eqref{eq:sunflower-2},
\[0 =  v_j\cdot(\on{Taylor}_i((g_{h_1},g_{h_2},g_{h_3}))-\on{Taylor}_i((g_{h_1'},g_{h_2},g_{h_3})) = M_j (a_j-a_j') + \mb{Z} \pm (MD/\rho)^{O_s(d^{O_s(1)})}N^{-i},\]
where $M_j$ is an nonzero integer bounded by $(MD/\rho)^{O_s(d^{O_s(1)})}$. (That is, $M_j(a_j-a_j')$ is within $(MD/\rho)^{O_s(d^{O_s(1)})}N^{-i}$ of an integer.)

We may now use this information about such indices $j$ in conjunction with \eqref{eq:sunflower-1} and \eqref{eq:sunflower-2}. We deduce that for all $h_1,h_1'\in\mc{T}$,
\[\on{dist}(\on{Taylor}_i(g_{h_1}) - \on{Taylor}_i(g_{h_1'}), \pi_{123}(J_i)^\ast + {T_1}^{-1}\mb{Z}^{\dim(\on{Horiz}_i(G))}) \le (MD/\rho)^{O_s(d^{O_s(1)})}N^{-i},\]
where $T_1$ is an integer of size bounded by $(MD/\rho)^{O_s(d^{O_s(1)})}$. Analogously,
\[\on{dist}(\on{Taylor}_i(g_{h_1}) - \on{Taylor}_i(g_{h_1'}), \pi_{124}(J_i)^\ast + {T_1'}^{-1}\mb{Z}^{\dim(\on{Horiz}_i(G))}) \le (MD/\rho)^{O_s(d^{O_s(1)})}N^{-i}\]
where $T_1'$ is an integer bounded by $(MD/\rho)^{O_s(d^{O_s(1)})}$. Putting it together, we may deduce that
\[\on{dist}(\on{Taylor}_i(g_{h_1}) - \on{Taylor}_i(g_{h_1'}), \pi_{123}(J_i)^\ast\cap\pi_{124}(J_i)^\ast + T_2^{-1}\mb{Z}^{\dim(\on{Horiz}_i(G))})\le (MD/\rho)^{O_s(d^{O_s(1)})}N^{-i}\]
with $T_2$ a nonzero integer bounded by $(MD/\rho)^{O_s(d^{O_s(1)})}$. To see this, simply construct a bounded integral basis of the orthogonal complement of $\pi_{123}(J_i)^\ast\cap\pi_{124}(J_i)^\ast$ (treated as a subspace of the dual space to $G_{(i,1)}/G_{(i,2)}\simeq\mb{R}^{\dim(\on{Horiz}_i(G))}$). Then the two input inequalities imply that any basis vector for the intersection space dual will map $T_1T_1'(\on{Taylor}_i(g_{h_1})-\on{Taylor}_i(g_{h_1'}))$ to a near-integral scalar, which gives the claim.

Now by Lemma~\ref{lem:coeff-mult} we therefore have 
\begin{equation}\label{eq:sunflower-3}
\on{dist}(\on{Taylor}_i(g_{h_1}g_{h_1'}^{-1}), \pi_{123}(J_i)^\ast\cap\pi_{124}(J_i)^\ast + T_2^{-1}\mb{Z}^{\dim(\on{Horiz}_i(G))})\le (MD/\rho)^{O_s(d^{O_s(1)})}N^{-i}.
\end{equation}
It is also trivial by restricting the factorization to the first coordinate that
\begin{equation}\label{eq:sunflower-4}
\on{dist}(\on{Taylor}_i(g_{h_1}), \pi_1(J_i) + T_3^{-1}\mb{Z}^{\dim(\on{Horiz}_i(G))}) \le (MD/\rho)^{O_s(d^{O_s(1)})}N^{-i}
\end{equation}
with $T_3$ a nonzero integer bounded by $(MD/\rho)^{O_s(d^{O_s(1)})}$.

\noindent\textbf{Step 6: Extracting initial factorizations.}
For the remainder of the proof fix $h_1^\ast\in\mc{T}$. Given $h_1'\in\mc{T}$ we have
\[g_{h_1'} = g_{h_1'}g_{h_1^\ast}^{-1} \cdot g_{h_1^\ast}.\]

Note that $\pi_{1}(J_i)$ and $\pi_{123}(J_i)^\ast\cap\pi_{124}(J_i)^\ast$ may each be defined as the kernel of a set of $i$-th horizontal characters (on $G$) of height at most $(MD/\rho)^{O_s(d^{O_s(1)})}$. Recall \eqref{eq:sunflower-3} and \eqref{eq:sunflower-4}. Scaling the horizontal characters by at most $(MD/\rho)^{O_s(d^{O_s(1)})}$ and applying Lemma~\ref{lem:factor}, we may write 
\begin{align*}
g_{h_1'}g_{h_1^\ast}^{-1} &= \eps_{h_1'}\cdot\wt{g_{h_1'}} \cdot\gamma_{h_1'},\\
g_{h_1^\ast} &= \eps \cdot \wt{g'} \cdot\gamma',
\end{align*}
where:
\begin{itemize}
    \item $\eps_{h_1'}(0) = \wt{g_{h_1'}}(0) = \gamma_{h_1'}(0) = \eps(0) = g'(0) = \gamma'(0) = \mr{id}_G$;
    \item $\on{Taylor}_i(\wt{g_{h_1'}})\in\pi_{123}(J_i)^\ast\cap\pi_{124}(J_i)^\ast$ and $\on{Taylor}_i(\wt{g'})\in\pi_{1}(J_i)$;
    \item $\gamma_{h_1'},\gamma'$ are $(MD/\rho)^{O_s(d^{O_s(1)})}$-rational;
    \item $d_{G}(\eps_{h_1'}(n),\eps_{h_1'}(n-1)) + d_{G}(\eps(n),\eps(n-1))\le (MD/\rho)^{O_s(d^{O_s(1)})}N^{-1}$ for $n\in[N]$.
\end{itemize}
Therefore 
\begin{align*}
g_{h_1'} &= g_{h_1'}g_{h_1^\ast}^{-1} \cdot g_{h_1^\ast} =\eps_{h_1'}\eps \cdot (\eps^{-1} \wt{g_{h_1'}} \eps) \cdot (\eps^{-1}\gamma_{h_1'}\eps \gamma_{h_1'}^{-1}) \cdot (\gamma_{h_1'} \cdot \wt{g'} \gamma_{h_1'}^{-1}) \cdot\gamma_{h_1'}\gamma'.
\end{align*}

By Lemma~\ref{lem:coeff-mult}, we have 
\begin{align*}
\on{Taylor}_i(\gamma_{h_1'} \wt{g'} \gamma_{h_1'}^{-1}) = \on{Taylor}_i(\wt{g'}) &\in\pi_{1}(J_i),\\    
\on{Taylor}_i((\eps^{-1} \wt{g_{h_1'}} \eps) \cdot (\eps^{-1}\gamma_{h_1'}\eps \gamma_{h_1'}^{-1})) =  \on{Taylor}_i(\wt{g_{h_1'}})&\in\pi_{123}(J_i)^\ast\cap  \pi_{124}(J_i)^\ast.
\end{align*}

We say that $h_1',h_1''\in\mc{T}$ have matching rational parts if
\[(\gamma_{h_1'}\gamma')^{-1} \cdot (\gamma_{h_1''}\gamma')\]
is a polynomial sequence valued in $\Gamma$. By restricting $\mc{T}$ to an appropriate subset of density $(MD/\rho)^{-O_s(d^{O_s(1)})}$, we may assume that all $h_1'\in\mc{T}$ have matching rational parts. (This is most easily seen in first-kind coordinates: if $\gamma_{h_1'}\gamma'$ and $\gamma_{h_1''}\gamma'$ have all coefficients differing by $T_4 \cdot\on{span}(\mc{X},\mb{Z})$ where $T_4$ is an appropriate integer of size bounded by $(MD/\rho)^{O_s(d^{O_s(1)})}$ then two sequences match up to a polynomial sequence in $\Gamma$.)

So, ultimately we may assume that for all $h_1'\in\mc{T}$ we have 
\begin{align*}
g_{h_1'} = \eps_{h_1'}^\ast \cdot g_{h_1'}^\ast \cdot\gamma^{\ast} \cdot \wt{\gamma_{h_1'}}
\end{align*}
where:
\begin{itemize}
    \item $\eps_{h_1'}^\ast (0) = g_{h_1'}^\ast(0) = \gamma^\ast(0) = \wt{\gamma_{h_1'}}(0) = \mr{id}_G$;
    \item $\on{Taylor}_i(g_{h_1'}^\ast\cdot (g_{h_1''}^\ast)^{-1})\in\pi_{123}(J_i)^\ast\cap\pi_{124}(J_i)^\ast$ and $\on{Taylor}_i(g_{h_1'}^\ast)\in\pi_{1}(J_i)$ for all $h_1''\in\mc{T}$;
    \item $\gamma^\ast$ is $(MD/\rho)^{O_s(d^{O_s(1)})}$-rational
    \item $\wt{\gamma_{h_1'}}$ takes values in $\Gamma$;
    \item $d_{G}(\eps_{h_1'}^\ast(n),\eps_{h_1'}^\ast(n-1))\le (MD/\rho)^{O_s(d^{O_s(1)})}N^{-1}$ for $n\in[N]$.
\end{itemize}

\noindent\textbf{Step 7: Removing periodic and smooth pieces of factorization.}
Let $Q$ be the period of $\gamma^\ast\Gamma$ and define $\delta = (MD/\rho)^{-O_s(d^{O_s(1)})}$ where $\delta$ is to be chosen later. We break $[N]$ into a collection of arithmetic progressions with difference $Q$ and length between $\delta N$ and $2\delta N$; there are at most $\delta^{-1}$ such progressions. Call these progressions $P_1,\ldots,P_\ell$ and note that 
\[\norm{\mb{E}_{n\in[N]}(\Delta_h f)(n) \sum_{i=1}^\ell\mbm{1}_{n\in P_i} \cdot\chi(h,n)\otimes \chi_h(n) \cdot\psi_h(n)}_{\infty}\ge\rho\]
where $\psi_h(n)$ is the degree $(s-2)$ nilsequence coming from the condition $\Delta_hf(n)\otimes\ol{\chi(h,n)}\otimes\ol{\chi_h(n)}\in\on{Corr}(s-1,\rho,M,d)$ from the original correlation structure. For $h\in\mc{T}$ we may write
\[\chi_h(n) = F(g_h(n)\Gamma) = F(\eps_h^\ast g_h^\ast \gamma^\ast \Gamma);\]
here we are using that $\wt{\gamma_{h_1'}}$ takes values in $\Gamma$ so may be dropped for the remainder of the analysis. 

Since $Q$ is the period of $\gamma^\ast$, we may replace $\gamma^\ast$ by a value $\gamma_{P_i}$ for each progression where $\gamma_{P_i}\gamma^\ast(n)^{-1}\in\Gamma$ for $n\in P_i$ and $\snorm{\psi(\gamma_{P_i})}_{\infty}\le 1$. Then
\[\norm{\mb{E}_{n\in[N]}(\Delta_h f)(n) \cdot\chi(h,n)\otimes \bigg(\sum_{i=1}^\ell\mbm{1}_{n\in P_i}F(\eps_h^\ast g_h^\ast \gamma_{P_i}\Gamma)\bigg) \cdot\psi_h(n)}_{\infty}\ge\rho.\]
Furthermore as $\eps_h^\ast$ is sufficiently smooth we may replace $\eps_h^\ast$ with the constant $\eps_{h,P_i} = \eps_h^\ast(\min(P_i))$ and have 
\[\norm{\mb{E}_{n\in[N]}(\Delta_h f)(n) \cdot\chi(h,n)\otimes \bigg(\sum_{i=1}^\ell\mbm{1}_{n\in P_i}F(\eps_{h,P_i}g_h^\ast \gamma_{P_i}\Gamma)\bigg) \cdot\psi_h(n)}_{\infty}\ge\rho/2,\]
as long as $\delta$ was chosen sufficiently small.

By the triangle inequality there exists some $P_i$ which is distance at least $\delta^{1/2}N$ from the ends of the interval $[N]$ such that 
\[\snorm{\mb{E}_{n\in[N]}(\Delta_h f)(n) \cdot\chi(h,n)\otimes \mbm{1}_{n\in P_i}F(\eps_{h,P_i}g_h^\ast \gamma_{P_i}\Gamma) \cdot\psi_h(n)}_{\infty}\ge\delta^{2}.\]
By paying a $\delta^{O(1)}$-fraction in the size of $\mc{T}$ we may assume that the choice of index $i$ is independent of $h$, hence writing $P_i = P$. Furthermore note that there is a $\delta^{O(1)}$-net of size $\delta^{-O_s(d)}$ for the set of $g$ satisfying $d_{G}(g,\mr{id}_G)\le (MD/\rho)^{O_s(d^{O_s(1)})}$. If the net size is chosen small enough, we may shift $\eps_{h,P_i}$ to a nearby value in the net without much loss. Then we can pay a $\delta^{O_s(d)}$-fraction in the size of $\mc{T}$ to Pigeonhole onto a single point in the net, writing $\eps_{h,P_i} = \eps_P$.

Overall, for all $h\in\mc{T}$ we have 
\[\snorm{\mb{E}_{n\in[N]}(\Delta_h f)(n) \cdot\chi(h,n)\otimes \mbm{1}_{P}(n)F(\eps_P g_h^\ast \gamma_P\Gamma) \cdot\psi_h(n)}_{\infty}\ge\delta^{3}\]
for some $P$ at least $\delta^{1/2}N$ from the endpoints of the interval. Thus by Lemma~\ref{lem:major-arc}, for each $h\in\mc{T}$ there exists $\Theta_h$ with $\snorm{Q\cdot\Theta_h}_{\mb{R}/\mb{Z}}\le\delta^{-O(1)}N^{-1}$ and 
\begin{equation}\label{eq:sunflower-5}
\snorm{\mb{E}_{n\in[N]}(\Delta_h f)(n) \cdot\chi(h,n)\otimes e(\Theta_h n)F(\eps^\ast g_h^\ast \gamma_{P}\Gamma) \cdot\psi_h(n)}_{\infty}\ge\delta^{O(1)}.
\end{equation}
Rounding $\Theta_h$ to a net of distance $\delta^{O(1)}N^{-1}$ and paying a $Q^{-1}\delta^{O(1)}$-fraction in the size of $\mc{T}$ to Pigeonhole the resulting point, we may write $\Theta_h = \Theta$ for all $h\in\mc{T}$. We are now finally in position to define the output data. Define
\begin{align*}
g_h' &= \gamma_P^{-1}g_h^\ast \gamma_P,\\
F' &= F(\eps^\ast \gamma_P \cdot),\\
\chi_h'(n) &= F'(g_h'(n)),\\
\chi'(h,n) &= \chi(h,n) \cdot e(\Theta n).
\end{align*}
Note that $g'(h,n) = (g(h,n),\Theta n)$ is the polynomial sequence underlying $\chi'$, and $\chi'(h,n)={F^\ast}'(g'(h,n){\Gamma^\ast}')$. It is easy to check the relevant properties of Definition~\ref{def:correlation-structure} to see that we obtain a degree-rank $(s-1,r^\ast)$ correlation structure with appropriately modified underlying parameters (we set $H'$ to be the final refined version of $\mc{T}$); in particular, \eqref{eq:sunflower-5} demonstrates the necessary correlation fact.

Finally, taking 
\[V_{i,\mr{Dep}} = \pi_{123}(J_i)^\ast\cap\pi_{124}(J_i)^\ast\text{ and }V_{i} = \pi_1(J_i)\]
we finish the proof: in particular, the result from Step 3 demonstrates the final item of the conclusion, and the result from Step 6 demonstrates the third item.
\end{proof}

\section{Linearization Step}\label{sec:linear}
We now come to the second crucial argument of this paper. Prior this stage we have modified the degree-rank $(s-1,r^\ast)$ to one in which various Taylor coefficients of $g_h$ for $h\in H$ (upon factoring) differ only on certain special subspaces. In this next stage, we deduce that either these Taylor coefficients differ on a further refined subspace which is seen to be essentially ``annhilated'' by $\eta$ or $g_h$ has a certain ``bracket linear'' form. This step is ultimately where we invoke the results of Sanders \cite{San12b} on quasi-polynomial bounds for the Bogolyubov lemma. 

This step is closely modeled after \cite[Step~2]{GTZ11} and the closely related proof of \cite[Lemma~11.5]{GTZ12}; a quantitative version for the $U^4$-inverse theorem due to the first author can be found in \cite{Len23}. The precise statement of the lemma should also be compared with \cite[Theorem~11.1(ii)]{GTZ12}.
\begin{lemma}\label{lem:linear-output}
Fix $s\ge 2$ and $1\le r^\ast\le s-1$. Let $f\colon[N]\to\mb{C}$ be a $1$-bounded function. Suppose that $f$ has a degree rank $(s-1,r^\ast)$ correlation structure with parameters $\rho$, $M$, $d$, and $D$ and that $N\ge (MD/\rho)^{O_s(d^{O_s(1)})}$. Furthermore let $\mc{X}_i = (\mc{X}\cap\log(G_{(i,1)}))/\log(G_{(i,2)})$.

We output a new degree-rank $(s-1,r^\ast)$ correlation structure for $f$ with parameters
\begin{align*}
\rho'^{-1}&\le\exp(O_s((d\log(MD/\rho))^{O_s(1)})),\quad M' \le O(M),\quad D' = D,\quad d'\le O(d),
\end{align*}
with set $H'\subseteq H$, with multidegree $(1,s-1)$ nilcharacter $\chi'(h,n)={F^\ast}'(g'(h,n){\Gamma^\ast}')$ on $(G^\ast)'=G^\ast\times\mb{R}$, with $h$-dependent nilcharacters $\chi_h'$ having underlying polynomial sequences $g_h'(n)=F'(g_h'(n)\Gamma)$ on $G'=G$. This correlation structure satisfies:
\begin{itemize}
    \item $(G^\ast)'$ is given the multidegree filtration 
    \[(G^\ast)'_{(i,j)} = (G^\ast)_{(i,j)} \times\{0\}\]
    if $(i,j)\neq (0,0)$ or $(0,1)$. For $(i,j)\in\{(0,0),(0,1)\}$, we set
    \[(G^\ast)'_{(i,j)} = (G^\ast)_{(i,j)} \times\mb{R}.\]
    We have ${F^\ast}'((x,z) (\Gamma^\ast\times\mb{Z})) = F^\ast(x \Gamma^\ast) \cdot e(z)$. We have $g'(h,n) = (g(h,n), \Theta n)$ for some appropriate value of $\Theta$;
    \item There is a collection of $\mb{R}$-vector spaces $W_{i,\ast}, W_{i,\mr{Lin}}, W_{i,\mr{Pet}}\leqslant G_{(i,1)}/G_{(i,2)}$ for each $i$;
    \item If $W_i := W_{i,\ast} +  W_{i,\mr{Lin}} + W_{i,\mr{Pet}}$ then $\dim(W_i) = \dim(W_{i,\ast}) +  \dim(W_{i,\mr{Lin}}) + \dim(W_{i,\mr{Pet}})$, i.e., the three spaces are linearly disjoint;
    \item There exist bases $\mc{X}_{i,\ast}$, $\mc{X}_{i,\mr{Lin}}$, and $\mc{X}_{i,\mr{Pet}}$ of the corresponding spaces which are composed of $(MD/\rho)^{O_s(d^{O_s(1)})}$-rational combinations of elements of $(\mc{X}\cap G_{(i,1)})/G_{(i,2)}$;
    \item For $1\le i\le s-1$ and $h,h_1,h_2\in H'$ we have
    \begin{align*}
    \on{Taylor}_i(g_h') &\in W_{i,\ast} + W_{i,\mr{Lin}} + W_{i,\mr{Pet}} = W_i,\\
    \on{Taylor}_i(g_{h_1}') - \on{Taylor}_i(g_{h_2}')&\in W_{i,\mr{Lin}} + W_{i,\mr{Pet}},\\
    \on{Proj}_{W_{i,\mr{Lin}}}(\on{Taylor}_i(g_h')) &= \sum_{Z_{i,j}\in\mc{X}_{i,\mr{Lin}}}\bigg(\gamma_{i,j} + \sum_{k=1}^{d^\ast}\alpha_{i,j,k}\{\beta_kh\}\bigg)Z_{i,j},
    \end{align*}
    with $d^\ast\le(d\log(MD/\rho))^{O_s(1)}$ and $\beta_k\in(1/N')\mb{Z}$ where $N'$ is a prime in $[100N,200N]$;
    \item $F'$ is $M'$-Lipschitz and has the same vertical frequency $\eta$ as $F$;
    \item For any integers $i_1 + \cdots + i_{r^\ast} = s-1$, suppose that $v_{i_j}\in V_{i_j}$ for all $j$. If for at least one index $\ell$ we have $v_{i_\ell}\in W_{i_\ell,\mr{Pet}}$, then if $w$ is any $(r^\ast-1)$-fold commutator of $v_{i_1},\ldots,v_{i_{r^\ast}}$ we have 
    \[\eta(w) = 0.\]
    Furthermore, if instead for at least two indices $\ell_1,\ell_2$ we have $v_{i_{\ell_1}}\in W_{i_\ell,\mr{Lin}}$ and $v_{i_{\ell_2}}\in W_{i_{\ell_2},\mr{Lin}}$, then if $w$ is any $(r^\ast-1)$-fold commutator of $v_{i_1},\ldots,v_{i_{r^\ast}}$ we have
    \[\eta(w) = 0.\]
\end{itemize}
\end{lemma}
\begin{remark*}
The projection map $\on{Proj}_{W_{i,\mr{Lin}}}\colon W_i\to W_{i,\mr{Lin}}$ is well-defined due to the linear disjointness condition. Furthermore we have written Taylor coefficients with additive notation, since $G_{(i,1)}/G_{(i,2)}$ can be identified with $\mb{R}^{\dim(\on{Horiz}_i(G))}$.
\end{remark*}
\begin{proof}
For the majority of the proof we will assume $s\ge 3$; we indicate the minor changes required for $s = 2$ for the end of the proof (and the case $s = 2$ is not used in the proof of Theorem~\ref{thm:main}). Note that the case when $\eta$ is trivial follows via taking $W_{i,\mr{Pet}} = G_{(i,1)}/G_{(i,2)}$, $W_{i,\mr{Lin}}$ and $W_{i,\ast}$ to be trivial, $g_h' = g_h$, and $g'(h,n) = (g(h,n), 0)$; therefore we may assume that $\eta$ is nontrivial for the remainder of the proof.

\noindent\textbf{Step 1: Applying Lemma~\ref{lem:sunflower} and linear-algebraic setup.}
We apply Lemma~\ref{lem:sunflower} and treat the resulting correlation structure as the input to the lemma. Up to changing implicit constants in the output this leaves the lemma unchanged except for noting that 
\[\chi(h,n) = e(\Theta n) \cdot F^\ast(g(h,n) \Gamma^\ast)\]
which is defined on the group $(G^\ast)'=G^\ast\times\mb{R}$. In particular, we will abusively overwrite notation and relabel the resulting $H'$ from the application of Lemma~\ref{lem:sunflower} as $H$, $g_h'$ as $g_h$, and $\chi'_h(n)$ as $\chi_h(n)$ and thus assume the output properties without further comment.

It will also be crucial to define certain linear-algebraic operators of $V_i$. Consider a basis for $V_{i,\mr{Dep}}$, an extension to a basis of $V_{i}$, and then to $G_{(i,1)}/G_{(i,2)}$ such that all basis elements are $(MD/\rho)^{O_s(d^{O_s(1)})}$-rational combinations of the basis $\exp(\mc{X}_i)\imod G_{(i,2)}$. In particular, write 
\begin{align*}
V_{i,\mr{Dep}} &= \on{span}_{\mb{R}}(w_{i,1},\ldots,w_{i,\dim(V_{i,\mr{Dep}})})\\
V_i &= \on{span}_{\mb{R}}(w_{i,1},\ldots,w_{i,\dim(V_{i,\mr{Dep}})},w_{i,\dim(V_{i,\mr{Dep}})+1},\ldots,w_{i,\dim(V_i)}),\\
 G_{(i,1)}/G_{(i,2)}&= \on{span}_{\mb{R}}(w_{i,1},\ldots,w_{i,\dim(\on{Horiz}_i(G))}).
\end{align*}
Given $v\in V_i$, there is a unique linear combination 
\[v = \sum_{j=1}^{\dim(V_{i})}\alpha_jw_{i,j}.\]
We define 
\[P_i v = \sum_{j=\dim(V_{i,\mr{Dep}})+1}^{\dim(V_{i})}\alpha_jw_{i,j},\qquad Q_i v = \sum_{j=1}^{\dim(V_{i,\mr{Dep}})}\alpha_jw_{i,j}.\]
By construction $P_i^2 = P_i$, $Q_i^2 = Q_i$, $Q_i(V_i) \cap P_i(V_i) = 0$, and $P_i v + Q_i v = v$ for $v\in V_i$. We also (abusively) extend the operator $P_i$ to $V_i^{\otimes \ell}$ and $(G_{(i,1)}/G_{(i,2)})^{\otimes 4}$ in the obvious manners by acting on each copy of $V_i$ separately (and zeroing out basis elements $w_{i,\dim(V_i)+1},\ldots,w_{i,\dim(\on{Horiz}_i(G))}$).

\noindent\textbf{Step 2: Invoking equidistribution theory.}
Applying Lemma~\ref{lem:four-fold-biased} when $s\ge 3$, we have 
\begin{align*}
\snorm{\mb{E}[\chi_{h_1}(n)\otimes \chi_{h_2}(n + h_1-h_4)\otimes \ol{\chi_{h_3}(n)}\otimes \ol{\chi_{h_4}(n + h_1 - h_4)}\cdot\psi_{\vec{h}}(g_{\vec{h}}(n)\Gamma')]}_{\infty} \ge (MD/\rho)^{-O_s(d^{O_s(1)})}
\end{align*}
for a $(MD/\rho)^{-O_s(d^{O_s(1)})}$ density of additive tuples. We define $G_{\mr{Error}}$, $\wt{G}$, and $\eta_{\mr{Prod}}$ as in the proof of Lemma~\ref{lem:sunflower} and as before we may assume that $g_{\vec{h}}(0) = \mr{id}_{G_{\mr{Error}}}$. Define 
\[g_{\vec{h}}^\ast(n) = (g_{h_1}(n),g_{h_2}(n + h_1-h_4),g_{h_3}(n),g_{h_4}(n + h_1-h_4),g_{\vec{h}}(n)).\]
By applying Corollary~\ref{cor:equi-deg}, we have
\[g_{\vec{h}}^\ast =\eps_{\vec{h}} \cdot\wt{g_{\vec{h}}}\cdot\gamma_{\vec{h}}\]
with
\begin{itemize}
    \item $\eps_{\vec{h}}(0) = \wt{g_{\vec{h}}}(0) = \gamma_{\vec{h}}(0) = \mr{id}_{\wt{G}}$;
    \item $\wt{g_{\vec{h}}}$ takes values in $K$;
    \item $\gamma_{\vec{h}}$ is $(MD/\rho)^{O_s(d^{O_s(1)})}$-rational;
    \item $d(\eps(n),\eps(n-1))\le(MD/\rho)^{O_s(d^{O_s(1)})}N^{-1}$ for $n\in[N]$.
\end{itemize}
where $\eta_{\mr{Prod}}(K\cap\wt{G}_{(s-1,r^\ast)}) = 0$ and $K$ is a $(MD/\rho)^{O_s(d^{O_s(1)})}$-rational subgroup of $\wt{G}$. By passing to a subset of additive quadruples of density $(MD/\rho)^{-O_s(d^{O_s(1)})}$ we may in fact assume that the group $K$ is independent of $\vec{h}$ under consideration.

\noindent\textbf{Step 3: Linear algebra deductions from equidistribution theory.}
Note that at present the subgroup $K$ does not account for the deductions given in Lemma~\ref{lem:sunflower}; these initial deductions are designed essentially to account for this. Let $\tau_i\colon\on{Horiz}_i(G)^{\otimes 4}\times\on{Horiz}_i(G_{\mr{Error}})\to\on{Horiz}_i(G)^{\otimes 4}$ be the natural projection to the four-fold product. We define the following set of vector spaces:
\begin{align*}
R_i &:= \{(Q_iv_1,Q_iv_2,Q_iv_3,Q_iv_4)\in V_i^{\otimes 4}\colon Q_iv_1 + Q_iv_2 - Q_iv_3 -Q_iv_4 = 0\},\\
K_i &:= \tau_i(K\cap\wt{G}_{(i,1)}\imod\wt{G}_{(i,2)}),\\
S_i &:= \{(v_1,v_2,v_3,v_4)\in V_{i}^{\otimes 4}\colon P_{i} v_1 =  P_{i} v_2 = P_{i} v_3 = P_{i} v_4\},\\
K_{i,1} &:= K_i \cap S_i,\\
\wt{K_i} &:= K_{i,1} + R_i,\\
L_i &:= \pi_1(\wt{K_i} \cap\{(v_1,v_2,v_3,v_4)\in V_i^{\otimes 4}\colon Q_i v_2 = Q_iv_3 = Q_i v_4 = 0\}) + Q_i(V_i).
\end{align*}
By inspection, we have $R_i\leqslant S_i$ hence $\wt{K_i}\leqslant S_i$. Note that 
\[\eta_{\mr{Prod}}([K_{i_1,1},\ldots,K_{i_r^{\ast},1}]) = 0\]
whenever one has that $i_1 + \cdots + i_r^\ast = s-1$ and $[\cdot,\ldots,\cdot]$ denotes any possible $(r^\ast -1)$-fold commutator bracket. This is a consequence of the fact that $\eta_{\mr{Prod}}(K\cap\wt{G}_{(s-1,r^\ast)}) = 0$ and noting that $K_{i,1}\leqslant K_i$. Note that we are implicitly using that $\eta_{\mr{Prod}}$ is trivial on $G_{\mr{Error}}$ as well, and we abusively descend $\eta_{\mr{Prod}}$ to $G^{\otimes 4}$.

We now claim that 
\[\eta_{\mr{Prod}}([v_{i_1},\ldots,v_{i_{r^\ast}}]) = 0\]
if $v_{i_\ell}\in S_{i_\ell}$ for all $\ell$ and there is at least one index $j$ such that $v_{i_j}\in R_{i_j}$.

To prove this, note by the final bullet point of Lemma~\ref{lem:sunflower} and multilinearity that 
\begin{align*}
\eta_{\mr{Prod}}([v_{i_1},\ldots,v_{i_{r^\ast}}]) &= \eta_{\mr{Prod}}([P_{i_1}v_{i_1},\ldots,P_{i_{r^\ast}}v_{i_{r^\ast}}]) + \sum_{k=1}^{r}\eta_{\mr{Prod}}([P_{i_1}v_{i_1},\ldots, Q_{i_k}v_{i_k},\ldots,P_{i_{r^\ast}}v_{i_{r^\ast}}])\\
&= \eta_{\mr{Prod}}([P_{i_1}v_{i_1},\ldots, Q_{i_j}v_{i_j},\ldots,P_{i_{r^\ast}}v_{i_{r^\ast}}])=0.
\end{align*}
The first equality uses that every bracket with at least two $Q_{i_k}v_{i_k}$ has two $V_{i,\mr{Dep}}$ terms so is $0$, the second equality uses $P_{i_j}v_{i_j} = 0$, and the third equality follows by noting that 
\[P_{i}v_{i} \in\{(v_1,v_2,v_3,v_4)\in V_i^{\otimes 4}\colon P_i v_1 =  P_i v_2 = P_i v_3 = P_i v_4, Q_i v_1 =  Q_i v_2 = Q_i v_3 = Q_i v_4 = 0\}\]
and $\eta_{\mr{Prod}} = (\eta,\eta,-\eta,-\eta)$.
Now, we may ultimately deduce
\[\eta_{\mr{Prod}}([\wt{K_{i_1}},\ldots,\wt{K_{i_{r^\ast}}}])=0\]
because $\wt{K_i} = K_{i,1} + R_i$ and $R_i,K_{i,1}\leqslant S_i$.

Finally, let $\pi_T$ for $T\subseteq [4]$ is as in the proof of Lemma~\ref{lem:sunflower} (namely, an appropriate projection map). We have
\[\pi_1(\wt{K_i})\leqslant L_i.\]
This follows because if
\[((Qv_1,Pv_1),(Qv_2,Pv_2),(Qv_3,Pv_3),(Qv_4,Pv_4))\in\wt{K_i} \]
then
\[((Q(v_1+v_2-v_3-v_4),Pv_1),(0,Pv_2),(0,Pv_3),(0,Pv_4))\in\wt{K_i}.\]

\noindent\textbf{Step 4: Constructing a decomposition of $Q_i(V_i)$.}
We will now decompose $Q_i(V_i) = V_{i,\mr{Dep}}$ into a pair of subspaces. On one of these subspaces we will deduce an improved vanishing for the commutator while on the other subspace we will deduce an approximate linearity for $\on{Taylor}_i(g_h)$. Let 
\[L_i^\ast = \{(v_1,v_2,v_3,v_4)\in S_i\colon Pv_1 = 0, v_2 = v_3 = v_4 = 0\} \cap\wt{K_i}.\]
Note that $L_i^\ast$ may abusively be viewed as a subspace of $V_i$ (instead of $V_i^{\otimes 4}$) and under this identification $L_i^\ast\leqslant Q_i(V_i)=V_{i,\mr{Dep}}\leqslant L_i$.

The key claim in our analysis is if $i_1 + \cdots + i_{r^\ast} = s-1$, $v_{i_\ell}\in L_{i_\ell}$ for all indices $\ell$, and $v_{i_j}\in L_{i_{j}}^\ast$ for at least one index $j$ we have 
\[\eta([v_{i_1},\ldots,v_{i_\ell}]) = 0.\]

To prove this, note that $Q_{i_j}v_{i_j} = v_{i_j}$ and $P_{i_j}v_{i_j} = 0$ and using the last bullet point of Lemma~\ref{lem:sunflower}, we have 
\[\eta([v_{i_1},\ldots,v_{i_\ell}]) = \eta([P_{i_1}v_{i_1},\ldots,Q_{i_j}v_{i_j},\ldots,P_{i_\ell}v_{i_\ell}]),\]
similar to the argument in Step 3.

Next note that $P_iQ_iv = 0$ for all $v\in V_i$ and therefore 
\[P_i(L_i) \leqslant P_i(\pi_1(\wt{K_i}\cap\{(v_1,v_2,v_3,v_4)\in V_i^{\otimes 4}\colon Q_iv_2 = Q_iv_3 = Q_iv_4 = 0\})).\]
Therefore we may lift $P_{i_\ell}v_{i_\ell}$ for $\ell\neq j$ to $\wt{v_{i_\ell}} = (P_{i_\ell}v_{i_\ell} + w_{i_\ell},P_{i_\ell}v_{i_\ell},P_{i_\ell}v_{i_\ell},P_{i_\ell}v_{i_\ell})\in\wt{K_{i_\ell}}$ where $w_{i_\ell}\in Q_{i_\ell}(V_{i_\ell})$. We lift $v_{i_j}$ to $\wt{v_{i_{j}}}$ which has the form $(Q_{i_j}v_{i_j},0,0,0)\in\wt{K_{i_j}}$.

Note that we have 
\begin{align*}
 0 &= \eta_{\mr{Prod}}([\wt{v_{i_{1}}},\ldots,\wt{v_{i_{r^\ast}}}]) = \eta([P_{i_{1}}v_{i_{1}} + w_{i_{1}},\ldots, Q_{i_j}v_{i_{j}},\ldots, P_{i_{r^\ast}}v_{i_{r^\ast}} + w_{i_{r^\ast}}])\\
&= \eta([P_{i_{1}}v_{i_{1}},\ldots, Q_{i_j}v_{i_{j}},\ldots, P_{i_{r^\ast}}v_{i_{r^\ast}}])
\end{align*}
where in the first equality we have used for all $\ell$ that $\wt{v_{i_\ell}}\in\wt{K_{i_\ell}}$ and the result from Step 3, in the second equality that $\wt{v_{i_{j}}}$ has the final three coordinates identically zero, and in the final equality that $w_{i_\ell}\in Q_{i_\ell}(V_{i_\ell})=V_{i_\ell,\mr{Dep}}$ and the final item of Lemma~\ref{lem:sunflower}. 

The desired decomposition of spaces for the lemma will have 
\begin{align*}
W_{i,\mr{Pet}} &:= L_i^\ast,\quad W_{i,\ast} := P_i(L_i)\leqslant L_i\cap P_i(V_i).
\end{align*}
The fact $P_i(L_i)\leqslant L_i\cap P_i(V_i)$ is deduced from $Q_i(V_i)\leqslant L_i$. $W_{i,\mr{Lin}}$ will be constructed explicitly in the next step but is chosen so that
\[W_{i,\mr{Lin}}\leqslant Q_i(L_i) = Q_i(V_i)=V_{i,\mr{Dep}}\]
and $W_{i,\mr{Lin}} + W_{i,\mr{Pet}} = V_{i,\mr{Dep}}\leqslant L_i$. Given these properties of $W_{i,\mr{Lin}}$, note that the above analysis, along with Lemma~\ref{lem:sunflower}, establishes the final bullet point for our output.

\noindent\textbf{Step 5: Controlling approximate homomorphisms.}
Recall $\wt{K_i}\leqslant S_i$ and there is a natural isomorphism of groups
\[S_i \simeq \{(v,v_1,v_2,v_3,v_4)\colon v\in P_i(V_i), v_1,\ldots,v_4\in Q_i(V_i)\}.\]
Using this as an identification, we may write 
\[\wt{K_i} = \bigcap_{j=1}^{\dim(S_i) - \dim(\wt{K_i})}\on{ker}((\xi_j^{P_i}, \xi_j^{Q_i},\xi_j^{Q_i},-\xi_j^{Q_i},-\xi_j^{Q_i}))\]
where $\xi_j^{P_i}\in P_i(V_i)^\vee$ and $\xi_j^{Q_i}\in Q_i(V_i)^\vee$ (i.e., corresponding dual vector spaces). Note that the annihilators all have the special form of $(\cdot, \xi_j^{Q_i},\xi_j^{Q_i},-\xi_j^{Q_i},-\xi_j^{Q_i})$ since
\[\{(Q_iv_1,Q_iv_2,Q_iv_3,Q_iv_4)\in V_i^{\otimes 4}\colon Q_iv_1 + Q_iv_2 - Q_iv_3 -Q_iv_4 = 0\}=R_i\leqslant\wt{K_i}.\]
Note that
\[L_i^\ast = \{v\in V_i\colon P_i v = 0\text{ and }\xi_j^{Q_i}(Q_iv) = 0\text{ for all }j\}\]
since $v\in L_i^\ast$ is equivalent under this identification to $(0,Q_iv,0,0,0)\in\wt{K_i}$. Without loss of generality we may assume that for $1\le j\le\dim(V_{i,\mr{Dep}}) - \dim(L_i^\ast)$, vectors $\xi_j^{Q_i}$ are independent in $Q_i(V_i)^\vee$ (and they must span the orthogonal space to $L_i^\ast$ within $Q_i(V_i)^\vee$).

By appropriate scaling, we may assume $\xi_j^{Q_i}(w_{i,j})$ is an integer bounded by $(MD/\rho)^{O_s(d^{O_s(1)})}$ for $1\le j\le\dim(V_{i,\mr{Dep}})$. We extend each $\xi_j^{Q_i}$ to an operator on $(G_{(i,1)}/G_{(i,2)})^\vee$ by setting $\xi_j^{Q_i}(w_{i,j}) = 0$ for $j>\dim(V_{i,\mr{Dep}})$. Possibly at the cost of another $(MD/\rho)^{O_s(d^{O_s(1)})}$ scaling, we may assume that $\xi_j^{Q_i}(\Gamma\cap G_{(i,1)}\imod G_{(i,2)})\in\mb{Z}$. We extend $\xi_j^{P_i}(\cdot)$ in an analogous manner to $(G_{(i,1)}/G_{(i,2)})^\vee$ by setting $\xi_j^{P_i}(w_{i,j}) = 0$ for $1\le j\le\dim(V_i,\mr{Dep})$ and $j>\dim(V_i)$. Again, we may scale such that $\xi_j^{P_i}(\Gamma\cap G_{(i,1)}\imod G_{(i,2)})\in\mb{Z}$. The crucial point here is that now $\xi_j^{P_i}$ and $\xi_j^{Q_i}$ are $i$-th horizontal characters of height at most $(MD/\rho)^{O_s(d^{O_s(1)})}$.

We have
\[\tau_i(\on{Taylor}_i(\wt{g}_{\vec{h}})) \in K_i\]
and thus
\[\on{dist}(\tau_i(\on{Taylor}_i(g_{\vec{h}}^\ast)), S_i + T^{-1}\on{Horiz}_i(\Gamma^{\otimes 4}))\le (MD/\rho)^{O_s(d^{O_s(1)})} N^{-i}\]
after Pigeonholing $\vec{h}$ appropriately. Here distance is in $L^{\infty}$ after expressing both of these expressions in the basis $\exp(\mc{X}_i)^{\otimes 4}$ and $T$ is an integer bounded by $(MD/\rho)^{O_s(d^{O_s(1)})}$. We have used Lemma~\ref{lem:coeff-mult} and the properties of the original factorization; a very similar argument appears in Step 5 of the proof of Lemma~\ref{lem:sunflower}.

Furthermore note that 
\[\tau_i(\on{Taylor}_i(g_{\vec{h}}^\ast))\in S_i\]
by Lemma~\ref{lem:sunflower}. So if we choose a set of horizontal characters of height $(MD/\rho)^{O_s(d^{O_s(1)})}$ relative to $\on{Horiz}_i(\Gamma^{\otimes 4})$ which cut out $\wt{K_i}$ as their common kernel, then noting that $K_i\cap S_i \leqslant \wt{K_i}$ and applying Lemma~\ref{lem:factor} we may assume that 
\begin{equation}\label{eq:linear-output-1}
\tau_i(\on{Taylor}_i(\wt{g}_{\vec{h}}))\in\wt{K_i}
\end{equation}
and $\eps_{\vec{h}},\gamma_{\vec{h}}$ have identical properties up to changing implicit constants. We will assume this refined property of the factorization for the remainder of our analysis. 

Given the factorization of $g_{\vec{h}}^\ast$, we thus deduce (taking an appropriate least common multiple)
\begin{align}
\norm{T_1\cdot\bigg(\xi_j^{P_i}&(\on{Taylor}_i(g_{h_2})) + \xi_j^{Q_i}(\on{Taylor}_i(g_{h_1})) + \xi_j^{Q_i}(\on{Taylor}_i(g_{h_2}))\notag\\
&\quad- \xi_j^{Q_i}(\on{Taylor}_i(g_{h_3})) - \xi_j^{Q_i}(\on{Taylor}_i(g_{h_4}))\bigg)}_{\mb{R}/\mb{Z}}\le (MD/\rho)^{O_s(d^{O_s(1)})} N^{-i}\label{eq:linear-output-2}
\end{align}
for all $1\le j\le\dim(V_{i,\mr{Dep}})-\dim(L_i^\ast)$ where $T_1$ is an integer bounded by $(MD/\rho)^{O_s(d^{O_s(1)})}$. Here we have used that $\xi_j^{P_i}(\on{Taylor}_i(g_h))$ is equal for all $h\in H$ by Lemma~\ref{lem:sunflower}.

We define functions $f,g\colon H\to\mb{R}^{\sum_{i=1}^{s-1}\dim(V_{i,\mr{Dep}})-\dim(L_i^\ast)}$ via
\begin{align*}
f(h) &= (T_1\xi_j^{Q_i}(\on{Taylor}_i(g_h)))_{1\le i\le s-1,~1\le j\le\dim(V_{i,\mr{Dep}})-\dim(L_i^\ast)},\\
g(h) &= (T_1\xi_j^{P_i}(\on{Taylor}_i(g_h)) + T_1\xi_j^{Q_i}(\on{Taylor}_i(g_h)))_{1\le i\le s-1,~1\le j\le\dim(V_{i,\mr{Dep}})-\dim(L_i^\ast)}.
\end{align*}
Note that for the additive quadruples on which we have \eqref{eq:linear-output-2}, we are exactly in the situation necessary to apply results on approximate homomorphisms. 

In particular, we may apply Lemma~\ref{lem:approximate}. We see that there exists $H'\subseteq H$ having density at least $\exp(-O_s(d\log(MD/\rho))^{O_s(1)})$ such that for all $i,j$ and $h\in H'$, we have
\begin{equation}\label{eq:linear-output-3}
\norm{T_1\xi_j^{Q_i}(\on{Taylor}_i(g_h)) - \bigg(\gamma_{i,j} + \sum_{k=1}^{d^\ast}\alpha_{i,j,k} \{\beta_k h\}\bigg)}_{\mb{R}/\mb{Z}}\le(MD/\rho)^{O_s(d^{O_s(1)})} N^{-i},
\end{equation}
where:
\begin{itemize}
    \item $d^\ast\le(d\log(MD/\rho))^{O_s(1)}$;
    \item $\beta_k\in(1/N')\mb{Z}$ where $N'$ is a prime between $100N$ and $200N$.
\end{itemize} 

At this point, for each $i$ we find elements $Z_{i,j}$ for $1\le j\le\dim(V_{i,\mr{Dep}})-\dim(L_i^\ast)$ which are $(MD/\rho)^{O_s(d^{O_s(1)})}$-rational combinations of $\{w_{i,j}\colon 1\le j\le\dim(V_{i,\mr{Dep}})\}$ such that 
\begin{equation}\label{eq:basis-def}
T_1\xi_j^{Q_i}(Z_{i,j}) = 1\text{ and }\xi_{j'}^{Q_i}(Z_{i,j}) = 0
\end{equation}
for $j'\neq j$ such that $1\le j,j'\le\dim(V_{i,\mr{Dep}})-\dim(L_i^\ast)$. We define
\[W_{i,\mr{Lin}} = \on{span}_{\mb{R}}((Z_{i,j})_{1\le j\le\dim(V_{i,\mr{Dep}})-\dim(L_i^\ast)}).\]

We see that there are no nontrivial linear relations between $W_{i,\ast}$ and $W_{i,\mr{Pet}}+W_{i,\mr{Lin}}$ since $W_{i,\ast}\leqslant P_i(V_i)$ and $W_{i,\mr{Pet}}+W_{i,\mr{Lin}}\leqslant Q_i(V_i)$. There are no linear relations between $W_{i,\mr{Pet}}$ and $W_{i,\mr{Lin}}$ as $W_{i,\mr{Pet}}$ lies in the joint kernel of the $\xi_j^{Q_i}$ and therefore using \eqref{eq:basis-def} one can prove any such relation is trivial. Furthermore, by construction we have $V_{i,\mr{Dep}} = Q_i(V_i) = W_{i,\mr{Lin}} + W_{i,\mr{Pet}}$. Finally $L_i = P_i(L_i) + Q_i(L_i)  = W_{i,\ast} + W_{i,\mr{Lin}} + W_{i,\mr{Pet}}$; this implicitly uses $Q(L_i) = Q_i(V_i) \leqslant L_i$.

\noindent\textbf{Step 6: Constructing the desired factorizations and completing the proof.}
Using the refined factorization \eqref{eq:linear-output-1} implies that
\[\pi_1(\tau(\on{Taylor}_i(\wt{g}_{\vec{h}}))) \in L_i\]
since $\pi_1(\wt{K_i}) \leqslant L_i$. Applying $g_{\vec{h}}^\ast = \eps_{\vec{h}}\cdot\wt{g_{\vec{h}}}\cdot\gamma_{\vec{h}}$ in the first coordinate then implies that
\begin{equation}\label{eq:linear-output-5}
\on{dist}(\on{Taylor}_i(g_{h_1}), L_i + T_2^{-1}\on{Horiz}_i(\Gamma))\le (MD/\rho)^{O_s(d^{O_s(1)})} N^{-i}
\end{equation}
for $h_1\in H'$ where distance is in $L^\infty$ after expressing values in terms of $\exp(\mc{X}_i)$. Here $T_2$ is an integer bounded by $(MD/\rho)^{O_s(d^{O_s(1)})}$. Furthermore recall from Lemma~\ref{lem:sunflower} that
\begin{equation}\label{eq:linear-output-6}
\on{Taylor}_i(g_h)-\on{Taylor}_i(g_{h'})\in V_{i,\mr{Dep}}=Q_i(V_i)
\end{equation}
for $h,h'\in H'\subseteq H$.

Let $Y_{i,j}\in\on{span}_{\mb{R}}(\mc{X}\cap\log(G_{(i,1)})\setminus \mc{X}\cap\log(G_{(i,2)}))$ be such that $\exp(Y_{i,j}) \imod G_{(i,2)}= Z_{i,j}$. Then for $h\in H'$, we define
\begin{equation}\label{eq:linear-output-7}
\wt{g}_h(n) = \prod_{i=1}^{s}\prod_{j=1}^{\dim(W_{i,\mr{Lin}})}\exp(Y_{i,j})^{T_1^{-1}\binom{n}{i}\cdot (\gamma_{i,j}+\sum_{k=1}^{d^\ast}\alpha_{i,j,k}\{\beta_k h\})}.
\end{equation}
By construction and Lemma~\ref{lem:coeff-mult}, for $h,h'\in H'$ we have 
\begin{align*}
\on{Taylor}_i(\wt{g}_h^{-1}g_h) - \on{Taylor}_i(\wt{g}_{h'}^{-1}g_{h'}) & \in Q_i(V_i),\\
\on{dist}(\on{Taylor}_i(\wt{g}_h^{-1}g_h), L_i + T_2^{-1} \on{Horiz}_i(\Gamma))&\le (MD/\rho)^{O_s(d^{O_s(1)})} \cdot N^{-i},\\
\snorm{T_1\xi_j^{Q_i}(\wt{g}_h^{-1}g_h)}_{\mb{R}/\mb{Z}}&\le (MD/\rho)^{O_s(d^{O_s(1)})} \cdot N^{-i},
\end{align*}
where $1\le i\le s-1$ and $1\le j\le\dim(V_{i,\mr{Dep}})-\dim(L_i^\ast)$. The first line comes from \eqref{eq:linear-output-6}, the second line from \eqref{eq:linear-output-5}, and the third from \eqref{eq:linear-output-3} and \eqref{eq:linear-output-7}, in conjunction with \eqref{eq:basis-def}.

We now fix an element $h_2\in H'$. For each $h_1\in H'$ we write 
\begin{align*}
g_{h_1}' &= \wt{g}_{h_1} \cdot (\wt{g}_{h_1}^{-1} g_{h_1}')= \wt{g}_{h_1} \cdot (\wt{g}_{h_1}^{-1} g_{h_1}') \cdot (\wt{g}_{h_2}^{-1} g_{h_2}')^{-1} \cdot (\wt{g}_{h_2}^{-1} g_{h_2}').
\end{align*}
By applying Lemma~\ref{lem:factor}, we may write 
\[(\wt{g}_{h_1}^{-1} g_{h_1}') \cdot (\wt{g}_{h_2}^{-1} g_{h_2}')^{-1} = \eps_{h_1}^\ast g_{h_1}^\ast \gamma_{h_1}^\ast,\qquad(\wt{g}_{h_2}^{-1} g_{h_2}') = \eps^\ast g^\ast \gamma^\ast\]
where $\gamma^{\ast},\gamma_{h_1}^\ast$ are $(MD/\rho)^{O_s(d^{O_s(1)})}$-rational, $\eps^{\ast},\eps_{h_1}^\ast$ are $((MD/\rho)^{O_s(d^{O_s(1)})},N)$-smooth, and we have $\on{Taylor}_i(g_{h_1}^\ast) \in L_i^\ast = W_{i,\mr{Pet}}$ using the first and third lines above and $\on{Taylor}_i(g^\ast) \in L_i^\ast+P_i(L_i)=W_{i,\ast} + W_{i,\mr{Pet}}$ using the second and third lines above. (Recall that $L_i^\ast\leqslant Q_i(V_i)$ is cut out by the $\xi_j^{Q_i}$.) Additionally, these sequences are the identity at $0$.

Therefore, for $h_1\in H'$ we have
\begin{align*}
g_{h_1}' &= \eps_{h_1}^\ast \eps^\ast ((\eps_{h_1}^\ast \eps^\ast)^{-1}\wt{g}_{h_1}(\eps_{h_1}^\ast \eps^\ast))((\eps^\ast)^{-1}g_{h_1}^\ast\eps^\ast)((\eps^\ast)^{-1}\gamma_{h_1}^\ast\eps^\ast(\gamma_{h_1}^\ast)^{-1})(\gamma_{h_1}^{\ast}g^\ast(\gamma_{h_1}^\ast)^{-1})(\gamma_{h_1}^\ast\gamma^\ast)\\
&=:(\eps_{h_1}^\ast\eps^\ast) \cdot g_{h_1}^{\triangle} \cdot (\gamma_{h_1}^\ast\gamma^\ast).
\end{align*}
So, for $h_3,h_4\in H$ we deduce using Lemma~\ref{lem:coeff-mult} and the above analysis that
\begin{align*}
\on{Taylor}_i(\wt{g}_{h_3}) &\in W_{i,\mr{Lin}},\\
\on{Taylor}_i(g_{h_3}^{\triangle}) &\in L_i =  W_{i,\ast} + W_{i,\mr{Lin}}  + W_{i,\mr{Pet}},\\
\on{Proj}_{W_i,\mr{Lin}}(\on{Taylor}_i(g_{h_3}^{\triangle})) &= \on{Proj}_{W_i,\mr{Lin}}(\on{Taylor}_i(\wt{g}_{h_3})),\\
\on{Taylor}_i(\wt{g}_{h_3}^{-1}g_{h_3}^{\triangle}) - \on{Taylor}_i(\wt{g}_{h_4}^{-1}g_{h_4}^{\triangle}) &\in W_{i,\mr{Pet}}.
\end{align*}
Furthermore note that $\eps_{h_1}^\ast \eps^\ast$ is sufficiently smooth and $\gamma_{h_1}^\ast\gamma^\ast$ is appropriately rational. This nearly gives the desired result except we need to remove the rational and smooth parts exactly as in Step 7 of Lemma~\ref{lem:sunflower}; we omit the details, although note that the only difference between $g_h^{\triangle}$ and the output is a conjugation by a fixed element which leaves all properties unchanged and the Fourier phase on the $\mb{R}$ part of $(G^\ast)'$ may be modified. Additionally, the set $H'$ will be made smaller by acceptable factors due to Pigeonhole.

\noindent\textbf{Step 7: Handling the exceptional case $s = 2$.}
In this exceptional case, we have $r^\ast = 1$ and $s=2$, and $\eta$ is nontrivial. The difference here versus the prior analysis is that the error term $\psi_h(g_{\vec{h}}(n)\Gamma')$ is replaced by $e(\Theta_{\vec{h}}n)$ with $\snorm{\Theta_{\vec{h}}}_{\mb{R}/\mb{Z}}\le(MD/\rho)^{O_s(d^{O_s(1)})}N^{-1}$ by using the Remark~\ref{rem:quadruple-2} regarding Lemma~\ref{lem:four-fold-biased} for $s=2$.

We take $G^{\mr{Error}} = \mb{R}$, $\Gamma^{\mr{Error}} = \mb{Z}$,  $g_{\vec{h}}(n) = \Theta_{\vec{h}}n$, and $\psi_{\vec{h}}(z) = e(z)$. $\wt{G}$ is defined as before. Taking $\eta^\ast = (\eta,\eta,-\eta,-\eta,1)$, by Corollary~\ref{cor:equi-deg} we may factor 
\[g_{\vec{h}}^\ast = \eps_{\vec{h}} \cdot\wt{g}_{\vec{h}} \cdot\gamma_{\vec{h}}\]
where 
$\eps_{\vec{h}}$ is $((MD/\rho)^{O_s(d^{O_s(1)})},N)$-smooth, $\gamma_{\vec{h}}$ is $(MD/\rho)^{O_s(d^{O_s(1)})}$-rational, and $\wt{g}_{\vec{h}}$ lies in a $(MD/\rho)^{O_s(d^{O_s(1)})}$-rational subgroup $K$ such that $\eta^\ast(K\cap\wt{G}_{(1,1)}) = 0$. Note however that 
\[g_{\vec{h}}^\ast = (\mr{id}_G,\mr{id}_G,\mr{id}_G,\mr{id}_G,\Theta_{\vec{h}} n)\cdot (\tau(g_{\vec{h}}^\ast),0)\]
where $\tau\colon\wt{G}\to G^{\otimes 4}$ is the natural projection. Let $K^\ast = K \cap (G^{\otimes 4}\times\{0\})$ and note that $K^\ast$ can be defined as the joint kernel of certain horizontal characters of height $(MD/\rho)^{O_s(d^{O_s(1)})}$ (namely, ones defining $K$ along with one of the form $(0,0,0,0,1)$). Since $\Theta_{\vec{h}}$ is small and is the only part in the fifth coordinate, arguments similar to before allow us to refine the first factorization (up to changing implicit constants) and instead assume that $\wt{g}_{\vec{h}}$ lies in $K^\ast$.

Furthermore note that if $\eta_{\mr{Prod}} = (\eta,\eta,-\eta,-\eta,0)$ we have that $\eta_{\mr{Prod}}(K^{\ast}) = 0$ as $\eta_{\mr{Prod}}$ and $\eta^\ast$ agree on the initial four groups. At this point we are exactly in the situation of the earlier analysis and we may complete the proof.\footnote{Various simplifications are possible in the case since the underlying groups are all abelian here; in particular, invoking Corollary~\ref{cor:equi-deg} reduces to summing a geometric series.}
\end{proof}

We remark that modulo minor annoyances, the strategy of using Lemma~\ref{lem:four-fold-biased}, deducing an approximate homomorphism, and then applying results coming from the Bogolyubov lemma was introduced by Gowers \cite{Gow98} in his seminal work on four-term arithmetic progressions. It was similarly applied in work of Green and Tao \cite{GT08b} on the $U^3$-inverse theorem. In a certain sense, the previous two sections can be thought of as showing that, given an appropriate equidistribution theorem and defining a number of notions for nilmanifolds, this analysis can be modified to make sense in the greater generality of nilmanifolds where the group is not abelian.

\section{Setup for extracting a \texorpdfstring{$(1,s-1)$}{(1,s-1)}-nilsequence}\label{sec:nil-first}
Before diving into the formal proof, we motivate how we extract the ``top degree-rank'' part and why lifting to the universal nilmanifold plays a role in our argument at this stage. We remark that Green, Tao, and Ziegler \cite{GTZ12} work with the universal nilmanifold throughout their argument (in the form of a representation of a degree-rank nilcharacter; see \cite[Definition~9.11]{GTZ12}).

Recall the bracket polynomial $U^5$-inverse sketch discussed in Section~\ref{sec:outline}; we started with functions
\[e\bigg(\sum_{i=1}^{d_1}a_{i,h} n[b_{i,h} n][c_{i,h}n] + \sum_{i=1}^{d_2}d_{i,h} n^2[e_{i,h} n] + \sum_{i=1}^{d_3} f_{i,h}n[g_{i,h}n] + j_hn^3 + \ell_h n^2 + m_hn\bigg)\]
which correlate with $\Delta_{h}f$. At this point, we have proven that 
\[\sum_{i=1}^{d_1}a_{i,h} n[b_{i,h} n][c_{i,h}n]\]
is equivalent to a bracket polynomial up to lower order terms of degree-rank of the form
\[\sum_{i = 1}^{d_1'} \delta_i \{\eps_i h\}n[\beta_{i, \ast}n][\gamma_{i, \ast}n].\]
Our goal at this stage is to isolate 
\[e\bigg(\sum_{i = 1}^{d_1'} \delta_i \{\eps_i h\}n[\beta_{i, \ast}n][\gamma_{i, \ast}n]\bigg);\]
in the next section we will then convert this ``top degree--rank'' bracket phase into a $(1,s-1)$--nilsequence. 

The reason lifting to a universal nilmanifold proves so technically useful is that it enables us to isolate various components of the horizontal tori as ``separate subgroups''. For the sake of simplicity, consider a $2$-step group $G$ in the $U^4$-inverse case given the degree-rank filtration $G_{(0,0)} = G_{(1,0)}=G_{(1,1)} = G$, $G_{(2,0)} = G_{(2,1)} = G_{(2,2)}= [G,G]$, where the remaining groups are trivial. In this case, the output of Lemma~\ref{lem:linear-output} gives the linearly disjoint subspaces $W_{\ast}$, $W_{\mr{Lin}}$, $W_{\mr{Pet}}$ of $V = G/[G,G]$ such that the commutator of any two elements in $W_{\mr{Lin}} + W_{\mr{Pet}}$ vanishes and the commutator of any element of $W_{\ast}$ and $W_{\mr{Pet}}$ vanishes.

Let $\mc{Z}_{\ast}$ denote the rational basis of $\log(W_{\ast}) \imod [G,G]$ and $\mc{Z}_{\mr{Lin}}$ and $\mc{Z}_{\mr{Pet}}$ be analogous. We also have a decomposition of our polynomial
\[g_h = g_{h, \ast} + g_{h, \mr{Lin}} + g_{h, \mr{Pet}} \imod [G, G]\]
where 
\begin{align*}
g_{h, \mr{Lin}}(n) &= \prod_{i = 1}^{\dim(W_{\mr{Lin}})} \exp(\delta_in \{\eps_i h\} Z_{i}^{\mr{Lin}}),\qquad g_{h, \ast}(n) = \prod_{i = 1}^{\dim(W_{\ast})} \exp(\beta_in Z_{i}^{\ast}),\\
g_{h, \mr{Pet}}(n) &= \prod_{i = 1}^{\dim(W_{\mr{Pet}})} \exp(\gamma_{i}^{h}n Z_{i}^{\mr{Pet}}).
\end{align*}
Therefore we may write 
\[g_h(n) = g_{h, \ast}(n)g_{h, \mr{Lin}}(n)g_{h, \mr{Pet}}(n) g_{h, \mr{Rem}}(n)\]
with $g_{h,\mr{Rem}}(n)\in[G,G]$ pointwise. The top order term which we seek to isolate is heuristically similar to
\[e\bigg(\sum_{i=1}^{\dim(W_{\mr{Lin}})}\sum_{j=1}^{\dim(W_{\ast})}\delta_in \{\eps_i h\}[\beta_j n] [Z_{i}^{\mr{Lin}},Z_{j}^{\ast}]\bigg).\]

Note that given the factorization of $g_h$, we have established no control over $g_{h, \mr{Pet}}$ and $g_{h, \mr{Rem}}$. This may suggest that we wish to quotient out by the subgroup $W_{\mr{Pet}}[G,G]$ in order to kill these terms; note however that $G/(W_{\mr{Pet}}[G, G])$ now abelian and such a projection ``kills'' the higher order degree-rank term calculated above. This suggest that the group $W_{\mr{Pet}}[G,G]$ is ``too large'' a quotient. The solution is to ``enlarge" the group $G$ so that the subgroup $[W_{\mr{Lin}},W_{\ast}]$ and the subgroup $G'$ that corresponds to the remaining phases $\gamma_h n^2 + \delta_h n$ are disjoint. We can then quotient by $W_{\mr{Pet}}G'$. This disjointness is accomplished by lifting to the universal nilmanifold of degree-rank $(2, 2)$.

\subsection{Unwinding the output of Lemma~\ref{lem:linear-output}}
We first require the following elementary lemma regarding lattice elements when presented in first-kind coordinates.
\begin{lemma}\label{lem:divis-basis}
Fix an integer $k\ge 1$. Consider a nilmanifold $G/\Gamma$ of dimension $d$ with a Mal'cev basis $\mc{X} = \{X_1,\ldots,X_d\}$ of $\log G$ which is $Q$-rational and such that $\mc{X}$ has the degree $k$ nesting property. Then there exists a positive integer $Q'\le O_k(Q^{O_k(d^{O_k(1)})})$ such that if $z_j\in Q'\cdot\mb{Z}$ then 
\[\exp\bigg(\sum_{j=1}^dz_jX_j\bigg)\in\Gamma.\]
\end{lemma}
\begin{proof}
Note that $\Gamma = \psi_{\mc{X}}(\mb{Z}^d)$. By \cite[Lemma~B.1]{Len23b}, $\psi_{\mc{X}} \circ \psi_{\mr{exp},\mc{X}}^{-1}$ is a degree $O_k(1)$ polynomial with coefficients of height at most $Q^{O_k(d^{O(1)})}$. The desired result then follows by taking $Q'$ to the least common multiple of all denominators of all coefficients present in this polynomial (since there are only $O_k(d^{O_k(1)})$ total coefficients). Note that the polynomial corresponding to $\psi_{\mc{X}} \circ \psi_{\mr{exp},\mc{X}}^{-1}$ has no constant term by observing the image of $\mr{id}_G$. 
\end{proof}

We next require the following additional elementary lemma which gives a Taylor series expansion which is ``graded by the Mal'cev basis''.
\begin{lemma}\label{lem:Taylor-mod}
Consider a nilmanifold $G/\Gamma$ of degree $k$ with an adapted Mal'cev basis $\mc{X} = \{X_1,\ldots,X_{\dim(G)}\}$ and a polynomial sequence $g(n)$. There exists a representation 
\[g(n) = \prod_{i=0}^{k}\prod_{j=\dim(G)-\dim(G_{k})+1}^{\dim(G)} \exp(X_{j})^{\alpha_{i,j}\cdot\frac{n^{i}}{i!}}\]
where $\alpha_{i,j}\in\mb{R}$.
\end{lemma}
\begin{proof}
Note via Baker--Campbell--Hausdorff and existence of Taylor expansions, we may write 
\[g(n) = \exp\bigg(\sum_{i=0}^{s}g_i\cdot\frac{n^{i}}{i!}\bigg)\]
with $g_i\in\log(G_i)$. Let $g_0(n)=g(n)$ and $g_{0,i}=g_i$. Then iteratively define $g_{\ell+1}(n)$ by the following process: write $\sum_{j=\dim(G)-\dim(G_{(\ell,0)})+1}^{\dim(G)}\alpha_{\ell,j}X_j=g_{\ell,\ell}$. Then let
\[g_{\ell + 1}(n) := \Bigg(\prod_{j= \dim(G)-\dim(G_\ell)+1}^{\dim(G)} \exp(X_{j})^{\alpha_{\ell,j}\cdot\frac{n^\ell}{\ell!}}\Bigg)^{-1} g_\ell(n)\]
and write
\[g_{\ell+1}(n) = \exp\bigg(\sum_{i=\ell+1}^{s}g_{\ell+1,i}\cdot\frac{n^i}{i!}\bigg)\]
in order to define $g_{\ell+1,i}$. There exists a valid choice of $\alpha_{\ell,j}$ at each step since $\mc{X}$ is a filtered Mal'cev basis and there exists a valid choice of $g_{\ell+1,i}$ for $i\ge\ell+1$ by Baker--Campbell--Hausdorff. This process terminates with the identity sequence, and unraveling gives the desired.
\end{proof}
\begin{remark*}
Note that in the above proof, the reason we do not use the basis $\binom{n}{i}$ is that $\binom{n}{i}\binom{n}{j}$ is \emph{not} a linear combination of polynomials of the form $\binom{n}{t}$ for $t\ge\max(i,j)+1$ and hence the Baker--Campbell--Hausdorff to construct $g_{\ell+1,i}$ fails (one needs lower-degree terms with $i\le\ell$).
\end{remark*}

We now explicitly unwind, for the sake of clarity, the conclusion of Lemma~\ref{lem:linear-output}. We will use the notation and conclusions here throughout the Sections~\ref{sec:nil-first} and \ref{sec:nil-first-2}. Suppose we have a $1$-bounded function $f\colon[N]\to\mb{C}$ with a degree-rank $(s-1,r^\ast)$ correlation structure with parameters $\rho, M,D,d$.\footnote{We apologize to the reader; there is a rather incredible amount of data which is floating around at this point. The crucial details to track are data regarding Taylor coefficient and the associated decompositions of the vector spaces corresponding to horizontal tori.} Then by Lemma~\ref{lem:linear-output} and some relabeling there exists a degree-rank $(s-1,r^\ast)$ correlation structure with parameters
\begin{align*}
\rho'^{-1}&\le\exp(O_s((d\log(MD/\rho))^{O_s(1)})),\quad M' \le O(M),\quad D' = D,\quad d' \le O(d)
\end{align*}
and
\begin{itemize}
    \item A subset $H\subseteq[N]$ with $|H|\ge\rho' N$;
    \item A multidegree $(1,s-1)$ nilcharacter $\chi(h,n)$ with a frequency $\eta^\ast$ with height at most $M$. Furthermore $\chi$ lives on a nilmanifold $(G^\ast\times\mb{R})/(\Gamma^\ast\times\mb{Z})$ with dimension bounded by $d'$, output dimension bounded by $D'$, complexity bounded by by $M'$, and the function underlying $\chi$ is $M'$-Lipschitz. We let $g(h,n)$ denote the underlying polynomial sequence;
    \item A collection of degree-rank $(s-1,r^\ast)$ nilcharacters $\chi_h(n)$ with a frequency $\eta$ with of height at most $M$. Furthermore $\chi_h$ lives on a nilmanifold $G/\Gamma$ with dimension bounded by $d$, output dimension bounded by $D$, $G/\Gamma$ has complexity bounded by $M$ and the function underlying $\chi_h$ (which is independent of $h$) is $M'$-Lipschitz. We let $g_h$ denote the underlying polynomial sequence and we have $g_h(0) = \mr{id}_G$;
    \item For all $h\in H$, we have 
    \[\Delta_h f(n) \otimes \chi(h,n) \otimes \chi_h(n) \in\on{Corr}(s-2,\rho', M', d');\]
    \item Then there exists a collection of subspaces $W_{i,\ast}, W_{i,\mr{Lin}}, W_{i,\mr{Pet}}\leqslant G_{(i,1)}/G_{(i,2)}$ for $1\le i\le s-1$ which are $(MD/\rho)^{O_s(d^{O_s(1)})}$-rational with respect to $\exp(\mc{X})\cap G_{(i,1)}\imod G_{(i,2)}$;
    \item If $W_i = W_{i,\ast} +  W_{i,\mr{Lin}} + W_{i,\mr{Pet}}$ then $\dim(W_i) = \dim(W_{i,\ast}) +  \dim(W_{i,\mr{Lin}}) + \dim(W_{i,\mr{Pet}})$;
    \item Let $Z_{i,1}^{\ast},\ldots,Z_{i,\dim(W_{i,\ast})}^\ast$ a sequence of integral linear combinations of $\mc{X}\cap G_{(i,1)}\setminus \mc{X}\cap G_{(i,2)}$ such that $\on{span}_\mb{R}(\exp(Z_{i,1}^\ast\imod G_{(i,2)},\ldots,\exp(Z_{i,\dim(W_{i,\ast})}^\ast)\imod G_{(i,2)}) = W_{i,\ast}$. We may let the coefficients of $Z_{i,j}^\ast$ be $(MD/\rho)^{O_s(d^{O_s(1)})}$-bounded and $\exp(Z_{i,j}^\ast)\in\Gamma$.
    \item Let $Z_{i,1}^{\mr{Lin}},\ldots,Z_{i,\dim(W_{i,\mr{Lin}})}^{\mr{Lin}}$ be a sequence of integral linear combinations of $\mc{X}\cap G_{(i,1)}\setminus\mc{X}\cap G_{(i,2)}$ such that $\on{span}_\mb{R}(\exp(Z_{i,1}^{\mr{Lin}})\imod G_{(i,2)},\ldots,\exp(Z_{i,\dim(W_{i,\mr{Lin}})}^{\mr{Lin}})\imod G_{(i,2)}) = W_{i,\mr{Lin}}$. We may let the coefficients of $Z_{i,j}^{\mr{Lin}}$ are $(MD/\rho)^{O_s(d^{O_s(1)})}$-bounded and $\exp(Z_{i,j}^{\mr{Lin}})\in\Gamma$.
    \item Let $Z_{i,1}^{\mr{Pet}},\ldots,Z_{i,\dim(W_{i,\mr{Pet}})}^{\mr{Pet}}$ be a sequence of integral linear combinations of $\mc{X}\cap G_{(i,1)}\setminus \mc{X}\cap G_{(i,2)}$ such that $\on{span}_\mb{R}(\exp(Z_{i,1}^{\mr{Pet}})\imod G_{(i,2)},\ldots,\exp(Z_{i,\dim(W_{i,\mr{Pet}})}^{\mr{Pet}})\imod G_{(i,2)}) = W_{i,\mr{Pet}}$. We may let the coefficients of $Z_{i,j}^{\mr{Pet}}$ be $(MD/\rho)^{O_s(d^{O_s(1)})}$-bounded and let $\exp(Z_{i,j}^{\mr{Pet}})\in\Gamma$.
    \item For $1\le i\le s-1$ and $h\in H$, we have 
    \begin{align*}
    \on{Taylor}_i(g_h) &= \prod_{j=1}^{\dim(W_{i,\ast})}\exp(Z_{i,j}^\ast)^{z_{i,j}^\ast}\cdot\prod_{j=1}^{\dim(W_{i,\mr{Pet}})} \exp(Z_{i,j}^{\mr{Pet}})^{z_{i,j}^{h,\mr{Pet}}}\\
    &\qquad \qquad\cdot\prod_{j=1}^{\dim(W_{i,\mr{Lin}})} \exp(Z_{i,j}^{\mr{Lin}})^{z_{i,j}^{h,\mr{Lin}}} \imod G_{(i,2)}
    \end{align*}
    where 
    \[z_{i,j}^{h,\mr{Lin}} = \gamma_{i,j} + \sum_{k=1}^{d^\ast}\alpha_{i,j,k}\{\beta_k h\}\]
    where $d^\ast\le(d\log(MD/\rho))^{O_s(1)}$ and $\beta_k\in(1/N')\mb{Z}$ where $N'$ is a prime in $[100N,200N]$.
    \item For any integers $i_1 + \cdots + i_{r^\ast} = s-1$, suppose that $v_{i_j}\in W_{i_j}$ for all $j$. If for at least one index $\ell$ we have $v_{i_\ell}\in W_{i_\ell,\mr{Pet}}$, then if $w$ is any $(r^\ast-1)$-fold commutator of $v_{i_1},\ldots,v_{i_{r^\ast}}$, we have
    \[\eta(w) = 0.\]
    Furthermore, if instead for at least two indices $\ell_1,\ell_2$ we have $v_{i_{\ell_1}}\in W_{i_{\ell_1},\mr{Lin}}$ and $v_{i_{\ell_2}}\in W_{i_{\ell_2},\mr{Lin}}$, then if $w$ which is any $(r^\ast-1)$-fold commutator of $v_{i_1},\ldots,v_{i_{r^\ast}}$ we have
    \[\eta(w) = 0.\]
\end{itemize}
We have relabeled as $H'$ by $H$, $g_h'$ by $g_h$, $g'(h,n)$ by $g(h,n)$, $\chi'$ by $\chi$, and $\chi_h'$ by $\chi_h$. We have applied Lemma~\ref{lem:divis-basis} and scaling to guarantee that $\exp(Z_{i,j}^\ast),\exp(Z_{i,j}^{\mr{Lin}}),\exp(Z_{i,j}^{\mr{Pet}})\in\Gamma$.

Let $\mc{X} = \{X_1,\ldots, X_{\dim(G)}\}$ denote the filtered Mal'cev basis given for $G/\Gamma$. Via Lemma~\ref{lem:Taylor-mod}, for $h\in H$ we may define 
\begin{align*}
g_h^\ast &= \prod_{i=1}^{s-1}\prod_{j=1}^{\dim(W_{i,\ast})}\exp(Z_{i,j}^\ast)^{z_{i,j}^\ast\cdot\frac{n^i}{i!}} \cdot\prod_{i=1}^{s-1}\prod_{j=1}^{\dim(W_{i,\mr{Lin}})}\exp(Z_{i,j}^{\mr{Lin}})^{\gamma_{i,j}\cdot\frac{n^i}{i!}},\\
g_h^{\mr{Lin}} &= \prod_{i=1}^{s-1}\prod_{j=1}^{\dim(W_{i,\mr{Lin}})}\exp(Z_{i,j}^{\mr{Lin}})^{(z_{i,j}^{h,\mr{Lin}}-\gamma_{i,j})\cdot\frac{n^i}{i!}},\\
g_h^{\mr{Pet}} &= \prod_{i=1}^{s-1}\prod_{j=1}^{\dim(W_{i,\mr{Pet}})}\exp(Z_{i,j}^{\mr{Pet}})^{z_{i,j}^{h,\mr{Pet}}\cdot\frac{n^i}{i!}},
\end{align*}
and define $g_h^{\mr{Rem}}$ via
\[g_h = g_h^\ast \cdot g_h^{\mr{Lin}} \cdot g_h^{\mr{Pet}} \cdot g_h^{\mr{Rem}}.\]
Using Lemma~\ref{lem:Taylor-mod} again, we may write 
\[g_h^{\mr{Rem}} = \prod_{i=1}^{s-1}\prod_{j=\dim(G)-\dim(G_{i,1})+1}^{\dim(G)}\exp(X_j)^{\kappa_{i,j}^{h}\cdot\frac{n^i}{i!}}.\]

The fact that when applying Lemma~\ref{lem:Taylor-mod} for $g_h^{\mr{Rem}}$ we observe no coefficients for $\frac{n^{i}}{i!}$ corresponding to basis elements in $\mc{X}\cap\log(G_{(i,1)})\setminus\mc{X}\cap\log(G_{(i,2)})$ follows from the fact that $g_h$ and $g_h^\ast \cdot g_h^{\mr{Lin}} \cdot g_h^{\mr{Pet}}$ have Taylor coefficients which match exactly for $1\le i\le s-1$.

We now reach the first stage of ``rewriting'' where we realize the nilsequence $\chi_h(n) = F(g_h(n)\Gamma)$ on a universal nilmanifold.

\subsection{Rewriting degree-rank nilsequences on the universal nilmanifold}
We recall the universal nilmanifold of a given degree-rank (see \cite[Definition~9.1]{GTZ12}). 
\begin{definition}\label{def:universal-1}
The \emph{universal nilmanifold of degree-rank $(s-1,r^\ast)$} and the associated discrete cocompact subgroup are defined as follows. We write $G_{\mr{Univ}} = G_{\mr{Univ}}^{\vec{D}}$ where $\vec{D} = \vec{D}^{\ast}+\vec{D}^{\mr{Lin}}+\vec{D}^{\mr{Pet}}$ with $\vec{D}^{\ast},\vec{D}^{\mr{Lin}},\vec{D}^{\mr{Pet}}\in (\mb{Z}_{\ge 0})^{s-1}$. We specify $G_{\mr{Univ}}^{\vec{D}}$ by formal generators of the Lie algebra
$e_{i,j}$ for $1\le i\le s-1$ and $1\le j\le D_i$ where $D_i = D_i^\ast + D_i^{\mr{Lin}} + D_i^{\mr{Pet}}$ with the relations:
\begin{itemize}
    \item Any $(r-1)$-fold commutator of $e_{i_1,j_1},\ldots, e_{i_r,j_r}$ with $i_1 + \cdots + i_r > (s-1)$ vanishes;
    \item Any $(r-1)$-fold commutator of $e_{i_1,j_1},\ldots, e_{i_r,j_r}$ with $i_1 + \cdots + i_r = (s-1)$ and $r>r^\ast$ vanishes.
\end{itemize}
The associated discrete group which we will be concerned with is $\Gamma_{\mr{Univ}}$ which is the discrete group generated by $\exp(e_{i,j})$ for $1\le i\le s-1$ and $1\le j\le D_i$.
\end{definition}
\begin{remark*}
Note that in this definition, $G_{\mr{Univ}}^{\vec{D}}$ depends only on $\vec{D}$; however, the quotient we will consider later depends on $\vec{D}^{\ast},\vec{D}^{\mr{Lin}},\vec{D}^{\mr{Pet}}$. Furthermore, we have presented $G_{\mr{Univ}}$ as a Lie algebra and not as a Lie group; via the general theory of nilpotent Lie algebras this is sufficient. Note that the Lie algebra defined is trivially seen to be nilpotent. By the Birkhoff Embedding Theorem (see remark following \cite[Theorem~1.1.11]{CG90}), we may realize any real nilpotent Lie algebra $\mf{g}$ as a Lie subalgebra of the $n\times n$ real strictly upper triangular matrices. The proof of \cite[Theorem~1.2.1]{CG90} then realizes the $n\times n$ real strictly upper triangular matrices as a logarithm of a connected, simply connected Lie group $N_n$ where the exponential map is bijective. The Baker--Campbell--Hausdorff formula then demonstrates $\mf{g}$ is the logarithm of a connected, simply connected subgroup $G\leqslant N_n$ (and by construction the logarithm is a bijection between $G$ and $\mf{g}$). The group $G$ constructed is unique up to isomorphism by Lie's third theorem.
\end{remark*}

We first prove the fact that $G_{\mr{Univ}}$ may be given a degree-rank filtration and that $G_{\mr{Univ}}/\Gamma_{\mr{Univ}}$ has reasonable complexity.
\begin{lemma}\label{lem:universal-complexity}
Let $G_{\mr{Univ}} = G_{\mr{Univ}}^{\vec{D}}$ and define $(G_{\mr{Univ}})_{(d,r)}$ by taking the group generated by all $(r'-1)$-fold iterated commutators of $\exp(t_{i_1,j_1}e_{i_1,j_1}),\ldots,\exp(t_{i_{r'},j_{r'}}e_{i_{r'},j_{r'}})$ with $t_{i_k,j_k}\in\mb{R}$, and either $i_1+\cdots+i_{r'}>d$ or $i_1+\cdots+i_{r'}=d$ and $r'\ge r$.

Then $(G_{\mr{Univ}})_{(d,r)}$ forms a valid degree-rank $(s-1,r^\ast)$ filtration of $G_{\mr{Univ}}$. Furthermore the dimension of $G_{\mr{Univ}}$ is bounded by $O_s(\snorm{D}_{\infty}^{O_s(1)})$ and one may find an adapted Mal'cev basis $\mc{X}_{\mr{Univ}}$ such that the complexity of $G_{\mr{Univ}}/\Gamma_{\mr{Univ}}$ is at most $\exp(\snorm{D}_{\infty}^{O_s(1)})$.
\end{lemma}
\begin{proof}
We will be brief with details; that the associated filtration is valid follows via a straightforward computation with Lemma~\ref{lem:com-check}. Note $(G_{\mr{Univ}})_{(i,0)} = (G_{\mr{Univ}})_{(i,1)}$ as $r'\ge 1$ in the set of generators always. Also, $(G_{\mr{Univ}})_{(0,0)} = (G_{\mr{Univ}})_{(0,1)}$ since for all generators $e_{i,j}$ we have $i\ge 1$.

To establish the complexity bounds, the key point is noting that taking all $(r'-1)$-fold iterated commutators of $e_{i_1,j_1},\ldots,e_{i_{r'},j_{r'}}$ with $i_1 + \cdots + i_{r'}\le s-2$ or $i_1 + \cdots + i_{r'}=s-1$ and $r'\le r^\ast$ gives a spanning set for $\log(G_{\mr{Univ}})$. This immediately gives the specified dimension bound. These generators are not linearly independent; however, all relations are generated by either antisymmetry ($[x,y] + [y,x] = 0$) or the Jacobi identity ($[x,[y,z]] + [y,[z,x]] + [z,[x,y]] = 0$) applied to the set of generators specified.

To simplify matters, note that all linear relations can be reduced to those between these generators with the ``same type'' (i.e., relations between the set of $(r'-1)$-fold commutators of a given set of generators $e_{i_1,j_1},\ldots,e_{i_{r'},j_{r'}}$). These can be collected into disconnected non-interacting ``components'' which are $O_s(1)$ in size. We may take a linearly spanning set within each group; each generator not in the spanning set may be written as a linear combination of height $O_s(1)$. Define $\mc{X}$ to be the union of all these spanning elements in $\log(G_{\mr{Univ}})$. This gives us a basis. Note the subspaces $\log((G_{\mr{Univ}})_{(d,r)})$ are clearly compatible with natural subsets of these ``components'' and their associated spanning sets, demonstrating that the basis is appropriate adapted to these vector spaces $\log((G_{\mr{Univ}})_{(d,r)})$.

The last matter to check is that there exists $C_s\ge 1$ such that $C_{s}\mb{Z}^{\dim(G_{\mr{Univ}})} \subseteq \psi_{\mr{exp},\mc{X}}(\Gamma_{\mr{Univ}})\subseteq C_{s}^{-1}\mb{Z}^{\dim(G_{\mr{Univ}})}$. This follows by noting that each element $\gamma\in\Gamma$ may be written as 
\[\gamma = \prod_{k=1}^t\exp(e_{i_k,j_k})^{s_k}\]
with $s_k\in\mathbb{Z}$. We prove the first implication first; we prove that $\log(\gamma)$ may be written as a linear combination of iterated commutators where $(r'-1)$-fold commutators have denominator bounded by $C_s^{r'}$. This is trivial to prove inductively via Baker--Campbell--Hausdorff and noting that all $s$-fold commutators vanish.

For the reverse direction, consider expressions of the form
\[\gamma' = \exp\bigg(\sum_{\alpha} c_{\alpha} e_{\alpha}\bigg)\]
where $e_{\alpha}$ ranges over all possible iterated commutators (here e.g.~$e_{[(1,2),(1,3)]} := [e_{(1,2)},e_{(1,3)}]$) where $c_{\alpha}$ are sufficiently divisible integers. Let $f_{\alpha}$ be defined as the commutator of the exponential of associated elements; e.g.~$f_{[(1,2),(1,3)]} = [\exp(e_{(1,2)}),\exp(e_{(1,3)})]$. Choose a generator $\alpha'$ with the fewest number of commutators in $\gamma'$ such that $c_{\alpha'}\neq 0$. It is straightforward to see via Baker--Campbell--Hausdorff that there is an integer $M_s$ such that if $c_{\alpha}$ are all divisible by $M_s$ then  
\[f_{\alpha'}^{-c_{\alpha'}}\gamma' = \exp\bigg(\sum_{\alpha} c_{\alpha}^\ast e_{\alpha}\bigg)\]
has each $c_{\alpha}^\ast$ still divisible by $M_s$ and $c_{\alpha'}^\ast =0$ (without introducing backwards corrections).

The desired result then follows from \cite[Lemma~B.11]{Len23b}, noting that $(G_{\mr{Univ}})_{(d,r)}$ is the degree-rank ordering forming a nested sequence of subgroups.
\end{proof}

We now represent the nilsequences $\chi_h(n) = F(g_h(n))$ on the universal nilmanifold. We define
\begin{align*}
D_i^\ast &=\dim(W_{i,\ast}) + \dim(W_{i,\mr{Lin}}),\quad D_i^{\mr{Pet}} = \dim(W_{i,\mr{Pet}}) + \dim(G_{(i,2)}),\quad D_i^{\mr{Lin}} = d^\ast\dim(W_{i,\mr{Lin}}).
\end{align*}
Note that $D_i^{\mr{Lin}}\le(d\log(MD/\rho))^{O_s(1)}$ and trivially $D_i^{\ast}, D_i^{\mr{Lin}}\le d$.

Recall that $\mc{X} = \{X_1,\ldots,X_{\dim(G)}\}$ is the filtered Mal'cev basis and $Z_{i,j}^\ast$, $Z_{i,j}^{\mr{Pet}}$, $Z_{i,j}^{\mr{Lin}}$ are representative of $\log(W_{i,\ast})$, $\log(W_{i,\mr{Lin}})$, and $\log(W_{i,\mr{Pet}})$ respectively.  

We define a homomorphism $\phi\colon G_{\mr{Univ}}\to G$ by defining the map on generators. Define
\begin{align*}
\phi(\exp(e_{i,j})) = 
\begin{cases}
\exp(Z_{i,j}^\ast) &\text{ if } 1\le j\le\dim(W_{i,\ast}),\\
\exp(Z_{i,j - \dim(W_{i,\ast})}^{\mr{Lin}}) &\text{ if } \dim(W_{i,\ast}) + 1\le j\le\dim(W_{i,\ast}) + \dim(W_{i,\mr{Lin}})=D_i^\ast,\\
\exp(Z_{i,\ell}^{\mr{Lin}}) &\text{ if } 1 + (\ell-1)d^\ast\le j - D_i^\ast\le\ell d^\ast\text{ for }1\le\ell\le\dim(W_{i,\mr{Lin}}),\\
\exp(Z_{i,j - D_i^\ast-D_i^{\mr{Lin}}}^{\mr{Pet}}) &\text{ if } D_i^{\ast}+D_i^{\mr{Lin}} + 1\le j\le D_i^{\ast}+D_i^{\mr{Lin}} + \dim(W_{i,\mr{Pet}}),\\
\exp(X_{j - D_i + \dim(G)}) &\text{ if } D_i - \dim(G_{(i,2)}) + 1\le j\le D_i.
\end{cases}
\end{align*}
That this is a homomorphism is an immediate consequence of the fact that the only relations on the universal nilmanifold are forced on the group $G$ since it has degree-rank $(s-1,r^\ast)$.

The function with which will be concerned is
\[\wt{F}(g\Gamma_{\mr{Univ}}) := F(\phi(g)\Gamma).\]
This is well-defined since $\phi(\Gamma_{\mr{Univ}})\leqslant\Gamma$; it suffices to check that the generators $\exp(e_{i,j})$ map to within $\Gamma$ but this is trivial by construction. (This is precisely why we scaled $Z_{i,\cdot}^\ast$, $Z_{i,\cdot}^{\mr{Lin}}$, and $Z_{i,\cdot}^{\mr{Pet}}$ so that when exponentiated they live within $\Gamma$.)

We now note a series of basic properties of $\wt{F}$ and the homomorphism $\phi$. 

\begin{lemma}\label{lem:lift-basic}
Given the above setup we have:
\begin{itemize}
    \item $\snorm{\wt{F}}_2 = 1$ for all $g\in G_{\mr{Univ}}$;
    \item $F$ has a vertical frequency $\eta_{\mr{Univ}}$ with height at most $(MD/\rho)^{O_s(\dim(G_{\mr{Univ}})^{O_s(1)})}$;
    \item $\wt{F}$ is $(MD/\rho)^{O_s(\dim(G_{\mr{Univ}})^{O_s(1)})}$-Lipschitz
    \item Consider $e_{i_1,j_1},\ldots,e_{i_{r^\ast},j_{r^\ast}}$ with $j_1 + \cdots + j_{r^\ast} = s-1$. If for at least one index $\ell$ we have $j_\ell>D_{i_\ell}^\ast + D_{i_\ell}^{\mr{Lin}}$, then 
    \[\eta_{\mr{Univ}}([\exp(e_{i_1,j_1}),\ldots,\exp(e_{i_{r^\ast},j_{r^\ast})}])= 0.\]
    Furthermore, if instead for two indices $\ell_1,\ell_2$ we have $j_{\ell_1}>D_{i_{\ell_1}}^\ast$ and $j_{\ell_2}>D_{i_{\ell_2}}^\ast$ then 
    \[\eta_{\mr{Univ}}([\exp(e_{i_1,j_1}),\ldots,\exp(e_{i_{r^\ast},j_{r^\ast}})]) = 0.\]
\end{itemize}
\end{lemma}
\begin{proof}
The first property is trivial. For the second property, note that $\phi$ is an $I$-filtered homomorphism (e.g.~$\phi((G_{\mr{Univ}})_{(s,r^\ast-1)}) \leqslant G_{(s,r^\ast-1)}$). Thus given $g\in G_{\mr{Univ}}$, $g'\in(G_{\mr{Univ}})_{(s,r^\ast-1)}$ we have 
\[\wt{F}(gg'\Gamma_{\mr{Univ}}) = F(\phi(g)\phi(g')\Gamma) = e(\eta(\phi(g')))F(\phi(g)\Gamma)\]
and therefore we may set $\eta_{\mr{Univ}} = \eta\circ\phi$. To check the complexity of $\eta_{\mr{Univ}}$ it suffices to check the magnitude of $\eta_{\mr{Univ}}$ on $[\exp(e_{i_1,j_1}),\ldots,\exp(e_{i_{r^\ast},j_{r^\ast}})]$ where we use Remark~\ref{rem:height} to convert between this notion and the notion of height defined. The resulting magnitude is bounded because $Z_{i,j}^\ast$, $Z_{i,j}^{\mr{Lin}}$, $Z_{i,j}^{\mr{Pet}}$ are appropriately bounded integral combinations of elements in $\mc{X}$ which itself has bounded complexity.

We omit a careful justification that $\wt{F}$ has an appropriately bounded Lipchitz constant. The crucial point is that the Mal'cev basis constructed in Lemma~\ref{lem:universal-complexity} is made up of appropriately bounded linear combinations of commutators of $e_{i_1,j_1},\ldots, e_{i_r,j_r}$ and each such commutator is seen to map to a bounded element of $G$ since $\phi$ maps each generator to a bounded element.
 
The final property is an immediate consequence of the properties of $W_{i,\mr{Lin}}$, $W_{i,\mr{Pet}}$, and $W_{i,\ast}$ established in Lemma~\ref{lem:linear-output} and recorded above. The additional generators which are lifted to the ``petal'' position on the $i$-th level come from $G_{(i,2)}$ and otherwise we have only artificially placed certain elements in the ``linear'' class upward to the ``$\ast$'' class. (These will correspond to the constant terms in the linear part of the nilsequences.)
\end{proof}

We now lift the polynomial sequences in question to the universal nilmanifold. We define:
\begin{align*}
g_h^{\ast,\mr{Univ}}(n) &= \prod_{i=1}^{s-1} \prod_{j=1}^{\dim(W_{i,\ast})} \exp(e_{i,j})^{z_{i,j}^\ast\cdot\frac{n^i}{i!}}  \prod_{i=1}^{s-1} \prod_{j=1}^{\dim(W_{i,\mr{Lin}})} \exp(e_{i,j +  \dim(W_{i,\ast})})^{\gamma_{i,j}\cdot\frac{n^{i}}{i!}},\\
g_h^{\mr{Lin},\mr{Univ}}(n)  &= \prod_{i=1}^{s-1} \prod_{j=1}^{\dim(W_{i,\mr{Lin}})} \prod_{k=1}^{d^\ast}\exp(e_{i,D_i^\ast + (j-1)d^\ast+k})^{\alpha_{i,j,k}\{\beta_k h\}\cdot\frac{n^i}{i!}},\\
g_h^{\mr{Pet},\mr{Univ}}(n) &=\prod_{i=1}^{s-1} \prod_{j=1}^{\dim(W_{i,\mr{Pet}})} \exp(e_{i,j + D_i^\ast + D_i^{\mr{Lin}}})^{z_{i,j}^{h,\mr{Pet}}\cdot\frac{n^i}{i!}},\\
g_h^{\mr{Rem},\mr{Univ}}(n)  &= \prod_{i=1}^{s-1} \prod_{j=1}^{\dim(G_{(i,2)})}\exp(e_{i,j + D_i - \dim(G_{(i,2)})})^{\kappa_{i,j}^h\cdot\frac{n^i}{i!}}.
\end{align*}
We define 
\[g_h^{\mr{Univ}} := g_h^{\ast,\mr{Univ}} \cdot g_h^{\mr{Lin},\mr{Univ}} \cdot g_h^{\mr{Pet},\mr{Univ}} \cdot g_h^{\mr{Rem},\mr{Univ}}.\]

The key claim, which is trivial by construction, is the following equality.
\begin{claim}\label{clm:univ-lift}
Given the above setup, we have 
\[\wt{F}(g_h^{\mr{Univ}}(n)\Gamma_{\mr{Univ}}) = F(g_h(n)\Gamma) = \chi_h(n).\]
\end{claim}
\begin{proof}
The final equality is by definition of $\chi_h(n)$. The first equality follows by checking that $\phi(g_h^{\ast,\mr{Univ}}) = g_h^{\ast},~\phi(g_h^{\mr{Lin},\mr{Univ}}) = g_h^{\mr{Lin}},~\phi(g_h^{\mr{Pet},\mr{Univ}}) = g_h^{\mr{Pet}},~\text{ and }\phi(g_h^{\mr{Rem},\mr{Univ}}) = g_h^{\mr{Rem}}$ by construction.
Therefore since $\phi$ is a homomorphism we conclude that $\phi(g_h^{\mr{Univ}}) = g_h$.
\end{proof}

Note that at this stage we have simply replace the group $G$ in our correlation structure with $G^{\mr{Univ}}$ as the cost of replacing $d$ by $\dim(G_{\mr{Univ}}) = d^{O_s(1)}\log(MD\rho^{-1})^{O_s(1)}$ and $M$ by $\exp(d^{O_s(1)}\log(MD\rho^{-1})^{O_s(1)})$. 

This may seem as if we have gone backwards, the key point is that in Lemma~\ref{lem:lift-basic} we have encoded various ``vanishing conditions'' on the commutator brackets at the level of the generators of the group. This will allow us to translate the ``vanishing conditions'' obtained in Lemma~\ref{lem:linear-output} into realizing we can, up to a degree-rank $(s-1,r^\ast-1)$-error term.

\subsection{Passing to a quotient nilmanifold}\label{sub:quot-def}
We now construct two additional nilmanifolds; there are essentially $G^\ast$ and $\wt{G}$ certain quotients constructed in \cite[Section~12]{GTZ12}.\footnote{There is a minor issue in \cite[p.~1309]{GTZ12} when defining $G^\ast$; we follow the definitions given in the erratum \cite{GTZ24}.}

\begin{definition}\label{def:univer-2}
We define $G_{\mr{Rel}} = G_{\mr{Rel}}^{\vec{D}^{\ast},\vec{D}^{\mr{Lin}},\vec{D}^{\mr{Pet}}}$ as the Lie subgroup of $G_{\mr{Univ}}$ where $\log(G_{\mr{Rel}})$ is spanned by:
\begin{itemize}
    \item Any $(r-1)$-fold commutator of $e_{i_1,j_1},\ldots, e_{i_r,j_r}$ with at least one index $\ell$ such that $j_\ell>D^\ast_{i_\ell} + D^{\mr{Lin}}_{i_\ell}$;
    \item Any $(r-1)$-fold commutator of $e_{i_1,j_1},\ldots, e_{i_r,j_r}$ with $j_\ell>D^\ast_{i_\ell}$ for at least two distinct indices $\ell$.
\end{itemize}
We then define $G_{\mr{Quot}}$ as $G_{\mr{Quot}}:= G_{\mr{Univ}}/G_{\mr{Rel}}$ and $\Gamma_{\mr{Quot}} = \Gamma_{\mr{Univ}}/(\Gamma_{\mr{Univ}}\cap G_{\mr{Rel}})$.
\end{definition}
\begin{remark}\label{rem:univ-2}
Note that we may set $r=1$ in the definition of $G_{\mr{Rel}}$; in particular $\exp(e_{i,j})\in G_{\mr{Rel}}$ for $j>D^\ast_{i} + D^{\mr{Lin}}_i$. Additionally, $\log(G_{\mr{Quot}})$ may be realized as the following. Consider formal generators of a Lie algebra, $\wt{e}_{i,j}$ for $1\le j\le D_i^\ast + D_{i}^{\mr{Lin}}$, with the property that:
\begin{itemize}
    \item Any $(r-1)$-fold commutator of $\wt{e}_{i_1,j_1},\ldots, \wt{e}_{i_r,j_r}$ with either $i_1 + \cdots + i_r > (s-1)$ or $i_1 + \cdots + i_r = (s-1)$ and $r>r^\ast$ vanishes;
    \item Any $(r-1)$-fold commutator of $\wt{e}_{i_1,j_1},\ldots, \wt{e}_{i_r^{\ast},j_r^\ast}$ with $j_\ell>D^\ast_{i_\ell}$ for at least two distinct indices $\ell$ vanishes.
\end{itemize}
This realization is given by taking $\wt{e}_{i,j} := \log(\exp(e_{i,j})\imod G_{\mr{Rel}})$.
\end{remark}

We first check that $G_{\mr{Quot}}$ is well-defined.
\begin{claim}\label{clm:normal}
For $\vec{D}^{\ast},\vec{D}^{\mr{Lin}},\vec{D}^{\mr{Pet}}\in (\mb{Z}_{\ge 0})^{s-1}$, $G_{\mr{Rel}}$ is a well-defined normal subgroup of $G_{\mr{Univ}}$. 
\end{claim}
\begin{proof}
It is clear from definition that $\log(G_{\mr{Rel}})$ is closed under brackets, so forms a Lie subalgebra within $\log(G_{\mr{Univ}})$. Thus $G_{\mr{Rel}}$ is indeed a Lie subgroup. To prove that $G_{\mr{Rel}}$ is normal it suffices to prove that it is furthermore a Lie algebra ideal, i.e., $[\log(G_{\mr{Univ}}),\log(G_{\mr{Rel}})]\leqslant \log(G_{\mr{Rel}})$.

Recall that $\log(G_{\mr{Univ}})$ is spanned by all the $(r-1)$-fold commutators $e_{i_1,j_1},\ldots, e_{i_r,j_r}$ (although as discussed in Lemma~\ref{lem:universal-complexity} this is not a basis). It suffices to check the containment at the level of generators of the respective Lie algebras. The result then follows since taking a commutator does not decrease the number of ``petal'' or ``linear'' generators.
\end{proof}

We also have the following complexity bound on $G_{\mr{Quot}}$. This may be done via the Lie algebra presentation given in Remark~\ref{rem:univ-2} and repeating the proof in Lemma~\ref{lem:universal-complexity}, or via noting that $G_{\mr{Rel}}$ is a sufficient rational subgroup of $G_{\mr{Univ}}$ and applying Lemma~\ref{lem:quotient-rat}. We omit the details.

\begin{lemma}\label{lem:universal-complexity-2}
Given the above setup, let $G_{\mr{Quot}} = G_{\mr{Quot}}^{\vec{D}}$ and note that $G_{\mr{Quot}}$ has a degree-rank $(s-1,r^\ast)$ filtration given by 
\[(G_{\mr{Quot}})_{(d,r)} = (G_{\mr{Univ}})_{(d,r)}/((G_{\mr{Univ}})_{(d,r)}\cap G_{\mr{Rel}}).\]
Furthermore the dimension of $G_{\mr{Univ}}$ is bounded by $O_s(\snorm{D}_{\infty}^{O_s(1)})$ and one may find an adapted Mal'cev basis $\mc{X}_{\mr{Quot}}$ such that the complexity of $G_{\mr{Quot}}/\Gamma_{\mr{Quot}}$ is $\exp(\snorm{D}_{\infty}^{O_s(1)})$.
\end{lemma}

A key point in this analysis is that this quotient is compatible with $\eta_{\mr{Univ}}$.
\begin{lemma}\label{lem:is-vert}
Given the above setup, define $\eta_{\mr{Quot}}\colon(G_{\mr{Quot}})_{(s-1,r^\ast)}\to\mb{R}$ via
\[\eta_{\mr{Quot}}(g\imod G_{\mr{Rel}}) := \eta_{\mr{Univ}}(g)\]
for all $g\in (G_{\mr{Univ}})_{(s-1,r^\ast)}$. The map 
$\eta_{\mr{Quot}}$ is well-defined and in fact is a vertical character of $G_{\mr{Quot}}$ of height at most $(MD/\rho)^{O_s(\dim(G_{\mr{Univ}})^{O_s(1)})}$.
\end{lemma}
\begin{proof}
To be well-defined as a map, it suffices to show that $G_{\mr{Rel}}\cap(G_{\mr{Univ}})_{(s-1,r^\ast)}\leqslant\on{ker}(\eta_{\mr{Univ}})$. This comes exactly from the final item of Lemma~\ref{lem:lift-basic}. That $\eta$ is a vertical character then follows as $\Gamma_{\mr{Quot}} = \Gamma_{\mr{Univ}}/(\Gamma_{\mr{Univ}} \cap G_{\mr{Rel}})$.

To bound the height of $\eta_{\mr{Quot}}$ note that taking a quotient by $G_{\mr{Rel}}$ maps $\exp(e_{i,j})$ to $\exp(\wt{e}_{i,j})$ in the sense of Remark~\ref{rem:univ-2}. Furthermore the construction of $\mc{X}_{\mr{Quot}}$ has the property that $\mc{X}_{\mr{Quot}}\cap\log((G_{\mr{Quot}})_{(s-1,r^\ast)})$ are sufficiently rational combinations of $(r^\ast-1)$-fold commutators of $\wt{e}_{i_1,j_1},\ldots,\wt{e}_{i_{r^\ast},j_{r^\ast}}$ with $i_1 + \cdots + i_{r^\ast} = s-1$. By Baker--Campbell--Hausdorff, we have that the $(r^\ast-1)$-fold commutator of $\exp(\wt{e}_{i_1,j_1}),\ldots,\exp(\wt{e}_{i_{r^\ast},j_{r^\ast}})$ is the same $\imod G_{\mr{Rel}}$ as the corresponding one for $\exp(e_{i_1,j_1}),\ldots,\exp(e_{i_{r^\ast},j_{r^\ast}})$. However, $\eta_{\mr{Univ}}$ maps the latter commutator to a sufficiently bounded integer by the complexity bound on $\eta_{\mr{Univ}}$ and the result follows.
\end{proof}

We will require 
\begin{align}
\begin{split}\label{eq:quot-nilsequences}
g_h^{\ast,\mr{Quot}}(n) &= \prod_{i=1}^{s-1} \prod_{j=1}^{\dim(W_{i,\ast})} \exp(\wt{e}_{i,j})^{z_{i,j}^\ast\cdot\frac{n^i}{i!}}  \prod_{i=1}^{s-1} \prod_{j=1}^{\dim(W_{i,\mr{Lin}})} \exp(\wt{e}_{i,j +  \dim(W_{i,\ast})})^{\gamma_{i,j}\cdot\frac{n^i}{i!}}\\
g_h^{\mr{Lin},\mr{Quot}}(n) &= \prod_{i=1}^{s-1} \prod_{j=1}^{\dim(W_{i,\mr{Lin}})} \prod_{k=1}^{d^\ast}\exp(\wt{e}_{i,D_i^\ast + (j-1)d^\ast+k})^{\alpha_{i,j,k}\{\beta_k h\}\cdot\frac{n^i}{i!}};
\end{split}
\end{align}
note that 
\[g_h^{\ast,\mr{Quot}} = g_h^{\ast,\mr{Univ}}\imod G_{\mr{Rel}},\qquad g_h^{\mr{Lin},\mr{Quot}} = g_h^{\mr{Lin},\mr{Univ}}\imod G_{\mr{Rel}}.\]
Furthermore we have 
\[g_h^{\mr{Pet},\mr{Univ}}\imod G_{\mr{Rel}} = g_h^{\mr{Rem},\mr{Univ}}\imod G_{\mr{Rel}} = \mr{id}_{G^{\mr{Quot}}}\]
pointwise. Finally we define 
\[g_h^{\mr{Quot}} := g_h^{\ast,\mr{Quot}} \cdot g_h^{\mr{Lin},\mr{Quot}}.\]

For the remainder of this section and Section~\ref{sec:nil-first-2}, fix a nilcharacter $F^\ast$ on $G_{\mr{Quot}}$ with a $G_{(s-1,r^\ast)}$-vertical frequency $\eta_{\mr{Quot}}$. Furthermore by Lemma~\ref{lem:nil-exist}\footnote{The lemma is stated for degree filtrations. However, one can give $G_{\mr{Quot}}$ the degree filtration 
\[(G_{\mr{Quot}})_{(0,0)} = (G_{\mr{Quot}})_{(1,0)}\geqslant (G_{\mr{Quot}})_{(2,0)}\geqslant \cdots \geqslant(G_{\mr{Quot}})_{(s-1,0)}\geqslant(G_{\mr{Quot}})_{(s-1,r^\ast)}\geqslant\mr{Id}_{G_{\mr{Quot}}};\]
a vertical nilcharacter with respect to this filtration is a vertical nilcharacter with respect to the original degree-rank filtration. $\mc{X}_{{\mr{Quot}}}$ is adapted to this degree-filtration (as it is adapted to the original degree-rank filtration).} we may take $F^\ast$ which is $(MD/\rho)^{O_s(\dim(G_{\mr{Univ}})^{O_s(1)})}$--Lipschitz with output dimension bounded by $2^{O_s(\dim(G_{\mr{Univ}}))}$.

The reason it will be sufficient to study $F^\ast(g_h^{\mr{Quot}} \Gamma^{\mr{Quot}})$ will be the following lemma which proves that it is equal to $\wt{F}(g_h^{\mr{Univ}} \Gamma^{\mr{Univ}})$ up a term which is lower-order in degree-rank.

\begin{lemma}\label{lem:reduct-quot}
Given the above setup, let 
\[G_{\mr{Univ}}^{\triangle} :=\{(g,g \imod G_{\mr{Rel}})\in G_{\mr{Univ}}\times G_{\mr{Quot}}\colon g\in G_{\mr{Univ}}\}\]
which is given the degree-rank filtration
\[(G_{\mr{Univ}}^{\triangle})_{(d,r)} :=\{(g,g \imod G_{\mr{Rel}})\in (G_{\mr{Univ}})_{(d,r)}\times (G_{\mr{Quot}})_{(d,r)}\colon g\in (G_{\mr{Univ}})_{(d,r)}\}.\]
Define 
$\Gamma_{\mr{Univ}}^{\triangle} = G_{\mr{Univ}}^{\triangle} \cap (\Gamma_{\mr{Univ}}\times\Gamma_{\mr{Quot}})$. We have:
\begin{itemize}
    \item $(g_h^{\mr{Univ}},g_h^{\mr{Quot}})$ is a polynomial sequence on $G_{\mr{Univ}}^{\triangle}$ with respect to the given degree-rank filtration;
    \item The function
    \[(g,g')\mapsto\wt{F}(g\Gamma_{\mr{Univ}}) \otimes \ol{F^\ast}(g'\Gamma_{\mr{Quot}})\]
    for $(g,g')\in G_{\mr{Univ}}^{\triangle}$ is $(G_{\mr{Univ}}^{\triangle})_{(s-1,r^\ast)}$-invariant;
    \item $G_{\mr{Univ}}^{\triangle}$ has complexity bounded by $(MD/\rho)^{O_s(\dim(G_{\mr{Univ}})^{O_s(1)})}$;
    \item Each coordinate of $ \wt{F}(g\Gamma_{\mr{Univ}}) \otimes \ol{F^\ast}(g'\Gamma_{\mr{Quot}})$ is $(MD/\rho)^{O_s(\dim(G_{\mr{Univ}})^{O_s(1)})}$-Lipschitz.
\end{itemize}
\end{lemma}
\begin{remark}\label{rem:quotient-pass}
The second item implies that $\wt{F}(g\Gamma_{\mr{Univ}}) \otimes \ol{F^\ast}(g'\Gamma_{\mr{Quot}})$ is $(G_{\mr{Univ}}^{\triangle})_{(s-1,r^\ast)}$-invariant and thus can be realized on a degree-rank $(s-1,r^\ast-1)$ nilmanifold $G_{\mr{Univ}}^{\triangle}/(G_{\mr{Univ}}^{\triangle})_{(s-1,r^\ast)}$ with $\Gamma_{\mr{Univ}}^{\triangle}/(\Gamma_{\mr{Univ}}^{\triangle}\cap(G_{\mr{Univ}}^{\triangle})_{(s-1,r^\ast)})$ being the lattice.
\end{remark}
\begin{proof}
It is trivial to verify that the degree-rank filtration on $G_{\mr{Univ}}^{\triangle}$ is valid. Noting that
\[\{(X_i, X_i \imod \log(G_{\mr{Rel}}))\colon X_i\in\mc{X}_{\mr{Univ}}\}\]
is a valid Mal'cev basis for $G_{\mr{Univ}}^{\triangle}$ bounds the complexity of $G_{\mr{Univ}}^{\triangle}$. The complexity bounds on $ \wt{F}(g\Gamma_{\mr{Univ}}) \otimes \ol{F^\ast}(g'\Gamma_{\mr{Quot}})$ follow by noting that $\wt{F}$ is appropriately Lipschitz on $G_{\mr{Univ}}/\Gamma_{\mr{Univ}}$
and similar for $\ol{F^\ast}$. For $F^\ast$, we note that each coordinate of $\{X_i \imod \log(G_{\mr{Rel}})\}$ is appropriately rational with respect to the Mal'cev basis for $\mc{X}_{\mr{Quot}}$, by construction. 

Furthermore for $(h,h \imod G_{\mr{Rel}})\in (G_{\mr{Univ}}^{\triangle})_{(s-1,r^\ast)}$ we have
\begin{align*}
&\wt{F}(g h\Gamma_{\mr{Univ}}) \otimes \ol{F^\ast}(g' (h \imod G_{\mr{Rel}})\Gamma_{\mr{Quot}}) \\
&\qquad= \wt{F}(g\Gamma_{\mr{Univ}}) \otimes \ol{F^\ast}(g'\Gamma_{\mr{Quot}}) \cdot e(\eta_{\mr{Univ}}(h))\ol{e(\eta_{\mr{Quot}}(h \imod G_{\mr{Rel}}))}\\
&\qquad=\wt{F}(g\Gamma_{\mr{Univ}}) \otimes \ol{F^\ast}(g'\Gamma_{\mr{Quot}})
\end{align*}
where in the final line we have used the definition of $\eta_{\mr{Quot}}$.

Finally to verify that $(g_h^{\mr{Univ}},g_h^{\mr{Quot}})$ is a polynomial sequence with respect to this degree-rank filtration, note via Taylor expansion (e.g. \cite[Lemma~B.9]{GTZ12}) that all polynomial sequences $h$ with respect to $G_{\mr{Univ}}^{\triangle}$ of the form $(h',h' \imod G_{\mr{Rel}})$ where $h'$ is a polynomial sequence with respect to $G_{\mr{Univ}}$ (and its specified degree-rank filtration). The result then follows due to the property 
\[g_h^{\mr{Univ}} \imod G_{\mr{Rel}} = g_h^{\mr{Quot}}\]
noted above, which was by construction.
\end{proof}

\section{Extracting a \texorpdfstring{$(1,s-1)$}{(1,s-1)}-nilsequence}\label{sec:nil-first-2}
The goal of this section is to realize
\[F^\ast(g_h^{\mr{Quot}}(n)\Gamma_{\mr{Quot}})\]
as a multidegree $(1,s-1)$ nilsequence in $(h,n)$. We accomplish this via a construction of Green, Tao, and Ziegler \cite[Section~12]{GTZ12} and then use this construction in order to complete the proof of Lemma~\ref{lem:induc}. After this, the main business of the paper is essentially done and all that remains to prove Theorem~\ref{thm:main} is the symmetrization argument which will be carried out in the next section.

\subsection{Constructing the \texorpdfstring{$(1,s-1)$}{(1,s-1)}-nilsequence}\label{sub:construct}
Our analysis at this point is essentially verbatim that of \cite[pp.~1313-1315]{GTZ12}. We reproduce the details here (and discuss various complexity issues which are completely routine in the appendix). For the sake of simplicity, we may clean up notation from \eqref{eq:quot-nilsequences} and write 
\begin{align*}
g_h^{\mr{Quot}}(n) &= \prod_{i=1}^{s-1}\prod_{j=1}^{D_i^\ast} \exp(\wt{e}_{i,j})^{\gamma_{i,j}\cdot\frac{n^i}{i!}}\cdot\prod_{i=1}^{s-1}\prod_{j=D_i^\ast + 1}^{D_i^\ast + D_i^{\mr{Lin}}}\exp(\wt{e}_{i,j})^{\alpha_{i,j}\{\beta_{i,j}h\}\cdot\frac{n^{i}}{i!}},
\end{align*}
where we have abusively reindexed various coefficients $\gamma,\alpha,\beta$ but nothing else.

We now define $G_{\mr{Lin}}$ to be the Lie subgroup of $G_{\mr{Quot}}$ such that $\log(G_{\mr{Lin}})$ is the subspace generated by all $(r-1)$-fold iterated commutators (with $r\ge 1$) of $\wt{e}_{i_1,j_1},\ldots,\wt{e}_{i_r,j_r}$ with $j_\ell>D_{i_\ell}^\ast$ for exactly one index $\ell$. We have the following pair of basic observations.

\begin{claim}\label{clm:abel}
We have that $G_{\mr{Lin}}$ is well-defined, abelian, and normal with respect to $G_{\mr{Quot}}$.  
\end{claim}
\begin{proof}
Similar to the proof of Claim~\ref{clm:normal}, $G_{\mr{Lin}}$ is well-defined and normal. The only modification to the proof is noting that a commutator of $\wt{e}_{i_k,j_k}$ with at least two indices $\ell$ with $j_\ell>D_{i_\ell}^\ast$ vanishes by the definition of $G_{\mr{Quot}}$.

To see that $G_{\mr{Lin}}$ is abelian, it suffices to prove that the commutator of any pair of generators is the identity. This immediately follows from the fact that commutators with at least two generators of the form $\wt{e}_{i_\ell,j_\ell}$ with $j_\ell>D_{i_\ell}^\ast$ vanish. 
\end{proof}

Due to normality, $G_{\mr{Quot}}$ acts on $G_{\mr{Lin}}$ via conjugation. In particular, we define $G_{\mr{Quot}}\ltimes G_{\mr{Lin}}$ with the group law given by 
\[(g,g_1) (g',g_1') := (gg',g_1^{g'}g_1') = (gg',((g')^{-1} g_1 g')g_1').\]
We now introduce a manner in which the additive group $R=\mb{R}^{\sum_{i=1}^{s-1}D_i^{\mr{Lin}}}$, with elements denoted
\[t = (t_{i,j})_{1\le i\le s-1,~D_{i,\ast}<j\le D_i + D_i^{\mr{Lin}}},\]
acts on $G_{\mr{Quot}}\ltimes G_{\mr{Lin}}$. Specifically, we will define an action $\rho(t)$ on this group for all $t\in R$ and use this to construct
\[G_{\mr{Multi}}=R\ltimes_\rho(G_{\mr{Quot}}\ltimes G_{\mr{Lin}}).\]
This action will allow us to simultaneously ``raise'' parts of $G_{\mr{Lin}}$ to various different fractional powers of $h$, allowing us to incorporate our ``$h$-linear'' family of nilsequences into a multidegree $(1,s-1)$ nilsequence (in variables $(h,n)$).

For each $t\in R$, we define the homomorphism $g\mapsto g^{t}$ from $G_{\mr{Quot}}$ to itself on generators. We map $\exp(\wt{e}_{i,j})\to\exp(\wt{e}_{i,j})^{t_{i,j}}$ for $1\le i\le s-1$ and $D_i^\ast<j\le D_i^\ast + D_i^{\mr{Lin}}$ while $\exp(\wt{e}_{i,j})$ is fixed
for $1\le i\le s-1$ and $1\le j\le D_i^\ast$. The defining relations of $G_{\mr{Quot}}$ are preserved by this transformation, so this is easily seen to be a well-defined homomorphism. At the Lie algebra this transformation is essentially replacing appropriate $\wt{e}_{i,j}$ by $t_{i,j} \wt{e}_{i,j}$.

For $g\in G_{\mr{Quot}}$ and $t,t'\in R$ we have 
\[(g^{t})^{t'} = g^{tt'},\]
and for $g,g'\in G_{\mr{Lin}}$ we have
\[g^{t}g^{t'} = g^{t+t'}\text{ and } g^{t}g'^{t} = (gg')^{t}.\]
This are trivial since $G_{\mr{Lin}}$ is abelian.

We next claim that if $g\in G_{\mr{Quot}}$ and $g'\in G_{\mr{Lin}}$ then
\begin{equation}\label{eq:rho-conjugation}
(gg'g^{-1})^{t} = g g'^{t} g^{-1}.
\end{equation}
To prove this note that it suffices to prove the claim for powers of generators of the groups $G_{\mr{Quot}}$ and $G_{\mr{Lin}}$ (since conjugation and $g\mapsto g^t$ are homomorphisms). If $g\in G_{\mr{Lin}}$ the result is trivial due to the abelian property, and if $g\notin G_{\mr{Lin}}$ (and is the power of a generator) then $g^t=g$ by definition so $(gg'g^{-1})^t=g^tg'^t(g^{-1})^t=gg'^tg^{-1}$ as desired.

We now define $\rho\colon R\to\on{Aut}(G_{\mr{Quot}}\ltimes G_{\mr{Lin}})$ by
\[\rho(t)(g,g_1) := (g\cdot g_1^{t},g_1).\]
The map $\rho(t)$ is clearly bijective and we have
\[\rho(s)(\rho(t)(g,g_1)) = \rho(s)((g\cdot g_1^{t},g_1)) = (g\cdot g_1^{t + s}, g_1)=\rho(t+s)(g,g_1),\]
so to check this is a group action it suffices to show $\rho(t)$ gives a valid homomorphism of $G_{\mr{Quot}}\ltimes G_{\mr{Lin}}$. This follows because 
\begin{align*}
\rho(t)((g,g_1)\cdot(g',g_1')) &= \rho(t)(gg', (g')^{-1}g_1g' g_1') \\
&= (gg' (g')^{-1}g_1^{t}g' (g_1')^{t}, (g')^{-1}g_1g' g_1') = (gg_1^{t}g' (g_1')^{t}, (g')^{-1}g_1g' g_1'),
\end{align*}
by \eqref{eq:rho-conjugation}, while
\begin{align*}
\rho(t)(g,g_1) \rho(t)(g',g_1') &= (gg_1^{t},g_1) \cdot (g'(g_1')^{t},g_1') = (gg_1^{t}g'(g_1')^{t},(g'(g_1')^{t})^{-1} g_1 g'(g_1')^{t} g_1')\\
&= (gg_1^{t}g'(g_1')^{t},(g_1')^{-t}((g')^{-1} g_1 g')(g_1')^{t} g_1') = (gg_1^{t}g'(g_1')^{t},(g_1')^{-t}(g_1')^{t}((g')^{-1} g_1 g') g_1')\\
&=(gg_1^{t}g' (g_1')^{t}, (g')^{-1}g_1g' g_1'),
\end{align*}
where we have used that $G_{\mr{Lin}}$ is abelian and normal.

We are now in position to define the group of interest which will support the multidegree $(1,s-1)$ nilsequence. Let 
\[G_{\mr{Multi}} = R \ltimes_\rho (G_{\mr{Quot}}\ltimes G_{\mr{Lin}})\]
where multiplication is given by 
\[(t,(g,g_1))(t',(g',g_1')) = (t+t', (\rho(t')(g,g_1)) \cdot (g',g_1')).\]
This is seen to be a connected, simply connected Lie group. We give it a multidegree filtration $(G_{\mr{Multi}})_{(d_1,d_2)}$ defined by:
\begin{itemize}
    \item If $d_1>1$ then $(G_{\mr{Multi}})_{(d_1,d_2)}= \mr{Id}_{G_{\mr{Multi}}}$;
    \item If $d_2>0$ then $(G_{\mr{Multi}})_{(1,d_2)} = \{(0,(g,\mr{id}_{G_{\mr{Lin}}}))\colon g\in (G_{\mr{Quot}})_{(d_2,0)}\cap G_{\mr{Lin}}\}$;
    \item $(G_{\mr{Multi}})_{(1,0)} = \{(t,(g,\mr{id}_{G_{\mr{Lin}}}))\colon t\in R, g\in (G_{\mr{Quot}})_{(0,0)}\cap G_{\mr{Lin}}\}$ or equivalently just $\{(t,(g,\mr{id}_{G_{\mr{Lin}}}))\colon t\in R, g\in G_{\mr{Lin}}\}$;
    \item If $d_2 > 0$ then $(G_{\mr{Multi}})_{(0,d_2)} = \{(0,(g,g_1))\colon g\in (G_{\mr{Quot}})_{(d_2,0)}, g_1\in (G_{\mr{Quot}})_{(d_2,0)}\cap G_{\mr{Lin}}\}$;
    \item $(G_{\mr{Multi}})_{(0,0)} = G_{\mr{Multi}}$.
\end{itemize}

\begin{claim}\label{clm:filtration}
$(G_{\mr{Multi}})_{(d_1,d_2)}$ is a valid multidegree filtration on $G_{\mr{Multi}}$.
\end{claim}
\begin{proof}
Note that 
\[(t,(g,g_1)) = (t,(\mr{id}_{G_{\mr{Quot}}},\mr{id}_{G_{\mr{Lin}}})) \cdot (0, (g,g_1))\]
and therefore $(G_{\mr{Multi}})_{(0,0)} = (G_{\mr{Multi}})_{(1,0)}\vee(G_{\mr{Multi}})_{(0,1)}$. We next check various commutator relations. First note that 
\[[(G_{\mr{Multi}})_{(1,0)},(G_{\mr{Multi}})_{(1,0)})] = \mr{Id}_{G_{\mr{Multi}}}.\]
This follows because if $g,h\in G_{\mr{Lin}}$ we have $gh=hg$ hence
\[(t,(g,\mr{id}_{G_{\mr{Lin}}})) \cdot (t',(h,\mr{id}_{G_{\mr{Lin}}}))) = (t + t',(gh,\mr{id}_{G_{\mr{Lin}}})) = (t',(h,\mr{id}_{G_{\mr{Lin}}}))) \cdot (t,(g,\mr{id}_{G_{\mr{Lin}}})).\]
Therefore it suffices to verify that
\begin{align*}
[(G_{\mr{Multi}})_{(0,a)},(G_{\mr{Multi}})_{(0,b)}]&\leqslant(G_{\mr{Multi}})_{(0,a + b)},\\
[(G_{\mr{Multi}})_{(1,a)},(G_{\mr{Multi}})_{(0,b)}]&\leqslant(G_{\mr{Multi}})_{(1,a + b)}.
\end{align*}

We first tackle the first claim, in which we may reduce to the case $a,b>0$. We wish to show
\[[(g,g_1),(g',g_1')] \in \{(h,h_1)\colon h\in (G_{\mr{Quot}})_{(a+b,0)}, h_1\in (G_{\mr{Quot}})_{(a+b,0)}\cap G_{\mr{Lin}}\}\]
if $g,g_1\in (G_{\mr{Quot}})_{(a,0)}$, $g',g_1'\in (G_{\mr{Quot}})_{(b,0)}$, and $g_1,g_1'\in G_{\mr{Lin}}$. Via Lemma~\ref{lem:com-check}, it suffices to prove $(G_{\mr{Quot}})_{(a+b,0)}$ is normal in $(G_{\mr{Quot}})_{(a,0)}$ and $(G_{\mr{Quot}})_{(b,0)}$ and then check at the level of generators. 

To check normality, we have
\begin{align*}
(g,g_1)(g',g_1')(g,g_1)^{-1} &= (g,g_1)(g',g_1')(g^{-1},gg_1^{-1}g^{-1})\\
&=(gg',(g')^{-1}g_1g' \cdot g_1')(g^{-1},gg_1^{-1}g^{-1})\\
& = (gg'g^{-1}, (g(g')^{-1})g_1 (g'g^{-1}) \cdot gg_1'g^{-1} \cdot gg_1^{-1}g^{-1})
\end{align*}
and the result follows noting that $G_{\mr{Lin}},(G_{\mr{Quot}})_{(j,0)}$ are normal in $G_{\mr{Quot}}$ for all $j\ge 0$. 

Since
\[(g,g_1) = (g,\mr{id}_{G_{\mr{Lin}}}) \cdot (\mr{id}_{\mr{Quot}}, g_1)\]
and it suffices to check the claim on generators, we may reduce to the case where exactly one of $g,g_1$ and exactly one of $g_1,g_1'$ are the identity. The result is clear when $g,g'$ are trivial, and the case when $g_1,g_1'$ are trivial follows from the fact that we have a valid filtration on $G_{\mr{Quot}}$. In the remaining cases we may assume by symmetry that $g_1 = \mr{id}_{G_{\mr{Lin}}}$ and $g' = \mr{id}_{G_{\mr{Quot}}}$. We have 
\[(g^{-1},\mr{id}_{G_{\mr{Lin}}})(\mr{id}_{G_{\mr{Quot}}},(g_1')^{-1})(g,\mr{id}_{G_{\mr{Lin}}})(\mr{id}_{G_{\mr{Quot}}},g_1') = (\mr{id}_{G_{\mr{Quot}}},g^{-1}(g_1')^{-1}gg_1')\]
and we see that the final coordinate satisfies $[g,g_1']\in (G_{\mr{Quot}})_{(a+b,0)}\cap G_{\mr{Lin}}$. We have finished verifying the first claim.

Now note that $\{(h,\mr{id}_{G_{\mr{Lin}}})\colon h\in G_{\mr{Lin}}\}$ is a normal subgroup of $G_{\mr{Quot}}\ltimes G_{\mr{Lin}}$, since $G_{\mr{Lin}}$ is abelian. Thus combining with the first claim gives the second claim, namely
\[[(G_{\mr{Multi}})_{(1,a)},(G_{\mr{Multi}})_{(0,b)})]\leqslant(G_{\mr{Multi}})_{(1,a + b)},\]
for $a>0$.

The only nontrivial case left is $a=0$ and $b>0$ for the second claim. Furthermore, combining what we know it suffices to check the case when $(t,(\mr{id}_{G_{\mr{Quot}}},\mr{id}_{G_{\mr{Lin}}}))$ is the element from $(G_{\mr{Multi}})_{(1,0)}$. Note however that
\begin{align*}
(t,&(\mr{id}_{G_{\mr{Quot}}},\mr{id}_{G_{\mr{Lin}}}))\cdot(0,(g,g_1))\cdot(-t,(\mr{id}_{G_{\mr{Quot}}},\mr{id}_{G_{\mr{Lin}}}))\cdot(0,(g,g_1))^{-1}\\
&= (t,(g,g_1))\cdot(-t,(\mr{id}_{G_{\mr{Quot}}},\mr{id}_{G_{\mr{Lin}}}))\cdot(0,(g^{-1},gg_1^{-1}g^{-1})) = (0,(gg_1^{-t},g_1))\cdot(0,(g^{-1},gg_1^{-1}g^{-1}))\\
&= (0,(gg_1^{-t}g^{-1},\mr{id}_{G_{\mr{Lin}}})).
\end{align*}
and the fact that if $g,g_1\in(G_{\mr{Quot}})_{(b,0)}$ and $g_1\in G_{\mr{Lin}}$ then $gg_1^{-t}g^{-1}\in(G_{\mr{Quot}})_{(b,0)}$. This follows because if $g_1\in(G_{\mr{Quot}})_{(b,0)}\cap G_{\mr{Lin}}$ then $g_1^t$ is in the same group.
\end{proof}

Writing $t=(t_{i,j})_{1\le i\le s-1,~D_i^\ast<j\le D_i + D_i^{\mr{Lin}}}$, we define
\[\Gamma_{\mr{Multi}} = \{(t,(g,g_1))\colon t_{i,j}\in\mb{Z}, g\in\Gamma_{\mr{Quot}},g_1\in\Gamma_{\mr{Quot}}\cap G_{\mr{Lin}}\}.\] 
To see this is a group, observe that for $g_1 \in\Gamma_{\mr{Quot}}$ we have $g^{t}\in\Gamma_{\mr{Quot}}$ if all coordinates of $t$ are integral. This is clear for the generators of $\Gamma_{\mr{Quot}}$ and the rest follows from recalling that ``taking $t$-th powers'' is a homomorphism on $G_{\mr{Quot}}$.

We now define the relevant functions which will be used to represent $F^\ast(g_h^{\mr{Quot}}(n)\Gamma^{\mr{Quot}})$. Let $\delta = \exp(-O_s((d\log(MD/\rho))^{O_s(1)}))$, where the implicit constants are chosen sufficiently large.

Let $\phi\colon\mb{R}\to\mb{R}$ be a $1$-bounded, $1$-periodic function such that:
\begin{itemize}
    \item $\phi(x)=1$ if $|\{x\}|\le 1/2-2\delta$;
    \item $\phi(x)=0$ if $|\{x\}|\ge 1/2-\delta$;
    \item $\phi$ is $O(1/\delta)$-Lipschitz.
\end{itemize}
Define $H^\ast\subseteq H$ such that for all $1\le i\le s-1$ and $D_i^\ast<j\le D_i^\ast + D_i^{\mr{Lin}}$ we have $|\{\beta_{i,j}h\}|\ge 1/2-\delta$. Using that $\beta_{i,j}\in(1/N')\mb{Z}$ where $N'$ is a prime between $100N$ and $200N$, we see that there are at most $O(\delta \cdot N\cdot\sum_{i=1}^{s-1} D_i^{\mr{Lin}})$ indices which do not satisfy the criterion and choosing $\delta$ sufficiently small, we may assume that $H^\ast$ is at least half the size of $H$.

Given $(t,(g,g_1))\in G_{\mr{Multi}}$, we may find $(t',(g',g_1'))\in (t,(g,g_1))\Gamma_{\mr{Multi}}$ such that $(t')_{i,j}\in(-1/2,1/2]$ for all $i,j$. Define
\[F_{\mr{Multi}}((t,(g,g_1)) \Gamma_{\mr{Multi}}) = F^\ast(g'\Gamma_{\mr{Quot}})\cdot\prod_{\substack{1\le i\le s-1\\ D_i^{\ast}<j\le D_i^\ast + D_i^{\mr{Lin}}}} \phi(t_{i,j}') ;\]
we check that this in fact gives a well-defined function on $G_{\mr{Multi}}/\Gamma_{\mr{Multi}}$. Note that if $(t',(g',g_1'))\in (t,(g,g_1))\Gamma_{\mr{Multi}}$ and $t_{i,j}'\in(-1/2,1/2]$ then $t_{i,j}' = \{t_{i,j}\}$ and hence $t'$ is unique. Furthermore note that 
\[(t',(g',g_1')) \cdot (0, (\gamma',\gamma_1')) = (t', (g',g_1')\cdot (\gamma',\gamma_1')) = (t', (g' \gamma', (\gamma')^{-1} g_1' \gamma' \gamma_1'))\]
and trivially 
\[F^\ast(g'\Gamma_{\mr{Quot}}) = F^\ast(g'\gamma'\Gamma_{\mr{Quot}})\]
if $\gamma'\in\Gamma_{\mr{Quot}}$. Now recall that 
\begin{align*}
g_h^{\mr{Quot}}(n) &= \prod_{i=1}^{s-1} \prod_{j=1}^{D_i^\ast} \exp(\wt{e}_{i,j})^{\gamma_{i,j}\cdot\frac{n^i}{i!}}\cdot\prod_{i=1}^{s-1} \prod_{j=D_i^\ast+1}^{D_i^\ast+D_i^{\mr{Lin}}}\exp(\wt{e}_{i,j})^{\alpha_{i,j}\{\beta_{i,j}h\}\cdot\frac{n^i}{i!}}.
\end{align*}
We set 
\begin{align*}
g_0(n) &= \prod_{i=1}^{s-1}\prod_{j=1}^{D_i^\ast} \exp(\wt{e}_{i,j})^{\gamma_{i,j}\cdot\frac{n^i}{i!}},\quad g_1(n) = \prod_{i=1}^{s-1}\prod_{j=D_i^\ast+1}^{D_i^\ast+D_i^{\mr{Lin}}} \exp(\wt{e}_{i,j})^{\alpha_{i,j}\cdot\frac{n^i}{i!}}
\end{align*}
and define
\begin{align*}
g_{\mr{Final}}(h,n) &= (0,(g_0(n),g_1(n)))\cdot ((\beta_{i,j}h)_{\substack{1\le i\le s-1\\D_i^{\ast}<j\le D_i^\ast + D_i^{\mr{Lin}}}},(\mr{id}_{G_{\mr{Quot}}},\mr{id}_{G_{\mr{Lin}}})) \\
&= (0,(g_0(n),\mr{id}_{G_{\mr{Lin}}}))\cdot(0,(\mr{id}_{G_{\mr{Quot}}},g_1(n)))\cdot((\beta_{i,j}h)_{\substack{1\le i\le s-1\\D_i^{\ast}<j\le D_i^\ast + D_i^{\mr{Lin}}}},(\mr{id}_{G_{\mr{Quot}}},\mr{id}_{G_{\mr{Lin}}})). 
\end{align*}
$g_{\mr{Final}}(h,n)$ is seen to be a polynomial sequence with respect to the filtration given to $G_{\mr{Multi}}$ as each piece is trivially a polynomial sequence and the polynomial sequences form a group under pointwise multiplication (see \cite[Corollary~B.4]{GTZ12}). 

Note that for all $h\in H$ we have
\begin{align*}
g_{\mr{Final}}(h,n)\Gamma_{\mr{Multi}} &= (0,(g_0(n),g_1(n)))\cdot ((\{\beta_{i,j}h\})_{\substack{1\le i\le s-1\\D_i^{\ast}<j\le D_i^\ast + D_i^{\mr{Lin}}}},(\mr{id}_{G_{\mr{Quot}}},\mr{id}_{G_{\mr{Lin}}}))\Gamma_{\mr{Multi}}\\
&= ((\{\beta_{i,j}h\})_{\substack{1\le i\le s-1\\D_i^{\ast}<j\le D_i^\ast + D_i^{\mr{Lin}}}},(\mr{id}_{G_{\mr{Quot}}},\mr{id}_{G_{\mr{Lin}}}))\cdot(0,(g_{0,h}^\ast(n),g_1(n)))\Gamma_{\mr{Multi}},
\end{align*}
writing
\[g_{0,h}^\ast(n)=g_0(n)(g_1(n))^{t(h)}\]
where $t(h)=(\{\beta_{i,j}h\})_{1\le i\le s-1,~D_i^{\ast}<j\le D_i^\ast + D_i^{\mr{Lin}}}\in R$. This is precisely the desired sense, discussed earlier, in which we have used the group action to ``raise'' parts of $G_{\mr{Lin}}$ to $h$-fractional powers.

Therefore, for all $h\in H^\ast$ we have 
\begin{equation}\label{eq:multi-rep}
F_{\mr{Multi}}(g_{\mr{Final}}(h,n)\Gamma_{\mr{Multi}}) = F^\ast(g_h^{\mr{Quot}}(n)\Gamma_{\mr{Quot}}).
\end{equation}

We now state various complexity claims regarding $G_{\mr{Multi}}/\Gamma_{\mr{Multi}}$ and the Lipschitz nature of the function $F_{\mr{Multi}}$. We defer the rather uninspiring task of checking these bounds to the end of Appendix~\ref{app:defer}.

\begin{lemma}\label{lem:multi-complex}
Given the above setup, we have that $G_{\mr{Multi}}/\Gamma_{\mr{Multi}}$ has the structure of a multidegree $(1,s-1)$ nilmanifold and it may be given a basis $\mc{X}_{\mr{Multi}}$ of complexity bounded by $\exp(O_s((d\log(MD/\rho))^{O_s(1)}))$. Furthermore $F_{\mr{Multi}}$ is $\exp(O_s((d\log(MD/\rho))^{O_s(1)}))$-Lipschitz under this metric.
\end{lemma}

\subsection{Extracting correlation}\label{sub:proof-of-induc}
We now complete the proof of Lemma~\ref{lem:induc}. The proof is little more than stitching results proven in this and the previous section and noting that if two nilcharacters ``differ by a lower degree-rank term'' then one may pass from to the other at the cost of introducing a lower order term. (This is essentially \cite[Lemma~E.7]{GTZ12}.)

\begin{proof}[Proof of Lemma~\ref{lem:induc}]
We return to the correlation structure discussed in Section~\ref{sec:nil-first} (that is output by Lemma~\ref{lem:linear-output}). Again, we will abuse notation slightly as discussed. So, for all $h\in H$ (where $|H|\ge\rho'N$) we have 
\[\snorm{\mb{E}_{n\in[N]}(\Delta_h f)(n) \otimes \ol{\chi(h,n)} \otimes \ol{\chi_h(n)} \cdot\ol{\psi_h(n)}}_{\infty}\ge\exp(-O_s((d\log(MD/\rho))^{O_s(1)}))\]
where $\psi_h$ is a complexity $M'$ nilsequence of degree $(s-2)$ and dimension at most $d'$. We adopt the notation developed in Sections~\ref{sec:nil-first} and \ref{sec:nil-first-2}. Applying Claim~\ref{clm:univ-lift}, we have 
\[\snorm{\mb{E}_{n\in[N]}(\Delta_h f)(n) \otimes \ol{\chi(h,n)} \otimes \ol{\wt{F}(g_h^{\mr{Univ}}(n)\Gamma_{\mr{Univ}})} \cdot\ol{\psi_h(n)}}_{\infty} \ge\exp(-O_s((d\log(MD/\rho))^{O_s(1)})).\]
Next note that
\[F^\ast(g'\Gamma_{\mr{Quot}})\otimes \ol{F^\ast(g'\Gamma_{\mr{Quot}})}\]
has trace equal to $1$ as $F^\ast$ is a nilcharacter. Since the output dimension of $F^\ast$ is bounded by $\exp((d\log(MD/\rho))^{O_s(1)})$, we have for all $h\in H$ that
\begin{align*}
\snorm{\mb{E}_{n\in[N]}(\Delta_h f)(n) \otimes \ol{\chi(h,n)} &\otimes \ol{\wt{F}(g_h^{\mr{Univ}}(n)\Gamma_{\mr{Univ}})} \otimes F^\ast(g_h^{\mr{Quot}}(n)\Gamma_{\mr{Quot}})\\
&\otimes \ol{F^\ast(g_h^{\mr{Quot}}(n)\Gamma_{\mr{Quot}})} \cdot\ol{\psi_h(n)}}_{\infty}\ge\exp(-O_s((d\log(MD/\rho))^{O_s(1)})).
\end{align*}
Using \eqref{eq:multi-rep}, we in fact may write for $h\in H^\ast$ that 
\begin{align*}
\snorm{\mb{E}_{n\in[N]}(\Delta_h f)(n) &\otimes \ol{\chi(h,n)} \otimes \ol{\wt{F}(g_h^{\mr{Univ}}(n)\Gamma_{\mr{Univ}})} \otimes F^\ast(g_h^{\mr{Quot}}(n)\Gamma_{\mr{Quot}})\otimes \ol{F_{\mr{Multi}}(g_{\mr{Final}}(h,n))} \cdot\ol{\psi_h(n)}}_{\infty} \\
&\ge\exp(-d^{O_s(1)}\log(MD\rho^{-1})^{O_s(1)}).
\end{align*}

Now we may pay a cost of $\exp(-(d\log(MD/\rho))^{O_s(1)})$ in the size of $H^\ast$ by Pigeonhole to choose a single coordinate function of $\ol{\wt{F}(g_h^{\mr{Univ}}(n)\Gamma_{\mr{Univ}})} \otimes F^\ast(g'\Gamma_{\mr{Quot}})$, call it $\psi_h^\ast(n)$, such that
\begin{align*}
\snorm{\mb{E}_{n\in[N]}(\Delta_h f)(n) &\otimes \ol{\chi(h,n)} \otimes \psi_h^\ast(n)\otimes \ol{F_{\mr{Multi}}(g_{\mr{Final}}(h,n))} \cdot\ol{\psi_h(n)}}_{\infty} \\
&\ge\exp(-O_s((d\log(MD/\rho))^{O_s(1)})).
\end{align*}
By Lemma~\ref{lem:reduct-quot} and using Remark~\ref{rem:quotient-pass}, $\psi_h^\ast(n)$ can be realized on a nilmanifold with a degree-rank $(s-1,r^\ast-1)$ filtration. Furthermore the function underlying $\psi_h^\ast(n)$ is has Lipschitz constant bounded by $\exp((d\log(MD/\rho))^{O_s(1)})$ and the nilmanifold it lives on has dimension at most $(d\log(MD/\rho))^{O_s(1)}$ and complexity bounded by $\exp((d\log(MD/\rho))^{O_s(1)})$ due to Lemma~\ref{lem:reduct-quot}.

By applying \cite[Lemma~A.6]{Len23b} with subgroup corresponding to the $(s-1,r^\ast-1)$ degree-rank and Pigeonholing in the associated vertical frequency, we may assume that $\psi_h^\ast(n)$ has a vertical frequency with height bounded by $\exp((d\log(MD/\rho))^{O_s(1)})$; this may reduce the subset of $H^\ast$ under consideration by a further admissible fraction. We then extend $\psi_h^\ast(n)$ to a nilcharacter by using Lemma~\ref{lem:nil-exist}\footnote{We have that $\psi_h^\ast$ lives on the group $G_{\mr{Univ}}^{\triangle}/(G_{\mr{Univ}}^{\triangle})_{(s-1,r^\ast)}$. We may give it the degree filtration 
\begin{align*}
G_{\mr{Univ}}^{\triangle}/(&G_{\mr{Univ}}^{\triangle})_{(s-1,r^\ast)} = (G_{\mr{Univ}}^{\triangle})_{(1,0)}/(G_{\mr{Univ}}^{\triangle})_{(s-1,r^\ast)}\geqslant (G_{\mr{Univ}}^{\triangle})_{(2,0)}/(G_{\mr{Univ}}^{\triangle})_{(s-1,r^\ast)}\\
&\geqslant \cdots \geqslant (G_{\mr{Univ}}^{\triangle})_{(s-1,0)}/(G_{\mr{Univ}}^{\triangle})_{(s-1,r^\ast)} \geqslant (G_{\mr{Univ}}^{\triangle})_{(s-1,r^\ast-1)}/(G_{\mr{Univ}}^{\triangle})_{(s-1,r^\ast)} \geqslant \mr{Id}_{G_{\mr{Univ}}^{\triangle}/(G_{\mr{Univ}}^{\triangle})_{(s-1,r^\ast)}}
\end{align*} 
and we apply Lemma~\ref{lem:nil-exist} to this filtration to get a nilcharacter $H$. We then embed $\psi_h^\ast$ by taking the underlying function, call it $Q$, and taking the nilcharacter $(Q/(2\cdot\snorm{Q}_{\infty}),\sqrt{1 - |Q/(2\cdot\snorm{Q}_{\infty})|^2} \cdot H)$.}; we refer to this nilcharacter as $\psi_h^{\mr{Output}}$ and note it is a degree-rank $(s-1,r^\ast-1)$ nilcharacter with appropriate complexity. We thus have 
\begin{align*}
\snorm{\mb{E}_{n\in[N]}(\Delta_h f)(n) &\otimes \ol{\chi(h,n)} \otimes \psi_h^{\mr{Output}}(n)\otimes \ol{F_{\mr{Multi}}(g_{\mr{Final}}(h,n))} \cdot\ol{\psi_h(n)}}_{\infty} \\
&\ge\exp(-O_s((d\log(MD/\rho))^{O_s(1)})).
\end{align*}
By Pigeonholing in $h$ once again we may pass to $F_{\mr{Multi}}^\ast$, which is a fixed coordinate of $F_{\mr{Multi}}$,
\begin{align*}
\snorm{\mb{E}_{n\in[N]}(\Delta_h f)(n) &\otimes \ol{\chi(h,n)} \otimes \psi_h^{\mr{Output}}(n)\cdot\ol{F_{\mr{Multi}}^\ast(g_{\mr{Final}}(h,n))} \cdot\ol{\psi_h(n)}}_{\infty} \\
&\ge\exp(-d^{O_s(1)}\log(MD\rho^{-1})^{O_s(1)}).
\end{align*}
on a $\exp(-(d\log(MD/\rho))^{O_s(1)})$ fraction of indices. $F_{\mr{Multi}}^\ast$ lives on the group $G_{\mr{Multi}}$ and via \cite[Lemma~A.6]{Len23b}, Pigeonholing in $h$ so that we have the same frequency, and embedding in a nilcharacter via Lemma~\ref{lem:nil-exist} similar to the above argument, we have for all $h\in H^\ast$ that
\begin{align*}
\snorm{\mb{E}_{n\in[N]}(\Delta_h f)(n) &\otimes \ol{\chi(h,n)} \otimes \psi_h^{\mr{Output}}(n)\otimes \ol{F_{\mr{Multi}}^{\mr{Output}}(g_{\mr{Final}}(h,n))} \cdot\ol{\psi_h(n)}}_{\infty} \\
&\ge\exp(-O_s((d\log(MD/\rho))^{O_s(1)}))
\end{align*}
where $F_{\mr{Multi}}^{\mr{Output}}$ is a multidegree $(1,s-1)$ nilcharacter on $G_{\mr{Multi}}$ with vertical frequency height, output dimension, and Lipschitz constant of each coordinate bounded by $\exp((d\log(MD/\rho))^{O_s(1)})$ while the dimension of the underlying nilmanifold is bounded by $(d\log(MD/\rho))^{O_s(1)}$.

This completes the proof with $\chi(h,n)\otimes F_{\mr{Multi}}^{\mr{Output}}(g_{\mr{Final}}(h,n))$ being the new multidegree $(1,s-1)$ nilcharacter, $\ol{\psi_h(n)^{\mr{Output}}}$ being the degree-rank $(s-1,r^\ast-1)$ nilcharacter and noting that the density of indices $h$ which remain is at least $\exp(-O_s((d\log(MD/\rho))^{O_s(1)}))$.
\end{proof}

\section{Symmetrization argument}\label{sec:sym}
We now perform the necessary symmetrization argument. In particular, at this stage in the argument due to Theorem~\ref{thm:iter-end} we have shown that for many $h$, $\Delta_h f$ correlates with $\chi(h,n)$ which is a multidegree $(1,s-1)$ nilcharacter. We now demonstrate that $\chi(h,n)$ is ``symmetric up to lower order terms'' in $h$ and $n$ (after multilinearizing the $n$ variable) via an argument of Green, Tao, and Ziegler \cite{GTZ12}, which in turn is closely related to an earlier argument of Green and Tao \cite{GT08b} which proved such a result for the $U^3$-norm. Our treatment is slightly simpler than in \cite{GTZ12}. Importantly, this argument is fundamentally based on a finite number of applications of Cauchy--Schwarz and a single call to equidistribution theory and therefore naturally comes with good bounds.

All references to Appendix~\ref{app:nilcharacters} are simply quantified versions of lemmas which appear in the work of Green, Tao, and Ziegler \cite[Appendix~E]{GTZ12} and a discussion of the correspondence is given more carefully in Appendix~\ref{app:nilcharacters}. The reader may benefit from glancing at the statements in Appendix~\ref{app:nilcharacters} or those in \cite[Appendix~E]{GTZ12}.

For the remainder of this section and Appendix~\ref{app:nilcharacters}, to lighten statements, we say a nilsequence $\chi$ has complexity $(M, d)$ if the underlying nilmanifold $G/\Gamma$ has complexity $M$, the underlying function is $M$-Lipschitz, and the dimension of $G$ is bounded by $d$. We will say a nilcharacter $\chi$ has complexity $(M,d)$ if the underlying nilmanifold $G/\Gamma$ has complexity $M$, the output dimension of $\chi$ is bounded by $M$, the underlying function has all coordinates being $M$-Lipschitz, the vertical character underlying $\chi$ has height bounded by $M$, and the dimension of $G$ is bounded by $d$. In this section, $M$ will always be of the form $M(\delta):=\exp(\log(1/\delta)^{O_s(1)})$ while the underlying $d$ will be of the form $d(\delta):=\log(1/\delta)^{O_s(1)}$ in our analysis, where the implicit constants may, by abuse of notation, vary from line to line.

We now recall the output of Theorem~\ref{thm:iter-end}. We have 
\[\mb{E}_{h\in[N]}\snorm{\mb{E}_{n\in[N]} \Delta_h f(n) \otimes \chi(h,n) \psi_h(n)}_{\infty} \ge M(\delta)^{-1}.\]
Here $\psi_h(n)$ is a degree $(s-2)$ nilsequence and $\chi(h,n) = F(g(h,n)\Gamma)$ is a multidegree $(1,s-1)$ nilcharacter. Furthermore $\chi$ has complexity $(M(\delta),d(\delta))$ while $\psi_h(n)$ has complexity $(M(\delta),d(\delta))$.

Our first step is to multilinearize $\chi$ in the $n$ variable, replacing it by a multidegree $(1,1,\ldots,1)$ nilcharacter which is symmetric in the final $(s-1)$ variables. 

\begin{lemma}\label{lem:periodic-start}
Fix $s\ge 2$. Suppose that 
\[\mb{E}_{h\in[N]}\snorm{\mb{E}_{n\in[N]} \Delta_h f(n) \otimes \chi(h,n) \cdot\psi_h(n)}_{\infty}\ge 1/M(\delta)\]
with $\chi(h,n)$ being a periodic multidegree $(1,s-1)$-nilcharacter and $\psi_h(n)$ are degree $(s-2)$ nilsequences each of complexity $(M(\delta),d(\delta))$. 

There exists $\wt{\chi}$ a multidegree $(1,\ldots,1)$ nilcharacter (with $s$ ones), $\wt{\psi}$ a degree $(s-1)$ nilsequence, and there exist $\wt{\psi_h}(n)$ which are degree $(s-2)$ nilcharacters all having complexity  complexity $(M(\delta),d(\delta))$ such that
\[\mb{E}_{h\in[N]}\snorm{\mb{E}_{n\in[N]} \Delta_h f(n) \otimes \wt{\chi}(h,n,\ldots,n) \cdot\wt{\psi}(n) \otimes \wt{\psi_h}(n)}_{\infty}\ge 1/M(\delta).\]
Furthermore $\wt{\chi}$ is symmetric in the final $(s-1)$ coordinates, i.e., for any $\sigma\in\mf{S}_{s-1}$ we have
\[\wt{\chi}(h,n_1,\ldots,n_{s-1}) = \wt{\chi}(h,n_{\sigma(1)},\ldots,n_{\sigma(s-1)}).\]
\end{lemma}
\begin{proof}
This is essentially an immediate consequence of multilinearization (see e.g. \cite[Theorem~E.10]{GTZ12}). By applying Lemma~\ref{lem:multilinear}, there is multidegree $(1,\ldots,1)$ nilcharacter $\wt{\chi}$ of complexity $(M(\delta),d(\delta))$ such that $\chi(h,n)$ and $\wt{\chi}(h,n,\ldots,n)$ are $(M(\delta),M(\delta),d(\delta))$-equivalent for degree $(s-1)$. Furthermore $\wt{\chi}$ is symmetric in the final $(s-1)$ coordinates. 

Thus applying Lemma~\ref{lem:equiv} (and the remark following), there exists a nilsequence $\psi^\ast(h,n)$ of degree $\le (s-1)$ and complexity $(M(\delta),d(\delta))$ such that 
\[\mb{E}_{h\in[N]}\snorm{\mb{E}_{n\in[N]} \Delta_h f(n) \otimes \wt{\chi}(h,n,\ldots,n) \otimes \psi^\ast(h,n) \cdot\psi_h(n)}_{\infty} \ge 1/M(\delta).\]

Note that a degree $(s-1)$ nilsequence of complexity $(M(\delta),d(\delta))$ in two variables $(h,n)$ is also a multidegree $(0,s-1) \cup (s-1,s-2)$ nilsequence of complexity $(M(\delta),d(\delta))$
via taking the filtration $G_{\vec{i}} := G_{|\vec{i}|}$. Therefore by Lemma~\ref{lem:split} and the first item of Lemma~\ref{lem:specialization}, there exist nilsequences $\wt{\psi}(n)$ and $\psi_h^\ast(n)$ of degree $(s-1)$ and $(s-2)$ respectively and complexity $(M(\delta),d(\delta))$ such that 
\[\mb{E}_{h\in[N]}\snorm{\mb{E}_{n\in[N]} \Delta_h f(n) \otimes \wt{\chi}(h,n,\ldots,n) \otimes \wt{\psi}(n) \cdot\psi_{h}^\ast(n) \cdot\psi_h(n)}_{\infty} \ge1/M(\delta).\]

Now, $\psi_h^\ast(n)\cdot\psi_h(n)$ is a degree $(s-2)$ nilsequence of complexity $(M(\delta),d(\delta))$. Applying \cite[Lemma~A.6]{Len23b}, we may replace this product by  $\psi_h'(n)$ which is a degree $(s-2)$ nilsequence of complexity $(M(\delta),d(\delta))$ with a vertical frequency of height $\exp(\log(1/\delta)^{O_s(1)})$. Finally apply Lemma~\ref{lem:nil-exist} and embed $\psi_h'(n)$ as a coordinate of a nilcharacter $\wt{\psi_h}(n)$ of complexity $(M(\delta),d(\delta))$, similar to in the proof of Lemma~\ref{lem:induc}. We thus have
\[\mb{E}_{h\in[N]}\snorm{\mb{E}_{n\in[N]} \Delta_h f(n) \otimes \wt{\chi}(h,n,\ldots,n) \cdot\wt{\psi}(n) \otimes \wt{\psi_h}(n)}_{\infty} \ge1/M(\delta)\]
where $\wt{\psi}$ and $\wt{\psi_h}$ have the appropriate properties.
\end{proof}

We are now in position to complete the proof of Theorem~\ref{thm:main} via a symmetrization argument. Our argument is analogous to that of Green, Tao, and Ziegler \cite[Section~13]{GTZ12} modulo certain minor simplifications to the underlying Cauchy--Schwarz arguments. 

\begin{proof}[Proof of Theorem~\ref{thm:main}]
We may assume that $s\ge 3$. The case $s=0$ is trivial, $s=1$ is standard Fourier analysis, and the case $s=2$ follows from work of Sanders \cite{San12b} (see \cite[Theorem~8]{Len22b}). Furthermore, throughout the analysis we will assume implicitly that $N\ge\exp(\log(1/\delta)^{\Omega_s(1)})$; in the case when $N$ is small one may deduce the statement via Fourier analysis. We proceed by induction, assuming that the inverse theorem is known for smaller $s$.

By Theorem~\ref{thm:iter-end} and then Lemma~\ref{lem:periodic-start} we may assume that 
\begin{equation}\label{eq:main-1}
\mb{E}_{h\in[N]}\snorm{\mb{E}_{n\in[N]} \Delta_h f(n) \otimes \chi(h,n,\ldots,n)\cdot\psi(n) \otimes \psi_h(n)}_{\infty}\ge1/M(\delta).
\end{equation}
Here $\chi$ is a multidegree $(1,\ldots,1)$ nilcharacter which is symmetric in the final $(s-1)$ variables, $\psi$ is a degree $(s-1)$ nilsequence, and $\psi_h$ are degree $(s-2)$ nilcharacters with complexities bounded by complexity $(M(\delta),d(\delta))$. For $h\notin [N]$, we take $\psi_h(n)$ to be the constant function $1$ (which is a degree $0$ nilcharacter) throughout the argument. Additionally, we may use differently indexed versions of functions $\psi$ that are defined at intermediate stages of the argument; although an abuse of notation, it will always be clear from context.

\noindent\textbf{Step 1: Initial setup for Cauchy--Schwarz argument.}
For the sake of shorthand, we will denote $\wt{\chi}(h,n) = \chi(h,n,\ldots,n)$ where there are $(s-1)$ copies of the variable $n$. By Lemma~\ref{lem:CS-basic} (taking $f_1 = f$ and $f_2 = \psi(n))$, we have 
\begin{align*}
&\mb{E}_{\substack{h_1 + h_2 = h_3 + h_4 \\ h_i\in[N]}}\snorm{\mb{E}_{n\in[N]} \wt{\chi}(h_1,n)\otimes \wt{\chi}(h_2,n + h_1-h_4) \otimes \ol{\wt{\chi}(h_3,n)}\otimes \ol{\wt{\chi}(h_4,n + h_1-h_4)} \\
& \qquad\qquad\qquad\otimes \psi_{h_1}(n)\otimes \psi_{h_2}(n + h_1 - h_4)\otimes \ol{\psi_{h_3}(n)}\otimes \ol{\psi_{h_4}(n + h_1 - h_4)}\cdot e(\Theta n)}_{\infty}\ge 1/M(\delta)
\end{align*}
for some $\snorm{\Theta}_{\mb{R}/\mb{Z}}\le M(\delta)/N$. Note that Lemma~\ref{lem:CS-basic} is stated for scalar function; here we are using that we may Pigeonhole on coordinates of the vector $\chi(h,n,\ldots,n)\cdot\psi(n) \otimes \psi_h(n)$ before using Lemma~\ref{lem:CS-basic}.

We next change variables with $h_1 = h + x$, $h_2 = h + y$, $h_3 = h + x + y$, and $h_4 = h$. The above then implies that
\begin{align*}
&\mb{E}_{h\in[N],x,y\in[\pm N]}\snorm{\mb{E}_{n\in[N]} \wt{\chi}(h + x,n)\otimes \wt{\chi}(h + y,n + x) \otimes \ol{\wt{\chi}(h + x + y,n)}\otimes \ol{\wt{\chi}(h, n + x)} \\
& \qquad\qquad\qquad\otimes \psi_{h + x}(n)\otimes \psi_{h + y}(n + x)\otimes \ol{\psi_{h + x + y}(n)}\otimes \ol{\psi_{h}(n + x)}e(\Theta n)}_{\infty}\ge1/M(\delta).
\end{align*}
By the first item of Lemma~\ref{lem:nil-basic-prop}, $\psi_{h+y}(n+x)$ and $\psi_{h+y}(n)$ are $(M(\delta),M(\delta),d(\delta))$-equivalent for degree $(s-3)$. We use that $s\ge 3$ precisely here so that this is a well-defined term.

Therefore by Lemma~\ref{lem:equiv}, there exists a collection $\psi_{h,x,y}(n)$ of degree $(s-3)$ nilsequences each of complexity $(M(\delta),d(\delta))$ such that
\begin{align*}
&\mb{E}_{h\in[N],x,y\in[\pm N]}\snorm{\mb{E}_{n\in[N]} \wt{\chi}(h + x,n)\otimes \wt{\chi}(h + y,n + x) \otimes \ol{\wt{\chi}(h + x + y,n)}\otimes \ol{\wt{\chi}(h, n + x)} \\
& \qquad\qquad\qquad\otimes \psi_{h + x}(n)\otimes \psi_{h+y}(n)\otimes \ol{\psi_{h + x + y}(n)}\otimes \ol{\psi_{h}(n + x)}\psi_{h,x,y}(n) \cdot e(\Theta n)}_{\infty}\ge1/M(\delta).
\end{align*}

We will use $B$ to denote vector-valued functions (which may vary term to term) with coordinates which are $1$-bounded such that the dimension is bounded by $M(\delta)$. The key point is that nearly all terms may be folded into $1$-bounded terms. In particular, we have
\begin{align*}
\mb{E}_{h\in[N],x,y\in[\pm N]}\snorm{\mb{E}_{n\in[N]} &\wt{\chi}(h + y,n + x) \cdot \psi_{h,x,y}(n)\otimes B(h,x,n) \otimes B(h,y,n) \otimes B(h,x+y, n)}_{\infty}\ge1/M(\delta).
\end{align*}
Noting that $\psi_{h,x,y}(n)$ may be twisted by an appropriate complex phase depending on $h$, we may in fact assume that
\begin{align*}
&\snorm{\mb{E}_{h,n\in[N],x,y\in[\pm N]} \wt{\chi}(h + y,n + x) \cdot\psi_{h,x,y}(n) \otimes B(h,x,n) \otimes B(h,y,n) \otimes B(h,x+y, n)}_{\infty}\ge1/M(\delta).
\end{align*}
By applying Pigeonhole in $h$, we may fix $h^\ast$ such that  
\begin{align*}
&\snorm{\mb{E}_{n\in[N],x,y\in[\pm N]} \wt{\chi}(h^\ast + y,n + x) \cdot\psi_{h^{\ast},x,y}(n)\otimes B(x,n) \otimes B(y,n) \otimes B(x+y, n)}_{\infty}\ge1/M(\delta).
\end{align*}
Taking the coordinate which achieves the infinity norm, we may assume that $B(\cdot,\cdot)$ are in fact all scalar and thus
\begin{align*}
&\snorm{\mb{E}_{x,y\in[\pm N]}\mb{E}_{n\in[N]} \wt{\chi}(h^\ast + y,n + x) \cdot\psi_{x,y}(n)
\cdot b(x,n) \cdot b(y,n) \cdot b(x+y, n)}_{\infty}\ge1/M(\delta);
\end{align*}
we have dropped $h^\ast$ in one subscript here.

By applying the second item of Lemma~\ref{lem:nil-basic-prop} and the second item of Lemma~\ref{lem:specialization}, we have that
$\wt{\chi}(h^\ast + y,n + x)$ and $\wt{\chi}(y,n + x)$ are $(M(\delta),M(\delta),d(\delta))$-equivalent for degree $(s-1)$. Thus by Lemma~\ref{lem:equiv} there exists a nilsequence $\psi^\ast(x,y,n)$ of degree $(s-1)$ and complexity $(M(\delta),d(\delta))$ such that 
\begin{align*}
&\snorm{\mb{E}_{x,y\in[\pm N]}\mb{E}_{n\in[N]} \wt{\chi}(y,n + x) \cdot\psi^\ast(x,y,n) \cdot\psi_{x,y}(n)\cdot
 b(x,n)  b(y,n)  b(x+y, n)}_{\infty}\ge1/M(\delta).
\end{align*}

Note that a degree $(s-1)$ nilsequence in variables $x,y,n$ is a multidegree $(s-1,s-1,s-3) \cup (1,0,s-2) \cup (0,1,s-2) \cup (0,0,s-1)$-nilsequence. Therefore applying Lemma~\ref{lem:split} and applying Pigeonhole, we may adjust $\psi_{x,y}$ and the $1$-bounded functions and remove $\psi^\ast$ and thus we may assume that 
\begin{align*}
&\snorm{\mb{E}_{x,y\in[\pm N]}\mb{E}_{n\in[N]} \wt{\chi}(y,n + x) \cdot\psi_{x,y}(n)
\cdot b(x,n) b(y,n) b(x+y, n)}_{\infty}\ge1/M(\delta);
\end{align*}
note that $\psi_{x,y}$ and $B$ have all been modified but we have abusively maintained the same notation. In particular, $\psi_{x,y}(n)$ is degree $(s-3)$.

By Lemma~\ref{lem:multi}, the second item of Lemma~\ref{lem:specialization}, and Lemma~\ref{lem:trans} (and symmetry of $\chi$ in the final $(s-1)$ coordinates), we have that $\wt{\chi}(y,n + x)$ and \[\bigotimes_{k=0}^{s-1} \chi(y,n,\ldots, n, x,\ldots, x)^{\otimes \binom{s-1}{k}}\]
are $(M(\delta),M(\delta),d(\delta))$-equivalent for degree $(s-1)$. In this notation there are $k$ copies of $n$ and $s-1-k$ copies of $x$. Now by Lemma~\ref{lem:equiv}, we have 
\begin{align*}
&\norm{\mb{E}_{x,y\in[\pm N]}\mb{E}_{n\in[N]} \psi^\ast(x,y,n)\cdot\bigotimes_{k=0}^{s-1} \chi(y,n,\ldots, n, x,\ldots, x)^{\otimes \binom{s-1}{k}}\cdot\psi_{x,y}(n)\\
&\qquad\qquad\qquad\qquad\cdot b(x,n) b(y,n) b(x+y, n) }_{\infty}\ge1/M(\delta)
\end{align*}
where $\psi^\ast(x,y,n)$ is a new degree $(s-1)$ nilsequence of complexity $(M(\delta),d(\delta))$. Applying Lemma~\ref{lem:split} as before, we may adjust $\psi_{x,y}(n)$ and the $1$-bounded functions and remove this term to have that
\begin{align*}
&\norm{\mb{E}_{x,y\in[\pm N]}\mb{E}_{n\in[N]} \bigotimes_{k=0}^{s-1} \chi(y,n,\ldots, n, x,\ldots, x)^{\otimes \binom{s-1}{k}}\cdot\psi_{x,y}(n)b(x,n) b(y,n) b(x+y, n)}_{\infty}\ge1/M(\delta).
\end{align*}
Note that the only terms of $\bigotimes_{k=0}^{s-1} \chi(y,n,\ldots, n, x,\ldots, x)^{\otimes \binom{s-1}{k}}$
which involve all of $x,y,n$ with $n$ appearing at least $s-2$ times have exactly one copy of $x$, one copy of $y$ and $n$ exactly $(s-2)$ times. Therefore taking the coordinate of 
\[\bigotimes_{\substack{0\le k\le s-1\\k\neq s-2}} \chi(y,n,\ldots, n, x,\ldots, x)^{\otimes \binom{s-1}{k}}\]
which achieves the infinity norm and adjusting $\psi_{x,y}$, $b$, and adding a term $b(x,y)$ we have 
\begin{align*}
&\norm{\mb{E}_{x,y\in[\pm N]}\mb{E}_{n\in[N]} \chi(y,x,n,\ldots, n)^{\otimes (s-1)}\cdot\psi_{x,y}(n) b(x,n) b(y,n) b(x+y, n) b(x,y)}_{\infty}\ge1/M(\delta).
\end{align*}

\noindent\textbf{Step 2: Cauchy--Schwarz to remove $1$-bounded functions.}
Applying Cauchy--Schwarz to each coordinate of the associated vector, duplicating the variable $y$, and using that $b(x,n)$ is $1$-bounded, we find that 
\begin{align*}
&\snorm{\mb{E}_{n\in[N],x\in[\pm N]} \mb{E}_{y,y'\in[\pm N]}\chi(y, x, n, \ldots, n)^{\otimes (s-1)}\otimes \ol{\chi(y', x, n, \ldots, n)^{\otimes (s-1)}} \cdot\psi_{x,y}(n)\ol{\psi_{x,y'}(n)}\\
&\qquad\qquad\qquad\qquad\cdot b(y,n) \ol{b(y',n)}\cdot b(x+y, n) \ol{b(x+y', n)}\cdot b(x,y)\ol{b(x,y')}}_{\infty}\ge1/M(\delta).
\end{align*}

By Lemma~\ref{lem:multi}, Lemma~\ref{lem:specialization}, and Lemma~\ref{lem:trans}, we have that
\[\chi(y, x, n, \ldots, n)^{\otimes (s-1)}\otimes \ol{\chi(y', x, n, \ldots, n)^{\otimes (s-1)}}\text{ and }\chi(y-y', x, n, \ldots, n)^{\otimes (s-1)}\]
are $(M(\delta),M(\delta),d(\delta))$-equivalent for
degree $(s-1)$. Therefore by Lemma~\ref{lem:equiv}, there exists $\psi^\ast(x,y,y',n)$ a degree $(s-1)$ nilsequence of complexity $(M(\delta),d(\delta))$ such that 
\begin{align*}
&\snorm{\mb{E}_{n\in[N],x,y,y'\in[\pm N]}\chi(y - y', x, n, \ldots, n)^{\otimes (s-1)}\cdot\psi^\ast(x,y,y',n)\cdot\psi_{x,y}(n)\ol{\psi_{x,y'}(n)} \\
&\qquad\qquad\cdot b(y,n) \ol{b(y',n)}\cdot b(x+y, n) \ol{b(x+y', n)}\cdot b(x,y)\ol{b(x,y')}}_{\infty}\ge1/M(\delta).
\end{align*}
Note that $z = x+y+y'$ ranges in the set $[-3N,3N]$. Take $\rho = \exp(-\log(1/\delta)^{O_s(1)})$ sufficiently small. Then there exists $z^\ast$ such that $z^\ast\in[-(3-\rho)N,(3-\rho)N]$ such that
\begin{align*}
&\snorm{\mb{E}_{n\in[N]}\mb{E}_{\substack{x,y,y'\in[\pm N]\\ x + y+y' = z^\ast}}\chi(y - y', x, n, \ldots, n)^{\otimes (s-1)}\cdot\psi^\ast(x,y,y',n)\cdot\psi_{x,y}(n)\ol{\psi_{x,y'}(n)} \\
&\qquad\qquad\cdot b(y,n) \ol{b(y',n)}\cdot b(x+y,n) \ol{b(x+y', n)}\cdot b(x,y)\ol{b(x,y')}}_{\infty}\ge1/M(\delta).
\end{align*}
This implies that 
\begin{align*}
&\snorm{\mb{E}_{n\in[N]}\mb{E}_{\substack{x,y,y'\in[\pm N]\\ x + y+y' = z^\ast}}\chi(y - y', z^\ast - y - y', n, \ldots, n)^{\otimes (s-1)}\cdot\psi^\ast(z^\ast - y - y',y,y',n)\cdot\psi_{x,y}(n)\ol{\psi_{x,y'}(n)} \\
&\cdot b(y,n) \ol{b(y',n)}\cdot b(z^\ast-y', n) \ol{b(z^\ast-y, n)}\cdot b(z^\ast-y-y',y)\ol{b(z^\ast-y-y',y')}}_{\infty}\ge1/M(\delta).
\end{align*}

By applying the first item of Lemma~\ref{lem:nil-basic-prop}, Lemma~\ref{lem:multi}, Lemma~\ref{lem:specialization}, and Lemma~\ref{lem:trans} we have that 
\[\chi(y - y', z^\ast - y - y', n, \ldots, n)^{\otimes (s-1)}\]
and 
\[\chi(y', y', n, \ldots, n)^{\otimes (s-1)}\chi(y', y, n, \ldots, n)^{\otimes (s-1)}\ol{\chi(y, y, n, \ldots, n)}^{\otimes (s-1)}\ol{\chi(y, y', n, \ldots, n)}^{\otimes (s-1)}\]
are $(M(\delta),M(\delta),d(\delta))$ equivalent for degree $(s-1)$. Thus by Lemma~\ref{lem:equiv} and letting $\wt{\psi}$ denote a degree $(s-1)$ nilsequence in $y,y',n$ of complexity $(M(\delta),d(\delta))$ we have that 
\begin{align*}
&\snorm{\mb{E}_{n\in[N]}\mb{E}_{\substack{y,y'\in[\pm N]\\ |z^\ast-y-y'|\le N}}\chi(y', y', n, \ldots, n)^{\otimes (s-1)}\chi(y', y, n, \ldots, n)^{\otimes (s-1)}\\
&\ol{\chi(y, y, n, \ldots, n)}^{\otimes (s-1)}\ol{\chi(y, y', n, \ldots, n)}^{\otimes (s-1)}\cdot\wt{\psi}(y,y',n)\cdot\psi_{z^\ast-y-y',y}(n)\ol{\psi_{z^\ast-y-y',y'}(n)}\\
&\cdot b(y,n) \ol{b(y',n)}\cdot b(z^\ast-y', n) \ol{b(z^\ast-y, n)}\cdot b(z^\ast-y-y',y) \cdot\ol{b(z^\ast-y-y',y')}}_{\infty}\ge1/M(\delta).
\end{align*}
Here we have ``folded'' in $\psi^\ast(z^\ast-y-y',y,y',n)$ via Lemma~\ref{lem:specialization} in $\wt{\psi}$. We may collapse various $1$-bounded functions (and pass to the coordinates of $\chi(y', y', n, \ldots, n)^{\otimes (s-1)}$ and $\ol{\chi(y, y, n, \ldots, n)}^{\otimes (s-1)}$ which achieve the $L^{\infty}$ norm) and obtain 
\begin{align*}
&\snorm{\mb{E}_{n\in[N]}\mb{E}_{\substack{y,y'\in[\pm N]\\ |z^\ast-y-y'|\le N}}\chi(y', y, n, \ldots, n)^{\otimes (s-1)}\ol{\chi(y, y', n, \ldots, n)}^{\otimes (s-1)}\\ 
&\qquad\qquad\cdot\wt{\psi}(y,y',n)\cdot\psi_{y,y'}(n) \cdot b(y,n)b(y',n)b(y,y')}_{\infty}\ge1/M(\delta);
\end{align*}
here the $\wt{\psi}_{y,y'}(n)$ are degree $(s-3)$ nilsequences of complexity $(M(\delta),d(\delta))$. Furthermore as $\wt{\psi}(y,y',n)$ is a degree $(s-1)$ nilsequence, we have that it is a multidegree $(s-1,s-1,s-3)\cup (1,0,s-1) \cup (0,1,s-1)$ nilsequence. Therefore by Lemma~\ref{lem:split} and Lemma~\ref{lem:specialization}, we may remove $\wt{\psi}$ at the cost of adjusting $b$ and $\psi_{y,y'}$ to obtain 
\begin{align*}
&\snorm{\mb{E}_{n\in[N]}\mb{E}_{\substack{y,y'\in[\pm N]\\ |z^\ast-y-y'|\le N}}\chi(y', y, n, \ldots, n)^{\otimes (s-1)}\ol{\chi(y, y', n, \ldots, n)}^{\otimes (s-1)}\\ 
&\qquad\qquad\cdot\psi_{y,y'}(n) \cdot b(y,n)b(y',n)b(y,y')}_{\infty}\ge1/M(\delta).
\end{align*}
This implies that
\begin{align*}
&\snorm{\mb{E}_{n\in[N]}\mb{E}_{y,y'\in[\pm N]}\chi(y', y, n, \ldots, n)^{\otimes (s-1)}\ol{\chi(y, y', n, \ldots, n)}^{\otimes (s-1)}\\ 
&\qquad\qquad\cdot\psi_{y,y'}(n) \cdot b(y,n) b(y',n) b(y,y') \cdot\mbm{1}_{|z^\ast-y-y'|\le N}}_{\infty}\ge1/M(\delta).
\end{align*}
as $\mb{P}[\mbm{1}_{|z^\ast-y-y'|\le N}]\gtrsim \rho^{2}$. Note that the final indicator may be absorbed into $b(y,y')$ to obtain
\begin{align*}
&\snorm{\mb{E}_{n\in[N]}\mb{E}_{y,y'\in[\pm N]}\chi(y', y, n, \ldots, n)^{\otimes (s-1)}\ol{\chi(y, y', n, \ldots, n)}^{\otimes (s-1)}\\ 
&\qquad\qquad\cdot\psi_{y,y'}(n) \cdot b(y,n) b(y',n) b(y,y')}_{\infty}\ge1/M(\delta).
\end{align*}
Define $G(y,y',n) = \chi(y, y', n, \ldots, n)^{\otimes (s-1)}\otimes \ol{\chi(y', y, n, \ldots, n)}^{\otimes (s-1)}$ and we have 
\begin{align*}
&\snorm{\mb{E}_{n\in[N],y,y'\in[\pm N]}G(y,y',n)\cdot\ol{\psi_{y,y'}(n)}\cdot b(y,n)b(y',n)b(y,y')}_{\infty}\ge1/M(\delta).
\end{align*}

Applying Cauchy--Schwarz in $n$, then $y$, and then $y'$ (analogously to as in Lemma~\ref{lem:CS-basic}) we may remove the bounded functions $b$ and we have 
\begin{align*}
&\snorm{\mb{E}_{n_1,n_2\in[N],y_1,y_2,y_1',y_2'\in[\pm N]}\bigotimes_{\eps\in\{1,2\}^{3}}\mc{C}^{|\eps|-1}G(y_{\eps_1},y_{\eps_2}',n_{\eps_3})\cdot\psi_{y_1,y_2,y_1',y_2'}(n_1) \cdot\ol{\psi_{y_1,y_2,y_1',y_2'}(n_2)}}_{\infty}\\
&\qquad\qquad\qquad\qquad\qquad\ge1/M(\delta),
\end{align*}
where $\mc{C}$ denotes conjugation and the $\psi_{y_1,y_2,y_1',y_2'}(n_i)$ are degree at most $(s-3)$ nilsequences of complexity $(M(\delta),d(\delta))$. Applying Pigeonhole in $n_2,y_2,y_2'$ and applying Lemma~\ref{lem:specialization} to specialize variables, reindexing $n_1, y_1,y_1'$ to $n,y,y'$, and taking the maximal coordinate we have 
\begin{align*}
&\snorm{\mb{E}_{n\in[N],y,y'\in[\pm N]}G(y,y',n)\cdot\psi_1(y,n) \cdot\psi_2(y',n) \cdot\psi_{y,y'}(n)}_{\infty}\ge1/M(\delta).
\end{align*}
Here $\psi_{y,y'}$ is degree at most $(s-3)$ in $n$ while $\psi_1(y,n)$ and $ \psi_2(y,n)$ are multidegree $(1,s-2)$ and all have complexity $(M(\delta),d(\delta))$. Finally by the triangle inequality we have 
\begin{align*}
&\mb{E}_{y,y'\in[\pm N]}\snorm{\mb{E}_{n\in[N]}G(y,y',n)\cdot\psi_1(y,n) \cdot\psi_2(y',n) \cdot\psi_{y,y'}(n)}_{\infty}\ge1/M(\delta).
\end{align*}

\noindent\textbf{Step 3: Converse of the inverse theorem and polarization.}
By the converse of the inverse theorem, see Lemma~\ref{lem:converse}, we have that 
\begin{align*}
&\mb{E}_{y,y'\in[N]}\snorm{G(y,y',\cdot)\psi_1(y,\cdot)\psi_2(y',\cdot)}_{U^{s-2}[N]}^{2^{s-2}}\ge1/M(\delta).
\end{align*}
Expanding out the definition of the $U^{s-2}$-norm, we find that 
\begin{align*}
&\norm{\mb{E}_{y,y'\in[\pm N]}\mb{E}_{n\in[N],h_1,\ldots,h_{s-2}\in[\pm N]}\bigotimes_{\eps\in\{0,1\}^{s-2}}\Big(\mc{C}^{|\eps|+s}(G(y,y',n + \eps \cdot\vec{h})\cdot\psi_1(y,n + \eps \cdot\vec{h}) \cdot\psi_2(y',n + \eps \cdot\vec{h})) \\
&\qquad\qquad\cdot\mbm{1}_{n + \eps \cdot\vec{h} \in[N]}\Big)}_{\infty} \ge1/M(\delta).
\end{align*}

The crucial point is that by repeatedly applying Lemma~\ref{lem:multi}, Lemma~\ref{lem:specialization}, and Lemma~\ref{lem:trans} we have that
\[\bigotimes_{\eps\in\{0,1\}^{s-2}}\mc{C}^{|\eps|+s}(G(y,y',n + \eps \cdot\vec{h}))\]
and 
\[\chi(y,y',h_1,\ldots,h_{s-2})^{\otimes (s-1)!}\cdot\ol{\chi(y',y,h_1,\ldots,h_{s-2})}^{\otimes(s-1)!}\]
are $(M(\delta),M(\delta),d(\delta))$-equivalent for degree $(s-1)$. Therefore by Lemma~\ref{lem:equiv}, there exists a nilsequence $\wt{\psi}$ of degree $(s-1)$ and complexity $(M(\delta),d(\delta))$ such that 
\begin{align*}
&\snorm{\mb{E}_{y,y'\in[\pm N]}\mb{E}_{n\in[N],h_1,\ldots,h_{s-2}\in[\pm N]}\chi(y,y',h_1,\ldots,h_{s-2})^{\otimes (s-1)!}\cdot\ol{\chi(y',y,h_1,\ldots,h_{s-2})}^{\otimes (s-1)!} \\
&\qquad\cdot\wt{\psi}(n,y,y',h_1,\ldots,h_{s-2}) \cdot\mbm{1}_{n + \eps \cdot\vec{h} \in[N]}}_{\infty}\ge1/M(\delta).
\end{align*}
Via Fourier expansion (a multidimensional version of the argument in Lemma~\ref{lem:major-arc}), we may fold in $\mbm{1}_{n + \eps \cdot\vec{h} \in[N]}$ into $\wt{\psi}(n,y,y',h_1,\ldots,h_{s-2})$.\footnote{To be precise, we convolve $\mbm{1}_{n + \eps \cdot\vec{h} \in[N]}$ with $\mbm{1}_{|n|\le\rho N} \cdot\prod_{i=1}^{s-2}\mbm{1}_{|h_i|\le\rho N}$ where $\rho = 1/M(\delta)$ is sufficiently small. This function has the necessary Fourier decay to apply the analysis in Lemma~\ref{lem:major-arc}} We reduce to 
\begin{align*}
&\snorm{\mb{E}_{y,y'\in[\pm N]}\mb{E}_{n\in[N],h_1,\ldots,h_{s-2}\in[\pm N]}\chi(y,y',h_1,\ldots,h_{s-2})^{\otimes (s-1)!}\cdot\ol{\chi(y',y,h_1,\ldots,h_{s-2})}^{\otimes (s-1)!} \\
&\qquad\cdot\wt{\psi}(n,y,y',h_1,\ldots,h_{s-2})}_{\infty}\ge1/M(\delta).
\end{align*}
Applying Pigeonhole in $n$ and applying the first item of Lemma~\ref{lem:specialization}, we reduce to
\begin{align}
\begin{split}\label{eq:main-2}
&\snorm{\mb{E}_{y,y'\in[\pm N]}\mb{E}_{h_1,\ldots,h_{s-2}\in[\pm N]}\chi(y,y',h_1,\ldots,h_{s-2})^{\otimes (s-1)!}\cdot\ol{\chi(y',y,h_1,\ldots,h_{s-2})}^{\otimes (s-1)!}\\
&\qquad\qquad\qquad\qquad\otimes  \wt{\psi}(y,y',h_1,\ldots,h_{s-2}) }_{\infty}\ge1/M(\delta);
\end{split}
\end{align}
once again we have abusively updated $\wt{\psi}$, which has degree $(s-1)$.

\noindent\textbf{Step 4: Invoking equidistribution theory.}
This is the unique moment we have the ability to apply equidistribution theory; up to this point we have been applying ``elementary'' facts regarding nilsequences. Let 
\[\chi(y,y',h_1,\ldots,h_{s-2}) = F(g(y,y',h_1,\ldots,h_{s-2})\Gamma)\]
and let $\xi$ denote the vertical $G_{(1,\ldots,1)}$ frequency of $F$ on the multidegree $(1,\ldots,1)$ nilmanifold $G/\Gamma$. We write 
\[\wt{\psi}(y,y',h_1,\ldots,h_{s-2}) = \wt{F}(g^\ast(y,y',h_1,\ldots,h_{s-2})\Gamma')\]
on the multidegree $(s-1)$ nilmanifold $G'/\Gamma'$. Note that 
\[(g(y,y',h_1,\ldots,h_{s-2}),g(y',y,h_1,\ldots,h_{s-2}),g^\ast(y,y_1,h_1,\ldots,h_{s-2}))\]
may be viewed as a polynomial sequence on $G\times G \times G'$ where $G'$ is given a degree $(s-1)$ filtration. $G\times G\times G'$ is given a degree $s$ filtration where the $t$-th group is
\[(G\times G\times G')_t=\bigvee_{|\vec{i}| = t}G_{\vec{i}}\times\bigvee_{|\vec{i}| = t}G_{\vec{i}} \times (G')_t.\]
Note that $F \otimes \ol{F} \otimes \wt{F}$ has $(G\times G \times G')_s$-vertical frequency $\xi' = (\xi,-\xi, 0)$, noting that $(G')_s=\mr{Id}_{G'}$. By applying Corollary~\ref{cor:equi-deg} with \eqref{eq:main-2} to 
\[F(g(y,y',h_1,\ldots,h_{s-2})\Gamma)^{\otimes (s-1)!} \otimes \ol{F(g(y',y,h_1,\ldots,h_{s-2})\Gamma)}^{\otimes (s-1)!} \cdot\wt{F}(g^\ast(y,y',h_1,\ldots,h_{s-2})\Gamma'),\]
and restricting the factorization to $G\times G$, we have 
\begin{align*}
&(g(y,y',h_1,\ldots,h_{s-2}),g(y',y,h_1,\ldots,h_{s-2})) \\
&\qquad= \eps(y,y_1,h_1,\ldots,h_{s-2}) \cdot g^{\mr{Output}}(y,y_1,h_1,\ldots,h_{s-2})\cdot\gamma(y,y_1,h_1,\ldots,h_{s-2}),
\end{align*}
where 
\begin{itemize}
    \item $g^{\mr{Output}}$ lives in an $M(\delta)$-rational subgroup $H$ such that $\xi'(H\cap (G\times G)_s) = 0$;
    \item $\gamma$ is an $M(\delta)$-rational polynomial sequence;
    \item $\eps$ is $(M(\delta),N)$-smooth.
\end{itemize}
Note that when apply Corollary~\ref{cor:equi-deg} the vertical frequency of the function we have is $(s-1)!\cdot\xi'$ and we obtain $(s-1)!\xi'(H\cap(G\times G)_s)=0$; we may divide by $(s-1)!$ to obtain the above. Additionally, we have implicitly used that $\xi'$ is trivial in the $G'$ part and abuse notation to descend $\xi'$ to $G\times G$.

Let $F^\ast = F\otimes\ol{F}$ and note that therefore
\begin{align*}
&\chi(h,n,\ldots,n) \otimes \ol{\chi(n,h,n,\ldots,n)} = F^\ast(\eps(h,n,\ldots,n) g^{\mr{Output}}(h,n,\ldots,n)\cdot\gamma(h,n,\ldots,n) (\Gamma\times\Gamma)).
\end{align*}

\noindent\textbf{Step 5: The finishing touch.}
We now recall from \eqref{eq:main-1} that
\[\mb{E}_{h\in[N]}\snorm{\mb{E}_{n\in[N]} \Delta_{h} f(n) \otimes \chi(h, n,\ldots,n)\cdot\psi(n) \cdot\psi_{h}(n)}_{\infty}\ge1/M(\delta);\]
here we have restricted to a coordinate of $\psi_h$ and we treat it as a degree $(s-2)$ nilsequence (rather than using the nilcharacter). By applying Pigeonhole there exist $q,q'\in[s]$ such that 
\[\mb{E}_{h\in[N/s]}\snorm{\mb{E}_{n\in[N/s]} \Delta_{sh + q'} f(sn + q) \otimes \chi(sh + q', sn + q,\ldots, sn + q)\cdot\psi(sn + q) \cdot\psi_{h}(sn + q)}_{\infty}\ge1/M(\delta).\]

By Lemma~\ref{lem:nil-basic-prop}, we have that
\[\chi(sh + q', sn + q,\ldots, sn+q) \text{ and } \chi(sh, sn,\ldots, sn)\]
are $(M(\delta),M(\delta),d(\delta))$-equivalent for degree $(s-1)$. Applying Lemma~\ref{lem:split} (splitting) and adjusting $\psi,\psi_h$, we may instead assume that
\begin{align*}
&\mb{E}_{h\in[N/s]}\snorm{\mb{E}_{n\in[N/s]} \Delta_{sh + q'} f(sn + q) \otimes \chi(sh, sn,\ldots, sn)\cdot\psi(n) \cdot\psi_{h}(n)}_{\infty}\ge1/M(\delta)
\end{align*}
for $\psi$ of degree $(s-1)$ and $\psi_h$ of degree $(s-2)$. Now define
\[T(h,n) := \chi(n+h,\ldots,n+h) \otimes \ol{\chi(h,n,\ldots,n)} \otimes \ol{\chi(n,h,\ldots,n)}^{\otimes (s-1)} \otimes \ol{\chi(n,n,\ldots,n)}.\]
Since this is a nilcharacter, we automatically know
\begin{align*}
&\mb{E}_{h\in[N/s]}\snorm{\mb{E}_{n\in[N/s]} \Delta_{sh + q'} f(sn + q) \otimes \chi(sh, sn,\ldots, sn)\cdot\psi(n) \cdot\psi_{h}(n)\\
&\qquad\qquad\qquad\qquad\qquad\qquad\qquad\qquad\otimes T(h,n)^{\otimes s^{s-1}} \otimes \ol{T(h,n)}^{\otimes s^{s-1}}}_{\infty}\ge1/M(\delta).
\end{align*}
We define
\begin{align*}
\wt{f}_1(n) &= f(sn + q) \cdot\ol{\chi(n,\ldots, n)}^{\otimes s^{s-1}},\\
\wt{f}_2(n + h) &= \ol{f(s(n+h) + q+q')} \cdot\chi(n+h,\ldots,n+h)^{\otimes s^{s-1}},
\end{align*}
which yields
\begin{align*}
&\mb{E}_{h\in[N/s]}\snorm{\mb{E}_{n\in[N/s]} \wt{f}_1(n)\otimes\wt{f}_2(n+h) \otimes \chi(sh, sn,\ldots, sn)\cdot\psi(n) \cdot\psi_{h}(n)\\
&\qquad\qquad \otimes (\ol{\chi(h,n,\ldots,n)} \otimes \ol{\chi(n,h,n,\ldots,n)}^{\otimes (s-1)})^{\otimes s^{s-1}} \otimes \ol{T(h,n)}^{\otimes s^{s-1}}}_{\infty}\ge1/M(\delta).
\end{align*}

By applications of Lemma~\ref{lem:multi}, Lemma~\ref{lem:specialization}, and Lemma~\ref{lem:trans} we have that $T(h,n)$ and 
\[\bigotimes_{k=1}^{s-1}\chi(h,h,\ldots,h,n,\ldots,n)^{\binom{s-1}{k}}\otimes\bigotimes_{k=2}^{s-1}\chi(n,h,\ldots,h,n,\ldots,n)^{\binom{s-1}{k}}\]
are $(M(\delta),M(\delta),d(\delta))$-equivalent for degree $(s-1)$. (There are $k+1$ many $h$'s in the first term and $k$ many $h$'s in the second term.) Applying Lemma~\ref{lem:split}, we may approximate each coordinate as a sum of products of multidegree $(s-1,s-2)$ and $(0,s-1)$ nilsequences in variables $(h,n)$. Furthermore, by the second item of Lemma~\ref{lem:specialization} this new nilsequence is of similar type. So, folding everything into $\psi(n)$ of degree $(s-1)$ and the $\psi_h(n)$ of degree $(s-2)$, we find
\begin{align*}
&\mb{E}_{h\in[N/s]}\norm{\mb{E}_{n\in[N/s]} \wt{f}_1(n)\otimes\wt{f}_2(n+h) \otimes \chi(sh, sn,\ldots, sn)\cdot\psi(n) \cdot\psi_{h}(n)\\
&\qquad \qquad\qquad\qquad\otimes (\ol{\chi(h,n,\ldots,n)} \otimes \ol{\chi(n,h,\ldots,n)}^{\otimes (s-1)})^{\otimes s^{s-1}}}_{\infty}\ge1/M(\delta).
\end{align*}

Furthermore note by Lemma~\ref{lem:nil-basic-prop} that
\[\chi(sh, sn,\ldots, sn) \text{ and } \chi(h, n,\ldots, n)^{\otimes s^{s}}\]
are $(M(\delta),M(\delta),d(\delta))$-equivalent for degree $(s-1)$. Applying Lemma~\ref{lem:equiv} and Lemma~\ref{lem:split} and adjusting $\psi$ and $\psi_h$ yet again we have
\begin{align*}
&\mb{E}_{h\in[N/s]}\norm{\mb{E}_{n\in[N/s]} \wt{f}_1(n)\otimes\wt{f}_2(n+h) \otimes \chi(h, n,\ldots, n)^{\otimes s^{s}}\cdot\psi(n) \cdot\psi_{h}(n)\\
&\qquad\qquad\qquad\qquad \otimes (\ol{\chi(h,n,\ldots,n)} \otimes \ol{\chi(n,h,\ldots,n)}^{\otimes (s-1)})^{\otimes s^{s-1}}}_{\infty}\ge1/M(\delta).
\end{align*}
Now by Lemma~\ref{lem:nil-basic-prop} we have that 
\[\chi(h, n,\ldots, n)^{\otimes s^s}\otimes\ol{\chi(h,n,\ldots,n)}^{\otimes s^{s-1}} \text{ and } \chi(h,n,\ldots,n)^{\otimes (s-1) \cdot s^{s-1}}\]
are $(M(\delta),M(\delta),d(\delta))$-equivalent for degree $(s-1)$. Thus applying Lemma~\ref{lem:equiv} and Lemma~\ref{lem:split} and adjusting $\psi$ and $\psi_h$ once again we have
\begin{align*}
&\mb{E}_{h\in[N/s]}\norm{\mb{E}_{n\in[N/s]} \wt{f}_1(n)\otimes\wt{f}_2(n+h) \cdot\psi(n) \cdot\psi_{h}(n) \otimes (\chi(h,n,\ldots,n) \otimes \ol{\chi(n,h,\ldots,n)})^{\otimes (s-1)s^{s-1}}}_{\infty}\\
&\qquad\qquad\qquad\qquad\qquad\ge1/M(\delta).
\end{align*}

This is finally where we may apply our earlier factorization for $\chi(h,n,\ldots,n) \otimes \ol{\chi(n,h,\ldots,n)}$. Recall that
\begin{align*}
&\chi(h,n,\ldots,n) \otimes \ol{\chi(n,h,n,\ldots,n)} = F^\ast(\eps(h,n,\ldots,n) g^{\mr{Output}}(h,n,\ldots,n)\cdot\gamma(h,n,\ldots,n) (\Gamma\times\Gamma))
\end{align*}
where $\gamma$ is $M(\delta)$-periodic and $\eps$ is $(M(\delta),N)$-smooth. Let $Q$ denote the period of $\gamma$ (i.e., changing any argument by a multiple of $Q$ keeps its $\Gamma\times\Gamma$ coset the same) and take $\rho = \exp(-\log(1/\delta)^{O_s(1)})$ where the implicit constant is sufficiently large. Break $[N/s]$ into arithmetic progressions of length roughly $\rho N$ and common difference $Q$; call these $\mc{P}_1,\ldots,\mc{P}_\ell$. There exist $\eps_{\mc{P}_{i,h}}$ and $\gamma_{\mc{P}_{i,h}}$ such that 
\begin{align*}
&\mb{E}_{h\in[N/s]}\norm{\mb{E}_{n\in[N/s]} \sum_{i=1}^\ell\mbm{1}_{n\in\mc{P}_i}\wt{f}_1(n)\otimes\wt{f}_2(n+h) \cdot\psi(n) \cdot\psi_{h}(n) \\
&\qquad\qquad\otimes (F^\ast(\eps_{\mc{P}_{i,h}} \gamma_{\mc{P}_{i,h}}(\gamma_{\mc{P}_{i,h}}^{-1} g^{\mr{Output}}(h,n,\ldots,n)\gamma_{\mc{P}_{i,h}}) (\Gamma\times\Gamma)))^{\otimes (s-1)s^{s-1}}}_{\infty}\ge1/M(\delta)
\end{align*}
where $d_{G\times G}(\eps_{\mc{P}_{i,h}},\mr{id}_{G\times G})+d_{G\times G}(\gamma_{\mc{P}_{i,h}},\mr{id}_{G\times G})\le\exp(\log(1/\delta)^{O_s(1)})$ and $\gamma_{\mc{P}_{i,h}}$ is $\exp(\log(1/\delta)^{O_s(1)})$-rational.

By Pigeonhole, there exists an index $i$ such that 
\begin{align*}
&\mb{E}_{h\in[N/s]}\norm{\mb{E}_{n\in[N/s]} \mbm{1}_{n\in\mc{P}_i}\cdot\wt{f}_1(n)\otimes\wt{f}_2(n+h) \cdot\psi(n) \cdot\psi_{h}(n) \\
&\qquad\qquad\otimes (F^\ast(\eps_{\mc{P}_{i,h}}\gamma_{\mc{P}_{i,h}} (\gamma_{\mc{P}_{i,h}}^{-1}g^{\mr{Output}}(h,n,\ldots,n)\cdot\gamma_{\mc{P}_{i,h}}) (\Gamma\times\Gamma)))^{\otimes (s-1)s^{s-1}}}_{\infty}\ge1/M(\delta).
\end{align*}
As $\gamma_{\mc{P}_{i,h}}$ is $\exp(\log(1/\delta)^{O_s(1)})$-rational and bounded, it takes on only $\exp(\log(1/\delta)^{O_s(1)})$ possible values. Thus by Pigeonhole, there is $\gamma\in\Gamma\times\Gamma$ such that
$d_{G\times G}(\gamma,\mr{id}_{G\times G})\le\exp(\log(1/\delta)^{O_s(1)})$ and $\gamma$ is $\exp(\log(1/\delta)^{O_s(1)})$-rational such that
\begin{align*}
&\mb{E}_{h\in[N/s]}\norm{\mb{E}_{n\in[N/s]}\mbm{1}_{n\in\mc{P}_i} \wt{f}_1(n)\otimes\wt{f}_2(n+h) \cdot\psi(n) \cdot\psi_{h}(n) \\
&\qquad\qquad\otimes (F^\ast(\eps_{\mc{P}_{i,h}}\gamma g^{\mr{Conj}}(h,n,\ldots,n)(\Gamma\times\Gamma)))^{\otimes(s-1)s^{s-1}}}_{\infty}\ge1/M(\delta),
\end{align*}
where $g^{\mr{Conj}}=\gamma^{-1}g^{\mr{Output}}\gamma$. Finally, rounding $\eps_{\mc{P}_{i,h}}\gamma$ to a $\exp(-\log(1/\delta)^{O_s(1)})$-net and noting it is $\exp(\log(1/\delta)^{O_s(1)})$-bounded, there exists $\eps$ such that 
\begin{align*}
&\mb{E}_{h\in[N/s]}\norm{\mb{E}_{n\in[N/s]} \mbm{1}_{n\in\mc{P}_i}\cdot\wt{f}_1(n)\otimes\wt{f}_2(n+h) \cdot\psi(n) \cdot\psi_{h}(n) \\
&\qquad\qquad\otimes (F^\ast(\eps g^{\mr{Conj}}(h,n,\ldots,n)(\Gamma\times\Gamma)))^{\otimes (s-1)s^{s-1}}}_{\infty}\ge1/M(\delta)
\end{align*}
and $d_{G\times G}(\eps,\mr{id}_{G\times G})\le\exp(\log(1/\delta)^{O_s(1)})$, as long as $\rho$ was chosen small enough.

By Lemma~\ref{lem:major-arc}, there exists $\Theta_h$ such that 
\begin{align*}
&\mb{E}_{h\in[N/s]}\norm{\mb{E}_{n\in[N/s]} e(\Theta_h n)\cdot\wt{f}_1(n)\otimes\wt{f}_2(n+h) \cdot\psi(n) \cdot\psi_{h}(n) \\
&\qquad\qquad\otimes (F^\ast(\eps g^{\mr{Conj}}(h,n,\ldots,n) (\Gamma\times\Gamma)))^{\otimes (s-1)s^{s-1}}}_{\infty}\ge1/M(\delta).
\end{align*}
As $(s-2)\ge 1$, we may absorb $e(\Theta_h n)$ into $\psi_h(n)$ and obtain 
\begin{align*}
&\mb{E}_{h\in[N/s]}\norm{\mb{E}_{n\in[N/s]} \wt{f}_1(n)\otimes\wt{f}_2(n+h) \cdot\psi(n) \cdot\psi_{h}(n) \\
&\qquad\qquad\otimes (F^\ast(\eps g^{\mr{Conj}}(h,n,\ldots,n) (\Gamma\times\Gamma)))^{\otimes (s-1)s^{s-1}}}_{\infty}\ge1/M(\delta).
\end{align*}
Replacing $F^\ast$ with $F^{\mr{Final}}(g) = F^\ast(\eps g\Gamma)$ and writing $g^{\mr{Final}}(h,n)=g^{\mr{Conj}}(h,n,\ldots,n)$, we have
\begin{align*}
&\mb{E}_{h\in[N/s]}\norm{\mb{E}_{n\in[N/s]} \wt{f}_1(n)\otimes \wt{f}_2(n+h) \cdot\psi(n) \cdot\psi_{h}(n) \\
&\qquad\qquad\otimes (F^{\mr{Final}}( g^{\mr{Final}}(h,n)(\Gamma\times\Gamma)))^{\otimes (s-1)s^{s-1}}}_{\infty}\ge1/M(\delta).
\end{align*}

Now $g^{\mr{Final}}(h,n)$ takes values in $\gamma^{-1}H\gamma$ such that $\xi'(\gamma^{-1}H\gamma\cap(G\times G)_s) = 0$. The key point is to note that $F^{\mr{Final}}$ is right-invariant under $(\gamma^{-1}H\gamma)\cap(G\times G)_s$ since it has $(G\times G)_s$-vertical frequency $\xi'$. Note that $\gamma^{-1}H\gamma$ has complexity bounded by $M(\delta)$ due to \cite[Lemma~B.15]{Len23b}. Furthermore $F^{\mr{Final}}$ is $M(\delta)$-Lipschitz on $\gamma^{-1}H\gamma$ by \cite[Lemma~B.9,~B.15]{Len23b}. Taking the quotient by $(\gamma^{-1}H\gamma)\cap(G\times G)_s$ gives that each coordinate of $(F^{\mr{Final}}(g^{\mr{Final}}(h,n)(\Gamma\times\Gamma)))^{\otimes s^{s-1}}$ may be realized a complexity $(M(\delta),d(\delta))$ nilsequence of degree $(s-1)$. 

Applying Pigeonhole in the coordinates of $(F^{\mr{Final}})^{\otimes s^{s-1}}$ and then Lemma~\ref{lem:split} to approximate as a sum of products of multidegree $(s-1,s-2)$ and $(0,s-1)$ nilsequences in variables $(h,n)$. So again folding everything into $\psi(n)$ of degree $(s-1)$ and the $\psi_h(n)$ of degree $(s-2)$, we find
\begin{align*}
&\mb{E}_{h\in[N/s]}\snorm{\mb{E}_{n\in[N/s]} \wt{f}_{1}(n)\otimes\wt{f}_{2}(n+h) \cdot\psi(n) \cdot\psi_{h}(n)}_{\infty}\ge1/M(\delta)
\end{align*}
The functions $\wt{f}_1$ and $\wt{f}_2$ are vector-valued, but by Pigeonhole there exist coordinates $j_1,j_2$ are coordinates of the vectors $\wt{f}_1$ and $\wt{f}_2$ such that 
\begin{align*}
&\mb{E}_{h\in[N/s]}|\mb{E}_{n\in[N/s]} \wt{f}_{1,j_1}(n) \wt{f}_{2,j_2}(n+h) \cdot\psi(n) \cdot\psi_{h}(n)|\ge1/M(\delta).
\end{align*}
Since $\psi_h(n)$ is a nilsequence of degree $(s-2)$ and complexity $(M(\delta),d(\delta))$, by the converse of the inverse theorem (see Lemma~\ref{lem:converse}) we have that 
\begin{align*}
&\mb{E}_{h\in[N/s]}\snorm{\wt{f}_{1,j_1}(\cdot)\wt{f}_{2,j_2}(\cdot+h) \psi(\cdot)}_{U^{s-1}[N/s]}^{2^{s-1}}\ge1/M(\delta).
\end{align*}

By the Gowers--Cauchy--Schwarz inequality (e.g.~\cite[Lemma~3.8]{Gow01}), we have that 
\begin{align*}
&\snorm{\mb{E}_{n\in[N/s]}\wt{f}_{1,j_1}(n) \psi(n)}_{U^{s}[N/s]}\ge1/M(\delta).
\end{align*}
By induction, there is a nilsequence $\Theta(n)$ of degree $(s-1)$ and complexity $(M(\delta),d(\delta))$ such that
\begin{align*}
|\mb{E}_{n\in[N/s]} \wt{f}_{1,j_1}(n) \psi(n) \Theta(n)|\ge1/M(\delta).
\end{align*}
Now recall that 
\[\wt{f}_1(n) = f(sn + q) \cdot\ol{\chi(n,\ldots, n)}^{\otimes s^{s-1}}.\]
Each coordinate of $\ol{\chi(n,\ldots, n)}^{\otimes s^{s-1}}$ is a degree $s$ nilsequence of complexity $(M(\delta),d(\delta))$; say $j_1$-th coordinate is $\Theta'(n)$ and thus we have 
\begin{align*}
|\mb{E}_{n\in[N/s]} f(sn + q) \Theta'(n) \psi(n) \Theta(n)|\ge1/M(\delta).
\end{align*}
This is equivalent to 
\begin{align*}
|\mb{E}_{n\in[N]} \mbm{1}[n\equiv q\imod s]f(n) \Theta'((n-q)/s) \psi((n-q)/s) \Theta((n-q)/s)|\ge1/M(\delta).
\end{align*}
Note the condition 
\[\mbm{1}[n\equiv q\imod s] = s^{-1}\sum_{j=0}^{s-1}e\bigg(\frac{j \cdot (n-q)}{s}\bigg)\]
and thus there $j$ such that 
\begin{align*}
|\mb{E}_{n\in[N]} f(n) \Theta'((n-q)/s) \psi((n-q)/s) \Theta((n-q)/s) e(jn/s)|\ge1/M(\delta).
\end{align*}
The desired nilsequence is then
\[\ol{\Theta'((n-q)/s) \psi((n-q)/s) \Theta((n-q)/s) e(jn/s)}\]
which is seen to have degree $s$ and complexity $(M(\delta),d(\delta))$. We have finally won.
\end{proof}

\appendix

\section{On approximate homomorphisms}\label{app:approx-hom}
In this section, we give a number of basic results regarding approximate homomorphisms. The results in this section are, by now, well known consequences of work of Sanders \cite{San12b}. The proof we give is essentially that in \cite{Len23}, modulo being forced to deal with slight error terms and operating over $\mb{Z}$. We dispose of these error terms via a rounding trick of Green, Tao, and Ziegler \cite[Appendix~C]{GTZ11}. 

\begin{lemma}\label{lem:approximate}
Fix $\delta\in (0,1/2)$, let $H_1,H_2,H_3,H_4\subseteq[N]$ and let functions $f_i\colon H_i\to\mb{R}^d$ be such that there are at least $\delta N^3$ additive tuples $h_1 + h_2 = h_3 + h_4$ with
\[\snorm{(f_1(h_1) + f_2(h_2) - f_3(h_3) - f_4(h_4))_j}_{\mb{R}/\mb{Z}}\le\eps_j\]
for all $1\le j\le d$. Then there exists $H_1'\subseteq H_1$ with $|H_1'|\ge\exp(-(d\log(1/\delta))^{O(1)})N$ such that
\[\norm{\bigg(f_1(h) - \sum_{i=1}^{d'}a_i \{\alpha_i h\} - b\bigg)_j}_{\mb{R}/\mb{Z}} \le \eps_j\]
for all $h\in H_1'$, for appropriate choices of $d'\le (d\log(1/\delta))^{O(1)}$, $a_i,b\in\mb{R}^d$, and $\alpha_i\in(1/N')\mb{Z}$ where $N'$ is a prime between $100N$ and $200N$.
\end{lemma}

We deduce the result from the following variant which is the same statement modulo not having an error term.

\begin{lemma}\label{lem:approximate-no-error}
Fix $\delta\in (0,1/2)$. Let $H_1,H_2,H_3,H_4\subseteq [N]$ and $f_i\colon H\to\mb{R}^d$ be such that there are at least $\delta N^3$ additive tuples $h_1 + h_2 = h_3 + h_4$ with
\[f_1(h_1) + f_2(h_2) - f_3(h_3) - f_4(h_4)\in\mb{Z}^{d}.\]
Then there exists $H_1'\subseteq H_1$ with $|H_1'|\ge\exp(-\log(1/\delta)^{O(1)})N$ such that
\[f_1(h) - \sum_{i=1}^{d'}a_i \{\alpha_i h\} - b\in\mb{Z}^d \]
for all $h\in H_1'$, for appropriate choices of $d'\le\log(1/\delta)^{O(1)}$, $a_i,b\in\mb{R}^d$, and $\alpha_i\in(1/N')\mb{Z}$ where $N'$ is a prime between $100N$ and $200N$.
\end{lemma}

We briefly give the deduction, and then in the sequel focus on Lemma~\ref{lem:approximate-no-error}.
\begin{proof}[Proof of Lemma~\ref{lem:approximate} given Lemma~\ref{lem:approximate-no-error}]
Round each value of $f_i$ to the nearest point in the lattice $(\eps_1 \mb{Z},\ldots,\eps_d \mb{Z})$ to form $\wt{f_i}$ (breaking ties arbitrarily). We have that 
\[\snorm{(\wt{f_1}(h_1) + \wt{f_2}(h_2) - \wt{f_3}(h_3) - \wt{f_4}(h_4))_j}_{\mb{R}/\mb{Z}}\le 5\eps_j\]
for at least $\delta N^{4}$ additive tuples. 

Note however that 
\[\wt{f_1}(h_1) + \wt{f_2}(h_2) - \wt{f_3}(h_3) - \wt{f_4}(h_4)\in (\eps_1 \mb{Z},\ldots,\eps_d \mb{Z})\]
and that there are at most $11^d$ lattice points in $(\eps_1 \mb{Z},\ldots,\eps_d \mb{Z})$ which are at most $5\eps_j$ in the $j$-th direction from the origin in all $d$ directions. Thus there is a vector $w\in (\eps_1 \mb{Z},\ldots,\eps_d \mb{Z})$ such that 
\[\wt{f_1}(h_1) + \wt{f_2}(h_2) - \wt{f_3}(h_3) - \wt{f_4}(h_4) + w \in\mb{Z}^d\]
for at least $11^{-d}\delta N^4$ additive tuples. Applying Lemma~\ref{lem:approximate-no-error} with $\wt{f_1}$, $\wt{f_2}$, $\wt{f_3}$, and $\wt{f_4} - w$ immediately gives the desired result. 
\end{proof}

We now require the notion of a Bohr set in an abelian group.
\begin{definition}\label{def:Bohr}
Given an abelian group $G$ and a set $S\subseteq \wh{G}$, we define the \emph{Bohr set} of radius $\rho$ to be
\[B(S,\rho) := \{x\in G\colon\snorm{s\cdot x}_{\mb{R}/\mb{Z}}\le\rho\emph{ for all }s\in S\}.\]
\end{definition}

We first require the fact that the four-fold sumset of a set with small doubling contains a Bohr set of small dimension and large radius. This is an immediate consequence of work of Sanders \cite[Theorem~1.1]{San12b} which produces a large symmetric coset progression and a proposition of Mili\'{c}evi\'{c} \cite[Propositon~27]{Mil21} which produces a Bohr set inside a large symmetric coset progression. This is explicitly \cite[Corollary~28]{Mil21}.
\begin{lemma}\label{lem:Bohr-set}
Let $A\subseteq\mb{Z}/N\mb{Z}$ be such that $|A|\ge N/K$. Then there exists $S\subseteq\wh{\mb{Z}/N\mb{Z}}$ with $|S|\le\log(2K)^{O(1)}$ and $1/\rho\le\log(2K)^{O(1)}$ such that $B(S,\rho)\subseteq 2A-2A$.
\end{lemma}

We next require the notion of a Freiman homomorphism.
\begin{definition}\label{def:fre-hom}
A function $f\colon A\to B$ (with $A$ and $B$ being subsets of possibly different abelian groups) is a \emph{$k$-Freiman homorphism} if for all $a_i,a_i'\in A$ satisfying
\[a_1 + \cdots + a_k = a_1' + \cdots + a_k'\]
we have
\[f(a_1) + \cdots + f(a_k) = f(a_1') + \cdots + f(a_k').\]
When $k$ is not specified, we will implicitly have $k = 2$.
\end{definition}

We will also require the follow basic lemma which converts the Freiman homomorphism on a Bohr set into a ``bracket'' linear function on a slightly smaller Bohr set; the proof is a simplification of \cite[Proposition~10.8]{GT08b}.
\begin{lemma}\label{lem:brack-form}
Consider $S\subseteq \wh{\mb{Z}/N\mb{Z}}$ and $\rho\in (0,1/4)$ with Freiman homomorphism $f\colon B(S,\rho)\to\mb{R}/\mb{Z}$. Taking $\rho' = \rho \cdot |S|^{-2|S|}$, we have for all $n\in B(S,\rho')$ that
\[f(n) - \Big(\sum_{\alpha_i\in S}a_i\{\alpha_i n\} + \gamma\Big) \in\mb{Z},\]
for appropriate choices of $a_i,\gamma\in\mb{R}$.
\end{lemma}
\begin{proof}
By \cite[Proposition~10.5]{GT08b}, we have that 
\[B(S,\rho \cdot |S|^{-2|S|}) \subseteq P \subseteq  B(S,\rho)\]
where $P$ is a proper generalized arithmetic progression $\{\sum_{i=1}^{d}\ell_i n_i\colon n_i\in[\pm N_i]\}$ of rank $d\le|S|$. Furthermore $(\{\alpha\cdot\ell_i\})_{\alpha\in S}$ for $1\le i\le d$ are linearly independent as vectors in $\mb{R}^{S}$.

Note that for $|n_i|\le N_i$, we have 
\begin{equation}\label{eq:f-linear}
 f\bigg(\sum_{i=1}^d \ell_i n_i\bigg) - f(0) = \sum_{i=1}^{d}n_i (f(\ell_i)-f(0)).    
\end{equation}
Furthermore letting $\Phi\colon B(S,\rho)\to\mb{R}^{S}$ denote $\Phi(x)=(\{\alpha\cdot x\})_{\alpha\in S}$ we have that 
\[\Phi(x) + \Phi(y) = \Phi(x+y);\]
we have used crucially that $\rho<1/4$ here. Therefore, by a simple inductive argument we see
\[\Phi\bigg(\sum_{i=1}^d \ell_i n_i\bigg) = \sum_{i=1}^{d}n_i \Phi(\ell_i)\]
if $n_i\in[\pm N_i]$ for all $1\\le i\le d$.

By the above linear independence, there exists $u_i\in\mb{R}^{S}$ such that $u_i \cdot\Phi(\ell_i) = 1$ and $u_i \cdot\Phi(\ell_j) = 0$ for $j\neq i$. Therefore if $n\in P$ is such that $n = \sum_{i=1}^{d}\ell_in_i$, we have that 
\[n_i = u_i\cdot\sum_{i=1}^d n_i\Phi(\ell_i) = u_i\cdot\Phi(n) = \sum_{\alpha\in S}(u_i)_{\alpha}\cdot\{\alpha n\}.\]
The lemma then follows by plugging into \eqref{eq:f-linear}. 
\end{proof}

We now recall the definition of additive energy. 
\begin{definition}
Given (finite) subsets $A_1,A_2,A_3,A_4$ of an abelian group $G$, define the \emph{additive energy} $E(A_1,A_2,A_3,A_4)$ to be
\[E(A_1,A_2,A_3,A_4) = \sum_{x_i \in A_i}\mbm{1}[x_1+x_2 = x_3 + x_4]\]
and let $E(A) = E(A,A,A,A)$.
\end{definition}

Note that one has the trivial bound $E(A)\le |A|^3$. Furthermore via a standard Cauchy--Schwarz argument (similar to e.g.~\cite[Corollary~2.10]{TV10}) we have
\[E(A_1,A_2,A_3,A_4) \le \prod_{i=1}^{4}E(A_i)^{1/4}.\]

\begin{proof}[Proof of Lemma~\ref{lem:approximate-no-error}]
Let $\Gamma_i = \{(h_i, f_i(h_i) \imod \mb{Z}^d)\colon h_i \in H_i\} \subseteq \mb{Z} \times (\mb{R}/\mb{Z})^d$, which is a graph (i.e., for every $x\in\mb{Z}$ there is at most one $y\in(\mb{R}/\mb{Z})^d$ with $(x,y)\in\Gamma_i$). By assumption we have 
\[E(\Gamma_1,\Gamma_2,\Gamma_3,\Gamma_4)\ge\delta N^3.\]
We have 
\[E(\Gamma_1,\Gamma_2,\Gamma_3,\Gamma_4)\le\prod_{i=1}^{4}E(\Gamma_i)^{1/4} \le E(\Gamma_1)^{1/4}N^{9/4}\]
and therefore $E(\Gamma_1) \ge\delta^{4}N^3$. By Balog--Szemer\'{e}di--Gowers (see \cite[Theorem~5.2]{GT08b}), there is $\Gamma'\subseteq \Gamma_1$ such that $|\Gamma'|\ge\delta^{O(1)}N$ while $|\Gamma' - \Gamma'|\le\delta^{-O(1)}N$. 

Let $A = (8 \Gamma' - 8\Gamma')\cap(\{0\}\times(\mb{R}/\mb{Z})^d)$. Since $\Gamma'$ is a graph, we have that $|\Gamma' + A| = |\Gamma'| |A|$. However $|\Gamma' + A|\le |9\Gamma' - 8\Gamma'|\le\delta^{-O(1)}N$ by the Pl{\"u}nnecke--Ruzsa inequality (e.g.~\cite[Theorem~5.3]{GT08b}) and thus $|A|\le\delta^{-O(1)}$.

Now, by abuse of notation we may view $A$ as a subset of $(\mb{R}/\mb{Z})^d$. We claim there exists $T\subseteq \mb{Z}^d$ with $|T|\le O(\log(1/\delta))$ such that $A\cap B(T,1/4) = \{0\}$; we give a proof which is essentially identical to that in \cite[Lemma~8.3]{GT08b}. Note that given any $w\in (\mb{R}/\mb{Z})^d\setminus \{0\}$ we have 
\[\limsup_{M\to\infty} \mb{P}_{v\in\{-M,\ldots,M\}^d}[\snorm{v\cdot w}_{\mb{R}/\mb{Z}}<1/4]\le 3/4.\]
This follows immediately from noting that if $w$ has an irrational coordinate the probability tends to $1/2$ by Weyl's equidistribution criterion while if $w$ is rational the limiting probability is at most say $2/3$. Choosing an integer vector $v$ which kills at least $1/4$ of the set iteratively then immediately gives the desired lemma.

Let $\psi\colon(\mb{R}/\mb{Z})^{d}\to(\mb{R}/\mb{Z})^{T}$ be defined as $\psi(\xi) = (t(\xi))_{t\in T}$. Now let $\tau = 2^{-7}$. By averaging there exists a cube $Q = \vec{x} + [0,\tau)^{T}$ such that 
\[\wt{\Gamma} :=\{(h,f_1(h))\in\Gamma'\colon\psi(f_1(h))\in Q\}\]
with $|\wt{\Gamma}|\ge\tau^{|T|} |\Gamma'|$, so $|\wt{\Gamma}|\ge\delta^{O(1)}N$. Fix such a cube $Q$.

We claim that $4\wt{\Gamma} - 4\wt{\Gamma}$ is a graph. For the sake of contradiction suppose not. Then there exist $h_1,\ldots,h_8$ and $h_1',\ldots,h_8'$ such that
\begin{align*}
h_1 + \cdots + h_4 - h_{5} - \cdots - h_{8} &= h_1' + \cdots + h_4' - h_{5}' - \cdots - h_{8}',\\
f_1(h_1) + \cdots + f_1(h_4) - f_1(h_{5}) - \cdots - f_1(h_{8}) &\not\equiv f_1(h_1') + \cdots + f_1(h_4') - f_1(h_{5}') - \cdots - f_1(h_{8}')\imod 1.
\end{align*}
However,
\begin{align*}
&\norm{\psi\big(\big(f_1(h_1) + \cdots + f_1(h_4)- f_1(h_{5}) - \cdots - f_1(h_{8})\big) - \big(f_1(h_1') + \cdots + f_1(h_4')-f_1(h_{5}') - \cdots - f_1(h_{8}')\big)\big)}_{\infty}\\
&\qquad\qquad\qquad\qquad\le 16 \cdot\tau <1/4
\end{align*}
by definition of $\wt{\Gamma}$. Since $A\cap B(T,1/4) = \{0\}$, it follows that
\[\big(f_1(h_1) + \cdots + f_1(h_4)- f_1(h_{5}) - \cdots - f_1(h_{8})\big) - \big(f_1(h_1') + \cdots + f_1(h_4')- f_1(h_{5}') - \cdots - f_1(h_{8}')\big) \in\mb{Z}^d\]
as desired. 

Let $H^\ast$ denote the projection of $\wt{\Gamma}$ onto the first coordinate. Since $f_1$ is an $8$-Freiman homomorphism on $H^\ast$ (because $4\wt{\Gamma}-4\wt{\Gamma}$ is a graph), we have that $f_1$ is a Freiman homorphism on $2H^\ast - 2H^\ast$ (where $f_1$ is extended via linearity). We now view $H^\ast$ (which is a subset of integers) as a subset of $\mb{Z}/N'\mb{Z}$ where $N'$ is a prime in $[100N, 200 N]$. Note here that $H^\ast\subseteq[-4N,4N]$ and thus $4\wt{\Gamma} -4\wt{\Gamma}$ when viewed as a subset of $(\mb{Z}/N'\mb{Z})\times (\mb{R}/\mb{Z})^d$ is still a graph. Note that $|H^\ast|\ge\delta^{O(1)}N$.

By Lemma~\ref{lem:Bohr-set}, we have that $2H^\ast - 2H^\ast$ contains a Bohr set $B(S,\rho)$ with $|S|,\rho^{-1}\le (\log(1/\delta))^{O(1)}$. Then by applying Lemma~\ref{lem:brack-form} to each coordinate of $f_1$ on $B(S,\rho')\subseteq 2H^\ast - 2H^\ast$ with $\rho'^{-1}\le\exp(\log(1/\delta)^{O(1)})$, we have that 
\begin{equation}\label{eq:approximate-no-error-1}
f_1(h_1) = \sum_{\alpha_i\in S}a_i \{\alpha_i h_1\} + \gamma\imod 1
\end{equation}
for all $h_1\in B(S,\rho')$, for appropriate choices of $a_i,\gamma\in\mb{R}^d$. Here $\alpha_i\in(1/N')\mb{Z}$.

We now undo this transformation and we abusively view $B(S,\rho')\subseteq 2H^\ast - 2H^\ast$ as a subset of integers in $[-4N,4N]$ instead of $\mb{Z}/N'\mb{Z}$, noting that the fractional part remains identical in both cases. As a slight technical annoyance, $B(S,\rho')$ might not intersect $H^\ast$. But, by Pigeonhole there exists $x^\ast\in[-5N,5N]$ such that $|(x^\ast + B(S, \rho'/2))\cap H^{\ast}|\ge\exp(-(\log(1/\delta))^{O(1)})N$. (This requires a lower bound on the size of a Bohr set, see \cite[Lemma~4.20]{TV10}.)

Fix $h^\ast\in B(S,\rho'/2)$ such that $x^\ast+h^\ast\in H^\ast$ and consider any $h_1\in B(S,\rho'/2)$ such that $h_1+x^\ast\in H^\ast$ we have that 
\[f_1(h_1-h^\ast)+f_1(x^\ast+h^\ast) = f_1(h_1+x^\ast)+f(0)\imod 1\]
since $4\wt{\Gamma} - 4\wt{\Gamma}$ is a graph (note that $h_1-h^\ast\in B(S,\rho')\subseteq 2H^\ast-2H^\ast$). Thus we have 
\begin{align*}
f_1(h_1+x^\ast) &= f_1(h_1-h^\ast)+f_1(x^\ast+h^\ast)-f(0)\imod 1\\
&= \sum_{\alpha_i\in S}a_i\{\alpha_i((h_1+x^\ast) - (x^\ast+h^\ast))\} + \gamma'\imod 1
\end{align*}
The second line holds since $x^\ast,h^\ast$ are viewed as fixed and $h_1-h^\ast\in B(S,\rho')$ hence we may apply \eqref{eq:approximate-no-error-1}.

So, letting $H'$ be the set of values $h_1+x^\ast\in H_1$ where $h_1\in B(S,\rho'/2)$, this nearly gives the desired result. The only issue is that there are shifts inside the brackets. Note that
\begin{align*}
\{z_1 + z_2\} &=\begin{cases}
\{z_1\} + \{z_2\} -1 \text{ if } \{z_1\} + \{z_2\} > 1/2,\\
\{z_1\} + \{z_2\} + 1\text{ if } \{z_1\} + \{z_2\} \le -1/2,\\
\{z_1\} + \{z_2\} \text{ otherwise.}
\end{cases}
\end{align*}
Given this, we may Pigeonhole possible values $h_1+x^\ast$ into one of $3^{|S|}$ cases based on the corresponding shift for each $\alpha_i\in S$. Applying the above relation with $z_1 = \alpha_i(h_1+x^\ast)$ and $z_2 = -\alpha_i(x^\ast+h^\ast)$ and taking the most common case then gives the desired result.
\end{proof}

\section{Miscellaneous deferred results}\label{app:defer}
We first require the following elementary lemma which will be used in the following deduction. 
\begin{lemma}\label{lem:linear-alg}
Fix an integer $H\ge 2$. Consider vectors $v_1,\ldots, v_\ell\in\mb{Z}^d$ with integer coordinates bounded by $H$ and $w\in\mb{R}^d$ such that $\on{dist}(v_i\cdot w, \mb{Z})\le\delta$ for $1\le i\le\ell$. We may write $w = w_{\mr{small}} + w_{\mr{rat}} + (w-w_{\mr{small}} - w_{\mr{rat}})$ where $w_{\mr{rat}}$ has coordinates which are rationals with denominators bounded by $H^{O(d^{O(1)})}$, $\snorm{w_{\mr{small}}}_{\infty}\le\delta \cdot H^{O(d^{O(1)})}$, and $v_i\cdot(w-w_{\mr{small}} - w_{\mr{rat}}) = 0$ for $1\le i\le\ell$.
\end{lemma}
\begin{proof}
Note that by passing to a subset we may assume that $v_1,\ldots, v_\ell\in\mb{Z}^d$ are linearly independent. By Cramer's rule, there exist $w_1,\ldots,w_\ell\in\mb{R}^d$ which have coordinates which are height $H^{O(d^{O(1)})}$ rationals such that $w_j\cdot v_k = \mbm{1}_{j=k}$. Taking $w_{\mr{rat}} = \sum_{j=1}^\ell(v_j\cdot w - \{v_j\cdot w\})\cdot w_j$ and $w_{\mr{small}} = \sum_{j=1}^\ell \{v_j\cdot w\}\cdot w_j$ we immediately have the desired result. Recall that we have chosen the fractional part $\{\cdot\}$ to live within $(-1/2,1/2]$. 
\end{proof}

We now prove the following elementary lemma which takes a set of horizontal characters (at potentially different levels) and produces a factorization.
\begin{lemma}\label{lem:factor}
Consider a nilmanifold $G/\Gamma$ of degree-rank $(s,r)$ of dimension $d$ and complexity $M$. Consider a polynomial sequence $g$ such that $g(0) = \mr{id}_G$ and consider a set of horizontal characters $\psi_{i,j}$ for $1\le j\le\ell_i$ and where $\psi_{i,\cdot}$ is an $i$-th horizontal character of height at most $H$. Furthermore suppose that for all $i,j$, 
\[\on{dist}(\psi_{i,j}(\on{Taylor}_i(g)), \mb{Z})\le H\cdot N^{-i}.\]
Then one may factor
\[g = \eps \cdot g' \cdot\gamma\]
where:
\begin{itemize}
    \item $\eps(0) = g'(0) = \gamma(0) = \mr{id}_G$;
    \item $\psi_{i,j}(\on{Taylor}_i(g')) = 0$;
    \item $\gamma$ is $(MH)^{O_s(d^{O_s(1)})}$-rational;
    \item $d_G(\eps(n),\eps(n-1))\le (MH)^{O_s(d^{O_s(1)})}\cdot N^{-1}$ for $n\in[N]$.
\end{itemize}
\end{lemma}
\begin{proof}
By the classification of polynomial sequences in terms of coordinates of the second-kind, we have that 
\[g(n) = \exp\Big(\sum_{k=1}^{s} \binom{n}{k}g_k\Big)\]
for some $g_k\in\log(G_{(k,0)}) = \log(G_{(k,1)})$. Note that \[\on{Taylor}_i(g) = \exp(g_k)\imod G_{(i,2)}\]
and note that each $\psi_{i,j}$ can be descended to a linear map on $\log(G_{(i,1)})$ with the property that $\psi_{i,j}(\log(\Gamma\cap G_{(i,1)}))\in\mb{Z}$ and $\psi_{i,j}(\log(G_{(i,2)})) = 0$. That $\psi_{i,j}$ descends uses the fact that $\log(x) + \log(y) \equiv \log(xy) \imod \log(G_{(i,2)})$ for $x,y\in G_{(i,1)}$, which follows from Baker--Campbell--Hausdorff.

We now apply Lemma~\ref{lem:linear-alg}. As $\on{dist}(\psi_{i,j}(\on{Taylor}_i(g)), \mb{Z})\le H\cdot N^{-i}$ by assumption, we may write $g_i = g_{i,\mr{small}} + g_{i,\mr{rat}} + (g_i - g_{i,\mr{small}} - g_{i,\mr{rat}})$ such that $g_{i,\mr{rat}}$ is an $H^{O_s(d^{O_s(1)})}$-rational combination of elements in $\mc{X}\cap\log(G_{(i,1)})$, such that $\snorm{g_{i,\mr{small}}}_{\infty}\le (MH)^{O_s(d^{O_s(1)})}\cdot N^{-i}$, and such that $\psi_{i,j}(g_i - g_{i,\mr{small}} - g_{i,\mr{rat}}) = 0$. Defining
\[\gamma := \exp\Big(\sum_{k=1}^{s} \binom{n}{k}g_{k,\mr{rat}}\Big),\quad\eps := \exp\Big(\sum_{k=1}^{s} \binom{n}{k}g_{k,\mr{small}}\Big),\]
and $g' := \eps^{-1} g \gamma^{-1}$, we immediately have that $\gamma\Gamma$ is $(MH)^{O_s(d^{O_s(1)})}$-periodic by \cite[Lemma~B.14]{Len23b}. That $\eps$ is sufficiently smooth is an immediate consequence of \cite[Lemmas~B.1,~B.3]{Len23b}.
\end{proof}

We next require the following result regarding the existence of a nilmanifold partition of unity. As a remark, a similar statement (e.g.~with $\sum_j\tau_j = 1$) appears as \cite[Lemma~2.4]{LSS24}. The proof there, strangely, does not adapt in a straightforward manner to here. 
\begin{lemma}\label{lem:nilmanifold-partition-of-unity}
Fix $\eps\in (0,1/2)$ and a nilmanifold $G/\Gamma$ of degree $s$, dimension $d$, and complexity $M$. There exists an index set $I$ and a collection of nonnegative smooth functions $\tau_j\colon G/\Gamma\to\mb{R}^{\ge 0}$ for $j\in I$ such that:
\begin{itemize}
    \item For all $g\in G$, we have $\sum_{j\in I}\tau_j(g\Gamma)^2 = 1$;
    \item $|I|\le (1/\eps)^{O_s(d^{O_s(1)})}$;
    \item For each $j\in I$, there exists $\beta\in[-2,2]^{d}$ so that for any $g\Gamma\in\on{supp}(\tau_j)$ there exists $g'\in g\Gamma$ such that $\psi_G(g')\in\prod_{i=1}^d[\beta_i-\eps,\beta_i+\eps]$;
    \item $\tau_j$ are $(M/\eps)^{O_s(d^{O_s(1)})}$-Lipschitz on $G/\Gamma$;
    \item For any $g\in G$, $g\Gamma$ is contained in the support of at most $2^{O_s(d)}$ terms.
\end{itemize}
\end{lemma}
\begin{proof}
We will prove the statement inductively based on the degree of the nilmanifold. For degree $1$ nilmanifolds $G$, note that $G\simeq \mb{T}^d$.There exists a set of function $\rho_1,\ldots,\rho_{2k}\colon\mb{T}\to\mb{R}^{\ge 0}$ such that:
\begin{itemize}
    \item $\on{supp}(\rho_j)\subseteq [j/(2k), j/(2k) + 1/k] \imod 1$;
    \item $\sum_{j=1}^{2k} \rho_{j}^2 = 1$;
    \item $\rho_{j}$ are $O(1/k)$-Lipschitz.
\end{itemize}
Taking $k = O(\eps^{-1})$, we have that 
\[1 = \sum_{(j_1,\ldots,j_d)\in[2k]^d}\prod_{\ell=1}^d\rho_{j_\ell}((\psi_{G}(g))_\ell)^2\]
where $(\psi_{G})_\ell$ denotes the $\ell$-th coordinate of $\psi_{G}$. For $\vec{j}\in[2k]^d$ we take 
\[\tau_{\vec{j}}(g) = \prod_{\ell=1}^d\rho_{j_\ell}((\psi_{G}(g))_\ell)\]
and note that this function is $\Gamma$-invariant since multiplying by an element in $\Gamma$ shifts all coordinates by an integer. Furthermore, by \cite[Lemma~B.3]{Len23b} we have that the standard $\ell^{\infty}$-metric on $G/\Gamma$ is equivalent to $d_{G/\Gamma}$ up to a factor of $O(M)^{O(d^{O(1)})}$. This completes the proof in this case.

When considering the case of a degree $s\ge 2$ filtration on $G$, suppose that $G_0 = G_1 \geqslant G_2 \geqslant \cdots \geqslant G_s \geqslant \mr{Id}_G$ is the given filtration. Note that if $\mc{X} = \{X_1,\ldots,X_d\}$ is the adapted Mal'cev basis for $G/\Gamma$ then 
\[\wt{\mc{X}} := \{X_1,\ldots,X_{\dim(G)-\dim(G_s)}\}\imod\log(G_s)\] is a valid Mal'cev basis for $\wt{G} := G/G_s$. Furthermore define $\wt{\Gamma} := \Gamma/(\Gamma\cap G_s)$. The complexity of $\wt{\mc{X}}$ is always bounded by $M$ by definition. The filtration on $\wt{G}$ is lower degree.

By induction, we have functions $(\tau_j)_{j\in I}$ with $|I|\le (M/\eps)^{O_s(d^{O_s(1)})}$ such that
\[1 = \sum_{j\in I}\wt{\tau_j}(\wt{g} \wt{\Gamma})^2\]
and satisfying various other appropriate properties. We may lift these functions to $G/\Gamma$ via 
\[\tau_j(g \Gamma) = \wt{\tau_j}((g \imod G_s) \wt{\Gamma}).\]
Note that this is well-defined since $g \Gamma\imod G_s = (g \imod G_s) \cdot (\Gamma\imod G_s) = (g \imod G_s)\wt{\Gamma}$.

We view each $\tau_j$ as a function on $\prod_{i=1}^{\dim(\wt{G})}(\beta_i - 1/2, \beta_i + 1/2] \times\mb{T}^{\dim(G_s)}$ which only depends on the first $\dim(\wt{G})$ coordinates and such that the support is only within some $\prod_{i=1}^{\dim(\wt{G})}[\beta_i - \eps, \beta_i + \eps] \times\mb{T}^{\dim(G_s)}$. This is via identifying the fundamental domain of $G/\Gamma$ via Mal'cev coordinates of the second-kind (see the proof of \cite[Lemma~B.6]{Len23b}). We let $\psi_{\beta}\colon G/\Gamma\to\prod_{i=1}^{\dim(\wt{G})}(\beta_i - 1/2, \beta_i + 1/2] \times\mb{T}^{\dim(G_s)}$ denote this identification. (Note that the choice of $\beta$ depends on $j\in I$, which we will fix through the remainder of the proof.)

We now have 
\[\tau_j(g \Gamma)^2 = \wt{\tau_j}((g \imod G_s) \wt{\Gamma})^2 \cdot\sum_{(t_1,\ldots,t_{\dim(G_s)})\in[2k]^{\dim(G_s)}} \prod_{\ell=1}^{\dim(G_s)}\rho_{t_\ell}((\psi_{\beta}(g\Gamma))_{\ell+\dim(\wt{G})})^2\]
where $k = O(1/\eps)$ and $\rho$ are defined as above.

The fact that each piece 
\[\tau_{j,\vec{t}}(g\Gamma)^2 := \tau_j(g \Gamma)^2 \cdot\prod_{\ell=1}^{\dim(G_s)}\rho_{t_\ell}((\psi_{\beta}(g\Gamma))_{\ell+\dim(\wt{G})})^2\]
is $\Gamma$-invariant on the right is trivial by construction, and the sum of squares property is trivial. 

Identifying $\rho_{j}$ with a function $\mb{R}\to\mb{R}^{\ge 0}$ where $\on{supp}(\rho_{j})\subseteq [j/(2k),j/(2k) + 1/k]$, we may identify $\tau_{j,\vec{t}}$ with a function on the fundamental domain (with respect to second-kind coordinates) of the form
\[\prod_{i=1}^{\dim(\wt{G})}(\beta_i - 1/2, \beta_i + 1/2] \times\prod_{\ell=1}^{\dim(G_s)} ((t_\ell+1)/(2k)-1/2, (t_\ell+1)/(2k) + 1/2].\]
To check that this function is sufficiently Lipschitz, we note that each element $g\Gamma$ has a unique representative in this domain.

Consider $\tau_{j,\vec{t}}(x\Gamma)$ and $\tau_{j,\vec{t}}(y\Gamma)$; by multiplying by the lattice we may assume that $\psi(x),\psi(y)$ are in the specified fundamental domain. Furthermore if $d_{G/\Gamma}(x\Gamma,y\Gamma)\ge\eps'= M^{-O_s(d^O_s(1))}$ we immediately win as $\tau_{j,\vec{t}}$ is $1$-bounded. 
We claim that if $d_{G/\Gamma}(x\Gamma,y\Gamma)\le\eps'$ then $d_{G/\Gamma}(x\Gamma,y\Gamma) = d_{G}(x,y)$. In particular, note that 
\begin{align*}
d_{G/\Gamma}(x\Gamma,y\Gamma) &=\min_{\gamma\in\Gamma}d_{G}(x\gamma,y)
\end{align*}
and that 
\[\min_{\gamma\in\Gamma\setminus\{\mr{id}_G\}}d_{G}(x\gamma,y) \ge M^{-O_s(d^{O_s(1)})}\cdot\min_{\gamma\in\Gamma\setminus\{\mr{id}_G\}}d_{G}(\gamma,x^{-1}y) \ge M^{-O_s(d^O_s(1))}\]
which gives the desired contradiction assuming that various implicit constants defining $\eps'$ are chosen appropriately.

Now we may assume that $x,y$ are such that 
\[\psi(x),\psi(y)\in\prod_{i=1}^{\dim(\wt{G})}[\beta_i - 2\eps, \beta_i + 2\eps) \times\prod_{\ell=1}^{\dim(G_s)} [t_\ell/(2k)-\eps, t_\ell/(2k) + 1/k+\eps),\]
else both function values vanish (again supposing $\eps'$ is sufficiently small). This is because $d_G(x,y)$ is equivalent to $\snorm{\psi(x)-\psi(y)}_{\infty}$ (up to a factor of $M^{O_s(d^{O_s(1)})}$) for bounded elements by \cite[Lemma~B.3]{Len23b}), and due to the condition on the support of $\rho_{t_\ell}$.

In particular, $\psi(x),\psi(y)$ are seen to lie in the interior of the domain. The result then follows immediately noting that $\tau_j$ is appropriately Lipschitz and $\rho_{t_\ell}$ is an appropriately Lipschitz function on $\mb{R}$. The claim that $g\Gamma$ is contained in the support of at most $2^{O_s(d)}$ terms follows trivially by construction. 
\end{proof}

Given this we are now in position to show the existence of nilcharacters on $G/\Gamma$.
\begin{lemma}\label{lem:nil-exist}
Fix $\eps\in (0,1/2)$ and a nilmanifold $G/\Gamma$ of degree $s$, dimension $d$, and complexity $M$. Fix $\eta$ a vertical $G_s$-frequency with height bounded by $M$. There exists a nilcharacter $F$ with frequency $\eta$ such that the output dimension is bounded by $2^{O_s(d^{O_s(1)})}$ and each coordinate is $O_s(M)^{O_s(d^{O_s(1)})}$-Lipschitz.
\end{lemma}
\begin{proof}
Let $\wt{G} = G/G_s$ and $\wt{\Gamma} = \Gamma/(\Gamma\cap G_s)$. Apply Lemma~\ref{lem:nilmanifold-partition-of-unity} on $\wt{G}/\wt{\Gamma}$ with $\eps=1/4$ to obtain $\wt{\tau_j}$ for $j\in I$. For $\eta = 0$, we may take the coordinates of $F$ to be 
\[\tau_j(g\Gamma) = \wt{\tau_j}((g\imod G_d) \wt{\Gamma}).\]
In general, for appropriate $\beta$ depending on $j$, we have that $g\Gamma$ is naturally identified with a unique point inside $\prod_{i=1}^{\dim(\wt{G})}(\beta_i-1/2,\beta_i + 1/2] \times\mb{T}^{\dim(G_s)}$ as in the proof of Lemma~\ref{lem:nilmanifold-partition-of-unity} and we let $\psi_{\beta}(g\Gamma)$ denote this map. The key point is to write
\[\tau_j(g\Gamma) = \wt{\tau_j}((g\imod G_s) \wt{\Gamma}) \cdot\exp(\eta\cdot\psi_{\beta}(g\Gamma))\]
and note that $\sum_{j\in I}|\tau_j(g\Gamma)|^2 = 1$ as before. Here we have identified $\eta$ with an integer vector using the last $\dim(G_s)$ elements of the Mal'cev basis and extending by $0$. Note that this is trivially a function on $G/\Gamma$ and by construction it has the $G_s$-vertical frequency $\eta$. The only technical point is verifying that this function is indeed Lipschitz, which we check for each coordinate $\tau_j$.

Consider $x\Gamma$ and $y\Gamma$. If $\tau_j(x\Gamma)=\tau_j(y\Gamma)=0$ the Lipschitz condition is obviously satisfied. Thus at least one value is nonzero, and without loss of generality we may assume $\tau_j(x\Gamma)\neq 0$. Furthermore, noting that $\tau_j$ is $1$-bounded, we may assume that $d_{G/\Gamma}(x\Gamma,y\Gamma)\le M^{-O_s(d^{O_s(1)})}$. As $\tau_j(x\Gamma)\neq 0$, possibly shifting $x$ on the right by an element in the lattice allows us to assume 
\[\psi(x)\in\prod_{i=1}^{\dim(\wt{G})}(\beta_i-1/4,\beta_i + 1/4] \times (0,1]^{\dim(G_s)}.\]
Via an argument analogous to that in the proof of Lemma~\ref{lem:nilmanifold-partition-of-unity}, there exists $y'$ such that $y'\Gamma = y\Gamma$, 
\[\psi(y')\in\prod_{i=1}^{\dim(\wt{G})}(\beta_i-1/3,\beta_i + 1/3] \times (-1/2,3/2]^{\dim(G_s)},\]
and $\snorm{\psi(x)-\psi(y')}_{\infty}\le M^{O_s(d^{O_s(1)})} d_{G/\Gamma}(x\Gamma,y\Gamma)$. Since $\vec{z}\mapsto\exp(\eta\cdot\vec{z})$ is an appropriately Lipschitz function on the torus if $\eta\in\mb{Z}^{\dim(G_s)}$, the desired result follows immediately. 
\end{proof}

We will also require the following converse of the $U^{s+1}$-inverse theorem; this is verbatim in \cite[Appendix~G]{GTZ11} modulo various complexity details being omitted.
\begin{lemma}\label{lem:converse}
Fix $\eps\in(0,1/2)$ and let $G/\Gamma$ be a degree $s$ nilmanifold of dimension $d$ and complexity $M$, and let $g(n)$ be a polynomial sequence with respect to this filtration. Furthermore let $F\colon G/\Gamma\to\mathbf{C}$ satisfy $\snorm{F}_{\mr{Lip}}\le M$. If $f\colon[N]\to\mb{C}$ is a $1$-bounded function such that
\[\big|\mb{E}_{n\in[N]}f(n) \ol{F(g(n)\Gamma)}\big|\ge\eps,\]
then 
\[\snorm{f}_{U^{s+1}[N]}\ge (\eps/M)^{O_s(d^{O_s(1)})}.\]
\end{lemma}
\begin{proof}
In the degenerate case when $s = 0$, we take a degree $s$ nilsequence of complexity $M$ to be a constant function $\psi$ bounded by $M$. This implies that 
\[|\mb{E}_{n\in[N]}f(n)|\ge\eps/M\]
and by Cauchy--Schwarz we have 
\[\mb{E}_{n,n'\in[N]}f(n)\ol{f(n')}\ge (\eps/M)^2.\]
By unwinding definitions this implies the case $s = 0$.

For larger $s$, by applying \cite[Lemma~A.6]{Len23b} we may assume that 
\[\big|\mb{E}_{n\in[N]}f(n) \ol{F_{\xi}(g(n)\Gamma)}\big|\ge (\eps/M)^{O_s(d^{O_s(1)})}\]
where $F_{\xi}$ is a $(M/\eps)^{O_s(d^{O_s(1)})}$-Lipschitz function with $G_s$-vertical frequency $\xi$ bounded in height by $(M/\eps)^{O_s(d^{O_s(1)})}$, after Pigeonhole. Cauchy--Schwarz implies that
\[\mb{E}_{n,n'\in[N]}f(n)\ol{f(n')} F_{\xi}(g(n')\Gamma)\ol{F_{\xi}(g(n)\Gamma)}\ge (\eps/M)^{O_s(d^{O_s(1)})}.\]
Note that we may rewrite this as 
\[\mb{E}_{n\in[N],h\in[\pm N]}f(n)\ol{f(n+h)} F_{\xi}(g(n + h)\Gamma)\ol{F_{\xi}(g(n)\Gamma)}\ge (\eps/M)^{O_s(d^{O_s(1)})},\]
where we extend $f$ by $0$ in the usual manner. We define 
\[G^{\Box} = \{(g,g')\colon g,g'\in G,g^{-1}g'\in G_2\}\]
and note that this has a filtration $(G^{\Box})_i = \{(g,g')\colon g,g'\in G_i,g^{-1}g'\in G_{i+1}\}$ by \cite[Lemma~A.3]{Len23b} (with $G^{\Box} = (G^{\Box})_{1}$). Let $\Gamma^{\Box} = (\Gamma\times\Gamma)\cap G^{\Box}$ and note that 
\[\wt{F}_{\xi}((x,y) (\Gamma\times\Gamma)) := F_{\xi}(x \Gamma) \ol{F_{\xi}(y\Gamma)}\]
is invariant under $G^{\Box}_{s}$. Note that $\wt{F}_{\xi}$ is $(M/\eps)^{O_s(d^{O_s(1)})}$-Lipschitz on $G\times G$ and on $G^{\Box}$, and $G^{\Box}/\Gamma^{\Box}$ is a nilmanifold of appropriate complexity by \cite[Lemma~A.3]{Len23b}.

Let 
\[(g(0),g(h)) = \{(g(0),g(h))\} \cdot [(g(0),g(h))]\]
with $d_{G\times G}(\{(g(0),g(h))\})\le M^{O_s(d^{O_s(1)})}$ and $[(g(0),g(h))]\in\Gamma\times\Gamma$. Define
\[g_h'(n) = \{(g(0),g(h))\}^{-1}(g(n),g(n+h))[(g(0),g(h))]^{-1};\]
this is easily seen to be a polynomial sequence with respect to $G^{\Box}$. Thus
\[\mb{E}_{n\in[N],h\in[\pm N]}f(n)\ol{f(n+h)\wt{F}_{\xi}(\{(g(0),g(h))\}g_h'(n) (\Gamma\times\Gamma))}\ge (\eps/M)^{O_s(d^{O_s(1)})}.\]
Define $\wt{F}_{\xi,h}(x,y) := \wt{F}_{\xi}(\{(g(0),g(h))\}(x,y) (\Gamma\times\Gamma))$ and note that it is $(M/\eps)^{O_s(d^{O_s(1)})}$-Lipschitz on $G\times G$ and on $G^{\Box}$ by \cite[Lemma~B.4]{Len23b}. Applying the triangle inequality and restricting to $G^{\Box}$ we have 
\[\mb{E}_{h\in[\pm N]}\Big|\mb{E}_{n\in[N]} \Delta_h f(n) \cdot\ol{\wt{F}_{\chi,h}(g_h'(n) \Gamma^{\Box}})\Big|\ge (\eps/M)^{O_s(d^{O_s(1)})}.\]
Since $\wt{F}_\xi$ is invariant under $(G^{\Box})_s$, passing to $G^{\Box}/(G^{\Box})_{s}$ gives a nilmanifold of degree $(s-1)$ and complexity $M^{O_s(d^{O_s(1)})}$. Thus we may apply by induction, and deduce that 
\[\mb{E}_{h\in[\pm N]}\snorm{\Delta_h f}_{U^s[N]}\ge (\eps/M)^{O_s(d^{O_s(1)})}.\]
Since
\[\mb{E}_{h\in[\pm N]}\snorm{\Delta_h f}_{U^s[N]}^{2^s}\lesssim_s\snorm{f}_{U^{s+1}[N]}^{2^{s+1}},\]
the desired result follows.
\end{proof}

We now check the deferred Lemma~\ref{lem:multi-complex}.
\begin{proof}[Proof of Lemma~\ref{lem:multi-complex}]
We first construct a weak basis for $G_{\mr{Quot}}\ltimes G_{\mr{Lin}}$. Note that each element in $(g,g')\in G_{\mr{Quot}}\ltimes G_{\mr{Lin}}$ may be written as
\[(g,g') = (g,\mr{id}_{G_{\mr{Lin}}})\cdot (\mr{id}_{G_{\mr{Quot}}},g').\]

Consider $\wt{e}_{i,j}$ and consider $(r-1)$-fold commutators of $\wt{e}_{i_1,j_1},\ldots,\wt{e}_{i_r,j_r}$ with $i_1 + \cdots + i_r\le s-2$ or $i_1 + \cdots + i_r=s-1$, $r\le r^\ast$ and at most one generator has $i_{\ell}>D_{i_{\ell}}^\ast$. We define the type of the commutator to be given by the multiset $\{\wt{e}_{i_1,j_1},\ldots,\wt{e}_{i_r,j_r}\}$ and we say that said type is linear if $i_{\ell}>D_{i_{\ell}}^\ast$ for exactly one index $\ell$. We define the degree of a commutator to be $i_1+\cdots+i_r$. As discussed in Lemmas~\ref{lem:universal-complexity} and \ref{lem:universal-complexity-2}, commutators of all types span $\log(G_{\mr{Quot}})$ and commutators of linear type span $\log(G_{\mr{Lin}})$, and all relations between these elements are spanned by relations between commutators of the same type of height $O_s(1)$.

Given this, for each collection of commutators of a given type choose a subset which ``spans the type'' (similar to in the proof of Lemma~\ref{lem:universal-complexity}). Let $\mc{X}_{1}$ denote the set of selected commutators and $\mc{X}_2$ denote the selected commutators which are of linear type. Our weak basis for $G_{\mr{Quot}}\ltimes G_{\mr{Lin}}$ will be 
\[\mc{X} = \{(X,0)\colon X\in\mc{X}_1\} \cup \{(0,X)\colon X\in\mc{X}_2\};\]
this is seen to be a basis for the Lie algebra of $G_{\mr{Quot}}\ltimes G_{\mr{Lin}}$. That it spans is trivial, and if there were a relation note that there could be no elements of the form $(X,0)$ in the relation since projecting onto the first coordinate we recover multiplication in $G_{\mr{Quot}}$. Given that there are no elements of the form $(X,0)$, within this relation multiplication then acts exactly as in $G_{\mr{Lin}}$ and the result claimed independence follows.

We give $G_{\mr{Quot}}\ltimes G_{\mr{Lin}}$ a multidegree filtration by taking the multidegree filtration of $G_{\mr{Multi}}$ and intersecting with the subgroup of elements of the form $(0,(g,g_1))$. We see that all the subgroups of the filtration are in fact spanned subsets by subsets of $\mc{X}$. This is simply by taking the generators in $\mc{X}$ of the appropriate degree-rank; for instance 
\[(G_{\mr{Quot}}\ltimes G_{\mr{Lin}})_{(0,d)} = \{(g,g_1)\colon g\in (G_{\mr{Quot}})_{(d,0)}, g_1\in(G_{\mr{Quot}})_{(d,0)}\cap G_{\mr{Lin}}\}\]
and we take the subsets of $\{(X,0)\colon X\in\mc{X}_1\}$ and $\{(0,X)\colon X\in\mc{X}_2\}$ where $X$ has degree at least $d$. This is similarly true for $\bigvee_{|\vec{i}| = k}(G_{\mr{Quot}}\ltimes G_{\mr{Lin}})_{(i_1,i_2)}$ which will ultimately form the underlying degree filtration for $G_{\mr{Quot}}\ltimes G_{\mr{Lin}}$. Furthermore ordering the basis according to whether they lie in the degree ordering associated to $G_{\mr{Quot}}\ltimes G_{\mr{Lin}}$ proves that the basis has the degree $O_s(1)$ nesting property. Thus it suffices to check the complexity of various commutators.

Note the identity
\begin{align*}
[V,W] &= \frac{d}{ds}\frac{d}{dt} \exp(sV)\exp(tW)\exp(-sV)\exp(-tW)\bigg|_{s,t=0}
\end{align*}
which holds for any Lie group and the associated Lie bracket. It is therefore immediate that
\[[(X,0),(X',0)] = ([X,X'],0)\text{ and }[(0,X),(0,X')] = (0,[X,X']),\]
and we have
\begin{align*}
&[(X,0),(0,X')] \\
&\qquad= \frac{d}{ds}\frac{d}{dt} (\exp(sX), \mr{id}_{G_{\mr{Lin}}}) \cdot (\mr{id}_{G_{\mr{Quot}}}, \exp(tX')) \cdot (\exp(-sX), \mr{id}_{G_{\mr{Lin}}}) \cdot (\mr{id}_{G_{\mr{Quot}}}, \exp(-tX'))\bigg|_{s,t=0} \\
&\qquad= \frac{d}{ds}\frac{d}{dt} (\mr{id}_{G_{\mr{Quot}}}, \exp(sX)\exp(tX')\exp(-sX)\exp(-tX'))\bigg|_{s,t=0} \\
&\qquad=(0,[X,X']).
\end{align*}
This immediately implies that the structure constants associated to the weak basis $\mc{X}$ are of height $O_s(1)$.

When including the semi-direct action, we will use the weak basis given by taking elements $\log((\vec{e}_{ij},(\mr{id}_{G_{\mr{Quot}}},\mr{id}_{G_{\mr{Lin}}})))$ where $\vec{e}_{i,j}$ denotes the elementary basis vector in the corresponding direction in $R$, placed at the start of $\mc{X}$. This is easily seen to preserve the nesting property.

To compute the associated structure constants, first note that 
\[[(\vec{e}_{ij},(\mr{id}_{G_{\mr{Quot}}},\mr{id}_{G_{\mr{Lin}}})),(0,(g,\mr{id}_{G_{\mr{Lin}}}))] = \mr{id}_{G_{\mr{Multi}}}\]
and thus the all Lie bracket structure constants of the corresponding form vanish. Furthermore note that
\begin{align*}
&[\log((\vec{e}_{ij},(\mr{id}_{G_{\mr{Quot}}},\mr{id}_{G_{\mr{Lin}}}))),(0,(0,X'))] = \frac{d}{ds}\frac{d}{dt} (0,(\exp(tX')^{s \cdot\vec{e}_{i,j}},\mr{id}_{G_{\mr{Lin}}}))\bigg|_{s,t=0}.
\end{align*}
We have that if the type of $X'$ does not contain $\wt{e}_{i,j}$ then $\exp(tX')^{s \cdot\vec{e}_{i,j}} = \mr{id}_{G_{\mr{Quot}}}$ and otherwise $\exp(tX')^{s \cdot\vec{e}_{i,j}} = \exp(stX')$ (recall the definition of exponentiation by elements of $R$ given in Section~\ref{sub:construct}). In either case the structure constant is appropriately rational. Therefore we may construct a Mal'cev basis adapted to $G_{\mr{Multi}}$ with the appropriate complexity by applying \cite[Lemma~B.11]{Len23b} to $\mc{X}$ to construct a Mal'cev basis for $G_{\mr{Quot}}\ltimes G_{\mr{Lin}}$ and adding the semi-direct Mal'cev basis elements described above to the front of the list. We define this basis to be $\mc{X}_{\mr{Multi}}$ and define initial segment corresponding to the semi-direct Mal'cev basis elements to be $\mc{X}_{\mr{Multi}, R}$ and the remaining elements to be  $\mc{X}_{\mr{Multi},G_{\mr{Quot}}\ltimes G_{\mr{Lin}}}$.

We finally check the that $F_{\mr{Multi}}$ is an appropriately Lipschitz function. Let $\delta$ be defined as in Section~\ref{sub:construct}. Fix a pair $x,y\in G_{\mr{Multi}}$. Note that if 
\[d_{G_{\mr{Multi}}/\Gamma_{\mr{Multi}}}(x\Gamma_{\mr{Multi}},y\Gamma_{\mr{Multi}})\ge\delta^{O_s(d^{O_s(1)})}\]
we have that 
\[\frac{F_{\mr{Multi}}(x\Gamma_{\mr{Multi}})-F_{\mr{Multi}}(y\Gamma_{\mr{Multi}})}{d_{G_{\mr{Multi}}/\Gamma_{\mr{Multi}}}(x\Gamma_{\mr{Multi}},y\Gamma_{\mr{Multi}})}\le\delta^{-O_s(d^{O_s(1)})} \cdot 2\snorm{F_{\mr{Multi}}}_{\infty}\]
which is sufficiently bounded. Therefore to check the Lipschitz constant it suffices to consider $x,y$ such that $d_{G_{\mr{Multi}}/\Gamma_{\mr{Multi}}}(x\Gamma_{\mr{Multi}},y\Gamma_{\mr{Multi}})\le\delta^{O_s(d^{O_s(1)})}$ (where the implicit constants are chosen sufficiently large for the remainder of the argument). By multiplying by elements in the lattice, we may assume that $d_{G_{\mr{Multi}}/\Gamma_{\mr{Multi}}}(x\Gamma_{\mr{Multi}},y\Gamma_{\mr{Multi}}) = d_{G_{\mr{Multi}}}(x,y)$, that $F_{\mr{Multi}}(x\Gamma_{\mr{Multi}})\neq 0$, and 
\[\psi_{\mc{X}_{\mr{Multi}}}(x)\in[-1/2,1/2)^{\dim(G_{\mr{Multi}})}.\]
Note that to assume that $F_{\mr{Multi}}(x\Gamma_{\mr{Multi}})\neq 0$ we may need to swap $x$ and $y$ (if both are zero there is nothing to check with respect to the Lipschitz constant).

Since $F_{\mr{Multi}}(x\Gamma_{\mr{Multi}})\neq 0$ we in fact have that the first $\sum_{i=1}^{s-1}D_i^{\mr{Lin}}$ coordinates of $\psi_{\mc{X}_{\mr{Multi}}}(x)$ are in $[-1/2 + \delta, 1/2-\delta]$. This implies, due to the distance bound between $x$ and $y$ and by \cite[Lemma~B.3]{Len23b}, that the first $\sum_{i=1}^{s-1}D_i^{\mr{Lin}}$ coordinates of $\psi_{\mc{X}_{\mr{Multi}}}(y)$ are in $[-1/2 + \delta/2, 1/2-\delta/2]$. Therefore if $x = (t_1,(g_1,g_1'))$ and $y = (t_2,(g_2,g_2'))$ then 
\begin{align*}
F_{\mr{Multi}}(x\Gamma_{\mr{Multi}}) &= F^\ast(g_1'\Gamma_{\mr{Quot}})\cdot\prod_{\substack{1\le i\le s-1\\ D_i^{\ast}<j\le D_i^\ast + D_i^{\mr{Lin}}}} \phi((t_1)_{i,j}),\\
F_{\mr{Multi}}(x\Gamma_{\mr{Multi}}) &= F^\ast(g_2'\Gamma_{\mr{Quot}})\cdot\prod_{\substack{1\le i\le s-1\\ D_i^{\ast}<j\le D_i^\ast + D_i^{\mr{Lin}}}} \phi((t_2)_{i,j}).
\end{align*}
Note that
\begin{align*}
|F_{\mr{Multi}}(x\Gamma_{\mr{Multi}}) - F_{\mr{Multi}}(y\Gamma_{\mr{Multi}})|&\le\snorm{F^\ast}_{\infty} \cdot\bigg|\prod_{\substack{1\le i\le s-1\\ D_i^{\ast}<j\le D_i^\ast + D_i^{\mr{Lin}}}} \phi((t_1)_{i,j}) - \prod_{\substack{1\le i\le s-1\\ D_i^{\ast}<j\le D_i^\ast + D_i^{\mr{Lin}}}} \phi((t_2)_{i,j})\bigg| \\
& + |F^\ast(g_1'\Gamma_{\mr{Quot}}) - F^\ast(g_2'\Gamma_{\mr{Quot}})|,
\end{align*}
where we have used that $\phi$ is $1$-bounded. Next note that distance in $\psi_{\mc{X},\mr{exp}}$ controls the distance in $\psi_{\mc{X}}$ for bounded elements by \cite[Lemma~B.1]{Len23b} and distance in $\psi_{\mc{X}}$ controls distance in $d_{G_{\mr{Multi}}}$ by \cite[Lemma~B.3]{Len23b}. The first term is therefore sufficiently bounded as $\phi$ is $O(1/\delta)$-Lipschitz.

Finally note that 
\begin{align*}
(g_1,g_1') &= \prod_{(X_i,X_i')\in\mc{X}_{\mr{Multi},G_{\mr{Quot}}\ltimes G_{\mr{Lin}}}} \exp((X_i,X_i'))^{x_i},\\
(g_2,g_2') &= \prod_{(X_i,X_i')\in\mc{X}_{\mr{Multi},G_{\mr{Quot}}\ltimes G_{\mr{Lin}}}} \exp((X_i,X_i'))^{y_i},
\end{align*}
where $x_i$ and $y_i$ are the coordinates of $x$ and $y$ in $\psi_{\mc{X}_{\mr{Multi}}}$ in the coordinates corresponding to $\mc{X}_{\mr{Multi},G_{\mr{Quot}\ltimes G_{\mr{Lin}}}}$. This is using that
\[(t,(\mr{id}_{G_{\mr{Quot}}},\mr{id}_{G_{\mr{Lin}}}))\cdot(0,(g,g')) = (t,(g,g')).\]
Therefore we have that 
\begin{align*}
g_1 &= \prod_{(X_i,X_i')\in\mc{X}_{\mr{Multi},G_{\mr{Quot}}\ltimes G_{\mr{Lin}}}} \exp(X_i)^{x_i}\\
g_2 &= \prod_{(X_i,X_i')\in\mc{X}_{\mr{Multi},G_{\mr{Quot}}\ltimes G_{\mr{Lin}}}} \exp(X_i)^{y_i};
\end{align*}
and note that $\exp(X_i)$ are appropriately bounded elements in $G_{\mr{Quot}}$ since $\mc{X}_{\mr{Multi},G_{\mr{Quot}}\ltimes G_{\mr{Lin}}}$ are low height combinations of elements in $\mc{X}$. Via telescoping, and using that the metric $d_{G_{\mr{Quot}}}$ is right-invariant and essentially left-invariant under multiplication by bounded elements (e.g.~\cite[Lemma~B.4]{Len23b}), we have that
\begin{align*}
d_{G_{\mr{Quot}}}(g_1,g_2) &\le\delta^{-O_s(d^{O_s(1)})} \cdot\sum_{(X_i,X_i')\in\mc{X}_{\mr{Multi},G_{\mr{Quot}}}}d_{G_{\mr{Quot}}}(\exp(X_i)^{x_i-y_i},\mr{id}_{G_{\mr{Quot}}})\\
&\le\delta^{-O_s(d^{O_s(1)})} \cdot\sum_{(X_i,X_i')\in\mc{X}_{\mr{Multi},G_{\mr{Quot}}}}|x_i-y_i|\\
&\le\delta^{-O_s(d^{O_s(1)})}\cdot\snorm{\psi_{\mc{X}_{\mr{Multi}}}(x) - \psi_{\mc{X}_{\mr{Multi}}}(y)}\le\delta^{-O_s(d^{O_s(1)})}\cdot d_{G_{\mr{Multi}}}(x,y).
\end{align*}
This completes the proof upon noting that $F^\ast$ is appropriately Lipschitz on $G_{\mr{Quot}}$.
\end{proof}

\section{Nilcharacters}\label{app:nilcharacters}
This section is essentially a straightforward quantification of various statements regarding nilcharacters proven in \cite[Appendix~E]{GTZ12}.

We first require that two nilcharacters being equivalent is a transitive relationship; this is a quantified version of \cite[Lemma~E.7]{GTZ12}. Recall the notion of complexity $(M,d)$ that we carry over from Section~\ref{sec:sym}.
\begin{lemma}\label{lem:trans}
Consider three nilcharacters $\chi_1,\chi_2,\chi_3$ each of complexity $(M,d)$ and such that the pair $\chi_1$ and $\chi_2$ and the pair $\chi_2$ and $\chi_3$ are $(M,D,d)$-equivalent for multidegree $J$. Then $\chi_1$ and $\chi_3$ are $((MD)^{O_{|J|}(1)},(MD)^{O_{|J|}(1)},O(d))$-equivalent for multidegree $J$.
\end{lemma}
\begin{proof}
Notice that each coordinate of $\chi_1\otimes\ol{\chi_3}$ may be expressed as the sum of at most $D$ coordinates of the nilcharacter 
\[\chi_1\otimes (\ol{\chi_2} \otimes \chi_2)\otimes \ol{\chi_3};\]
this follows since the trace of $\ol{\chi_2} \otimes \chi_2$ is $1$. The result then follows by rewriting
\[\chi_1\otimes (\ol{\chi_2} \otimes \chi_2)\otimes \ol{\chi_3} = (\chi_1\otimes \ol{\chi_2}) \otimes (\chi_2\otimes \ol{\chi_3})\]
and applying the assumption.
\end{proof}

We will generally require the following specialization lemmas; these are rather straightforward consequences of the definitions modulo the need to handle slight filtration issues.
\begin{lemma}\label{lem:specialization}
We have the following:
\begin{itemize}
    \item Consider a nilsequence $\chi(h_1,\ldots,h_k)$ of multidegree $(s_1,\ldots,s_k)$ and complexity $(M,d)$. Given $h^\ast\in\mb{Z}$, the function $\chi(h^{\ast},h_2,\ldots,h_k)$, treating $h^\ast$ as fixed, is a multidegree $(s_2,\ldots,s_k)$ nilsequence of complexity $(M^{O_{|\vec{s}|}(d^{O_{|\vec{s}|}(1)})},d)$.
    \item Consider homomorphisms $L_i\colon\mb{Z}^{\ell}\to\mb{Z}$. If $\chi(h_1,\ldots,h_k)$ is a nilsequence of degree $s$ of complexity $(M,d)$ then $\chi(L_1(t_1,\ldots,t_{\ell}),\ldots,L_k(t_1,\ldots,t_{\ell}))$ is a degree $s$ nilsequence in variables $t_1,\ldots, t_{\ell}$ of complexity $(M,d)$.
    \item If $\chi(h_1,\ldots,h_k)$ is a nilsequence of multidegree $(s_1,\ldots,s_k)$ of complexity $(M,d)$ then it is also a nilsequence of degree $s_1 + \cdots + s_k$ of complexity $(M,d)$. 
\end{itemize}
\end{lemma}
\begin{remark*}
This result allows us to interpret expressions such as $\chi(h_1+h_1',h_2,\ldots,h_k)$ as an appropriate degree nilcharacter in $k+1$ variables, if $\chi$ is a nilcharacter in $k$ variables with multidegree $(s_1,\ldots,s_k)$.
\end{remark*}
\begin{proof}
We handle these items in reverse order (as this is also the difficulty of these claims). Let
\[\chi(h_1,\ldots,h_k) = F(g(h_1,\ldots,h_k)\Gamma)\]
with the underlying nilmanifold being $G/\Gamma$ and the specified Mal'cev basis being $\mc{X}$.

For the last claim, note that $\mc{X}$ (by the definition of complexity for multidegree nilmanifolds) is adapted to the degree filtration $G_t = \bigvee_{|\vec{i}|=t}G_{\vec{i}}$. Furthermore by the inclusion given on \cite[p.~1264]{GTZ12} or direct inspection given the Taylor expansion in \cite[Lemma~B.9]{GTZ12}, we have that $g(h_1,\ldots,h_k)$ is a polynomial sequence with respect to the degree filtration $G_0 = G_1\geqslant G_2 \geqslant \cdots \geqslant G_{s_1 + \cdots + s_k} \geqslant \mr{Id}_{G}$. The desired result follows immediately.

For the second item, notice that if a polynomial $P(x_1,\ldots,x_k)$ has total degree $s$, then for any linear maps $L_i\colon\mb{R}^{\ell} \to\mb{R}$ we have that  $P(L_1(y_1,\ldots,y_\ell),\ldots,L_k(y_1,\ldots,y_\ell))$ has total degree $s$. This coupled with Taylor expansion \cite[Lemma~B.9]{GTZ12} and the fact that the set of polynomial sequences with respect to a given $I$-filtration is a group (by \cite[Corollary~B.4]{GTZ12}) implies the result.

We now handle the first item; this is the only nontrivial part. Write $g(h^{\ast},0,\ldots,0) = \{g(h^{\ast},0,\ldots,0)\}[g(h^{\ast},0,\ldots,0)]$ with $\psi_{G,\mc{X}}(\{g(h^{\ast},0,\ldots,0)\})\in[0,1)^{\dim(G)}$ and $[g(h^{\ast},0,\ldots,0)]\in\Gamma$. We replace the polynomial sequence $g$ by $g' = \{g(h^{\ast},0,\ldots,0)\}^{-1}g[g(h^{\ast},0,\ldots,0)]^{-1}$ and $F$ by the function $F'(\cdot) = F(\{g((h^{\ast},0,\ldots,0))\} \cdot)$. We may thus assume, at the cost of replacing $M$ by $M^{O_{|\vec{s}|}(d^{O_{|\vec{s}|}(1)})}$, that $g(h^{\ast},0,\ldots,0) = \mr{id}_G$.

We now apply \cite[Lemma~B.9]{GTZ12} to see
\[g(h_1,\ldots,h_k) = \prod_{i_1,\ldots,i_k}g_{(i_1,\ldots,i_k)}^{\binom{h_1}{i_1}\cdots \binom{h_k}{i_k}}\]
where we order $(i_1,\ldots,i_k)$ lexicographically with indices considered in reverse order in the product (in particular, the first few terms are $(0,\ldots,0)$, $(1,\ldots,0)$, $(2,\ldots,0)$, and so on) and $g_{(i_1,\ldots,i_k)}\in G_{(i_1,\ldots,i_k)}$. As $g(h^{\ast},0,\ldots,0) = \mr{id}_G$, we have that 
\[g(h^{\ast},h_2,\ldots,h_k) = \prod_{\substack{i_2,\ldots,i_k\\ i_2 + \cdots + i_k>0}}g_{(i_1,\ldots,i_k)}^{\binom{h^\ast}{i_1}\cdot\binom{h_2}{i_2}\cdots \binom{h_k}{i_k}}.\]
It then follows that $g(h^{\ast},h_2,\ldots,h_k)$ is a polynomial with respect to
\[G^\ast = \bigvee_{\ell=2}^k G_{\vec{e}_{\ell}}\]
which we give a multidegree filtration 
$G^\ast_{(i_2,\ldots,i_k)} = G_{(0,i_2,\ldots,i_k)}$
for $i_2 + \cdots + i_k>0$ and $G^\ast_{(0,\ldots,0)} = G^\ast$. Note that all subgroups in this filtration are $M$-rational with respect to $\mc{X}$ and that $(G^\ast)_t = \bigvee_{|\vec{i}| = t}G^\ast_{\vec{i}}$ is a degree $|\vec{s}|-s_1$ filtration. Therefore applying \cite[Lemma~B.11]{Len23b} guarantees that we may find a Mal'cev basis $\mc{X}^\ast$ for $G^\ast$ (which is $M^{O_{|\vec{s}|}(d^{O_{|\vec{s}|}(1)})}$-rational with respect to $\mc{X}$). Descending $F$ to $G^\ast$ gives the desired result with the necessary Lipschitz bound following from \cite[Lemma~B.9]{Len23b}.
\end{proof}

We now state a quantified version of \cite[Lemma~E.8]{GTZ12}. Recall the notion of equivalence (Definition~\ref{def:equiv}).
\begin{lemma}\label{lem:nil-basic-prop}
Consider a nilcharacter $\chi$ with complexity $(M,d)$ of multidegree $\vec{s} = (s_1,\ldots,s_k)$ with $|\vec{s}| = s_1 + \cdots + s_k$. We have that:
\begin{itemize}
    \item The nilcharacters
    \[\chi(\cdot)\emph{ and } \chi(\cdot)\]
    are $(M^{O_{|\vec{s}|}(1)},M^{O_{|\vec{s}|}(1)},O(d))$-equivalent for multidegree $<(s_1,\ldots,s_k)$.\footnote{This means we take the down-set generated by $(s_1,\ldots,s_k)$ and then remove $(s_1,\ldots,s_k)$.}
    \item Fix $h^\ast\in\mb{Z}$. The nilcharacters
    \[\chi(\cdot + h^\ast\vec{e}_j) \emph{ and } \chi(\cdot)\]
    are $(M^{O_{|\vec{s}|}(d^{O_{|\vec{s}|}(1)})},M^{O_{|\vec{s}|}(d^{O_{|\vec{s}|}(1)})},O(d))$-equivalent for multidegree $<(s_1,\ldots,s_k)$.
    \item Fix $q\in\mb{Z}$. Then
    \[\chi^{\otimes q^{|\vec{s}|}}(\cdot) \emph{ and }\chi(q\cdot)\]
    are $(M^{O_{|\vec{s}|,q}(d^{O_{|\vec{s}|,q}(1)})},M^{O_{|\vec{s}|,q}(d^{O_{|\vec{s}|,q}(1)})}, d^{O_q(1)})$-equivalent for multidegree $<(s_1,\ldots,s_k)$.
    \item Fix $q\in\mb{Z}^{>0}$. There exists a nilcharacter $\wt{\chi}$ of complexity $(M^{O_{|\vec{s}|,q}(d^{O_{|\vec{s}|,q}(1)})},d^{O_q(1)})$ such that
    \[\chi(\cdot) \emph{ and } \wt{\chi}^{\otimes q}(\cdot)\]
    are $(M^{O_{|\vec{s}|,q}(d^{O_{|\vec{s}|,q}(1)})},M^{O_{|\vec{s}|,q}(d^{O_{|\vec{s}|,q}(1)})}, d^{O_q(1)})$-equivalent for multidegree $<(s_1,\ldots,s_k)$.
\end{itemize}  
\end{lemma}
\begin{remark*}
$\chi^{-\otimes q}$ for $q\in\mb{Z}^{>0}$ is interpreted as $\ol{\chi}^{\otimes q}$.
\end{remark*}
\begin{proof}
Throughout the proof, we let
\[\chi(\vec{n}) = F(g(\vec{n})\Gamma)\]
where the underlying nilmanifold is $G/\Gamma$ and the underlying Mal'cev basis is $\mc{X}$. When going from item to item, we may reuse variables (e.g., $G'$ will be defined in multiple different manners throughout the proof). Additionally, the following analysis implicitly uses that $|\vec{s}|\ge 1$; in the remaining case $\vec{s}=0$ all nilsequences become fixed constants and the result is obvious.

For the first item, note that coordinates of $\chi \otimes \ol{\chi}$ are multidegree $(s_1,\ldots,s_k)$ polynomial sequences with respect to group $G' = \{(g,g)\colon g\in G\}$ given the filtration 
\[G'_{\vec{i}} = \{(g,g)\colon g\in G_{\vec{i}}\}.\]
As all coordinates of $\chi$ have the same vertical frequency, we have that the coordinates of $\chi \otimes \ol{\chi}$ are invariant under $G'_{(s_1,\ldots,s_k)}$. This immediately gives the desired result upon taking a quotient and using Lemma~\ref{lem:quotient-rat}.

For the second item, note that $G^{+\vec{e}_j} = (G_{\vec{i}+\vec{e}_j})_{\vec{i}\in I}$ is a shifted filtration. Note that this is an $I$-filtration with respect to the multidegree ordering. We define the group
\[G' = \bigvee_{\ell=1}^k(G_{\vec{e}_\ell}\times\mr{Id}_{G})\vee \{(g,g)\colon g\in G\},\]
let $\Gamma' = G'\cap (\Gamma\times\Gamma)$, and define the following $I$-filtration with respect to the multidegree ordering:
\[G'_{\vec{i}} = \bigvee_{\ell=1}^k(G_{\vec{i}+\vec{e}_\ell}\times\mr{Id}_{G})\vee\{(g,g)\colon g\in G_{\vec{i}}\}.\]
By using Lemma~\ref{lem:com-check} we may see that this is a valid. We define the cocompact groups similarly. Now the proof of \cite[Lemma~E.8]{GTZ12} shows that  
\[(g(\vec{n} + h\vec{e}_j),g(\vec{n}))\]
is a polynomial sequence with respect to this filtration and that 
\[\wt{F}((x,y)(\Gamma\times\Gamma)) = F(x\Gamma)\otimes \ol{F(y\Gamma)}\]
is invariant under the action of $G_{\vec{s}}' = \{(g,g)\colon g\in G_{(s_1,\ldots,s_k)}\}$.

We first construct a Mal'cev basis $\mc{X}'$ on $G'$. Define
$G_{t}^\ast = \bigvee_{|\vec{i}| = t}G_{\vec{i} + \vec{e}_j}$ and note that $G^\ast = G_{0}^\ast$ has a degree filtration
\[G_{0}^\ast = G_{0}^\ast\geqslant G_{1}^\ast\geqslant G_{2}^\ast \geqslant \cdots\]
and all these subgroups are $M^{O_{|\vec{s}|}(1)}$-rational with respect to $\mc{X}$. Therefore $G^\ast$ has a Mal'cev basis $\mc{X}^\ast$ which is adapted to this filtration and all elements are height at most $M^{O_{|\vec{s}|}(d^{O_{|\vec{s}|}(1)})}$ combinations of elements in $\mc{X}$ by \cite[Lemma~B.11]{Len23b}. Note that
\[\{(X,X)\colon X\in\mc{X}\}\cup \{(X^{\ast},0)\colon X^\ast\in\mc{X}^\ast\}\]
is easily shown to be a weak basis of rationality $M^{O_{|\vec{s}|}(d^{O_{|\vec{s}|}(1)})}$ for $G'/\Gamma'$ and has the degree $O_{|\vec{s}|}(1)$ nesting property. Letting $G_{t}' = \bigvee_{|\vec{i}| = t}G_{\vec{i}}'$ we see that 
\[G' = G'\geqslant G_1'\geqslant G_2'\geqslant \cdots\]
form a sequence of subgroups such that $[G',G_i']\leqslant G_{i+1}'$ for $i\ge 0$ (with $G_0' = G'$). Thus by \cite[Lemma~B.11]{Len23b} we can find a Mal'cev basis $\mc{X}'$ adapted to this sequence such that each element is a height $M^{O_{|\vec{s}|}(d^{O_{|\vec{s}|}(1)})}$ linear combination of
\[\{(X,X)\colon X\in\mc{X}\}\cup \{(X^{\ast},0)\colon X^\ast\in\mc{X}^\ast\}.\] 

At present, however, we see that $G'$ has not been given a multidegree filtration (only an $I$-filtration with respect to the multidegree ordering; recall Definition~\ref{def:filt-prec}). We replace $G'$ by 
\[\wt{G} = \bigvee_{\ell=1}^k G'_{\vec{e}_\ell} = G_1'\]
and note that $\wt{G}$ is appropriately rational with respect to $\mc{X}'$ and $G'_{\vec{i}} \leqslant \wt{G}$ for $\vec{i}\neq 0$. Note that $\wt{G}$ is easily seen to have a multidegree $(s_1,\ldots,s_k)$ filtration. Furthermore, removing the initial $\dim(G')-\dim(\wt{G})$ elements, we see that the truncation of $\mc{X}'$ is valid Mal'cev basis for $\wt{G}$ of complexity $M^{O_{|\vec{s}|}(d^{O_{|\vec{s}|}(1)})}$ and all subgroups in the multidegree filtration are $M^{O_{|\vec{s}|}(d^{O_{|\vec{s}|}(1)})}$-rational.

We now write $(g(h^\ast\vec{e}_j),g(0)) = \{(g(h^\ast\vec{e}_j),g(0))\}[(g(h^\ast\vec{e}_j),g(0))]$ where
\[\snorm{\psi_{\mc{X}'}(\{(g(h^\ast\vec{e}_j),g(0))\})}_{\infty}\le M^{O_{|\vec{s}|}(d^{O_{|\vec{s}|}(1)})}\]
and $[(g(h^\ast\vec{e}_jj),g(0))]\in\Gamma'$. We consider the modified polynomial sequence
\[g'(\vec{n}) = \{(g(h^\ast\vec{e}_j),g(0))\}^{-1}(g(\vec{n} + h^\ast\vec{e}_j),g(\vec{n}))[(g(h^\ast\vec{e}_j),g(0))]^{-1};\]
evaluating at $\vec{n}=0$ this is now seen to be a polynomial sequence in $\wt{G}$. Defining
\[F'((x,y)(\Gamma\times\Gamma)) = \wt{F}(\{(g(h^\ast\vec{e}_j),g(0))\}(x,y)(\Gamma\times\Gamma)),\]
we have that $F'$ is invariant under $(\wt{G})_{(s_1,\ldots,s_k)}$ and 
\[F'(g'(\vec{n})) = F(g(\vec{n} + h^\ast\vec{e}_j))\ol{F}(g(\vec{n})).\]
We may pass to the quotient group $\wt{G}/\wt{G}_{(s_1,\ldots,s_k)}$ and the desired result is essentially an immediate consequence of Lemma~\ref{lem:quotient-rat}.

We now come to the third item; we only maintain the notation from the first sentence of the proof. Note that via writing $g(0) = \{g(0)\}[g(0)]$ with $[g(0)]\in\Gamma$ and $\psi_{\mc{X}}(\{g(0)\})\in[0,1)^{\dim(G)}$, replacing $g(\vec{n})$ by $\{g(0)\}^{-1}g(\vec{n})[g(0)]^{-1}$ and $F$ by $F(\{g(0)\} \cdot)$, up to replacing $M$ by $M^{O_{|\vec{s}|}(1)}$ we may assume that $g(0) = \mr{id}_G$.

Define
\[G'_{\vec{i}} = \bigvee_{\vec{j}>\vec{i}}(G_{\vec{j}}\times G_{\vec{j}})\vee\bigvee \{(g^{q^{|\vec{i}|}},g)\colon g\in G_{\vec{i}}\};\]
here $\vec{j}>\vec{i}$ means $\vec{j}$ is coordinate-wise at least as large as $\vec{i}$ and not identical. Furthermore $\Gamma' = G'\cap (\Gamma\times\Gamma)$; note that $G'$ is isomorphic to $G\times G$ however we have given the group an alternate filtration. This is verified to be an $I$-filtration with respect to the multidegree ordering in \cite[p.~1356]{GTZ12}. Furthermore the proof of \cite[Lemma~E.8]{GTZ12} shows that  
\[(g(q\vec{n}),g(\vec{n}))\]
is a polynomial sequence with respect to this filtration and that 
\[\wt{F}((x,y)(\Gamma\times\Gamma)) = F(x\Gamma)\otimes\ol{F(y\Gamma)}^{\otimes q^{|\vec{s}|}}\]
is invariant under the action of $G_{\vec{s}} = \{(g^{q^{|\vec{s}|}},g)\colon g\in G_{(s_1,\ldots,s_k)}\}$. The primary technical issue, as before, is that while this is an $I$--filtration with respect to the multidegree ordering this is not a multidegree filtration (Definition~\ref{def:filt-prec}).

We first give $G'$ a Mal'cev basis. Note that $\mc{X}$ is adapted to the degree filtration on $G$ given by $G_t = \bigvee_{|\vec{i}| = t}G_{\vec{i}}$ (Definition~\ref{def:complex-2}). It is immediate to see that 
\[\mc{X}^\ast = \{(X,0)\colon X\in\mc{X}\cap G_1\} \cup \{(0,X)\colon X\in\mc{X}\cap G_1\}\]
is a Mal'cev basis for the product filtration on $G$. Then using \cite[Lemma~B.11]{Len23b} on
\[G_0'= G_0'\geqslant G_1'\geqslant\cdots\]
where $G_t'=\bigvee_{|\vec{i}| = t}G_{\vec{i}}'$, which is seen to satisfy $[G_0',G_t']\leqslant G_{t+1}'$ for $t\ge 0$, we easily construct a Mal'cev basis $\mc{X}'$ for $G'/\Gamma'$ coming from combinations of $\mc{X}^\ast$. (We implicitly use that $G=\bigvee_j G_{\vec{e}_j}$.) As $G$ is has complexity $M$, it is trivial to see that all subgroups in the filtration of $G'$ are $M^{O_{|\vec{s}|,q}(d^{O_{|\vec{s}|,q}(1)})}$-rational. The Mal'cev basis $\mc{X}'$ clearly has the nesting property of order $|\vec{s}|$ since $\mc{X}$ does.

We define $\wt{G}$ as
\[\wt{G} = \bigvee_{\ell=1}^k G'_{\vec{e}_\ell}\]
and this group is seen to be appropriately rational with respect to $\mc{X}'$ and is given the multidegree filtration $\wt{G}_{\vec{i}} = G'_{\vec{i}}$ for $\vec{i}\neq 0$. Noting that the constant term of the Taylor expansion of $(g(q\vec{n}),g(\vec{n}))$ is $(\mr{id}_G,\mr{id}_G)$, we have that this is in fact a polynomial sequence with respect to the multidegree filtration given to $\wt{G}$. (This is where we use that we reduced to $g(0) = \mr{id}_{G}$.) Furthermore, letting $\wt{G}_{t} = \bigvee_{|\vec{i}| = t} \wt{G}_{\vec{i}}$, we see that a truncation of $\mc{X}'$ is an adapted Mal'cev basis to $\wt{G}_0 = \wt{G}_1\geqslant \wt{G}_2 \geqslant \cdots $ where each element is an $M^{O_{|\vec{s}|,q}(d^{O_{|\vec{s}|,q}(1)})}$-rational combination of $\mc{X}^\ast$. As $\wt{F}$ is invariant under $\wt{G}_{(s_1,\ldots,s_k)}$, by passing to the quotient $\wt{G}/\wt{G}_{(s_1,\ldots,s_k)}$ and applying Lemma~\ref{lem:quotient-rat} we immediately finishe the proof.

We finally deduce the fourth item from the third item. Let $g'(n) = g(n/q)$; note that $g$ may be extended to take on rational input via using Mal'cev coordinates and we may treat $g'$ as a valid polynomial sequence. By applying the third item, we have that 
\[F(g(n))\text{ and } F(g'(n))^{\otimes q^{|\vec{s}|}}\]
are $(M^{O_{|\vec{s}|,q}(d^{O_{|\vec{s}|,q}(1)})},M^{O_{|\vec{s}|,q}(d^{O_{|\vec{s}|,q}(1)})}, d^{O_q(1)})$-equivalent for multidegree $<(s_1,\ldots,s_k)$. Outputting $F(g'(n))^{\otimes q^{|\vec{s}|-1}}$ then gives the desired result.
\end{proof}

The next lemma is a quantified version of \cite[Lemma~13.2]{GTZ12}. The proof is once again essentially identical modulo noting slight changes in the filtration notions.
\begin{lemma}\label{lem:multi}
Consider $\chi\colon\mb{Z}^{k}\to\mb{C}$ which is a multidegree $(1,\ldots,1)$ nilcharacter of complexity $(M,d)$. Then 
\[\chi(h_1 + h_1', h_2, \ldots, h_k)\emph{ and }\chi(h_1, h_2, \ldots, h_k) \otimes \chi(h_1', h_2, \ldots, h_k)\]
are $(M^{O_k(d^{O_k(1)})},d^{O_k(1)})$-equivalent for degree $(k-1)$.
\end{lemma}
\begin{proof}
Let $\chi(h_1,\ldots,h_k) = F(g(h_1,\ldots,h_k)\Gamma)$ where the underlying nilmanifold is $G/\Gamma$ and the underlying Mal'cev basis is $\mc{X}$. Let $g(0,\ldots,0) = \{g(0,\ldots,0)\}[g(0,\ldots,0)]$ with $[g(0,\ldots,0)]\in\Gamma$ and $\psi_{G,\mc{X}}(\{g(0,\ldots,0)\})\in[0,1)^{\dim(G)}$. We have that 
\[F(g(h_1,\ldots,h_k)\Gamma) = F(\{g(0,\ldots,0)\} \cdot (\{g(0,\ldots,0)\}^{-1} g(h_1,\ldots,h_k)[g(0,\ldots,0)]^{-1}\Gamma)).\]
We let $F'(\cdot\Gamma) = F(\{g(0,\ldots,0)\} \cdot\Gamma)$ and $g'(h_1,\ldots,h_k) = (\{g(0,\ldots,0)\}^{-1} g(h_1,\ldots,h_k)[g(0,\ldots,0)]^{-1}$. Thus may assume replace $F$ by $F'$ and $g$ by $g'$ (at the cost of replacing $M$ by $M^{O_k(d^{O_k(1)})}$) and assume that $g(0) = \mr{id}_{G}$. 

Consider, for $t\ge 1$,
\[G'_t = \bigvee_{|\vec{i}|>t}(G_{\vec{i}}\times G_{\vec{i}}\times G_{\vec{i}}) \vee \bigvee_{|\vec{i}| = t}\{(g_1g_2,g_1,g_2):g_i\in G_{\vec{i}}\}\]
and take $G' = G'_1$. Via Baker--Campbell--Hausdorff this gives a valid degree $k$ filtration
\[G' = G'_1\geqslant G'_2\geqslant \cdots \geqslant G'_k\geqslant \mr{Id}_{G'}.\]
We define $\Gamma' = G'\cap (\Gamma\times\Gamma\times\Gamma)$. We now verify that 
\[(g(h_1 + h_1', h_2, \ldots, h_k), g(h_1, h_2, \ldots, h_k),g(h_1', h_2, \ldots, h_k))\]
is a polynomial sequence with respect to this degree filtration. This is immediate nothing that by Taylor expansion \cite[Lemma~B.9]{GTZ12} and the condition at $0$, we have 
\[g(h_1,\ldots,h_k) = \prod_{(i_1,\ldots,i_k)\in\{0,1\}^k\setminus \{\vec{0}\}} g_{i_1,\ldots,i_k}^{\binom{h_1}{i_1} \cdots \binom{h_k}{i_k}}\]
with $g_{i_1,\ldots,i_k}\in G_{(i_1,\ldots,i_k)}$. The desired polynomiality of the tripled sequence then follows easily from Baker--Campbell--Hausdorff and the fact that the degree of the exponents in the Taylor expansion for the $h_1$ term is at most $1$.

We now construct a Mal'cev basis for $G'$. Let $G_t = \bigvee_{|\vec{i}|=t}G_{\vec{i}}$ and note that by definition $\mc{X}$ is adapted to $G_t$. We may prove that
\begin{align*}
&\mc{X}^{'} = \{(X,0,0)\colon X\in\mc{X}\cap\log(G_2) \} \cup \{(0,X,0)\colon X\in\mc{X}\cap\log(G_2)\} \cup \{(0,0,X):X\in\mc{X}\cap\log(G_2)\} \\
&\cup\{(X,X,0)\colon X\in \mc{X}\cap\log(G_1)\setminus\mc{X}\cap\log(G_2)\}\cup \{(X,0,X)\colon X\in \mc{X}\cap\log(G_1)\setminus\mc{X}\cap\log(G_2)\}
\end{align*}
is a weak basis for $G'$. Furthermore this basis is easily seen to have the nesting property of order $k$ and that all subgroups $G_t'$ are $M^{O_k(d^{O_k(1)})}$-rational. Thus applying \cite[Lemma~B.11]{Len23b} we may find a Mal'cev basis $\wt{\mc{X}}$ for $G'$ adapted to the given filtration of complexity $M^{O_k(d^{O_k(1)})}$ and such that all basis elements are $M^{O_k(d^{O_k(1)})}$-rational combinations of $\mc{X}'$.

The function $\wt{F}$ we will consider is 
\[\wt{F}((x,y,z)\Gamma^{\otimes 3}) = F(x\Gamma) \otimes \ol{F(y\Gamma)} \otimes \ol{F(z\Gamma)}.\]
This is easily seen to be Lipschitz on $G\times G\times G$ when given the Mal'cev basis $\mc{X}^\ast = \{(X,0,0)\colon X\in\mc{X}\} \cup \{(0,X,0)\colon X\in\mc{X}\} \cup \{(0,0,X)\colon X\in\mc{X}\}$. As $\wt{\mc{X}}$ has basis elements which are $M^{O_k(d^{O_k(1)})}$ height rational combination of $\mc{X}^\ast$, we find that $\wt{F}$ is $M^{O_k(d^{O_k(1)})}$-Lipschitz on $G'/\Gamma'$.

Note that $\wt{F}$ is invariant under the group $G_k^{3}$ and therefore taking the output quotient group $G^{3}/G_k^{3}$ with lattice $\Gamma^{3}/(\Gamma^{3}\cap G_k^{3})$ and using Lemma~\ref{lem:quotient-rat} completes the proof.
\end{proof}

We now come to the most technical of the complexity justifications we will need to perform, multilinearization. We will give a rather barebones analysis (citing much from \cite[Proposition~E.9,~E.10]{GTZ12}); the reader may find the discussion in \cite[pp.~1360-1363]{GTZ12} where an extended example is discussed useful. (We prove a slightly weaker statement which is all that is used in the analysis to ease checking extra complexity details.)

\begin{lemma}\label{lem:multilinear}
Consider nilcharacter $\chi(h_1,\ldots,h_k)$ of multidegree $(s_1,\ldots,s_k)$ and complexity $(M,d)$. There exists a multidegree $(1,\ldots,1)$ nilcharacter
\[\chi'(h_{1,1},\ldots,h_{1,s_1},h_{2,1},\ldots,h_{2,s_2},\ldots,h_{k,1},\ldots,h_{k,s_k})\]
of complexity $(M^{O_{|\vec{s}|}(d^{O_{|\vec{s}|}(1)})},d^{O_{|\vec{s}|}(1)})$ such that
\[\chi(h_1,\ldots,h_k) \emph{ and } \chi'(h_1,\ldots,h_1,h_2,\ldots,h_2,\ldots,h_k,\ldots,h_k)\]
are $(M^{O_{|\vec{s}|}(d^{O_{|\vec{s}|}(1)})}, M^{O_{|\vec{s}|}(d^{O_{|\vec{s}|}(1)})},d^{O_{|\vec{s}|}(1)})$-equivalent for degree $|\vec{s}| - 1$ and furthermore, for each $1\le i\le k$, $\chi'$ is symmetric in the variables $h_{i,1},\ldots,h_{i,s_i}$.
\end{lemma}
\begin{remark*}
We will only require the above lemma for multidegree $(1,s-1)$ nilsequences. 
\end{remark*}
\begin{proof}
By Lemma~\ref{lem:nil-basic-prop}, there exists $\chi^\ast$ such that 
\[\chi(h_1,\ldots,h_k) \text{ and }\chi^\ast(h_1,\ldots,h_k)^{\otimes \prod_{i=1}^{k}s_i!}\]
are $(M^{O_{|\vec{s}|}(d^{O_{|\vec{s}|}(1)})}, M^{O_{|\vec{s}|}(d^{O_{|\vec{s}|}(1)})},d^{O_{|\vec{s}|}(1)})$-equivalent for degree $|\vec{s}| - 1$. Therefore by Lemma~\ref{lem:equiv}, it suffices to produce $\chi'$ such that 
\[\chi^\ast(h_1,\ldots,h_k)^{\otimes \prod_{i=1}^{k}s_i!} \text{ and } \chi'(h_1,\ldots,h_1,h_2,\ldots,h_2,\ldots,h_k,\ldots,h_k)\]
are $(M^{O_{|\vec{s}|}(d^{O_{|\vec{s}|}(1)})}, M^{O_{|\vec{s}|}(d^{O_{|\vec{s}|}(1)})},d^{O_{|\vec{s}|}(1)})$-equivalent for degree $|\vec{s}|-1$.

Let
\[\chi'(h_1,\ldots,h_k) = F(g(h_1,\ldots,h_k)\Gamma)\]
with the underlying nilmanifold being $G/\Gamma$ and the associated Mal'cev basis being $\mc{X}$. Via a standard manipulation which has been perform several times already, we may assume that $g(0) = \mr{id}_G$ (at the cost of an insignificant change in parameters). Furthermore assume that $\eta$ is the vertical character, so
\[F(g_{(s_1,\ldots,s_k)} x \Gamma) = e(\eta(g_{(s_1,\ldots,s_k)})) \cdot F(x\Gamma).\]

Given $J\subseteq [|\vec{s}|]$, we denote 
\[\|J\| := (|J \cap\{s_1 + \cdots + s_{i - 1} + 1, \ldots, s_1 + \cdots  + s_{i - 1} + s_i\}|)_{1 \le i \le k}.\]
The group $\wt{G}$ we will ultimately use to construct our nilsequence will be given by constructing the associated nilpotent Lie algebra. We take 
\[\log(\wt{G}) = \bigoplus_{\emptyset \neq J\subseteq [|\vec{s}|]} \log(G_{\|J\|})\]
and for each $\emptyset \neq J\subseteq [|\vec{s}|]$ let $\iota_J\colon\log(G_{\|J\|})\hookrightarrow\log(\wt{G})$ denote the embedding into the direct sum. We endow $\wt{G}$ with a Lie bracket such that if $J\cap K \neq \emptyset$ then
\[[\iota_J(x_J),\iota_K(y_K)] = 0\]
and if $J\cap K = \emptyset$ then 
\[[\iota_J(x_J),\iota_K(y_K)] = \iota_{J\cup K}([x_J,y_K]),\]
where the bracket between $x_J,x_K$ is taken in the ambient space $\log(G)$ and is seen to lie in $\log(G_{J\cup K})$ by the commutator property of the original filtration on $G$.

To verify that this gives a valid Lie algebra it suffices to verify this operation is antisymmetric and satisfies the Jacobi relations. Furthermore to verify it suffices to verify these relations on the generators. For antisymmetry for $\iota_J(x_J),\iota_K(y_K)$, if $J\cap K \neq \emptyset$ it is trivial and otherwise 
\[[\iota_J(x_J),\iota_K(y_K)] = \iota_{J\cup K}([x_J,y_K]) = -\iota_{J\cup K}([y_J,x_K]) = -[\iota_K(y_K),\iota_J(x_J)]\]
as desired. For the Jacobi identity, when checked on generators $\iota_J(x_J),\iota_K(y_K),\iota_L(z_L)$, if $(J\cap K) \cup (K\cap L) \cup (L\cap K)\neq \emptyset$ the result is trivial. Otherwise we have 
\begin{align*}
&[\iota_J(x_J),[\iota_K(y_K),\iota_L(z_L)]] + [\iota_K(y_K),[\iota_L(z_L),\iota_J(x_J)]] + [\iota_L(z_L),[\iota_J(x_J),\iota_K(y_K)]]\\
& = \iota_{J\cup K \cup L}([x_J,[y_K,z_L]] + [y_K,[z_L,x_J]] + [z_L,[x_J,y_K]])  = 0
\end{align*}
as desired. 

The associated $I$-filtration with respect to the multidegree ordering is given as follows. For any $(a_1, \ldots, a_{|\vec{s}|}) \in\mb{N}^{|\vec{s}|}$, let $\log(\wt{G}_{(a_1, \ldots, a_{|\vec{s}|})})$ be the Lie subalgebra of $\log(\wt{G})$ generated by $\iota_J(x_J)$ for which $1_J(j) \ge a_j$ for each $j = 1, \ldots, |\vec{s}|$, and $x_J \in G_{\|J\|}$. It follows this is an $I$-filtration with respect to the multidegree ordering because for vectors $a,b\in\{0,1\}^{|\vec{s}|}$, if $1_J(j) \ge a_j$ and $1_K(j)\ge b_j$ we either have $J\cap K\neq\emptyset$ in which case the commutator is trivial or $1_{K\cup J}(j)\ge a_j+b_j$ in which case the result also follows easily. Noting that by construction $\wt{G} = \bigvee_{\ell=1}^{|\vec{s}|}\wt{G}_{\vec{e}_{\ell}}$, the above immediately implies that we have a multidegree $(1,\ldots,1)$ filtration on $\wt{G}$. 

We now construct a weak basis for $\wt{G}$. Recall we have a Mal'cev basis $\mc{X}$ for $G$. Given $\|J\|$, we define the filtration 
\[G_{\|J\|}^{t} = \bigvee_{|\vec{i}| = t}G_{\|J\| + \vec{i}}.\]
Note that $G_{\|J\|} = G_{\|J\|}^{0}$ and that $G_{\|J\|}^{0} = G_{\|J\|}^{0} \geqslant G_{\|J\|}^{1} \geqslant G_{\|J\|}^{2} \geqslant \cdots$ is a valid degree filtration when $\|J\|\neq\vec{0}$. Thus by \cite[Lemma~B.11]{GTZ12}, we may find a Mal'cev basis $\mc{X}^{\|J\|}$ for each $G_{\|J\|} $ which is an $M^{O_{|\vec{s}|}(d^{O_{|\vec{s}|}(1)})}$-rational combination of $\mc{X}$.

Define
\[\wt{\mc{X}} = \bigcup_{\emptyset \neq J\subseteq [|\vec{s}|]}\iota_J(\mc{X}^{\|J\|}).\]
Furthermore define $\wt{\Gamma}$ to be the group generated by $\exp(L! \cdot\iota_J(\Gamma\cap G_{\|J\|}))$ where $L$ is a sufficiently large constant depending only on $|\vec{s}|$ (and in particular not on $M$ or $d$). Direct computation with Baker--Campbell--Hausdorff implies that $\wt{\Gamma}\cap\wt{G}_{(1,\ldots,1)}$ is contained in $\iota_{[|\vec{s}|]}(\Gamma\cap G_{(s_1,\ldots,s_k)})$. Furthermore we see that $\wt{G}/\wt{\Gamma}$ is compact, $\wt{\mc{X}}$ is a weak basis of rationality $M^{O_{|\vec{s}|}(d^{O_{|\vec{s}|}(1)})}$ for $\wt{G}$, and $\wt{\mc{X}}$ has the degree $O_{|\vec{s}|}(1)$ nesting property. As all groups within the multidegree filtration are $M^{O_{|\vec{s}|}(d^{O_{|\vec{s}|}(1)})}$-rational with respect to $\wt{\mc{X}}$, by applying \cite[Lemma~B.11]{Len23b} we may construct a basis with respect to the canonical associated degree filtration of $\wt{G}$ which certifies that $\wt{G}$ with the given multidegree filtration has complexity bounded by $M^{O_{|\vec{s}|}(d^{O_{|\vec{s}|}(1)})}$. Furthermore the adapted Mal'cev basis $\wt{\mc{X}}^\ast$ is an $M^{O_{|\vec{s}|}(d^{O_{|\vec{s}|}(1)})}$-rational combination of $\wt{\mc{X}}$ (lifted to $\log(\wt{G})$ appropriately).

We define the $\wt{G}_{(1, \ldots, 1)}$-vertical frequency as 
\[\wt{\eta}(\exp(\iota_{(1, \ldots, 1)}(\log(g_{(s_1,\ldots,s_k)})))) := \eta(g_{(s_1,\ldots,s_k)})\]
and it is trivial to use the construction of $\wt{\mc{X}}^\ast$ to certify that $\wt{\eta}$ has height bounded by $M^{O_{|\vec{s}|}(d^{O_{|\vec{s}|}(1)})}$. We take $\wt{F}$ to be a nilcharacter with frequency $\wt{\eta}$ produced by Lemma~\ref{lem:nil-exist} (which is applied to the canonical degree filtration of $\wt{G}$) and this construction gives output dimension $M^{O_{|\vec{s}|}(d)}$ and Lipschitz constant $M^{O_{|\vec{s}|}(d^{O_{|\vec{s}|}(1)})}$ with respect to $\wt{\mc{X}}^\ast$.

We now define $\wt{g}$. Note that
\[g(h_1,\ldots,h_k) = \prod_{\vec{0}\neq (i_1,\ldots,i_k)\le (s_1,\ldots,s_k)} (g_{(i_1,\ldots,i_k)}^{\mr{Tay}})^{h_1^{i_1}\cdots h_k^{i_k}}\]
via \cite[Lemma~B.9]{GTZ12} and the condition at $0$ to rule out need a coefficient where $(i_1,\ldots,i_k)=\vec{0}$. (We are using monomials instead of binomials, which is a minor but easy alteration.) The product here is taken in increasing lexicographic order. We define 
\[\wt{g}(h_1, \ldots, h_{|\vec{s}|}) := \prod_{\vec{0}\neq (i_1,\ldots,i_k)\le (s_1,\ldots,s_k)}  \exp\bigg(i_1!\cdots i_k! \sum_{\substack{J \subseteq \{1, \ldots, |\vec{s}|\}\\ \|J\| = (i_1,\ldots,i_k)}} \big(\prod_{i \in J} h_i\big)\cdot\iota_J(\log(g_{(i_1,\ldots,i_k)}^{\mr{Tay}}))\bigg).\]

Let $G^\ast$ denote the subgroup of $G\times\wt{G}$ generated by 
\[G^{\ast}:= \{(g_{(s_1,\ldots,s_k)}, \exp(s_1!\cdots s_k!\iota_{(1, \ldots, 1)}(\log(g_{(s_1,\ldots,s_k)}))))\colon g_{(s_1,\ldots,s_k)} \in G_{(s_1,\ldots,s_k)}\}.\]
Note that the function
\[(g,\wt{g})\mapsto F(g\Gamma)^{\otimes s_1!\cdots s_k!} \otimes \ol{\wt{F}(\wt{g}\wt{\Gamma})}\]
is invariant under the action of $G^\ast$.

We will construct $G'$ which is a subgroup of $G\times\wt{G}$ with a degree $|\vec{s}|$ filtration such that the final group is $G^\ast$ and such that 
\[(g(h_1,\ldots,h_k),\wt{g}(h_1,\ldots,h_1,h_2,\ldots,h_2,\ldots,h_k,\ldots,h_k))\]
is a polynomial sequence with respect to this filtration. Let $G_j'$ (for $j\ge 1$) be generated by elements of the form
\begin{equation}\label{eq:splitting-1}
\bigg(g_{(i_1,\ldots,i_k)}, \exp\bigg(i_1!\cdots i_k! \sum_{\substack{J \subseteq \{1, \ldots, |\vec{s}|\}\\ \|J\| = (i_1,\ldots,i_k)}} \iota_J(\log(g_{(i_1,\ldots,i_k)}))\bigg)\bigg)
\end{equation}
where $|\vec{i}| = j$, as well as
\[\big(g_{(i_1,\ldots,i_k)}, \mr{id}_{\wt{G}}\big) ,\text{ and }\big(\mr{id}_{G}, \wt{G}_{J}\big) \]
where $|\vec{i}|\ge j+1$ in the first case, and $|J|\ge j+1$ in the second case. Furthermore set $G'=G_0':= G_1'$; we trivially see that $G_{|\vec{s}|} = G^\ast$. That this is a filtration follows from liberal application of Baker--Campbell--Hausdorff; we use crucially that the number of ways to break a set of size $(i+j)$ into two labeled sets of size $i$ and $j$ which are disjoint is $(i+j)!/(i! \cdot j!)$, which modifies the factorial prefactors in \eqref{eq:splitting-1} appropriately.

Furthermore it is trivial to see that the $G_j'$ are $M^{O_{|\vec{s}|}(d^{O_{|\vec{s}|}(1)})}$-rational with respect to the Mal'cev basis for $G\times\wt{G}$ given by
\[\{(X,0)\colon X\in\mc{X}\} \cup \{(0,X)\colon X\in\wt{\mc{X}}^\ast\},\]
and therefore applying \cite[Lemma~B.11]{Len23b} we may construct a Mal'cev basis $\mc{X}'$ of complexity $M^{O_{|\vec{s}|}(d^{O_{|\vec{s}|}(1)})}$ for $G'/(G' \cap (\Gamma\times\wt{\Gamma}))$. Furthermore $F^{\otimes s_1!\cdots s_k!} \otimes \ol{\wt{F}}$ is appropriately Lipschitz with respect to $\mc{X}'$. Finally, since 
\[\bigg(g_{(i_1,\ldots,i_k)}, \exp\big(i_1!\cdots i_k! \sum_{\substack{J\subset \{1,\ldots, |\vec{s}|\\\|J\| = (i_1,\ldots,i_k)}} \iota_J(\log(g_{(i_1,\ldots,i_k)}))\big)\bigg)\]
is in $G_{i_1+\cdots+i_k}'$ by definition, we see that
\[(g(h_1,\ldots,h_k),\wt{g}(h_1,\ldots,h_1,h_2,\ldots,h_2,\ldots,h_k,\ldots,h_k))\]
is a polynomial sequence with respect to the filtration. Quotienting out by $G^\ast = G_{|\vec{s}|}'$ (using that $F^{\otimes s_1!\cdots s_k!} \otimes \ol{\wt{F}}$ is invariant under $G^\ast$) and using Lemma~\ref{lem:quotient-rat}, we finally complete the proof.
\end{proof}

We now reach the final technical lemma of the paper which states that a nilsequence of multidegree $J\cup J'$ can be approximated by a sum of products of nilsequences in $J$ and $J'$. This ``splitting'' lemma is a quantified version of \cite[Lemma~E.4]{GTZ12}; the proof here is ever so slightly different as we are forced to not use the Stone--Weierstrass theorem. 
\begin{lemma}\label{lem:split}
Let $J$ and $J'$ be finite downsets in $\mb{N}^k$ and fix $\eps \in (0,1/2)$. Suppose that $\beta(h_1,\ldots,h_k)$ is a nilsequence of multidegree $J\cup J'$ with complexity $(M,d)$. Then there exists $1\le L\le (M/\eps)^{O_{J,J'}(d^{O_{J,J'}(1)})}$ such that 
\[\norm{\beta(h_1,\ldots,h_k) - \sum_{j=1}^{L}\beta_j(h_1,\ldots,h_k)\beta_j'(h_1,\ldots,h_k)}_{L^{\infty}(\mb{Z}^k)}\le\eps\]
with the $\beta_j$ being nilsequences of multidegree $J$, the $\beta_j'$ being  nilsequences of multidegree $J'$, and $\beta_j,\beta_j'$ having complexity $((M/\eps)^{O_{J,J'}(d^{O_{J,J'}(1)})},d^{O_{J,J'}(1)})$.
\end{lemma}
\begin{proof}
We let 
\[\beta(h_1,\ldots,h_k) = F(g(h_1,\ldots,h_k)\Gamma)\]
where the underlying nilmanifold is $G/\Gamma$. As is standard, we may assume that $g(0,\ldots,0)=\mr{id}_G$ up to the insignificant change of adjusting $M$ to $M^{O_{J,J'}(d^{O_{J,J'}(1)})}$. Furthermore let the adapted Mal'cev basis for $G$ be $\mc{X}$. 

We have for each $\vec{j}\neq\vec{0}$ that the groups
\[G_{t}^{\vec{j}} = \bigvee_{|\vec{i}| = t}G_{\vec{j}+\vec{i}}\]
form a degree filtration $G_{0}^{\vec{j}} = G_{0}^{\vec{j}} \geqslant G_{1}^{\vec{j}}\geqslant G_{2}^{\vec{j}} \geqslant \cdots$, where the length of the filtration is $O_{J,J'}(1)$. As these subgroups are all $M$-rational with respect to $\mc{X}$, there exists a Mal'cev basis $\mc{X}^{\vec{j}}$ adapted to this filtration of complexity $M^{O_{J,J'}(d^{O_{J,J'}(1)})}$ where each element is an $M^{O_{J,J'}(d^{O_{J,J'}(1)})}$-rational combination of elements in $\mc{X}$ by \cite[Lemma~B.11]{Len23b}.

Using a variant of Lemma~\ref{lem:Taylor-mod}, adapted to multidegree filtrations, we may write
\[g(h_1,\ldots,h_k) = \prod_{\vec{j}\neq \vec{0}}\prod_{X_{\vec{j},i}\in\mc{X}^{\vec{j}}}\exp(X_{\vec{j},i})^{\alpha_{\vec{j},i}\prod_{\ell=1}^{k}(h_{\ell}^{j_{\ell}}/j_{\ell}!)}.\]
The product here is taken in $\vec{j}$ is increasing $|\vec{j}|$ and then lexicographic order and $X_{\vec{j},i}$ taken in increasing order of $i$. The modified proof of such a representation involves iteratively handling terms in increasing order of $|\vec{j}|$ (and handling these terms in an arbitrary order); we omit a careful proof. 

The first key part of the proof is lifting to the universal nilmanifold. We define the universal nilmanifold $\wt{G}$ to be generated by generators $\exp(e_{\vec{j},i})^{t_{\vec{j},i}}$ for $\vec{j}\neq \vec{0}$, $1\le i\le\dim(G_{\vec{j}})$, and $t_{\vec{j},i}\in\mb{R}$. The only relations these generators satisfy is that any $(r-1)$-fold commutator for $r\ge 1$ between $\exp(e_{\vec{j_1},i_1}),\ldots,\exp(e_{\vec{j_r},i_r})$ vanishes if $\vec{j_1} + \cdots + \vec{j_r}$ is not in $J\cup J'$. We give $\wt{G}$ the structure of a multidegree $J\cup J'$ nilmanifold by letting $(\wt{G})_{\vec{j}^\ast}$ be generated by the set of $(r-1)$-fold commutators (for any $r\ge 1$) of $\exp(e_{\vec{j_1},i_1}),\ldots,\exp(e_{\vec{j_r},i_r})$ where $\vec{j_1} + \cdots + \vec{j_r}\geq \vec{j}^\ast$ (here $\geq$ means that each coordinate is larger). This is easily proven to be an $I$-filtration with respect to the multidegree ordering and note that since we have no generators with $\vec{j} = \vec{0}$, this is in fact a multidegree filtration. Finally we let $\wt{\Gamma}$ be the lattice generated by $\exp(e_{\vec{j},i})$.

The analysis in Lemma~\ref{lem:universal-complexity} can easily be extended to prove that $\wt{G}$ has a filtered Mal'cev basis $\wt{X}$ of complexity $M^{O_{J,J'}(d^{O_{J,J'}(1)})}$ where the basis elements are height $M^{O_{J,J'}(d^{O_{J,J'}(1)})}$ linear combinations of $(r-1)$-fold commutators of $e_{\vec{j_1},i_1},\ldots,e_{\vec{j_r},i_r}$. Furthermore note that the dimension of $\wt{G}$ is $d^{O_{J,J'}(1)}$. 

We now lift $\beta$ to $\wt{G}/\wt{\Gamma}$. Define the homomorphism $\phi\colon\wt{G}\to G$ via 
\[\phi(\exp(e_{\vec{j},i})) = \exp(X_{\vec{j},i});\]
here we are writing $\mc{X}^{\vec{j}} = \{X_{\vec{j},1},\ldots,X_{\vec{j},\dim(G_{\vec{j}})}\}$. That this is a homormorphism follows from noting that all relations in $\wt{G}$ are present in $G$ because $G$ has multidegree $J\cup J'$. We next lift the polynomial sequence $g$ to
\[\wt{g}(h_1,\ldots,h_k) = \prod_{\vec{j}\neq \vec{0}}\prod_{i=1}^{\dim(G_{\vec{j}})}\exp(e_{\vec{j},i})^{\alpha_{i,\vec{j}}\prod_{\ell=1}^{k}(h_{\ell}^{j_{\ell}}/j_{\ell}!)}\]
and $F$ to $\wt{F}$ via
\[\wt{F}(\wt{g}\wt{\Gamma}) = F(\phi(\wt{g})\Gamma).\]
Note that since $\phi(\wt{\Gamma})\leqslant\Gamma$, this is a well-defined function on $\wt{G}/\wt{\Gamma}$. Furthermore, noting various properties of $\wt{X}$ and that elements of $\mc{X}^{\vec{j}}$ are appropriately bounded and rational linear combinations of $\mc{X}$, we have that $\wt{F}$ is $M^{O_{J,J'}(d^{O_{J,J'}(1)})}$-Lipschitz with respect to the Mal'cev basis specified by $\wt{\mc{X}}$. Therefore for the remainder of the proof we operate with the nilsequence
\[\wt{F}(\wt{g}(h_1,\ldots,h_k)\wt{\Gamma}).\] 

For the remainder of the analysis we furthermore assume that there exists $g^\ast\in\wt{G}$ with $d_{\wt{G},\wt{\mc{X}}}(g^\ast)\le M^{O_{J,J'}(d^{O_{J,J'}(1)})}$ such that if $\psi_{\mr{exp},\wt{G}}(g^\ast) + (-1/2,1/2]^{\dim(G)}$ is identified with $\wt{G}/\wt{\Gamma}$ then $\on{supp}(\wt{F})$ lies in 
$\psi_{\mr{exp},\wt{G}}(g^\ast) + (-\delta,\delta]^{\dim(G)}$. We will ultimately take $\delta = M^{-O_{J,J'}(d^{O_{J,J'}(1)})}$ sufficiently small. If we prove the proposition with $\eps' = \eps \cdot\delta^{O_{J,J'}(d^{O_{J,J'}(1)})}$ for functions with such restricted support then the result in generality follows by Lemma~\ref{lem:nilmanifold-partition-of-unity}.

The use of the universal nilmanifold comes precisely when defining the following two nilmanifolds for the split terms. Let $\wt{G}_{>J}$ be the group generated by $\wt{G}_{\vec{i}}$ with $\vec{i}\in J'\setminus J$ and $\wt{G}_{>J'}$ be the group generated by $\wt{G}_{\vec{i}}$ with $\vec{i}\in J\setminus J'$. It is trivial to see that $\wt{G}_{>J},\wt{G}_{>J'}$ are normal and $M^{O_{J,J'}(d^{O_{J,J'}(1)})}$-rational with respect to $\wt{\mc{X}}$. Let $\wt{\mc{X}}^{J}$ and $\wt{\mc{X}}^{J'}$ be bases for the Lie algebras of $\log(\wt{G}_{>J})$ and $\log(\wt{G}_{>J'})$ which are $M^{O_{J,J'}(d^{O_{J,J'}(1)})}$-rational bounded combinations of $\wt{\mc{X}}$.

We consider the nilmanifolds $\wt{G}/(\wt{G}_{>J}\wt{\Gamma})$ and $\wt{G}/(\wt{G}_{>J'}\wt{\Gamma})$. The first is clearly a multidegree $J'$ nilmanifold while the second is a multidegree $J$ nilmanifold, each of complexity $M^{O_{J,J'}(d^{O_{J,J'}(1)})}$. Furthermore we can choose underlying Mal'cev bases which are $M^{O_{J,J'}(d^{O_{J,J'}(1)})}$-rational combinations of $\wt{\mc{X}} \imod \wt{\mc{X}}^{J}$ and $\wt{\mc{X}} \imod \wt{\mc{X}}^{J'}$, respectively. (See, e.g., the arguments regarding $G_{\mr{Quot}}$ in Section~\ref{sub:quot-def}.) Note here that $\wt{\Gamma}_{>J} = \wt{\Gamma}/(\wt{\Gamma}\cap\wt{G}_{>J})$ and analogously for $\wt{\Gamma}_{>J'}$.

The key point is that by construction, $\wt{G}_{>J}\cap\wt{G}_{>J'}=\mr{Id}_{\wt{G}}$. This implies that there exist linear maps $A$ and $B$ such that
\begin{equation}\label{eq:splitting-2}
A\circ \psi_{\exp,\wt{G}/\wt{G}_{>J}}(z \imod \wt{G}_{>J}) + B\circ \psi_{\exp,\wt{G}/\wt{G}_{>J'}}(z \imod \wt{G}_{>J'}) = \psi_{\exp,\wt{G}}(z)
\end{equation}
for all $z\in\wt{G}$. Furthermore one can take $A$ and $B$ bounded in the sense that 
\[d_{\wt{G}}(A\circ \psi_{\exp,\wt{G}/\wt{G}_{>J}}(\exp(\wt{X}_i) \imod \wt{G}_{>J}),\mr{id}_{\wt{G}})\le M^{O_{J,J'}(d^{O_{J,J'}(1)})}\]
for all $\wt{X}_i\in\wt{\mc{X}}$ and analogously for $B$ and $J'$.

We now identify $\wt{G}/\wt{\Gamma}$ via $\psi_{\wt{G}}$ with the domain $\psi_{\mr{exp},\wt{G}}(g^\ast) + (-1/2,1/2]^{\dim(\wt{G})}$ and we only have support of $\wt{F}$ in $\psi_{\mr{exp},\wt{G}}(g^\ast) + (-\delta,\delta]^{\dim(\wt{G})}$. Given $x\in\wt{G}$ such that $\psi_{\mr{exp},\wt{G}}(x)\in\psi_{\mr{exp},\wt{G}}(g^\ast) + (-\delta,\delta]^{\dim(\wt{G})}$, we have that 
\begin{align*}
\psi_{\mr{exp},\wt{G}/\wt{G}_{>J}}(x \imod \wt{G}_{>J}) \in\psi_{\mr{exp},\wt{G}/\wt{G}_{>J}}(g^\ast \imod \wt{G}_{>J}) + (-\delta,\delta]^{\dim(\wt{G}/\wt{G}_{>J})} \cdot M^{O_{J,J'}(d^{O_{J,J'}(1)})},\\
\psi_{\mr{exp},\wt{G}/\wt{G}_{>J'}}(x \imod \wt{G}_{>J'}) \in\psi_{\mr{exp},\wt{G}/\wt{G}_{>J'}}(g^\ast \imod \wt{G}_{>J'}) + (-\delta,\delta]^{\dim(\wt{G}/\wt{G}_{>J'})} \cdot M^{O_{J,J'}(d^{O_{J,J'}(1)})}.
\end{align*}
Given that $\delta = M^{-O_{J,J'}(d^{O_{J,J'}(1)})}$ is sufficiently small, these are contained
\begin{align*}
\psi_{\mr{exp},\wt{G}/\wt{G}_{>J}}(g^\ast \imod \wt{G}_{>J}) + (-1/4,1/4]^{\dim(\wt{G}/\wt{G}_{>J})},\\
\psi_{\mr{exp},\wt{G}/\wt{G}_{>J'}}(g^\ast \imod \wt{G}_{>J'}) + (-1/4,1/4]^{\dim(\wt{G}/\wt{G}_{>J'})},
\end{align*}
respectively.

Identify $\psi_{\mr{exp},\wt{G}}(g^\ast) + (-1/2,1/2]^{\dim(G)}$ with the torus (note that the boundaries are glued differently than in $\wt{G}/\wt{\Gamma}$, but we are near the center so it is not an issue). We have that $\wt{F}$ is an $(M/\delta)^{O_{J,J'}(d^{O_{J,J'}(1)})}$-Lipschitz function with respect to the standard torus metric (see e.g.~\cite[Lemma~2.3]{LSS24} and \cite[Lemma~B.3]{Len23b}). Thus for $x\in\wt{G}$ such that $\psi_{\mr{exp},\wt{G}}(x)\in\psi_{\mr{exp},\wt{G}}(g^\ast) + (-1/2,1/2]^{\dim(G)}$, via standard Fourier approximation (see e.g.~\cite[Lemma~A.8]{PSS23}), for $\xi\in\mb{Z}^{\dim(\wt{G})}$ there exist $c_\xi$ with $|c_{\xi}|\le (M/(\delta\eps'))^{O(\dim(\wt{G}))}$ such that 
\[\norm{\wt{F}(x\wt{\Gamma}) - \sum_{\snorm{\xi}_{\infty}\le (M/(\delta\eps))^{O(\dim(\wt{G}))}} c_{\xi}e(\xi \cdot\psi_{\exp}(x))}_{\infty}\le\eps'\]
where the sum is over $\xi\in\mb{Z}^{\dim(\wt{G})}$. Using \eqref{eq:splitting-2} we may write this equivalently as
\[\norm{\wt{F}(x\wt{\Gamma}) - \sum_{\xi} c_{\xi}e(\xi \cdot (A \circ \psi_{\exp,\wt{G}_{>J}}(x \imod \wt{G}_{>J})))e(\xi \cdot (B \circ \psi_{\exp,\wt{G}_{>J'}}(x \imod \wt{G}_{>J'}))) }_{\infty}\le\eps',\]
where again the sum is over $\snorm{\xi}_\infty\le(M/(\delta\eps'))^{O(\dim(\wt{G}))}$.

For $z$ such that $\psi_{\mr{exp},\wt{G}/\wt{G}_{>J}}(z)-\psi_{\mr{exp},\wt{G}/\wt{G}_{>J}}(g^\ast \imod \wt{G}_{>J})\in(-1/2,1/2]^{\dim(\wt{G}/\wt{G}_{>J})}$, we let
\[\tau_{>J,\xi}(z) = \rho(\snorm{\psi_{\mr{exp},\wt{G}/\wt{G}_{>J}}(z)-\psi_{\mr{exp},\wt{G}/\wt{G}_{>J}}(g^\ast \imod \wt{G}_{>J})}) \cdot e(\xi \cdot (A \circ \psi_{\exp,\wt{G}_{>J}}(z)))\]
with $\rho(x) = 1$ for $|x|\le 1/4$, $\rho(x) = 0$ for $|x|\ge 1/3$, and $\rho$ is $O(1)$-Lipschitz and extends to $\wt{G}/(\wt{G}_{>J}\wt{\Gamma})$ via periodicity. $\tau_{>J}$ is seen to be an $(M/(\delta\eps'))^{O_{J,J'}(d^{O_{J,J'}(1)})}$-Lipschitz function on $\wt{G}/(\wt{G}_{>J}\wt{\Gamma})$. This follows via the size of $\xi$ and that distance in $d_{\wt{G}}$ controls distance in first-kind coordinates (see e.g.~\cite[Lemmas~B.1,~B.3]{Len23b}). Define $\tau_{>J',\xi}$ in the same manner. We have that
\[\norm{\wt{F}(x\wt{\Gamma}) - \sum_{\snorm{\xi}_{\infty}\le (M/(\delta\eps'))^{O(\dim(\wt{G}))}} c_{\xi}\tau_{>J,\xi}((x \imod \wt{G}_{>J}) \wt{\Gamma}_{>J})\tau_{>J,\xi}((x \imod \wt{G}_{>J'}) \wt{\Gamma}_{>J'})}_{\infty}\le\eps'.\]
As this holds for all $x$ such that $\psi_{\mr{exp},\wt{G}}(x)\in\psi_{\mr{exp},\wt{G}}(g^\ast) + (-1/2,1/2]^{\dim(G)}$ and the approximating function is invariant under $\wt{\Gamma}$, this holds for all $x\in\wt{G}$. This completes the proof, plugging in $x = \wt{g}(h_1,\ldots,h_k)$ and noting that $\wt{g}\imod\wt{G}_{>J}$ and $\wt{g}\imod\wt{G}_{>J'}$ are multidegree $J'$ and $J$ polynomial sequences on $\wt{G}/(\wt{G}_{>J}\wt{\Gamma})$ and $\wt{G}/(\wt{G}_{>J'}\wt{\Gamma})$ respectively.
\end{proof}

\bibliographystyle{amsplain1.bst}
\bibliography{main.bib}

\end{document}